\documentclass{amsart}
\usepackage[headings]{fullpage}
\usepackage{amsfonts,amssymb,amsthm,amsmath,euscript}
\usepackage{units}
\usepackage{subfigure}
\usepackage{graphicx}
\usepackage{tikz}
\usepackage{pinlabel}
\usepackage[colorlinks,citecolor=blue,linkcolor=red]{hyperref}
\input{xy}
\xyoption{all}

\newcommand{\Out}{\mathrm{Out}}
\newcommand{\Inn}{\mathrm{Inn}}

\newcommand{\Aut}{\mathrm{Aut}}

\newcommand{\abs}[1]{\left\vert#1\right\vert}
\newcommand{\norm}[1]{\left\Vert#1\right\Vert}
\newcommand{\inv}{^{-1}}

\newcommand{\MM}{{\mathbb M}}
\newcommand{\M}{{\mathcal M}}
\newcommand{\R}{{\mathbb R}}
\newcommand{\sone}{{\mathbb S}^1}
\newcommand{\Z}{{\mathbb Z}}
\newcommand{\Q}{{\mathbb Q}}
\newcommand{\C}{\EuScript{C}}
\newcommand{\A}{{\mathcal A}}

\newcommand{\D}{\EuScript{D}}
\newcommand{\Csec}{\EuScript{S}}
\newcommand{\XC}{X^{\EuScript{C}}}

\newcommand{\BNS}{\Sigma}
\newcommand{\QBNS}{\widehat{\mathbb{Q}\BNS}}

\newcommand{\Hom}{\mathrm {Hom}}
\newcommand{\f}{f} 
\newcommand{\fee}{\varphi} 
\newcommand{\fib}{\eta}    
\newcommand{\rk}{\mathrm{rk}}

\newcommand{\flow}{\psi}   
\newcommand{\tflow}{\widetilde \psi} 
\newcommand{\tpsi}{\widetilde \psi} 
\newcommand{\tX}{\widetilde X} 
\newcommand{\tTheta}{\widetilde \Theta}
\newcommand{\hTheta}{\widehat \Theta}
\newcommand{\hGamma}{\widehat \Gamma}
\newcommand{\tGamma}{\widetilde \Gamma}
\newcommand{\tV}{\widetilde{\mathcal V}}

\newcommand{\tOmega}{\widetilde{\Omega}}

\newcommand{\poly}{\mathfrak m} 

\newcommand{\Rose}{\mathcal R}

\theoremstyle{plain}
\newtheorem{theorem}{Theorem}[section]
\newtheorem{lemma}[theorem]{Lemma}
\newtheorem{corollary}[theorem]{Corollary}

\newtheorem{proposition}[theorem]{Proposition}
\newtheorem{prop-defn}[theorem]{Proposition-Definition}

\newtheorem{claim}[theorem]{Claim}
\newtheorem{theor}{Theorem} 
\newtheorem*{theorem:Lipschitzflows}{Theorem \ref{T:lipschtiz flows}} 
\newtheorem*{theorem:conesections}{Theorem \ref{T:cone_of_sections}}
\newtheorem*{theorem:splittings}{Theorem \ref{T:splittings}}
\newtheorem*{theorem:determinant}{Theorem \ref{T:determinant formula}}
\newtheorem*{theorem:continuity-convexity}{Theorem \ref{T:continuity/convexity again}}
\newtheorem*{theorem:cones-are-equal}{Theorem \ref{T:cones_are_equal}}
\newtheorem*{theorem:cone-is-BNScpt}{Theorem \ref{T:BNS}}
\newtheorem*{theorem:transversefoliations}{Theorem \ref{T:transverse foliations}}
\theoremstyle{definition}
\newtheorem{defn}[theorem]{Definition}
\newtheorem{remark}[theorem]{Remark}
\newtheorem{question}[theorem]{Question}
\newtheorem*{remark*}{Remark}
\newtheorem*{remarks*}{Remarks}
\newtheorem{conv}[theorem]{Convention}
\newtheorem{example}[theorem]{Example}

\makeatletter
\let\c@equation\c@theorem
\makeatother
\numberwithin{equation}{section}

\begin{document}

\title[McMullen polynomials and Lipschitz flows for free-by-cyclic groups]{McMullen polynomials and Lipschitz flows for\\free-by-cyclic groups}
\author{Spencer Dowdall, Ilya Kapovich, and Christopher J. Leininger}

\address{\tt  Department of Mathematics, Vanderbilt University,
1326 Stevenson Center, Nashville, TN 37240, U.S.A\\
  \newline \indent http://www.math.vanderbilt.edu/\~{}dowdalsd/ } \email{\tt spencer.dowdall@vanderbilt.edu}

\address{\tt  Department of Mathematics, University of Illinois at Urbana-Champaign,
  1409 West Green Street, Urbana, IL 61801
  \newline \indent http://www.math.uiuc.edu/\~{}kapovich/} \email{\tt kapovich@math.uiuc.edu}

\address{\tt  Department of Mathematics, University of Illinois at Urbana-Champaign,
  1409 West Green Street, Urbana, IL 61801
  \newline \indent http://www.math.uiuc.edu/\~{}clein/} \email{\tt clein@math.uiuc.edu}
  
\begin{abstract}
Consider a group $G$ and an epimorphism $u_0\colon G\to \Z$ inducing a splitting of $G$ as a semidirect product $\ker(u_0)\rtimes_\varphi \Z$ with $\ker(u_0)$ a finitely generated free group and $\varphi\in \Out(\ker(u_0))$ representable by an expanding irreducible train track map. Building on our earlier work \cite{DKL}, in which we realized $G$ as $\pi_1(X)$ for an Eilenberg-Maclane $2$--complex $X$ equipped with a semiflow $\flow$, and inspired by McMullen's Teichm\"uller polynomial for fibered hyperbolic $3$--manifolds, we construct a polynomial invariant $\poly \in \Z[H_1(G;\Z)/\text{torsion}]$ for $(X,\flow)$ and investigate its properties.

Specifically, $\poly$ determines a convex polyhedral cone $\C_X\subset H^1(G;\R)$, a convex, real-analytic function $\mathfrak H\colon \C_X\to \R$, and \emph{specializes} to give an integral Laurent polynomial $\poly_u(\zeta)$ for each integral $u\in \C_X$.
We show that $\C_X$ is equal to the ``cone of sections'' of $(X,\flow)$ (the convex hull of all cohomology classes dual to sections of of $\flow$), and that for each (compatible) cross section $\Theta_u\subset X$ with first return map $f_u\colon \Theta_u\to \Theta_u$, the specialization $\poly_u(\zeta)$ encodes the characteristic polynomial of the transition matrix of $f_u$.   More generally, for \emph{every} class $u\in \C_X$ there exists a geodesic metric $d_u$ and a codimension--$1$ foliation $\Omega_u$ of $X$ defined by a ``closed $1$--form'' representing $u$ transverse to $\flow$ so that after reparametrizing the flow $\flow^u_{s}$ maps leaves of $\Omega_u$ to leaves via a local $e^{s\mathfrak H(u)}$--homothety.

Among other things, we additionally prove that $\C_X$ is equal to (the cone over) the component of the BNS-invariant $\BNS(G)$ containing $u_0$ and, consequently, that each primitive integral $u\in \C_X$ induces a splitting of $G$ as an ascending HNN-extension $G = Q_u\ast_{\phi_u}$ with $Q_u$ a finite-rank free group and $\phi_u\colon Q_u\to Q_u$ injective. For any such splitting, we show that the stretch factor of $\phi_u$ is exactly given by $e^{\mathfrak H(u)}$.  
In particular, we see that $\C_X$ and $\mathfrak H$ depend only on the group $G$ and epimorphism $u_0$.
\end{abstract}  
  

\thanks{The first author was partially supported by the NSF
  postdoctoral fellowship, NSF MSPRF no. 1204814. The second author
  was partially supported by the NSF grant DMS-0904200, DMS-1405146 and by
  the Simons Foundation Collaboration grant no. 279836. 
  The third author was partially supported by the NSF grant DMS-1207183 and acknowledges support from NSF grants DMS 1107452, 1107263, 1107367 (the ``GEAR Network'').
}


\maketitle


\vspace{-5pt}
\tableofcontents

\section{Introduction}

Given an outer automorphism $\varphi$ of the finite-rank free group $F_N$, consider the free-by-cyclic group
\[ G = G_\fee = F_N \rtimes_\varphi \Z = \langle F_N, r \mid r^{-1} w r = \Phi(w), w \in F_N \rangle \]
(here $\Phi \in \Aut(F_N)$ is any representative of $\varphi$). The epimorphisms $G\to \Z$ are exactly the primitive integral points of the vector space $H^1(G;\R) = \Hom(G,\R)$, and every such $u\colon G\to \Z$ gives rise to a \emph{splitting}
\[ \xymatrix{ 1  \ar[r] &  \ker(u) \ar[r] & G \ar[r]^{u} & \Z \ar[r] & 1}\]
of $G$ and an associated \emph{monodromy} $\fee_u\in \Out(\ker(u))$ generating the outer action of $\Z$ on $\ker(u)$.  We are most interested in the case that $\ker(u)$ is finitely generated, or more generally, finitely generated over the monoid generated by $\fee_u$. In this case there is a canonically defined stretch factor $\Lambda(u)$ measuring the exponential growth rate of $\fee_u$ acting on $\ker(u)$; see \S\ref{sec:Setup} for the precise definitions of $\fee_u$ and $\Lambda(u)$.

We aim to illuminate the dynamical and topological properties of splittings of $G_\varphi$ in the case that $\varphi$ may be represented by an expanding irreducible train-track map $f$. For example, one broad goal is to understand how  $\fee_u$ and $\Lambda(u)$ vary with $u$. We began this undertaking in \cite{DKL} where, inspired by the work of Thurston and Fried on hyperbolic $3$--manifolds and their fibrations over the circle, we defined a $K(G,1)$ space called the {\em folded mapping torus} $X = X_f$ that comes equipped with a semi-flow $\flow$.
This determines an open convex cone $\A=\A_f \subset H^1(X;\R)= \Hom(G,\R)$ containing the natural projection $u_0\colon G\to G/F_N\cong \Z$ \emph{associated} with the given splitting $G = F_N\rtimes_\fee \Z$.  The cone $\A$ has the property that every primitive integral $u\in\A$ has $\ker(u)$ a finitely generated free group and monodromy $\fee_u$ which is again represented by an expanding irreducible train-track map $f_u$. Moreover, if the original automorphism $\fee \in \Out(F_N)$ is hyperbolic and fully irreducible, then so is $\fee_u$ for every primitive integral $u\in \A$.  We also proved in~\cite{DKL} that there exists a convex, continuous, homogeneous of degree $-1$ function $\mathfrak H\colon \A\to \R_+$ such that $\mathfrak H(u)$ is equal to the topological entropy of $f_u$ and also to $\log(\Lambda(u))$ for all primitive integral $u$. 
We refer the reader to \cite{DKL} for a detailed discussion of this material and other related results such as those in \cite{AR, BStrian,Gau1,Gau3,Gau4,Wang}.

Comparing our results on $G$ and $(X,\flow)$ from \cite{DKL} with those of a fibered hyperbolic $3$--manifold leads to several additional questions. The capstone of the $3$--manifold picture is McMullen's \emph{Teichm\"uller polynomial} \cite{Mc} which encapsulates nearly all the relevant information regarding fibrations of the manifold. In particular, it detects a canonically defined cone---which has topological, dynamical and algebraic significance---and gives specific information about every class in this cone. 
It is thus tantalizing to ask: Is there an analogous polynomial in the free-by-cyclic setting and, more importantly, what does it tell you about $G$ and $(X,\flow)$?
Any such analog would naturally determine a cone in $H^1(X;\R)$ and a convex, real-analytic function on that cone. What is the relationship of these to the cone $\A$ and function $\mathfrak H$ found in \cite{DKL} or to the canonically defined stretch function $\Lambda$? Is there a geometric, algebraic or dynamical characterization of the cone (e.g., in terms of foliations of $X$, Fried's cone of homology directions \cite{FriedS}, or the BNS-invariant)? What, if any, topological, geometric, or dynamical structure is there associated to the \emph{irrational} points of the cone?

The centerpiece of the present paper is the construction of exactly such a polynomial $\poly$, which we call the {\em McMullen polynomial}, together with a detailed analysis of $\poly$ that answers all of the above questions. Our multivariate, integral Laurent polynomial $\poly$ is constructed as an element of the integral group ring $\Z[H]$, where $H = H_1(G;\R)/\text{torsion}\cong \Z^b$, for $b = b_1(G)$.  As such, $\poly$ has a natural {\em specialization} $\poly_u$ for every $u \in H^1(X;\R)$; this is a single variable Laurent polynomial when $u$ is an integral class and a ``power sum'' in general. We now briefly summarize the the main results of the paper and explain the information packaged in $\poly$. Over the following several pages, we then provide detailed and expanded statements of our results together with discussion of related and other recent results in the literature and a more in depth comparison with the motivating $3$--manifold setting.

\begin{conv}\label{conv:main}
For the entirety of this paper (except in \S\ref{S:preliminaries}), $G=G_\fee$ will denote the free-by-cyclic group associated to an outer automorphism $\varphi\in \Out(F_N)$ represented by an expanding irreducible train track map $f \colon \Gamma \to \Gamma$ as above, and $X=X_f$ will denote the folded mapping torus built from $f$.
\end{conv}

\pagebreak

\noindent \underline{\underline{\bf Meta-Theorem}}

\smallskip

{\em 
\noindent {\bf \underline{I. Canonical cone}} There is an open, convex, rationally defined, polyhedral cone $\C_X \subset H^1(G;\R) = H^1(X;\R)$ containing $\A$, which can be characterized by any of the following:
\begin{enumerate}
\item[(1)] The dual cone of a distinguished vertex of the Newton polytope of $\poly$ (Theorem~\ref{T:cones_are_equal});
\item[(2)] The open convex hull of rays through integral $u \in H^1(X;\R)$ dual to sections of $\flow$ (Proposition~\ref{P:Csec convex});
\item[(3)] $\{ u \in H^1(X;\R) \mid$ $u$ is represented by a ``flow-regular closed $1$--form'' $\}$ (Definition~\ref{D:cone_sections});
\item[(4)] $\{u \in H^1(X;\R) \mid $ $u$ is positive on all closed orbits of $\flow$ $\}$ (Definition~\ref{D:Fried_cone});
\item[(5)] $\{u \in H^1(X;\R) \mid $ $u$ is positive on a specific finite set of closed orbits of $\psi$ $\}$ (Theorem~\ref{T:cone_of_sections});
\item[(6)] $\{ u \in H^1(X;\R) \mid $ $u$ is represented by a cellular $1$--cocycle which is positive on a specific set of $1$--cells of the ``circuitry cell structure'' of $X$ $\}$ (Proposition~\ref{P:homological_characterization});
\item[(7)] The cone on the component of the BNS--invariant containing $u_0$ (Theorem~\ref{T:BNS}).
\end{enumerate}}

\medskip

Item (1) shows that the cone $\C_X$ is determined by the McMullen polynomial $\poly$. Given $u \in \C_X$ primitive integral, the section $\Theta_u \subset X$ of $\flow$ dual to $u$ provided by (2) has first return map we denote $f_u \colon \Theta_u \to \Theta_u$.  By (7) we can realize $G$ (not necessarily uniquely) as an ascending HNN-extension $G = Q_u \ast_{\phi_u}$ dual to $u$ over a finitely generated free group $Q_u$.  For any class $u \in \C_X$ (not necessarily integral), the flow-regular closed $1$--form provided by (3) determines a ``foliation''  $\Omega_u$ of $X$.

\medskip

{\em 
\noindent {\bf \underline{II. Specialization: entropy and stretch factors}} There is a convex, real-analytic function $\mathfrak H \colon \C_X \to \R$ that is homogeneous of degree $-1$, extends the previously defined function on $\A$, and tends to $\infty$ at $\partial \C_X$ such that for every primitive integral $u\in \C_X$:
\begin{enumerate}
\item[(1)] $e^{\mathfrak H(u)}>1$ is equal to the largest root of $\poly_u(\zeta)$.
\item[(2)] $f_u \colon \Theta_u \to \Theta_u$ is an expanding irreducible train track map ``weakly representing'' $\phi_u$.
\item[(3)] The characteristic polynomial of the transition matrix for $f_u$ is $\pm \zeta^k \poly_u(\zeta)$, for some $k \in \Z$.
\item[(4)] The entropy of $f_u$ is given by $h(f_u) = \mathfrak H(u)$;
\item[(5)] The stretch factor of $\phi_u$ is $\Lambda(u) = \lambda(\phi_u) = e^{\mathfrak H(u)}$.
\end{enumerate}

\medskip

\noindent
{\bf \underline{III. Determinant formula}} For any primitive integral $u \in H^1(X;\R)$, the first return map $f_u \colon \Theta_u \to \Theta_u$ determines a matrix $A_u({\bf t}) = A_u(t_1,\ldots,t_{b-1})$ with entries in the ring of Laurent polynomials $\Z[t_1^{\pm 1},\ldots,t_{b-1}^{\pm 1}]$, for which $A_u(1,\ldots,1)$ is the transition matrix for $f_u$ and whose characteristic polynomial is $\poly$:
\[ \poly(x,{\bf t}) = \det(xI - A_u({\bf t}))\in\Z[x,{\bf t}] \cong \Z[H]. \]

\medskip

\noindent
{\bf \underline{IV. Foliations and Flows}} From the construction of $\poly$, for any class $u \in \C_X$ (not necessarily integral), there is a geodesic-metric $d_u$ on $X$ and a reparameterization $\flow^u$ of $\flow$ so that
\begin{enumerate}
\item[(1)] The flow lines $s \mapsto \flow_s^u(\xi)$ are $d_u$--geodesics for all $\xi \in X$;
\item[(2)] The ``leaves'' of $\Omega_u$ are (possibly infinite) graphs.
\item[(2)] $\flow^u$ maps leaves to leaves;
\item[(3)] The restriction of $\flow^u_s$ to any leaf is a local $e^{s {\mathfrak H(u)}}$--homothety with respect to the induced path metric.
\end{enumerate}
}

\medskip

\noindent
{\bf Some remarks:}
\begin{itemize}
\item We are indebted to Nathan Dunfield who suggested the scheme to prove Meta-Theorem I(7).
\item We consider left actions instead of right actions, and so the BNS--invariant discussed here agrees with $\Sigma^1(G;\Z)$ from \cite{BieriRenz}, but is $-\Sigma_{G'}(G)$ from \cite{BNS}; see Section 1.3 of \cite{BieriRenz}.
\item Despite the ostensible dependence on the dynamical system $(X,\flow)$, Meta-Theorem I--II shows that $\C_X$ and $\mathfrak H$ in fact depend only on the group $G$ and the initial cohomology class $u_0\colon G\to \Z$. 
\item We record that Theorem \ref{T:cone_of_sections}, Proposition~\ref{P:integral_classesin_Csec}, and Proposition \ref{P:homological_characterization} give the precise meaning to, and prove the equivalence of, (2)--(6) in Meta-Theorem I. Meanwhile Theorems~\ref{T:cones_are_equal} and \ref{T:BNS} prove the equivalence of these with (1) and (7). Meta-Theorem II follows from Theorems~\ref{T:splittings}, \ref{T:main_theorem}, and \ref{T:continuity/convexity again}. Meta-Theorem III is just Theorem~\ref{T:determinant formula}, and Meta-Theorem IV is an abbreviated version of Theorem~\ref{T:lipschtiz flows}.
\end{itemize}

\begin{remark}
Contemporaneously with the release of this paper, Algom-Kfir, Hironaka, and Rafi \cite{AHR} independently introduced a related polynomial $\Xi\in \Z[H]$ that provides information about $G$ and $X$. We stress that the results of the present paper (outlined above) have very little overlap with those of \cite{AHR}. Specifically, the polynomial $\Xi$ similarly determines a cone $\mathcal{T}\subset H^1(G;\R)$ and a convex, real-analytic function $L\colon \mathcal{T}\to \R$. Algom-Kfir, Hironaka, and Rafi prove that $\A$ (the cone defined in \cite{DKL}) is contained $\mathcal{T}$ and that $L(u) = \log(\Lambda(u))$ for every integral class $u\in \A$.  By real-analyticity, this implies that $L=\mathfrak H$ and that $\mathcal{T} = \C_X$.  Rather than study classes $u\in \mathcal{T}\setminus \A$ directly, the analysis in \cite{AHR} focuses on the construction of alternate ``$\A$--cones'' inside $\mathcal{T}$, and also produces an interesting interpretation of $\Xi$ in terms of a ``cycle polynomial" (which we learned from the authors of \cite{AHR}, and employ in the proof of Theorem~\ref{T:cones_are_equal} below). 
However, \cite{AHR} provides no geometric, algebraic, or dynamical interpretation of the significance of the cone $\mathcal{T}$ (c.f., Meta-Theorem I), nor for the value $L(u)$ when $u$ is not in an $\A$--cone (c.f. Meta-Theorem II); note that Example~\ref{Ex:the_cones} shows that $\C_X$ is not a union of $\A$--cones, in general. Furthermore, \cite{AHR} does not provide any discussion of, nor structure for, the irrational points of $\C_X$ (c.f. Meta-Theorem IV).

We additionally remark that in general the polynomial $\Xi$ from \cite{AHR} is not equal to $\poly$. In fact, $\Xi$ has the advantage of depending only on the pair $(G,u_0)$ and not on the folded mapping torus $(X,\flow)$.  The construction of \cite{AHR} furthermore ensures that $\Xi$ is a factor of $\poly$. However, $\Xi$ consequently cannot in general compute the characteristic polynomials for first return maps of cross sections to $\flow$ (c.f. Meta-Theorem II(3)). 
\end{remark}

We now turn to a more detailed discussion of the main results and ideas in this paper.\\

\subsection{Cross sections}
Some of the central objects of our current investigation are the \emph{cross sections} of $\flow$; these are embedded graphs in $X$ which are transverse to the flow and intersect every flowline of $\flow$; see \S\ref{S:cross_sections} for details. Every section $\Theta \subset X$ determines a dual cohomology class $u = [\Theta]$, and we construct an open convex cone $\Csec\subset H^1(X;\R)$ containing $\A$ with the property that an integral element $u \in H^1(X;\R)$ is dual to a section if and only if $u \in \Csec$; see Proposition~\ref{P:integral_classesin_Csec}. This \emph{cone of sections} $\Csec$ has a convenient definition in terms of closed $1$--forms (in the sense of \cite{FGS10}) satisfying an additional property we call {\em flow-regularity}; see \S\ref{S:closed 1-forms cohomology}.  
To aid in our analysis of $\Csec$, in \S\ref{S:cone_of_sections} we introduce another cone, the \emph{Fried cone} $\D\subset H^1(X;\R)$, consisting of those cohomology classes which are positive on all closed orbits of $\flow$; in fact, $\D$ is determined by finitely many closed orbits (Proposition~\ref{P:Fried finiteness}). The definition of $\Csec$ in terms of flow-regular closed $1$--forms easily implies that $\Csec\subseteq \D$. In Appendix~\ref{S:characterizing_sections} we give a combinatorial characterization of $\D$ (Proposition \ref{P:homological_characterization}) that is similar in spirit to the definition of $\A$ and which allows us to show the reverse containment $\D\subseteq \Csec$. Our first theorem summarizes these properties of $\Csec$ and $\D$:

\begin{theor}[Cone of sections]
\label{T:cone_of_sections}
There is an open convex cone $\Csec\subset H^1(X;\R)=H^1(G;\R)$ containing $\A$ (and thus containing $u_0$) such that a primitive integral class $u\in H^1(X_f,\R)$ is dual to a section of $\flow$ if and only if $u \in \Csec$. Moreover, $\Csec$ is equal to to the Fried cone $\D$, and there exist finitely many closed orbits $\mathcal{O}_1,\dotsc,\mathcal{O}_k$ of $\flow$ such that $u\in H^1(X;\R)$ lies in $\Csec$ if and only if $u(\mathcal{O}_i) > 0$ for each $1\leq i \leq k$.  In particular, $\Csec$ is an open, convex, polyhedral cone with finitely many rationally defined sides.
\end{theor}

\begin{remark}
We emphasize that the closed orbits $\mathcal{O}_i$ appearing in Theorem~\ref{T:cone_of_sections} are \emph{explicitly} defined in terms of the data used to construct $X$; see \S\ref{S:cone_of_sections}.
\end{remark}

This theorem mirrors Thurston's result in the $3$--manifold setting, where the cone of sections is the cone over a face of the (polyhedral) Thurston norm ball \cite{ThuN} and hence is defined by finitely many rational inequalities, and also Fried's characterization of this cone in terms of ``homology directions'' \cite{FriedS}.  In his unpublished thesis \cite{Wang}, Wang uses Fried's notion of homology directions to provide a similar characterization of the cone of sections for the mapping torus of a free group automorphism. Since there are semiflow equivariant maps between the mapping torus and $X$, Theorem~\ref{T:cone_of_sections} implies Wang's result in his setting.

For every primitive integral class $u\in \Csec$ we henceforth use $\Theta_u$ to denote a cross section dual to $u$; see Convention~\ref{conv:notation_convention}. For technical reasons we impose an additional assumption on the section $\Theta_u$ we call ``$\mathcal F$--compatibility'', but as we show in \S\ref{S:constructing_sections}, every primitive integral $u \in \Csec$ is dual to such a section. Every cross section $\Theta\subset X$ has an associated \emph{first return map} $\Theta\to \Theta$ induced by the semiflow, and we will use $f_u\colon\Theta_u\to\Theta_u$ to denote the first return map of $\Theta_u$.

\subsection{Splittings}
We will see that, in general, the sections $\Theta_u$ dual to primitive integral $u\in \Csec$ are not as nicely behaved as those dual to elements of $\A$. In particular, the inclusion  $\Theta_u\hookrightarrow X$ need not be $\pi_1$--injective and $f_u$ need not be a homotopy equivalence. Despite these shortcomings, our next theorem shows that for each section $\Theta_u$, the induced endomorphism $(f_u)_*$ of $\pi_1(\Theta_u)$ nevertheless provides a great deal of information about the corresponding splitting of $G$; namely it gives a realization of this splitting as an ascending HNN-extension over a finitely generated free group and provides a way to calculate $\Lambda(u)$.

For the statement, we need the following terminology regarding arbitrary endomorphisms $\phi$ of a finitely generated free group $F$:
As we explain in \S\ref{S:HNN-like}, each such $\phi$ naturally descends to an injective endomorphism $\bar\phi$ of the \emph{stable quotient} of $F$ by the \emph{stable kernel} consisting of all elements that become trivial under some power of $\phi$. 
Additionally, each endomorphism $\phi$ has as \emph{growth rate} or \emph{stretch factor} defined as
\begin{equation*}
\lambda(\phi) = \sup_{g \in F} \liminf_{n \to \infty} \sqrt[n]{\norm{\phi^n(g)}_A},
\end{equation*}
where $A$ is any free basis of $F$ and $\norm{\phi^n(g)}_A$ denotes the cyclically reduced word length of $\phi^n(g)$ with respect to $A$. One can show that this definition is independent of the the free basis $A$; see \S\ref{S:growth} for details. With this terminology, we prove that the first return maps $f_u\colon \Theta_u\to\Theta_u$ provide the following information about the splittings of $G$ determined by primitive integral classes $u\in \Csec$.

\begin{theor}[Splittings and ascending HNN-extensions]
\label{T:splittings}
Let $u\in \Csec$ be a primitive integral class with $\mathcal F$--compatible dual section $\Theta_u \subset X$ and first return map $f_u \colon \Theta_u \to\Theta_u$. Let $Q_u$ be the stable quotient of $(f_u)_*$ and let $\phi_u = \overline{(f_u)}_*$ be the induced endomorphism of $Q_u$. Then
\begin{enumerate}
\item $f_u$ is an expanding irreducible train track map.
\item $Q_u$ is a finitely generated free group and $\phi_u\colon Q_u\to Q_u$ is injective.
\item $G$ may be written as an HNN-extension
\[G \cong Q_u \ast_{\phi_u} = \langle Q_u, r \mid r\inv q r = \phi_u(q) \text{ for all }q\in Q_u\rangle\]
such that $u\colon G\to \Z$ is given by the assignment $r\mapsto 1$ and $Q_u\mapsto 0$.
\item If $J_u\leq \ker(u)$ denotes the image of $\pi_1(\Theta_u)$ induced by the inclusion $\Theta_u\hookrightarrow X$, then there is an isomorphism $Q_u\to J_u$ conjugating $\phi_u$ to $\Phi_u\vert_{J_u}$ for some automorphism $\Phi_u\in\Aut(\ker(u))$ representing the monodromy $\fee_u$.
\item The topological entropy of $f_u$ is equal to $\log(\lambda(\phi_u)) = \log(\lambda(\Phi_u\vert_{J_u}))$ and also to $\log(\Lambda(u))$.
\item $\ker(u)$ is finitely generated if and only if $\phi_u$ is an automorphism, in which case we have $\ker(u) \cong Q_u$ and $\fee_u = [\phi_u] \in \Out(Q_u)$.
\end{enumerate}
\end{theor}

\begin{remark}
Throughout this paper -- for example in Theorem~\ref{T:splittings}(1) -- we use ``train track map'' to mean a graph map that satisfies the usual dynamical properties of train track maps in $\Out(F_N)$ theory but which may not be a homotopy equivalence; see \S\ref{S:graphs_and_maps} for the precise definition.
\end{remark}

As noted above, the inclusion $\Theta_u \subset X$ is not necessarily $\pi_1$--injective and $f_u$ may fail to be a homotopy equivalence in two different ways. First, it can happen that  $(f_u)_* \colon \pi_1(\Theta_u) \to\pi_1(\Theta_u)$ may be injective but non-surjective, in which case $Q_u = \pi_1(\Theta_u)$, $\phi_u = (f_u)_*$ and $\ker(u)$ is not finitely generated even though $u\in\Csec$. Second, it can happen that $(f_u)_*$ is non-injective, in which case the Hopficity of the free group $\pi_1(\Theta_u)$ necessitates that $(f_u)_*$ non-surjective as well.  We will see (in Remark~\ref{R:S_has_both_kinds_of_kerels}) that both of these possibilities can in fact occur when $u \in \Csec\setminus\A$. In the second case it can moreover happen that the injective endomorphism $\phi_u \colon Q_u\to Q_u$ is surjective and consequently an automorphism of $Q_u\cong \ker(u)$. These kinds of phenomena are not present in the $3$--manifold setting; see \S\ref{S:contrasting_3-manifolds} for a more detailed discussion.

\subsection{The McMullen polynomial}
Having shown that a section $\Theta_u$ dual to $u\in \Csec$ leads to an algebraic description of the corresponding splitting of $G$, we now turn our attention to the dynamical properties of these splittings. In this regard, our main result is the introduction of a polynomial invariant $\poly$, termed  the \emph{McMullen polynomial}, that simultaneously encodes information about all of the first return maps $f_u$ and the stretch factors of their associated endomorphisms $(f_u)_*$. This result is analogous to McMullen's construction of the Teichm\"uller polynomial in the setting of fibered hyperbolic $3$--manifolds \cite{Mc}. Let
\[ H = H_1(G;\Z)/\text{torsion} \cong \Z^b,\]
where $b = b_1(G)$ is the first Betti number of $G$. The McMullen polynomial of $(X,\flow)$ is a certain element $\poly$ in the group ring $\Z[H]$, that is, it is a finite sum
\[\poly = \sum_{h\in H} a_h \, h \in \Z[H]\]
of elements of $H$ with integral coefficients. Note that if we expresses the elements $h\in H$ in terms of a multiplicative basis $t_1,\ldots,t_b$, then $\poly = \poly(t_1,\ldots,t_b)$ becomes an integral Laurent polynomial in $t_1,\ldots,t_b$.

For every integral class $u\in H^1(G;\Z)$, one may then form the \emph{specialization of $\poly$ at $u$}, which is the $1$--variable integral Laurent polynomial
\[\poly_u(\zeta) = \sum_{h\in H} a_h \zeta^{u(h)} \in \Z[\zeta^{\pm1}].\]
Our next theorem establishes an intimate relationship between these specializations $\poly_u$ and the first return maps $f_u$ of sections $\Theta_u$ dual to $u\in \Csec$.

\begin{theor}[The McMullen polynomial and its specializations]
\label{T:main_theorem}
There exists an element $\poly\in \Z[H]$ as above with the
following properties. Let $u\in \Csec\subset H^1(X;\R)$ be a primitive
integral class with dual compatible section $\Theta_u$, first return map $f_u$, and injective endomorphism $\phi_u = \overline{(f_u)}_*$ as in Theorem~\ref{T:splittings}. Then
\begin{enumerate}
\item The specialization $\poly_u(\zeta)\in \Z[\zeta^{\pm1}]$ is, up to multiplication by $\pm \zeta^k$ for some $k \in \Z$, equal to the characteristic polynomial of the transition
matrix of $f_u$.  
\item The largest positive root $\lambda_u$ of $\poly_u$ is equal to the stretch factor $\lambda(\phi_u)$ of $\phi_u$, to the spectral radius of the transition matrix of $f_u$, and to the stretch factor of $(f_u)_\ast$. In particular, $\lambda_u = \Lambda(u)$. Additionally, $\log(\lambda_u)$ is the topological entropy $f_u$.
\end{enumerate}
\end{theor}

\begin{remark}
Using different methods, Algom-Kfir, Hironaka and Rafi \cite{AHR} have independently produced a polynomial and have obtained a similar result to Theorem~\ref{T:main_theorem}, but only for primitive integral elements $u \in \A$.  Their polynomial has an additional minimality property which guarantees that it is a factor of $\poly$ and that it depends only on the outer automorphism $f_\ast$ represented by $f\colon \Gamma\to\Gamma$. Because of this last property, their polynomial need not specialize to the characteristic polynomial of the train track map $f_u$ for primitive integral $u \in \A$ (even up to units; c.f.~Theorem~\ref{T:determinant-subdivision} and Example \ref{Ex:subdivision matters}).
\end{remark}

The McMullen polynomial $\poly$ in Theorem~\ref{T:main_theorem} is constructed in the following manner, which is in some sense ``dual'' to McMullen's construction of the Teichm\"uller polynomial \cite{Mc}.   Let $\tX$ denote the universal torsion-free abelian cover of $X$ (which has deck group $H$) and consider the foliation $\mathcal F$ of $\tX$ by (lifted) flowlines of $\flow$. In \S\ref{S:foliation} we then construct a \emph{module of transversals to $\mathcal F$}; this is a module over the group ring $\Z[H]$ and is denoted by $T(\mathcal F)$. A key fact regarding this module, which we prove in \S\ref{secc:module_isom}, is that $T(\mathcal F)$ is finitely presented as a $\Z[H]$--module. The McMullen polynomial of $(X,\flow)$ is then defined to be the g.c.d.~of the Fitting ideal of $T(\mathcal F)$; see \S\ref{S:McMullen_polynomial} for details.

This abstract definition of $\poly$ is, however, rather opaque and thus somewhat unsatisfying. In particular, it gives no indication as to why $\poly$ should enjoy the properties described in Theorem~\ref{T:main_theorem}. Ultimately, these properties follow from the fact that $T(\mathcal F)$ encodes a great deal of information about the semiflow $\flow$, but it is not readily apparent how this information is imparted to the fitting ideal and consequently to $\poly$.

To remedy this situation we give an alternate description of the McMullen polynomial, showing that $\poly$ is a very concrete and explicitly computable object. This alternate description is in terms of \emph{graph modules} for sections of $\flow$. More precisely, for each primitive integral class $u \in \Csec$ with dual compatible section $\Theta_u$, we consider the preimage $\hTheta_u \subset \tX$ in the universal torsion-free abelian cover and define a corresponding \emph{graph module} $T(\hTheta_u)$; this module is a certain quotient of the free module on the edges of $\hTheta_u$ that encodes information about the first return map of $\hTheta_u$ to itself. 

The module $T(\hTheta_u)$ may be explicitly described as follows: Choosing a component $\tTheta_u \subset \hTheta_u$, the stabilizer of $\tTheta_u$ in $H$ is a rank $b-1$ subgroup $H_u \subset H$, and $H$ splits (noncanonically) as $H \cong H_u \oplus \Z$. There is a canonical submodule of the free module of the edges of $\hTheta_u$ consisting of finite sums of edges of $\tTheta_u$, and if we choose a finite set $E$ of $H_u$--orbit representatives of the edges of $\tTheta_u$, then this submodule is naturally isomorphic to the {\em finitely generated} free $\Z[H_u]$--module $\Z[H_u]^E$. Choosing a multiplicative basis $t_1,\ldots,t_{b-1}$ for $H_u$, this latter module may be regarded as the finitely generated free module $\Z[t_1^{\pm1},\dotsb,t_{b-1}^{\pm1}]^E$ over the ring of integer Laurent polynomials in $b-1$ variables.  The first return map $f_u \colon \Theta_u \to \Theta_u$ then lifts to a map $\tTheta_u \to \tTheta_u$ which has a well-defined $\abs{E}\times\abs{E}$ ``transition matrix'' $A_u(t_1,\ldots,t_{b-1})$ with entries in $\Z[t_1^{\pm 1},\ldots,t_{b-1}^{\pm 1}]$; see \S\ref{S:transition_matrices}.

Choosing a multiplicative generator $x$ for the complement of $H_u < H$ produces an isomorphism between  $\Z[H]$ and the ring $\Z[t_1^{\pm 1},\ldots,t_{b-1}^{\pm 1},x^{\pm 1}]$ of integral Laurent polynomials.  With respect to this isomorphism, the McMullen polynomial $\poly$ is given by the following ``determinant formula'', which is analogous to McMullen's formula for the Teichm\"uller polynomial in \cite{Mc} and which implies Meta-Theorem III.

\begin{theor}[Determinant formula] \label{T:determinant formula}
For any primitive integral $u \in \Csec$ we have
\[ \poly(t_1,\ldots,t_{b-1},x) = \det(xI - A_u(t_1,\ldots,t_{b-1}))\]
up to units in $\Z[t_1^{\pm 1},\ldots,t_{b-1}^{\pm 1},x^{\pm 1}]$.
\end{theor}

While the compatibility condition imposed on the cross sections $\Theta_u$ can be computationally expensive, as it can lead to the addition of a large number of valence--$2$ vertices (see Definition~\ref{D:standard_graph}), it is essential for the above discussion because the module of transversals $T(\mathcal F)$, and consequently $\poly$, is sensitive to the choice of distinguished \emph{vertex leaves} of $\mathcal F$. These leaves are determined by the original graph structure on $\Gamma$, and subdividing $\Gamma$ can indeed affect the McMullen polynomial by introducing extra factors. To illuminate this dependence and increase computational flexibility, we derive a secondary determinant formula (Theorem~\ref{T:determinant-subdivision}) that explains exactly how the McMullen polynomial $\poly$ changes under the addition of new vertex leaves resulting from any subdivision of $\Gamma$ or of any cross section $\Theta_u$.

\subsection{The McMullen cone $\C_X$}
The McMullen polynomial $\poly\in \Z[H]$ naturally determines yet another convex cone $\C_X\subset H^1(X;\R)$, which we term the \emph{McMullen cone}. Specifically, $\C_X$ is the dual cone associated to a certain vertex of the ``Newton polytope'' of $\poly\in \Z[H]$; see Definition~\ref{D:McMcone}.
The next result states that $\C_X$ is in fact equal to the cone of sections $\Csec$.
The proof of this result appeals to Theorems \ref{T:cone_of_sections} and \ref{T:main_theorem} and a formula for $\poly$ as a ``cycle polynomial'' introduced by Algom-Kfir, Hironaka and Rafi \cite{AHR}.  The proof of this formula involves the use of directed graphs labeled by homology classes, which also appears in the work of Hadari \cite{HadariComm}, and is related to McMullen's recent work on the ``clique polynomial'' of a weighted directed graph \cite{McClique}.  Thus in addition to supplying dynamical information about all cross sections of $\flow$, our polynomial invariant also detects exactly which integral cohomology classes are dual to sections.

\begin{theor}[McMullen polynomial detects $\Csec$]
\label{T:cones_are_equal}
The McMullen cone $\C_X$ is equal to the cone of sections $\Csec$.
\end{theor}

Theorem~\ref{T:main_theorem} shows that for primitive integral points $u$ of $\Csec$, the canonically defined stretch factor $\Lambda(u)$ can be calculated, without any mention of splittings of $G$ or cross sections of $\flow$, in terms of the specializations $\poly_u$ of the McMullen polynomial. As such it is perhaps unsurprising that the algebraic object $\poly$ might impose strong regularity on the function $\Lambda$. Indeed, following McMullen \cite{Mc}, we use $\poly$ and properties of Perron-Frobenius matrices with entries in the ring of integer Laurent polynomials to prove that the assignment $u\mapsto\log(\lambda(\phi_u))=\log(\Lambda(u))$ extends to a real-analytic, convex and homogeneous function on the entire cone $\Csec$. 
Thus we obtain a new proof of \cite[Theorem D]{DKL} and also extend that result to $\Csec$.
Together with Theorem~\ref{T:cones_are_equal}, our argument additionally shows that this function blows up on the boundary of $\Csec$. 

\begin{theor}[Convexity of stretch factors]
\label{T:continuity/convexity again}
There exists a real-analytic, homogeneous of degree $-1$ function $\mathfrak H\colon \Csec \to \R$ such that:
\begin{enumerate}
\item $1/\mathfrak H$ is positive and concave, hence $\mathfrak H$ is convex.
\item For every primitive integral $u\in \Csec\subset H^1(X;\R)$ with dual compatible section $\Theta_u$, first return map $f_u$, and injective endomorphism $\phi_u = \overline{(f_u)}_*$ as in Theorem~\ref{T:splittings} we have 
\[\mathfrak H(u)=\log(\Lambda(u)) = \log(\lambda(\phi_u)) =\log(\lambda(f_u)) = h(f_u).\]
\item $\mathfrak H(u)$ tends to infinity as $u\to\partial \Csec$.
\end{enumerate}
\end{theor}

\begin{remark}
In \cite{AHR} the authors have also constructed a function defined on a cone containing $\A$, and have proved a result similar to Theorem \ref{T:continuity/convexity again} (with (2) holding for $u \in \A$).   Real-analyticity of $\mathfrak H$ and of the analogous function from \cite{AHR}, combined with the fact that these functions necessarily agree on $\A \subset \C_X$, ensures that the function and cone constructed in \cite{AHR} agree with $\mathfrak H$ and $\C_X$, respectively.   McMullen has obtained similar results in a purely graph theoretic setting; see \cite[Theorem 5.2]{McClique}.  In \cite{Had2}, Hadari generalizes and further analyzes the kinds of polynomials produced here (and in \cite{AHR}, \cite{Mc}, and \cite{McClique}).
\end{remark}

\subsection{Foliations and Lipschitz flows}
Up to this point our focus has been on properties of sections, or equivalently on the integral points of $\Csec$.  However there is a rich geometric, topological, and dynamical structure associated to the non-integral points as well. To describe this structure, we recall that any element $u \in \Csec$ is represented by a closed $1$--form $\omega^u$ that is flow-regular.  Associated to $\omega^u$ is a ``foliation of $X$ transverse to $\flow$'' denoted by $\Omega_u$ with leaves that are (typically infinite) graphs.  Concretely, the $1$--form determines a (flow-regular, $u$--equivariant) function on the universal torsion-free abelian cover $\tX \to \R$, and each ``leaf'' of $\Omega_u$ is the image in $X$ of a fiber of this map; see \S\ref{S:closed 1-forms and foliations}.  The leaves $\Omega_{y,u}$ are thus in one-to-one correspondence with points $y \in \R/u(\pi_1(G))$.  As with sections, we are able to draw the strongest conclusions about the closed $1$--forms representing points in $\A$. 

\begin{theor}[$\pi_1$--injective foliations]
\label{T:transverse foliations}
Given $u \in \A$, there exists a flow-regular closed $1$--form
$\omega^u$ representing $u$ with associated foliation $\Omega_u$ of
$X$ having the following property.  There is a reparameterization of $\flow$, denoted $\flow^u$, so that for each $y \in \R$ the inclusion of the fiber $\Omega_{y,u} \to X$ is $\pi_1$--injective and induces an isomorphism $\pi_1(\Omega_{y,u}) \cong \ker(u)$.
Furthermore, for every $s \geq 0$ the restriction
\[ \flow^u_s \colon \Omega_{y,u} \to \Omega_{y+s,u} \]
of $\flow_s^u$ to any leaf $\Omega_{y,u}$ is a homotopy equivalence.
\end{theor}

This result should be compared to the situation for a fibered hyperbolic $3$--manifold $M$, where elements in the cone on a fibered face of the Thurston norm ball are represented by closed, nowhere vanishing $1$--forms.  The kernel of such a $1$--form is then tangent to a {\em taut foliation} of $M$, which has $\pi_1$--injective leaves by the Novikov-Rosenberg Theorem; see e.g.~\cite{Calegari07}.  Furthermore, appropriately reparameterizing the suspension flow on $M$ will map leaves to leaves by homeomorphisms.

Another application of the McMullen polynomial $\poly$ mirrors McMullen's construction of a ``Teichm\"uller flow'' for each cohomology class in the cone on a fibered face of the Thurston norm ball; see Theorems 1.1 and 9.1 of \cite{Mc}. In the setting of a fibered $3$--manifold $M$, there is a metric on $M$ so that the reparameterized flow mapping leaves to leaves is actually a Teichm\"uller mapping on each leaf.

In our setting, associated to any $u \in \Csec$, we construct a metric on $X$ so that the reparameterized flow, which maps leaves of $\Omega_u$ to leaves, has ``constant stretch factor''.  See \S\ref{S:closed 1-forms cohomology} and \S\ref{S:lipschitz flows} for precise definitions.

\begin{theor}[Lipschitz flows]  
\label{T:lipschtiz flows}
For every $u \in \Csec$, let $\mathfrak H(u)$ be as in Theorem \ref{T:continuity/convexity again}, $\omega^u$ a tame flow-regular closed $1$--form representing $u$, $\flow^u$ the associated reparameterization of $\flow$, and $\Omega_u$ the foliation defined by $\omega^u$.  Then there is a geodesic metric $d_u$ on $X$ such that:
\begin{enumerate}
\item The semiflow-lines $s \mapsto \psi_s^u(\xi)$ are local geodesics for all $\xi \in X$.
\item The metric $d_u$ induces a path metric on each (component of a) leaf $\Omega_{y,u}$ of the foliation $\Omega_u$ defined by $\omega^u$ making it into a (not necessarily finite) simplicial metric graph.
\item The restrictions of the reparameterized semiflow to any leaf
\[ \{\psi_s^u \colon \Omega_{y,u} \to \Omega_{y+s,u}\}_{s \geq 0}\]
are $\lambda^s$--homotheties on the interior of every edge, where $\lambda = e^{\mathfrak H(u)}$.
\item The restriction of $d_u$ to the interior of any $2$--cell of $X$ is locally isometric to a constant negative curvature Riemannian metric.
\end{enumerate}
\end{theor}

The proof of this theorem resembles McMullen's construction of
Teichm\"uller flows in some ways.  In particular, the metric $d_u$ in
this theorem relies on the construction of a kind of ``twisted
transverse measure'' on the foliation $\mathcal F$ of $X$ associated
to the ray in $\Csec$ through $u$; see \S\ref{S:twisted transverse
  measures}.  These measures are similar in spirit to the ``affine
laminations'' of Oertel~\cite{Oertel} and the ``twisted measured
laminations'' of McMullen~\cite{Mc} in the context of hyperbolic
surfaces and $3$--manifolds.

\subsection{Relation to the BNS-invariant}
Recall that a rational point of $H^1(G;\R)$ projects into the BNS-invariant $\BNS(G)$ if and only if $\ker(u)$ is finitely generated as a $\langle \fee_u\rangle_+$--module \cite{BNS} (see the remarks following the Meta-Theorem). We note that the stretch function $\Lambda$ is naturally defined on the \emph{rational BNS-cone} $\QBNS(G)$ consisting of classes $u\in H^1(G;\R)$ with $[u]\in \BNS(G)$ and $u(G)$ discrete; see Definition~\ref{D:stretch_function}.
 Theorem~\ref{T:splittings} implies that the rational points of $\Csec$ project into $\BNS(G)$, and Theorem~\ref{T:DKL_BC} shows that the rational points of $\A$ in fact land in the \emph{symmetrized} BNS-invariant $\BNS_s(G)=\BNS(G)\cap -\BNS(G)$. Our above investigations into the foliations $\Omega_u$ of $X$ dual to non-rational points $u\in \Csec$ can be used to show that these inclusions hold for irrational points as well; namely $\Csec= \C_X$ projects into $\BNS(G)$ (Proposition~\ref{P:S is in BNS}) and $\A$ projects into $\BNS_s(G)$ (Corollary~\ref{C:A is in BNS_0}). Combining the properties of the stretch function $\Lambda$ provided by Theorem~\ref{T:continuity/convexity again} and Proposition~\ref{prop:stretch_is_locally_bounded}, we moreover find that $\C_X$ projects onto a full connected component of $\BNS(G)$:

\begin{theor}[McMullen polynomial detects a component of $\BNS(G)$]
\label{T:BNS}
The McMullen cone $\C_X$ projects onto a full component of the BNS-invariant $\BNS(G)$. That is, $\{[u] \mid u\in \C_X\}$ is a connected component of $\BNS(G)$.
\end{theor}

This shows that the component of $\BNS(G)$ containing $u_0$ is polyhedral and that $\Lambda$ extends to a real analytic function on the cone over that component. It is an interesting question whether \emph{every} component of $\BNS(G)$ has this structure:

\begin{question}
\label{Q:all_of_BNS}
Does every component of $\BNS(G)$ contain a class $u\colon G\to \Z$ so that $\ker(u)$ is finitely generated and $\fee_u$ can be represented by an expanding irreducible train track map?
\end{question}

If the answer to Question~\ref{Q:all_of_BNS} is ``yes,'' then the theory developed in this paper would imply that $\BNS(G)$ is a union of rationally defined polyhedra and that $\Lambda\colon\QBNS(G)\to \R_+$ admits a real-analytic, convex extension.

In a recent paper~\cite[Theorem 5.2]{CL} Cashen and Levitt computed $\BNS(G)$ for the case where $G$ is the mapping torus of a polynomially growing automorphism $\phi$ of $F_N$. They showed that in that situation $\BNS(G)$ is centrally symmetric and consists of the complement of finitely many rationally defined hyperplanes in $H^1(G,\mathbb R)$.  For a polynomially growing $\phi$ we have $\lambda(\phi)=1$, and such $\phi$ does not admit an expanding train track representative. Thus the results of \cite{CL} are disjoint from our results in the present paper.

\subsection{Contrasting the $3$--manifold setting}
\label{S:contrasting_3-manifolds}
Our results illustrate some qualitatively different behavior for free-by-cyclic groups as compared with similar results for $3$--manifold groups. For a fibered hyperbolic $3$--manifold $M$, there are several natural cones that one may consider in $H^1(M;\R)$, such as the cone determined by McMullen's Teichm\"uller polynomial, the analog of the cone of sections $\Csec$ (which can be viewed as the dual on Fried's \cite{FriedS} cone of ``homology directions''), and the components of both the BNS-invariant $\BNS(\pi_1(M))$ and its symmetrization $\BNS_s(\pi_1(M))$ containing $u_0$. These cones all turn out to be the same and are in fact equal to the cone on a ``fibered face'' of the Thurston norm ball \cite{ThuN,Mc,FriedS,BNS}.

Above we have seen that in the free-by-cyclic setting the positive cone $\A$ is contained in a component of $\BNS_s(G)$ and that the McMullen cone $\C_X$ is equal to a component of $\BNS(G)$. However our computations show that in general this component of $\BNS_s(G)$ can be strictly smaller than $\C_X = \Csec$. 
While it was already known that, unlike the $3$--manifold case,  the BNS-invariant of a free-by-cyclic group need not be symmetric (precisely because a free-by-cyclic group can also split as a strictly ascending HNN-extension of a finite rank free group, as an example in Brown's 1987 paper~\cite{Brown} illustrates), one still might have hoped that the geometric property of being dual to a section was sufficient to ensure containment in $\BNS_s(G)$. Evidently this is not the case.

In particular, in the ``running example'' group $G=G_\fee$ that we analyze in detail throughout this paper, the positive cone $\A\subseteq H^1(G,\R)$ is equal to a component of $\BNS_s(G)$ but is a \emph{proper} subcone of $\Csec = \C_X$. We moreover exhibit a specific primitive integral element $u_1\in \Csec \cap \partial \A$ such that $\ker(u_1)$ is not finitely generated but which does, in accordance with Theorem~\ref{T:splittings}, induce a splitting of $G$ as a strictly ascending HNN-extension over a finitely generated free group. Thus $u_1$ belongs to $\Csec = \C_X$ but not to $\BNS_s(G)$; see Examples~\ref{Ex:(-1,2) introduced}, \ref{Ex:(-1,2)_splitting}, \ref{Ex:the_cones}, and ~\ref{Ex:specializing} for the relevant computations regarding $u_1$. For the running example we also exhibit a primitive integral class $u_2\in \Csec \setminus \overline{\A}$ with dual section $\Theta_2$ such that the first return map $f_{2}\colon\Theta_2\to\Theta_2$ fails to be a homotopy equivalence but nevertheless the induced endomorphism descends to an automorphism of the finitely generated free group $\ker(u_2)$; see Examples~\ref{Ex:(-1,1) introduced} and \ref{Ex:(-1,1)_splitting} for the relevant computations. This shows that $\Csec$ can contain multiple distinct components of the symmetrized BNS-invariant $\BNS_s(G)$.

\begin{remark}
In Example \ref{Ex:(-1,2) introduced} and continued in Example
\ref{Ex:(-1,2)_splitting} and Remark \ref{Rm:infinitely generated
  kernels exist}, we find a section $\Theta_1 \subset X$ dual to a class
$u_1\in\Csec$ (mentioned above) with  $\ker(u_1)$ infinitely generated and which gives rise to a splitting of $G$ as a strictly ascending HNN-extension of a finite rank free group over an injective but non-surjective endomorphism. This example
illustrates that some of the results similar to Theorem~\ref{T:splittings}
that are claimed in Wang's thesis~\cite{Wang} are incorrect.  Specifically
\cite[Lemma 1.3]{Wang} produces, for any section of the semiflow on
the mapping torus, a kind of ``fibration'' of a related
complex. According to \cite[Lemma 4.1.3]{Wang}, this would imply $u_1$ has finitely
generated kernel (equal to the fundamental group of a fiber of the
associated ``fibration'') and thus that $u_1$ would define a splitting of
$G$ as a (f.g.~free)-by-cyclic group.  However, our computations show that this is not the case.
\end{remark}

\subsection{Acknowledgments:} We are grateful to Nathan Dunfield for many
useful discussions regarding  the Alexander norm and the BNS-invariant, and in particular for suggesting the approach to proving Theorem \ref{T:BNS}.
We would also like to thank Curtis McMullen for his help with clarifying some results from
the Perron-Frobenius theory related to his earlier work and to the proof of Theorem~\ref{T:continuity/convexity again}, Asaf Hadari for an interesting conversation regarding labeled graphs, and Robert Bieri for his reference to \cite{BieriRenz} clarifying the sign convention for the BNS--invariant.

\section{Background}
\label{S:preliminaries}
\subsection{Graphs and graph maps}
\label{S:graphs_and_maps}
We briefly review the relevant terminology regarding graph maps while referring the reader to \cite[\S2]{DKL} for a more detailed discussion of these definitions. A continuous map $f\colon \Gamma\to\Gamma$ of a finite graph $\Gamma$ is said to be a \emph{topological graph map} if it sends vertices to vertices and edges to edge paths. We always assume that our graphs do not have valence--$1$ vertices. If $\Gamma$ is equipped with either a metric structure or a (weaker) linear structure, then $f$ is furthermore said to be a \emph{(linear) graph map} if it enjoys a certain piecewise linearity on edges.  For simplicity, we also typically assume that $f$ is a {\em combinatorial graph map}, which means that $\Gamma$ is given a metric structure in which every edge has length one, and the restriction of $f$ to each edge $e$ is a local $d_e$--homothety, where $d_e$ is the number of edges in edge path $f(e)$.  The map is called \emph{expanding} if for every edge $e$ of $\Gamma$, the combinatorial lengths of the edge paths $f^n(e)$ tend to infinity as $n\to \infty$. 

Every graph map $f\colon \Gamma\to\Gamma$ has an associated \emph{transition matrix} $A(f)$, which is the $\abs{E\Gamma}\times\abs{E\Gamma}$ matrix whose $(e,e')$--entry records the number of times the edge path $f(e')$ crosses the edge $e$ (in either direction). We caution that, while this definition of $A(f)$ agrees with the definition given in \cite{BH92}, it gives the transpose of what was used to mean the transition matrix in \cite{DKL}. The graph map $f$ is then said to be \emph{irreducible} if its transition matrix $A(f)$ is irreducible. We additionally use $\lambda(f)$ to denote the spectral radius of $A(f)$.

A \emph{train track map} is a linear graph map $f\colon \Gamma\to\Gamma$ such that $f^k$ is locally injective at each valence--$2$ vertex and on each edge of $\Gamma$ for all $k\geq 1$. Note that this definition is slightly more general than the standard notion of a train track map in $\Out(F_N)$ theory (in particular, more general than what was used in \cite{DKL}) since here we don't require $f$ to be a homotopy equivalence. A train track map that is irreducible always has $\lambda(f)\geq 1$; furthermore this inequality is strict if and only if $f$ is expanding.

If $f\colon\Gamma\to\Gamma$ is an irreducible  train-track map with $\lambda(f)>1$, then by Corollary~A.7 of \cite{DKL} there exists a volume--1 metric structure $\mathcal L$ on $\Gamma$ with respect to which the map $f$ is a local $\lambda(f)$--homothety on every edge of $\Gamma$.  As in \cite{DKL}, we call  $\mathcal L$ the \emph{canonical metric structure} on $\Gamma$ and call the corresponding metric $d_\mathcal L$ on $\Gamma$ the \emph{canonical eigenmetric}.

\subsection{Markings and representatives}
\label{S:markings}
When $\Gamma$ is a finite connected graph and $f\colon\Gamma\to\Gamma$ is
topological graph map which is a homotopy equivalence, it induces an
automorphism $f_*$ of the free group $\pi_1(\Gamma)$ that is
well-defined up to conjugacy; accordingly we say that $f$ is a
\emph{topological representative} of the outer automorphism $[f_*]\in
\Out(\pi_1(\Gamma))$. Even when $f$ is not a homotopy equivalence, there is still a sense in
which $f$ represents an endomorphism of a free group, as we now explain.

A \emph{marking} on the free group $F_N$ consists of a finite connected graph $\Gamma$ without valence--1 vertices and an isomorphism $\alpha\colon F_N\to \pi_1(\Gamma)$. If $\mathcal L$ is a metric structure on $\Gamma$, then the pair $(\alpha,\mathcal L)$ is called a \emph{marked metric structure} on $F_N$. Lifting $\mathcal L$ to the universal cover $\widetilde \Gamma$ defines a metric $d_{\widetilde {\mathcal L}}$ on $\widetilde \Gamma$ such that $T=(\widetilde \Gamma, d_{\widetilde{\mathcal L}})$ is an $\R$--tree equipped with a free and discrete isometric action (via $\alpha$) of $F_N$ by covering transformations. For an element $g\in F_N$, let  $\norm{g}_T:=\inf_{x\in T} d_{\widetilde{\mathcal L}}(x,gx)$ denote the \emph{translation length} of $g$ on $T$. Note that $\norm{g}_T$ is equal to the $\mathcal L$--length of the unique immersed loop in $\Gamma$ which is freely homotopic to the loop $\alpha(g)\in \pi_1(\Gamma)$. If $A=\{a_1,\dots, a_N\}$ is a free basis of $F_N$, $\Gamma$ is a rose with $N$ petals corresponding to the elements of $A$, and $\mathcal L$ is the metric structure assigning every edge of $\Gamma$ length $1$, then $T$ is exactly the Cayley tree of $F_N$ corresponding to $A$. In this case for $g\in F_N$ we have $\norm{g}_T=\norm{g}_A$, where $\norm{g}_A$ is the cyclically reduced word length of $g$ with respect to $A$. Note that if $g,g'\in F_N$ are conjugate elements then $\norm{g}_T=\norm{g'}_T$.

A \emph{topological representative} of a free group endomorphism  $\phi\colon F_N\to F_N$ consists of a marking  $\alpha\colon F_N\to \pi_1(\Gamma,v)$ (where $v\in V\Gamma$) and a topological graph map $f\colon \Gamma\to\Gamma$ such that for some inner automorphism $\tau$ of $F_N$ we have $\tau\circ \phi= \alpha^{-1}\circ f_* \circ \alpha$. Here
$f_*\colon\pi_1(\Gamma,v)\to\pi_1(\Gamma,v)$ is the endomorphism of $\pi_1(\Gamma,v)$ defined as $f_*(\gamma)=\beta f(\gamma) \beta^{-1}$ for some fixed edge path $\beta$ in $\Gamma$ from $v$ to $f(v)$.  Thus the only difference with the standard definition of a topological representative is that here the graph map $f\colon \Gamma\to\Gamma$ need not a homotopy equivalence. Indeed, we note that a topological representative $f$ as above will be a homotopy equivalence if and only if $\phi$ is an actual automorphism of $F_N$. As usual, we often suppress the marking and, by abuse of notation, simply talk about $f\colon \Gamma\to\Gamma$ being a topological representative of $\phi$.  Most of the free group train track theory is developed for automorphisms of $F_N$, but topological and train track representatives of  endomorphisms of $F_N$  appear, for example, in \cite{AR,DV,Reynolds}.

\subsection{Growth}
\label{S:growth}
Let $\alpha\colon F_N\to \pi_1(\Gamma)$ be a marking on $F_N$, let $\mathcal L$ be a metric structure on $\Gamma$, and let $T=(\widetilde \Gamma, d_{\widetilde{\mathcal L}})$ be the corresponding $\R$--tree as above. For any endomorphism $\phi\colon F_N\to F_N$ and any element $g\in F_N$ put
\[
\lambda (\phi; g, T):=\liminf_{n\to\infty} \sqrt[n]{\norm{\phi^n(g)}_T}.
\]
It is not hard to check that the $\liminf$ in the above formula is in fact always a limit. Moreover, if $T'$ is an $\R$--tree corresponding to another marked metric structure on $F_N$, then the trees $T$ and $T'$ are $F_N$--equivariantly quasi-isometric. This fact  implies that $\lambda (\phi; g, T)=\lambda (\phi; g, T')$ for all $g\in F_N$. Thus we may unambiguously define $\lambda (\phi; g):=\lambda(\phi; g, T)$, where $T$ is the $\R$--tree corresponding to any marked metric structure on $F_N$. 

\begin{defn}[Growth of an endomorphism]\label{D:growth}
Let $\phi\colon F_N\to F_N$ be an arbitrary endomorphism. The \emph{growth rate} or \emph{stretch factor} of $\phi$ is defined to be 
\[
\lambda(\phi):=\sup_{g\in F_N} \lambda(\phi; g).
\]
We say that $\phi$ is \emph{exponentially growing} if $\lambda(\phi)>1$.
\end{defn}
It is not hard to check that if $\phi$ is not exponentially growing then $\lambda (\phi; g)=1$ or $\lambda(\phi; g)=0$ for every $g\in F_N$. The case $\lambda(\phi; g)=0$ is possible since $\phi$ need not be injective; indeed, $\lambda(\phi; g)=0$ if and only if $\phi^k(g)=1\in F_N$ for some $k\ge 1$. Furthermore, since the translation length of an element for an isometric group action is invariant under conjugation in the group, Definition~\ref{D:growth} implies that $\lambda(\phi) = \lambda(\tau\phi) = \lambda(\phi\tau)$ for any inner automorphism $\tau\in \Inn(F_N)$ of $F_N$.

Any irreducible train track representative of $\phi$ (assuming such a representative exists) can be used to compute the growth rate $\lambda(\phi)$:

\begin{proposition}\label{P:stretch}
Let $\phi\colon F_N\to F_N$ be an endomorphism and let $f\colon \Gamma\to\Gamma$ be an irreducible train track representative of $\phi$. Then $\lambda(\phi)=\lambda(f)$ and $\log(\lambda(f))=h(f)$, where $h(f)$ is the topological entropy of $f$.
\end{proposition}
\begin{proof}
If $\lambda(f)=1$, then $f$ is a simplicial automorphism of $\Gamma$ and hence $\phi$ is an automorphism of $F_N$ which has finite order in $\Out(F_N)$. Therefore $\lambda(\phi)=1$ and $\lambda(\phi)=\lambda(f)$, as required. Also, in this case it is easy to see that $h(f)=0$ and thus that $\log(\lambda(f))=h(f)$ holds as well.

Suppose now that $\lambda(f)>1$, so that $f$ is expanding. Proposition~A.1 of \cite{DKL} then implies the equality $\log(\lambda(f))=h(f)$. 
Let $\mathcal L$ be the canonical metric structure on $\Gamma$. Then
for every edge $e$ of $\Gamma$ we have $\mathcal L(f(e))=\lambda(f)
\mathcal L(e)$ and hence for every edge-path $\gamma$ in $\Gamma$ we
have 
\[
\mathcal L(f_\#(\gamma)) \le \mathcal L(f(\gamma))=\lambda(f)
\mathcal L(\gamma),\]
where $f_\#(\gamma)$ is the tightened form of $f(\gamma)$. This
implies that if $T$ is the $\R$--tree corresponding to $(\Gamma,\mathcal L)$, 
then for every $g\in F_N$ we have $\norm{\phi(g)}_T\le
\lambda(f)\norm{g}_T$. Therefore for $n\ge 1$ we have
$\norm{\phi^n(g)}_T\le \lambda(f)^n \norm{g}_T$ and hence
$\lambda(\phi)\le \lambda(f)$.

Since $f$ is an expanding graph map, for some edge $e$ there exists $m\ge 1$ such that $f^m(e)$ contains at least two occurrences of the same oriented edge. Therefore there exists a nontrivial immersed loop $\gamma$ in $\Gamma$ such that $\gamma$ is a subpath of $f^m(e)$. Since $f$ is a train track map, for every $n\geq 1$ the loop $f^n(\gamma)=(f^n)_\#(\gamma)$ is immersed in $\Gamma$ and therefore satisfies
\[
\mathcal L((f^n)_\#(\gamma))=\mathcal L
(f^n(\gamma))=\lambda(f)^n\mathcal L(\gamma).
\]
Taking $g_0\in F_N$ to be any element whose conjugacy class is
represented by $\gamma$, we see that $\norm{\phi^n(g_0)}_T= \lambda(f)^n \norm{g_0}_T$ for all $n\ge 1$. Hence $\lambda
(\phi; g_0, T)=\lambda(f)$. Together with $\lambda(\phi)\le
\lambda(f)$, this implies that $\lambda(\phi)=
\lambda(f)$.
\end{proof}

\begin{proposition}\label{P:subgroup_stretch}
Let $\phi\colon F_N\to F_N$ be an arbitrary free group endomorphism. Given any $i\geq 0$, set $J = \phi^i(F_N)$ and let $\xi=\phi\vert_J \colon J\to J$. Then $\lambda(\phi)=\lambda(\xi)$.
\end{proposition}
\begin{proof}
Pick a free basis $A$ of $F_N$ and let $T=T_A$ be the Cayley graph of $F_N$ with respect to $A$. Let $i\ge 0$ be arbitrary and let $J=\phi^i(F_N)$. Thus $J$ is a finitely generated free group of rank $\le N$. Choose a free basis $B$ for $J$ and let $T'$ be the Cayley graph of $J$ with respect to $B$. Let $\xi=\phi|_J\colon J\to J$. 
For any $g\in F_N$ the definitions imply that $\lambda(\phi; g, T)=\lambda(\phi; \phi^i(g), T)$.
The subgroup $J\le F_N$ is quasi-isometrically embedded in $F_N$. Hence there exists $C\ge 1$ such that
\[\frac{1}{C} \norm{w}_{T} \le \norm{w}_{T'}\le C \norm{w}_T\]
for all $w\in J$. Therefore for all $n\ge 1$ and $g\in F_N$ we have $\phi^{i+n}(g)=\xi^n(\phi^i(g))$ and
\[
\frac{1}{C} \norm{\phi^{n+m}(g)}_{T} \le \norm{\xi^n(\phi^m(g))}_{T'}\le C \norm{\phi^{n+m}(g)}_T.
\]
Hence 
\[
\lambda(\phi; g, T)=\lambda(\phi; \phi^m(g), T)=\lambda(\xi; \phi^m(g), T')
\]
The definition of the growth rate now implies that $\lambda(\phi)=\lambda(\xi)$. 
\end{proof}

\subsection{Endomorphisms and HNN-like presentations}
\label{S:HNN-like}
In this subsection we elaborate on the observation of Kapovich~\cite{K02}
about the algebraic structure of HNN-like presentations based on
arbitrary (and possibly non-injective) endomorphisms.

Let $\phi\colon G\to G$ be an arbitrary endomorphism of any group $G$. We then use the notation $G\ast_{\phi}$ to denote the group given by the ``HNN-like'' presentation
\begin{equation}\label{eqn:HNN-like}
G\ast_\phi := \langle G, r \mid r\inv g r = \phi(g),\text{ for all }
g\in G\rangle.
\end{equation}
The generator $r$ here is called the \emph{stable letter} of $G\ast_\phi$, and the group has a \emph{(natural) projection} $G\ast_\phi\to \Z$ defined by sending $r\mapsto 1$ and $G\mapsto 0$. 

Presentations as above with $\phi$ non-injective are difficult to work
with, since in that case $G$ does not embed into 
$G\ast_\phi$ and Britton's Lemma (the normal form theorem for HNN
extensions) does not hold.
However, as we will see below, presentation~\eqref{eqn:HNN-like}
does define a group which is a genuine HNN-extension along an injective endomorphism of a quotient
group of $G$.

To this end, define the \emph{stable kernel of $\phi\colon G\to G$} to be the normal subgroup
\[K_\phi := \bigcup_{i=1}^\infty \ker(\phi^i) \lhd G\]
obtained as the union of the increasing chain $\ker(\phi)\leq
\ker(\phi^2)\leq\dotsb$ of subgroups $\ker(\phi^i)\lhd G$.
We also denote $\bar{G}^\phi=G/K_\phi$ and call $\bar{G}^\phi$ the
\emph{stable quotient of $\phi$}.

\begin{proposition}
\label{P:stable_quotient_HNN}
Let $G$ be a group and $\phi \colon G\to G$ be an arbitrary endomorphism. Let $G\ast_\phi$ be as in presentation~\eqref{eqn:HNN-like} above and let $\bar G=\bar G^\phi$. Then the following hold:
\begin{enumerate}
\item $\phi\colon G\to G$ descends to an injective endomorphism $\bar\phi\colon \bar G\to \bar G$. 
\item The natural homomorphism $G\to G\ast_\phi$  has kernel equal
to $K_\phi$, so that the image of $G$ in $G\ast_\phi$ is canonically
identified with $\bar G$.
\item The quotient map $G\to \bar G$ (which we denote $g\mapsto \bar{g}$) induces an isomorphism
\[
G\ast_\phi \to \bar G\ast_{\bar \phi} 
= \langle \bar{G}, \bar{r} \mid \bar{r}\inv\bar{g} \bar{r} = \bar{\phi}(\bar g), \mbox{ for all }\bar g\in
\bar{G} \rangle,\]
whose composition with the projection $\bar{G}\ast_{\bar\phi}\to \Z$ yields the natural projection of $G\ast_\phi$. Thus $G\ast_\phi$ is canonically isomorphic to the (genuine) HNN-extension $\bar G\ast_{\bar \phi}$ over the injective endomorphism  $\bar\phi$.
\end{enumerate}
\end{proposition}
\begin{proof}
From the definition of $K_\phi$ we see that $g\in K_\phi$ if and only
if $\phi(g)\in K_\phi$. Therefore $\phi$ does indeed descend to an
endomorphism $\bar\phi\colon \bar G\to \bar G$, and moreover
$\bar\phi$ is injective. Thus (1) is established.

If $g\in K_\phi$ then $g\in\ker(\phi^n)$ for some $n\ge 1$. Then
$\phi^n(g)=1$ in $G$ and therefore in the group $G\ast_\phi$ we have
$g=r^{n}\phi^n(g)r^{-n}=1$. Therefore, by applying a Tietze transformation, we can rewrite
presentation~\eqref{eqn:HNN-like} of $G\ast_\phi$ as
\begin{align*}
G\ast_\phi &= \langle G, r \mid r\inv g r = \phi(g) \text{ for all }
g\in G\rangle\\
&= \langle G, r \mid r\inv g r = \phi(g) \text{ for all }
g\in G; \;\text{ and }  g=1  \text{ for all } g\in K_\phi\rangle\\
&= \langle G/K_\phi,\ \bar r \mid \bar{r}\inv \bar g \bar{r} = \bar\phi(\bar g) \text{ for all }
\bar g\in G/K_\phi\rangle
= \bar{G} \ast_{\bar\phi}
\end{align*}
This implies parts (2) and (3) of the proposition.
\end{proof}

The following proposition implies that in certain situations the
isomorphism $G\ast_\phi\cong \bar G\ast_{\bar \phi}$ from Proposition~\ref{P:stable_quotient_HNN} can 
actually be seen from {\em inside} $G$.

\begin{proposition}
\label{P:injective_on_image}
Let $\phi\colon G\to G$ be an arbitrary endomorphism, and let $\bar{G}$ and $\bar\phi$ be as in Proposition~\ref{P:stable_quotient_HNN}. Suppose that $i\ge 1$ is such that for $J = \phi^i(G)$ the endomorphism $\xi = \phi\vert_J\colon J\to J$ is injective.
Then $K_\phi=\ker(\phi^i)$ and there is a canonical isomorphism $\tau\colon \bar{G}\to J$ conjugating $\bar\phi$ to $\xi$.
\end{proposition}
\begin{proof}
Since $\xi = \phi\vert_J$ is injective by assumption, it follows that $\phi^n\vert_J$ is injective for all $n\geq 0$. Therefore, if $g\in G$ is such that $\phi^i(g)\in J$ is nontrivial, then $\phi^{n+i}(g) = \phi^n(\phi^i(g))$ is also nontrivial for all $n\geq 0$. This implies that $\phi^i(k)=1\in J$ for all $k\in K_\phi$ and thus that $K_\phi\subset \ker(\phi^i)$. As we also have $\ker(\phi^i)\subset K_\phi$ by definition, it follows that $\ker(\phi^i) = K_\phi$.

Thus $\bar G=G/\ker(\phi^i)$ and so the First Isomorphism Theorem provides a canonical isomorphism $\tau\colon\bar G = G/K_\phi\to J$ given by $\tau(gK_\phi)=\phi^i(g)$ for all $g\in G$. Then for any $g\in G$ we have
\[ \xi(\tau(gK_\phi))=\xi(\phi^i(g))=\phi^{i+1}(g)=\tau(\phi(g)K_\phi)=\tau(\bar \phi (gK_\phi)).\]
Thus $\tau$ indeed conjugates $\bar\phi$ to
$\xi$, and the statement of the proposition follows.
\end{proof}

It was observed by Kapovich \cite{K02} that the assumption of Proposition~\ref{P:injective_on_image} is always satisfied when $G$ is a finite-rank free group $F_N$.

\begin{proposition}\label{P:Kapovich subgroup}
Let $\phi\colon F_N\to F_N$ be an arbitrary endomorphism of a finite-rank free group $F_N$. Then there exists $i\geq 0$ such that $b_1(\phi^i(F_N)) = b_1(\phi^{i+1}(F_N))$. Furthermore, if $J = \phi^i(F_N)$ for any such $i$, then $\xi=\phi|_J:J\to J$ is injective.
\end{proposition}
\begin{proof}
For every $m\ge 1$ the group $\phi^m(F_N)$ is a subgroup of $F_N$ and it is also a homomorphic image of $\phi^{m-1}(F_N)$. Therefore each image $\phi^m(F_N)$ is a finitely generated free group and the integral sequence $\{b_1(\phi^m(F_N))\}$ is nonincreasing. Therefore there exists $i\geq 0$ such that $b_1(\phi^i(F_N))=b_1(\phi^{i+1}(F_N))$ as claimed.

For any such $i\geq 0$, the free groups $\phi^i(F_N)$ and $\phi^{i+1}(F_N)$ have the same rank and are thus isomorphic. Moreover the homomorphism $\phi$ maps $\phi^i(F_N)$ onto $\phi^{i+1}(F_N)$. Since finitely generated free groups are Hopfian, it follows that $\phi$ maps
$\phi^i(F_N)$ isomorphically onto $\phi^{i+1}(F_N)$ and thus the restriction of $\phi\vert_{\phi^i(F_N)}$ is injective.
\end{proof}

Combining the above with the results from \S\ref{S:growth}, we obtain the following corollary which, in the case that $G$ is free, gives two (essentially equivalent) useful ways of realizing the group $G\ast_\phi$ from (\ref{eqn:HNN-like}) as an ascending HNN-extension of finitely generated free group along an injective endomorphism.

\begin{corollary}
\label{C:HNN_and_stretch_for_endos}
Let $G=F_N$, let $\phi\colon G\to G$ be an arbitrary endomorphism and let $\bar{G} = \bar{G}^\phi$ and $\bar{\phi}$ be as in Proposition~\ref{P:stable_quotient_HNN}. Then $\bar{G}$ is a finite-rank free group and there exists $J = \phi^i(G)$ for some $i\ge 0$ such that $\xi= \phi\vert_J\colon J\to J$ is injective. Furthermore, the respective stretch factors satisfy $\lambda(\phi) = \lambda(\bar\phi) = \lambda(\xi)$, and there are canonical isomorphisms
\[G\ast_\phi \cong \bar{G}\ast_{\bar\phi} \cong J\ast_{\xi}\]
respecting the natural projections of these groups to $\Z$.
\end{corollary}
\begin{proof}
The existence of such an $i\ge 0$ follows from Proposition~\ref{P:Kapovich subgroup}. Proposition~\ref{P:injective_on_image} then provides an isomorphism $\tau\colon \bar{G}\to J$ showing that $\bar{G}$ is a finite-rank free group. Furthermore, since $\tau$ conjugates $\bar\phi$ to $\xi$ it also induces an isomorphism $\bar{G}\ast_{\bar\phi}\to J\ast_{\xi}$ and shows that $\lambda(\xi) = \lambda(\bar\phi)$. Lastly, the equality $\lambda(\phi) = \lambda(\xi)$ follows from Proposition~\ref{P:subgroup_stretch}, and the isomorphism $G\ast_{\phi}\cong \bar{G}\ast_{\bar\phi}$ from Proposition~\ref{P:stable_quotient_HNN}.
\end{proof}

\subsection{Invariance of HNN stretch factor}
\label{sec:stretch_invariance}
Consider a finitely presented group $D$ that splits as an HNN-extension
\[D = A\ast_\phi = \langle A,r \mid r\inv ar = \phi(a)\text{ for all }a\in A\rangle,\]
where $A$ is a finite rank free group and $\phi\colon A\to A$ is an injective endomorphism of $A$. The natural projection $u\colon D\to \Z$ defined by $r\mapsto 1$ and $A \mapsto 0$ is then said to be \emph{dual} to the splitting $D = A\ast_\phi$.

Note that if $\ker(u) = ncl_D(A)$ is not finitely generated, then there are infinitely many splittings of $D$ that are all dual to the same $u$. Indeed, in this case $\phi\colon A\to A$ is non-surjective and so we may choose a finitely generated subgroup $C\le A$ such that $\phi(A)\lneq C \lneq A$. Put $B=rCr^{-1}$, so that $C\lneq A \lneq B$ and $r^{-1}Br=C\le B$. Defining $\psi\colon B\to B$ by $\psi(b)=r^{-1}br$, it is not hard to see that the HNN-extension $B\ast_\psi= \langle B, r \mid r\inv b r = \psi(b),\text{ for all } b\in B\rangle$ gives another splitting $D=B\ast_\psi$ which is again dual to $u$. Notice that this construction can be used to produce splittings $D=B\ast_\phi$ dual to $u$ where the rank of $B$ is any integer greater than the rank of $A$. Other variations and iterations of this constructions are also possible. 

This shows that if a homomorphism $u\colon D\to \Z$ is dual to an ascending HNN-extension splitting $D=A\ast_\phi$ of $D$, then the injective endomorphism $\phi\colon A\to A$ defining the splitting is in no way canonical. Nevertheless, its stretch factor $\lambda(\phi)$ \emph{is} uniquely determined by the homomorphism $u$:

\begin{proposition}\label{P:unique_stretch_for_HNN}
Suppose that $u\in \Hom(D,\Z)$ is dual to two splittings $D = A\ast_\phi$ and $D = B\ast_\psi$, where $\phi\colon A\to A$ and $\psi\colon B\to B$ are injective endomorphisms of finite rank free groups $A$ and $B$. Then $\lambda(\phi) = \lambda(\psi)$.
\end{proposition}
\begin{proof}
By assumption we may write
\[\langle A, r \mid r\inv a r = \phi(a)\text{ for all }a\in A\rangle \; = \; D \;=\; \langle B, s \mid s\inv b s = \psi(b)\text{ for all }b\in B\rangle,\]
where $u(r)=u(s) = 1$ and $u(A) = u(B) = 0$. In particular, $\ker(u)=ncl_D(A)=ncl_D(B)$.

If $\phi$ is an automorphism of $A$, then $A = \ker(u) = B$ so that $\psi$ defines an automorphism of $A$ in the same outer automorphism class as $\phi$. In this case we obviously have $\lambda(\phi) = \lambda(\psi)$. Thus we may assume neither $\phi$ nor $\psi$ is surjective.

The fact that $u(s)=u(r)$ implies $s=ra'$ for some $a'\in \ker(u) = \cup_{i=0}^\infty r^i A r^{-i}$.  Thus there exists $k\ge 1$ such that $a_0 :=r^{-k}a' r^k\in A$.
Conjugating the HNN presentation $D=B\ast_\psi$ by the element $r^{-k}\in D$, which does not change the dual homomorphism $u\colon D\to \Z$ and preserves the stretch factor of the defining endomorphism, we may henceforth assume $s=ra_0$ for some  $a_0\in A$. Note that we then have $s^{-1} A s\le A$.

We claim that $s^m A s^{-m} = r^m A r^{-m}$ for all $m\geq 0$. This is obviously true for $m=0$, so by induction assume it holds for some $m\geq 0$. Then
\[ s(s^m A s^{-m})s^{-1} = (ra_0) (r^m A r^{-m}) (a_0\inv r\inv) = r^{m+1} A r^{-m-1},\]
where here we have used the fact that $a_0 (r^m A r^{-m})a_0\inv = r^mAr^{-m}$ since $a_0\in r^m A r^{-m}$. The claim follows.

We therefore have $\ker(u) = \cup_{i=0}^\infty s^i A s^{-i}$. Since $B\le \ker(u)$ is finitely generated and $s\inv A s \le A$, it follows that there exists $n_0\geq 0$ for which $s^{-n_0} B s^{n_0} \le A$. Thus, conjugating the splitting $D=B\ast_\psi$ by the element $s^{-n_0}\in D$, which does not change the stable letter $s$ of the presentation, we may assume that $B\le A$.

Consider the injective endomorphism $\phi'\colon A\to A$ defined by $\phi'(a) = s\inv a s\in A$. Notice that for any $a\in A$, the elements $\phi'(a) = s\inv a s$ and $\phi(a) = r\inv a r$ are conjugate in $A$ (since $s = ra_0$) and therefore have the same cyclically reduced word length with respect to any free basis of $A$. Thus $\norm{\phi'(a)}_A = \norm{\phi(a)}_A$ for all $a\in A$, and so we obviously have $\lambda(\phi) = \lambda(\phi')$. On the other hand, since $\psi$ is defined as $\psi(b) = s\inv b s$ and $B\le A$, we see that $\psi = \phi'\vert_{B}$. Therefore, letting $K$ denote the maximum value of $\norm{\cdot}_A$ over all elements in a  free basis of $B$, we find that
\[ \lambda(\psi) 
\; =\; \sup_{b\in B} \liminf_{n\to \infty} \sqrt[n]{\norm{\psi^n(b)}_B}
\; \leq\; \sup_{b\in B} \liminf_{n\to \infty} \sqrt[n]{K\norm{\psi^n(b)}_A}
\; \leq\; \sup_{a\in A} \liminf_{n\to \infty} \sqrt[n]{\norm{\phi'^n(a)}_A}
\; =\; \lambda(\phi').\]
Thus $\lambda(\psi)\leq \lambda(\phi')=\lambda(\phi)$. By symmetry we have $\lambda(\phi)\leq \lambda(\psi)$ as well, and so the proposition follows.
\end{proof}

We remark that, in view of Corollary~\ref{C:HNN_and_stretch_for_endos}, the conclusion of Proposition~\ref{P:unique_stretch_for_HNN} holds even if one omits the requirement that the endomorphisms $\phi$ and $\psi$ be injective.

\section{Setup}
\label{sec:Setup}

Let $\Gamma$ be a finite graph with no valence--$1$ vertices, and let $f\colon \Gamma\to\Gamma$ be an expanding irreducible train track map representing an outer automorphism $\fee \in \Out(F_N)$ of the rank--$N$ free group $F_N\cong\pi_1(\Gamma)$. Let $G = F_N\rtimes_\fee \Z$ be the free-by-cyclic group determined by the outer automorphism $\fee$. Explicitly, $G$ is defined up to isomorphism by choosing a representative $\Phi\in \Aut(F_N)$ of $\fee$ and setting
\begin{equation}
\label{eqn:semi-direct}
G = F_N\rtimes_\fee \Z :=  \langle w, r \mid r^{-1}wr = \Phi(w)\text{ for } w\in F_N\rangle.
\end{equation}
The first cohomology $H^1(G;\R)$ of $G$ is simply the set of homomorphisms $\Hom(G;\R)$, and an element $u\in H^1(G;\R)$ is said to be \emph{primitive integral} if $u(G) = \Z$. In this situation, the element $u$ determines a splitting (i.e., a split extension)
\begin{equation}
\label{eqn:splitting}
\xymatrix{ 1  \ar[r] &  \ker(u) \ar[r] & G \ar[r]^{u} & \Z \ar[r] & 1}
\end{equation}
of $G$  and a corresponding \emph{monodromy} $\fee_u\in \Out(\ker(u))$. Namely, if $t_u\in G$ is any element such that $u(t_u)=1$, then the conjugation $g\mapsto t_u^{-1} g t_u$ defines an automorphism of $\ker(u)\lhd G$ whose outer automorphism class is the monodromy $\fee_u$. It is easy to see that $\fee_u\in \Out(\ker(u))$ depends only on $u$ and not on the choice of $t_u$. The homomorphism associated to the original splitting $G = F_N\rtimes_\fee \Z$ (that is, the homomorphism to $\Z$ with kernel $F_N\lhd G$ and sending the stable letter $r$ to $1\in \Z$) will be denoted by $u_0$; its monodromy is the given outer automorphism $\fee$.

We are particularly interested in the case where $\ker(u)$ is finitely generated, which automatically implies that $\ker(u)$ is free \cite{GMSW} and thus that $\fee_u$ is a free group automorphism. However, even when $\ker(u)$ is not finitely generated, we will see that in many cases $\fee_u$ has a naturally associated injective endomorphism -- of a finitely generated free subgroup of $\ker(u)$ -- with respect to which the splitting \eqref{eqn:splitting} is realized as an ascending HNN-extension; see \S\ref{S:endomorphisms}.

\subsection{The folded mapping torus}
\label{S:folded_mapping_torus}

Given a train track map $f$ as above, in \S4 of \cite{DKL} we constructed a $K(G,1)$ polyhedral $2$--complex $X = X_f$, called the \emph{folded mapping torus of $f$}.  This $2$--complex comes equipped with a semiflow $\flow$ and a natural map $\fib\colon X\to \sone$ that is a local isometry on flowlines  and whose induced map on fundamental groups is the homomorphism $u_0 = \fib_*\colon G\to \Z$.
The $1$--skeleton of $X$ consists of \emph{vertical $1$--cells} which are arcs of flowlines, and \emph{skew $1$--cells} which are transverse to the flowlines. Each $2$--cell of $X$ is a trapezoid whose top and bottom edges consist of skew $1$--cells and whose sides consist of vertical $1$--cells (it may be that one side is degenerate).

As in \cite{DKL}, we will assume that $f$ is a combinatorial graph map (see \S\ref{S:graphs_and_maps}), though this is only done to simplify the discussion and the exposition. We briefly recall the construction of $X$: Let $Z_f = \Gamma\times[0,1]/(x,1)\sim(f(x),0)$ be the usual mapping torus of $f$. There is a natural suspension semiflow on $Z_f$ given by flowing in the $[0,1]$ direction. The folded mapping torus $X$ is constructed as an explicit flow-respecting quotient of $Z_f$. In particular, the original graph $\Gamma$ may be identified with the image of $\Gamma\times\{0\}\subset \Gamma\times[0,1]$ in $X$. In this way, $\Gamma$ is realized as a cross section of the flow on $X$, and the first return map of $\flow$ to this cross section is exactly $f$. Furthermore, the natural map $Z_f\to \sone$ descends to a map $\fib\colon X\to \sone$ whose induced map on fundamental group is just the homomorphism $u_0\colon G\to\Z$.

We will work in the universal torsion-free abelian cover $p \colon \tX \to X$, which is the cover corresponding to the subgroup
\[ \ker( G \to H_1(G;\Z)/\mathrm{torsion}) < G.\]
Denote the covering group of $\tX \to X$ by $H \cong H_1(G;\Z)/\mathrm{torsion} \cong \Z^b$, where $b = b_1(G) = \rk(H)$. The semiflow $\flow$ lifts to a semiflow $\tflow$ on $\tX$. The cell structure on $X$ determines one on $\tX$ so that the covering map is cellular, and we will use the same terminology (skew, vertical, etc) to describe objects in $\tX$ with the same meaning as in $X$.

\begin{example}[Running example]
\label{E:introduce_running}
We recall here the `running example' that was introduced in Example 2.2 of \cite{DKL} and developed throughout that paper. In this paper we will continue our analysis of this example as we use it to illustrate key ideas. Let $\Gamma$ denote the graph in Figure \ref{F:train_tracks}, which has four edges oriented as shown and labeled $\{a,b,c,d\} = E_+\Gamma$. We consider a graph-map $f\colon \Gamma\to\Gamma$ under which the edges of $\Gamma$ map to the combinatorial edge paths $f(a)=d$, $f(b)=a$, $f(c)=b^{-1}a$, and $f(d)=ba^{-1}db^{-1}ac$. It is straight froward to check that $f$ is a train track map.

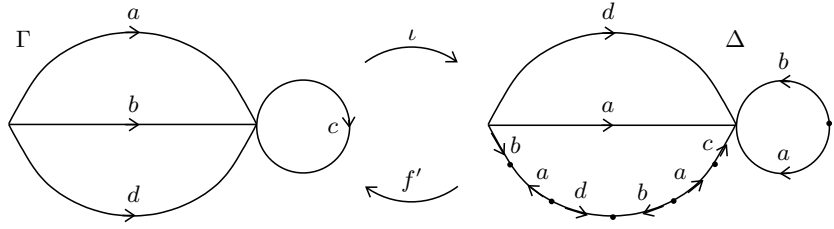
\begin{figure} [htb]
\labellist
\small\hair 2pt
\pinlabel $a$ [b] at 37 58
\pinlabel $b$ [b] at 37 31
\pinlabel $d$ [b] at 37 5
\pinlabel $c$ [r] at 98 28
\pinlabel $\iota$ [b] at 118 52
\pinlabel $f'$ [b] at 118 8
\pinlabel $d$ [b] at 175 58
\pinlabel $a$ [b] at 175 31
\pinlabel $b$ [l] at 145 23
\pinlabel $a$ [bl] at 153 12
\pinlabel $d$ [b] at 167 6
\pinlabel $b$ [b] at 185 5.5
\pinlabel $a$ [br] at 199 12
\pinlabel $c$ [r] at 207 23
\pinlabel $b$ [b] at 226 43.5
\pinlabel $a$ [b] at 226 17
\pinlabel $\Gamma$ at 5 54
\pinlabel $\Delta$ at 212 54
\endlabellist
\begin{center}
\includegraphics[height=2.7cm]{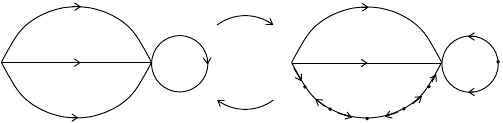} \caption{The example train track map. Left: Original graph $\Gamma$. Right: Subdivided graph $\Delta$ with labels. Our map $\f \colon \Gamma \to \Gamma$ is the composition of the ``identity'' $\iota\colon\Gamma\to\Delta$ with the map $f'\colon\Delta\to\Gamma$ that sends edges to edges preserving labels and orientations.}
\label{F:train_tracks}
\end{center}
\end{figure}

The graph map $f$ induces an automorphism $\fee = f_*$ of the free group $F_3 = F(\gamma_1,\gamma_2,\gamma_3)\cong \pi_1(\Gamma)$ generated by the loops $\gamma_1=b^{-1}a$, $\gamma_2=a^{-1}d$ and $\gamma_3=c$. Explicitly, the automorphism $\fee$ is given by $\fee(\gamma_1)=\gamma_2$, $\fee(\gamma_2)=\gamma_2^{-1}\gamma_1^{-1}\gamma_2\gamma_1\gamma_3$, and $\fee(\gamma_3)=\gamma_1$. Moreover, one may verify that $\fee$ is hyperbolic and fully irreducible.

Consider now the group extension $G = G_\fee$ defined by the presentation \eqref{eqn:semi-direct}. As this presentation contains relations $\gamma_2 = r\inv \gamma_1 r$ and $\gamma_3 = r \gamma_1 r\inv$, we see that $G$ is in fact generated by $\gamma_1$ and $r$. Writing $\gamma = \gamma_1$ and performing Tietze transformations to eliminate the generators $\gamma_2$ and $\gamma_3$, one may obtain the following two-generator one-relator presentation for $G$:
\begin{equation}
\label{eqn:example_presentation}
G=\langle \gamma, r \mid  \gamma^{-1} r \gamma^{-1} r  \gamma^{-1} r^{-1} \gamma r \gamma r \gamma r^{-3} = 1\rangle.
\end{equation}

\begin{center}
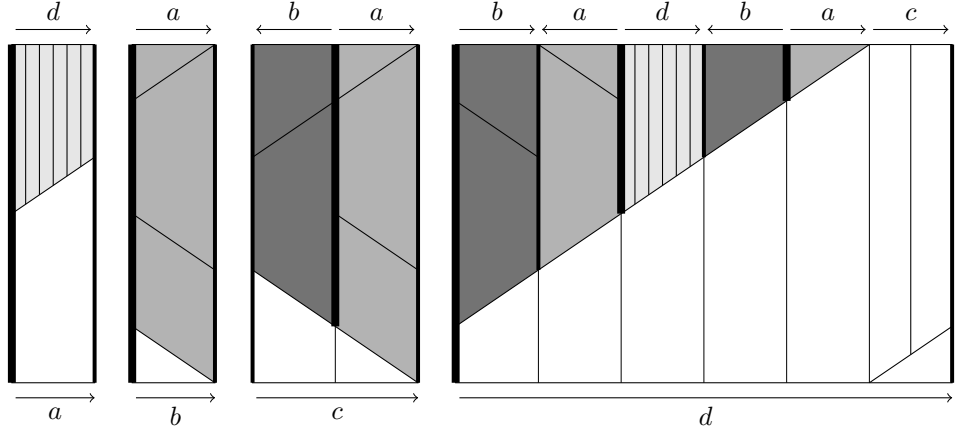
\begin{figure}[htb]
\begin{tikzpicture}
%

\pgfmathsetmacro\hstep{1.1}   
\pgfmathsetmacro\vstep{.75} 
\pgfmathsetmacro\hsp{.5}    
\pgfmathsetmacro\arsp{.05}  
\pgfmathsetmacro\varr{.2}    
\pgfmathsetmacro\subc{1.0/2.0} 
\pgfmathsetmacro\subd{1.0/6.0} 

\pgfmathsetmacro\vzero{0}                  
\pgfmathsetmacro\vone{\vzero   + 1*\vstep} 
\pgfmathsetmacro\vtwo{\vzero   + 2*\vstep} 
\pgfmathsetmacro\vthree{\vzero + 3*\vstep} 
\pgfmathsetmacro\vfour{\vzero  + 4*\vstep} 
\pgfmathsetmacro\vfive{\vzero  + 5*\vstep} 
\pgfmathsetmacro\vsix{\vzero   + 6*\vstep} 

\pgfmathsetmacro\azero{0}                  
\pgfmathsetmacro\aone{\azero + \hstep}     
\pgfmathsetmacro\bzero{\aone + \hsp}       
\pgfmathsetmacro\bone{\bzero + \hstep}     
\pgfmathsetmacro\czero{\bone + \hsp}       
\pgfmathsetmacro\cone{\czero + \hstep}     
\pgfmathsetmacro\ctwo{\czero + 2*\hstep}   
\pgfmathsetmacro\dzero{\ctwo  + \hsp}      
\pgfmathsetmacro\done{\dzero  + 1*\hstep}  
\pgfmathsetmacro\dtwo{\dzero  + 2*\hstep}  
\pgfmathsetmacro\dthree{\dzero+ 3*\hstep}  
\pgfmathsetmacro\dfour{\dzero + 4*\hstep}  
\pgfmathsetmacro\dfive{\dzero + 5*\hstep}  
\pgfmathsetmacro\dsix{\dzero  + 6*\hstep}  
 
\fill[black!30!white] 
(\bone,\vzero) -- (\bone,\vsix) -- (\bzero, \vsix) -- (\bzero,\vone)
(\ctwo,\vzero) -- (\ctwo,\vsix) -- (\cone, \vsix) -- (\cone,\vone)
(\done,\vtwo) -- (\done,\vsix) -- (\dtwo, \vsix) -- (\dtwo,\vthree)
(\dfour,\vfive) -- (\dfour,\vsix) -- (\dfive, \vsix);

\fill[black!55!white]
(\czero,\vtwo) -- (\czero,\vsix) -- (\cone,\vsix) -- (\cone,\vone)
(\done,\vtwo) -- (\done,\vsix) -- (\dzero,\vsix) -- (\dzero,\vone)
(\dthree,\vfour) -- (\dthree,\vsix) -- (\dfour,\vsix) -- (\dfour,\vfive);

\fill[black!10!white]
(\azero,\vthree) -- (\azero,\vsix) -- (\aone,\vsix) -- (\aone,\vfour)
(\dtwo,\vthree) -- (\dtwo,\vsix) -- (\dthree,\vsix) -- (\dthree,\vfour);

\draw (\azero,\vzero) rectangle (\aone,\vsix)
(\bzero,\vzero) rectangle (\bone,\vsix)
(\czero,\vzero) rectangle (\ctwo,\vsix)
(\dzero,\vzero) rectangle (\dsix,\vsix);

\draw [line width = 3]
(\azero,\vzero) -- (\azero,\vsix)
(\bzero,\vzero) -- (\bzero,\vsix)
(\dzero,\vzero) -- (\dzero,\vsix);
\draw [line width = 3] 
(\cone,\vone) -- (\cone,\vsix)
(\dtwo,\vthree) -- (\dtwo,\vsix)
(\dfour,\vfive) -- (\dfour,\vsix);
\draw [line width = 1.5]
(\aone,\vzero) -- (\aone,\vsix)
(\bone,\vzero) -- (\bone,\vsix)
(\czero,\vzero) -- (\czero,\vsix)
(\ctwo,\vzero) -- (\ctwo,\vsix)
(\dsix,\vzero) -- (\dsix,\vsix);
\draw [line width = 1.5]
(\done,\vtwo) -- (\done,\vsix)
(\dthree,\vfour) -- (\dthree,\vsix);
\draw [ultra thin] 
(\cone,\vzero) -- (\cone,\vone)
(\done,\vzero) -- (\done,\vtwo)
(\dtwo,\vzero) -- (\dtwo,\vthree)
(\dthree,\vzero) -- (\dthree,\vfour)
(\dfour,\vzero) -- (\dfour,\vfive)
(\dfive,\vzero) -- (\dfive,\vsix);

\draw 
(\azero+1*\subd*\hstep,\vthree+1*\subd*\vstep)--(\azero+1*\subd*\hstep,\vsix)
(\azero+2*\subd*\hstep,\vthree+2*\subd*\vstep)--(\azero+2*\subd*\hstep,\vsix)
(\azero+3*\subd*\hstep,\vthree+3*\subd*\vstep)--(\azero+3*\subd*\hstep,\vsix)
(\azero+4*\subd*\hstep,\vthree+4*\subd*\vstep)--(\azero+4*\subd*\hstep,\vsix)
(\azero+5*\subd*\hstep,\vthree+5*\subd*\vstep)--(\azero+5*\subd*\hstep,\vsix);

\draw 
(\dtwo+1*\subd*\hstep,\vthree+1*\subd*\vstep)--(\dtwo+1*\subd*\hstep,\vsix)
(\dtwo+2*\subd*\hstep,\vthree+2*\subd*\vstep)--(\dtwo+2*\subd*\hstep,\vsix)
(\dtwo+3*\subd*\hstep,\vthree+3*\subd*\vstep)--(\dtwo+3*\subd*\hstep,\vsix)
(\dtwo+4*\subd*\hstep,\vthree+4*\subd*\vstep)--(\dtwo+4*\subd*\hstep,\vsix)
(\dtwo+5*\subd*\hstep,\vthree+5*\subd*\vstep)--(\dtwo+5*\subd*\hstep,\vsix);
\draw
(\dfive+\subc*\hstep,\vzero+\subc*\vstep)--(\dfive+\subc*\hstep,\vsix);


\draw [->] (\azero+\arsp,\vzero -\varr)--node[below]{$a$}(\aone, \vzero -\varr); 
\draw [->] (\bzero+\arsp,\vzero -\varr)--node[below]{$b$}(\bone, \vzero -\varr); 
\draw [->] (\czero+\arsp,\vzero -\varr)--node[below]{$c$}(\ctwo, \vzero -\varr); 
\draw [->] (\dzero+\arsp,\vzero -\varr)--node[below]{$d$}(\dsix, \vzero -\varr); 
\draw [->] (\azero +\arsp,\vsix + \varr) -- node[above]{$d$} (\aone  -\arsp,\vsix + \varr);
\draw [->] (\bzero +\arsp,\vsix + \varr) -- node[above]{$a$} (\bone  -\arsp,\vsix + \varr);
\draw [<-] (\czero +\arsp,\vsix + \varr) -- node[above]{$b$} (\cone  -\arsp,\vsix + \varr);
\draw [->] (\cone  +\arsp,\vsix + \varr)  -- node[above]{$a$}(\ctwo  -\arsp,\vsix + \varr);
\draw [->] (\dzero +\arsp,\vsix + \varr) -- node[above]{$b$} (\done  -\arsp,\vsix + \varr);
\draw [<-] (\done  +\arsp,\vsix + \varr)  -- node[above]{$a$}(\dtwo  -\arsp,\vsix + \varr);
\draw [->] (\dtwo  +\arsp,\vsix + \varr)  -- node[above]{$d$}(\dthree-\arsp,\vsix + \varr);
\draw [<-] (\dthree+\arsp,\vsix + \varr) -- node[above]{$b$} (\dfour -\arsp,\vsix + \varr);
\draw [->] (\dfour +\arsp,\vsix + \varr) -- node[above]{$a$} (\dfive -\arsp,\vsix + \varr);
\draw [->] (\dfive +\arsp,\vsix + \varr) -- node[above]{$c$} (\dsix  -\arsp,\vsix + \varr);

\draw (\azero, \vthree) -- (\aone, \vfour); 
\draw (\bone, \vzero) -- (\bzero, \vone); 
\draw (\bone, \vtwo) -- (\bzero, \vthree); 
\draw (\bzero, \vfive) -- (\bone, \vsix); 
\draw (\ctwo, \vzero) -- (\czero, \vtwo); 
\draw (\ctwo, \vtwo) -- (\cone, \vthree); 
\draw (\czero, \vfour) -- (\ctwo, \vsix); 
\draw (\dzero,\vone) -- (\dfive,\vsix); 
\draw (\done,\vfour) -- (\dzero,\vfive); 
\draw (\dtwo,\vfive) -- (\done,\vsix); 

\draw (\dfive,\vzero)--(\dsix,\vone);
\end{tikzpicture}
\caption{The folded mapping torus $X_\f$, with its trapezoid cell structure, for the example $\f\colon \Gamma \to \Gamma$. The top is glued to the bottom as described by the labeling, and shaded cells with the same shading and shape are identified.}
\label{F:trapezoid_example}
\end{figure}
\end{center}

Using the specified train track representative $\f$ of $\fee$, we construct the corresponding folded mapping torus $X = X_\f$. This folded mapping torus $X$, equipped with its trapezoid cell structure, is illustrated in Figure \ref{F:trapezoid_example}. See the examples in \cite{DKL} for more details. The covering group $H \cong H_1(G;\Z)$ of the universal torsion-free abelian cover $\tX\to X$ is then just $\Z^2$. Moreover $H$ is freely generated by the image (under $G\to G^{ab}$) of the stable letter $r\in G$ and the image of $\gamma_1$ (which is also the image of $\gamma_2$ and of $\gamma_3$).
\end{example}

For the purposes of analyzing examples, we also note that the trapezoidal cell structure on $X$ is, in general, a subdivision of a cell structure with fewer cells. The $1$--cells of this {\em unsubdivided} cell structure are again either vertical or skew, and so may be oriented so that restriction of the orientation to a $1$--cell of the subdivision agrees with its original orientation. In particular, a positive (or nonnegative) $1$--cocycle for the canonical cell structure will give rise to one for the unsubdivided cell structure. We do not bother with a formal definition of this cell structure as it will only be used to simplify our discussion of examples, where it will be described.

\begin{example} \label{Ex:cell structure}
Let us employ the observations of the preceding paragraph to build a simplified cell structure on the folded mapping torus from Example \ref{E:introduce_running}.   This cell structure is shown in Figure~\ref{F:unsubdivided_torus} where we have removed the duplicate polygons to further simplify the picture.

\begin{figure}
\begin{center}
\includegraphics[width=10cm]{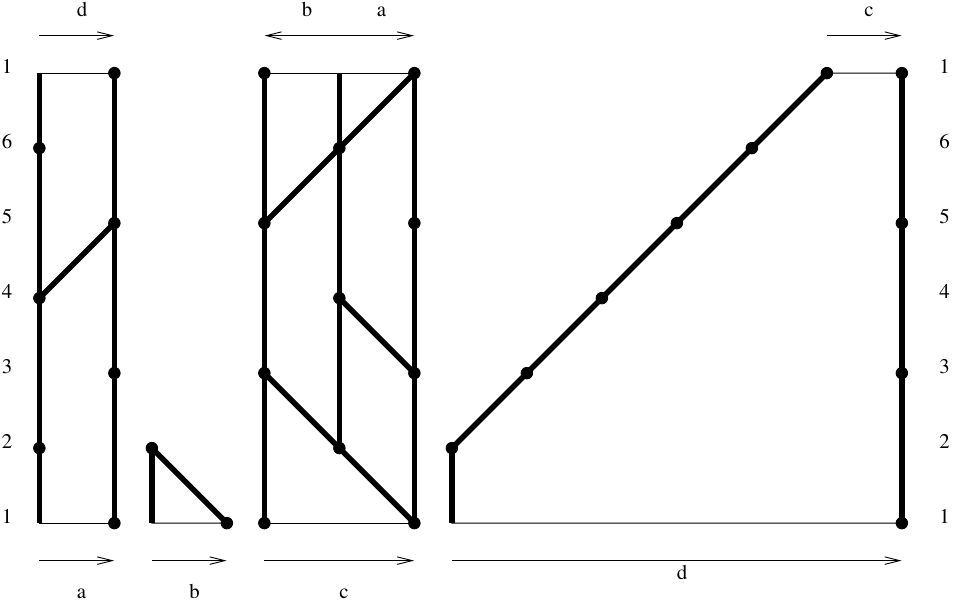} 
\label{F:01minimal}
\caption{A simpler picture of the folded mapping torus with the unsubdivided cell structure.  The six vertices $v_1,\ldots,v_6$ are those at the heights $1,\ldots,6$ as indicated.}
\label{F:unsubdivided_torus}
\end{center}
\end{figure}

There are just six vertices in the unsubdivided cells structure, and we label them $v_1,\ldots,v_6$ according to the heights as illustrated in Figure \ref{F:unsubdivided_torus}.  There are twelve $1$--cells, with at most one $1$--cell between every pair of vertices.  As such, we can label the oriented $1$--cells by their endpoints, and they are
\[ \left\{ \begin{array}{l} [v_1,v_3],[v_3,v_5],[v_5,v_1],[v_2,v_4],[v_4,v_6],[v_6,v_1],\\
 \quad [v_1,v_2],[v_2,v_3],[v_3,v_4],[v_4,v_5],[v_5,v_6],[v_6,v_1]  \end{array} \right\}. \]
We let $\{v_1^*,\ldots,v_6^*\}$ denote the dual basis of $0$--cochains.  Similarly, if $[v_i,v_j]$ is an oriented $1$--cell, we let $[v_i^*,v_j^*]$ denote the dual $1$--cochain, thus determining a basis of the $1$--cochains.

There are also six $2$--cells obtained by gluing the polygons in the figure. Orienting the $2$--cells, we can read off the vertices along the boundary.  Since there is at most one $1$--cell between any two vertices, this ordered list of vertices uniquely determines the loop in the $1$--skeleton which is the boundary of the $2$--cell. The boundaries of the $2$--cells are:
\[ \left\{ \begin{array}{l} (v_1,v_3,v_5,v_4,v_2,v_6),(v_1,v_2,v_6,v_5),(v_2,v_4,v_6,v_5,v_3),(v_1,v_3,v_4,v_2),\\
(v_3,v_5,v_1,v_6,v_4),(v_1,v_3,v_5,v_1,v_2,v_3,v_1,v_6,v_5,v_4,v_3,v_2,v_6,v_4,v_5) \end{array} \right\}. \]

\end{example}

\subsection{The BNS-invariant and the stretch function}
\label{sec:BNS_stretch}
Recall from \cite{BNS} that the \emph{BNS-invariant} of our finitely generated group $G$ is defined to be the subset $\BNS(G)$ of the sphere $S(G) = (H^1(G;\R)\setminus\{0\})/\R_+$ consisting of those directions $[u]$ for which $G'=[G,G]$ is finitely generated over a finitely generated submonoid of $\{g\in G : u(G)\geq 0\}$ (where here $G$ acts on $G'$ by conjugation---see the remarks following the Meta-Theorem). It follows from \cite[Proposition 4.3]{BNS} that for $u\in H^1(G;\R)$ primitive integral, the ray $[u]$ is in $\BNS(G)$ if an only if there exists $t\in G$ with $u(t)=1$ and a finitely generated subgroup $B\leq \ker(u)$ such that $t\inv Bt \leq B$ and $G = \langle B,t\rangle$. Defining $\phi\colon B\to B$ by $\phi(b) = t\inv b t$, we then see that $G$ splits as an ascending HNN-extension
\[G = B\ast_\phi = \langle B, t \mid t\inv b t = \phi(b)\text{ for all }b\in B\rangle\]
that is dual to $u$ (i.e., so that $u\colon G\to \Z$ is given by $B\mapsto 0$ and $t\mapsto 1$). Theorem 2.6 (with Remark 2.7) of \cite{GMSW} shows that in this case $B$ is necessarily a free group. Therefore $\phi$ is a free group endomorphism and we may consider its stretch factor $\lambda(\phi)$. By Proposition~\ref{P:unique_stretch_for_HNN}, $\lambda(\phi)$ is independent of the particular choices of $t\in G$ and $B\le \ker(u)$ and in fact only depends on $u$; thus we may unambiguously define $\Lambda(u) := \lambda(\phi)$.

\begin{defn}[Stretch function]
\label{D:stretch_function}
Let us define the \emph{rational BNS-cone of $G$} to be the set
\[\QBNS(G) = \{u\in H^1(G;\R) : [u]\in \BNS(G) \text{ and $u(G)$ is discrete}\}.\]
Given $u\in \QBNS(G)$, there is a unique $k>0$ so that $u'=ku$ is primitive integral, and we define $\Lambda(u):=\Lambda(u')^k$ (with $\Lambda(u')$ as defined above). This gives a well-defined function
\[\Lambda\colon \QBNS(G)\to \R_+\]
that we call the \emph{stretch function of $G$}. It is a canonical invariant of the free-by-cyclic group $G$ and satisfies the homogeneity property $\Lambda(ku) = \Lambda(u)^{1/k}$ for all $k>0$. 
\end{defn}

One of the main goals of this paper is to understand and explicitly compute $\Lambda$ on a large part of $\QBNS(G)$. Along the way, will need the following crucial feature of the stretch  function:

\begin{proposition}
\label{prop:stretch_is_locally_bounded}
The stretch function $\Lambda$ is locally bounded. That is, for all $u\in \QBNS(G)$, there exists a neighborhood $U\subset \QBNS(G)$ of $u$ and a finite number $R > 0$ so that $\Lambda(u')\leq R$ for all $u'\in U$.
\end{proposition}

This result will be obtained by expanding on some ideas in \cite{DKL}. As these considerations are somewhat far afield of our current discussion, the proof of Proposition~\ref{prop:stretch_is_locally_bounded} is relegated to Appendix~\ref{S:bounding_stretch}.

\section{The module of transversals and the McMullen polynomial}
\label{S:foliation}

The flowlines of $\tflow$ intersected with each trapezoidal $2$--cell determine a $1$--dimensional foliation of $\tX$. The two skew $1$--cells of each trapezoid are transverse to the flowlines, and the vertical $1$--cells are arcs of flowlines. The arcs of flowlines in a $2$--cell will be called {\em plaque arcs}.  A maximal, path connected, countable union of plaque arcs will be called a {\em leaf}. We will refer to this decomposition of $\tX$ into leaves as a {\em foliation} of $\tX$ and denote it $\mathcal F$ (we also use $\mathcal F$ to denote the actual foliation of any $2$--cell). This foliation descends to a foliation on $X$ for which the leaves are the images of the leaves in $\tX$, and by an abuse of notation we also refer to this foliation as $\mathcal F$ whenever it is convenient to do so.

Recall that the union of the vertical $1$--cells is preserved by $\flow$. Moreover, since the vertices of $\Gamma\subset X$ all lie on vertical $1$--cells, and since the preimages of these vertices under all powers of $f$ form a dense subset of $\Gamma$, it follows that the set of points that eventually flow into vertical $1$--cells form a dense subset of $\tX$. This set is, by definition, a union of leaves, and we refer to these as the \emph{vertex leaves} of $\mathcal F$.

\begin{defn}[Transversal]
\label{D:the_transversals}
A {\em transversal} $\tau$ to $\mathcal F$ is an arc contained in a $2$--cell of $\tX$ which is transverse to the foliation $\mathcal F$ and has both endpoints contained in vertex leaves.  We do not view a single point as an arc.
\end{defn}

Let $\mathbb F(\mathcal F)$ denote the free $\Z$--module generated by all transversals $\tau$ to $\mathcal F$.  Following McMullen, we define the module of transversals to $\mathcal F$, denoted $T(\mathcal F)$, 
to be the largest quotient $\mathbb F(\mathcal F) \to T(\mathcal F)$ in which the images $[\tau]$ of transversals $\tau$ satisfy the following {\em basic relations}:
\begin{enumerate}
\item $[\tau] = [\tau_1] +[\tau_2]$, if $\tau = \tau_1 \cup \tau_2$ and $\tau_1 \cap \tau_2$ is a single point in a vertex leaf, 
\item $[\tau] =[\tau']$, if $\tau$ flows homeomorphically onto $\tau'$ in the sense that there exists a continuous, nonnegative real function $\nu\colon \tau \to [0,\infty)$ so that $w\mapsto \tpsi_{\nu(w)}(w)$ defines a homeomorphism from $\tau$ onto $\tau'$. 
\end{enumerate}
More precisely, we make the following definition:

\begin{defn}[Module of transversals]
\label{D:transversals module}
Let $R \le \mathbb F(\mathcal F)$ be the submodule generated by all elements $\tau - \tau_1 - \tau_2$ and $\tau - \tau'$, with $\tau, \tau', \tau_1, \tau_2$ as in the basic relations (1)--(2) above. The {\em module of $\mathcal F$} is defined to be the quotient $T(\mathcal F) =\mathbb F(\mathcal F)/R$.
\end{defn}

The covering group $H$ acts on the set of transversals by taking preimages:  given an element $h \in H$ and transversal $\tau$, we have $h \cdot \tau = h^{-1}(\tau)$.  This is naturally a right action, and it is sometimes convenient to write
\[  \tau \cdot h =  h \cdot \tau = h^{-1}(\tau).\]
Since $H$ is abelian, the distinction between left versus right is not important, but taking preimages (as opposed to images) will be important.
This makes both $\mathbb F(\mathcal F)$ and $T(\mathcal F)$ into $\Z[H]$--modules, where $\Z[H]$ is the integral group ring of $H$. 
We also note that the quotient $\mathbb F(\mathcal F) \to T(\mathcal F)$ is a $\Z[H]$--module homomorphism (not just a $\Z$--module homomorphism).

\subsection{The McMullen polynomial}
\label{S:McMullen_polynomial}

We now define a multivariable polynomial which is analogous to the Teichm\"uller polynomial defined by McMullen in the $3$--manifold setting \cite{Mc}. In later sections we will see that this polynomial invariant encodes much information about cross sections to $\flow$ and various splittings of $G$. The definition relies on the following proposition, which will follow from the results in \S\ref{secc:module_isom}. 

\begin{proposition} \label{P:fgmodule}
The $\Z[H]$--module $T(\mathcal F)$ is finitely presented.
\end{proposition}

Choose any finite presentation of $T(\mathcal F)$ as a $\Z[H]$--module, say with $m$ generators and $r$ relations. This gives an exact sequence
\[ \xymatrix{ \Z[H]^r \ar[r]^{D} &  \Z[H]^m \ar[r] &T(\mathcal F) \ar[r] & 0 } \]
where $D$ is an $m\times r$ matrix with entries in $\Z[H]$. Recall that the {\em fitting ideal of $T(\mathcal F)$} is the ideal $\mathcal I \leq \Z[H]$ generated by all $m \times m$ minors of $D$, and that this ideal is independent of the chosen finite presentation of $T(\mathcal F)$ \cite[Ch XIII, \S10]{Lang} \cite{Northcott}.

\begin{defn}[The McMullen polynomial]
Define $\poly\in \Z[H]$ to be the g.c.d.~of the fitting ideal $\mathcal I$ of $T(\mathcal F)$, which is well-defined up to multiplication by units in $\Z[H]$ (note that $\Z[H]$ is a unique factorization domain). Explicitly, if $p_1,\ldots,p_k$ denote the minors generating $\mathcal I$, then we have
\[ \poly =\mathrm{gcd}\{ p \in \mathcal I \} = \mathrm{gcd}\{ p_1,\ldots,p_k \}.\]
Viewing $\Z[H]$ as the ring of integral Laurent polynomials in $b$ variables, we can think of $\poly$ as a Laurent polynomial which we call the {\em McMullen polynomial} of $(X,\psi)$. Note that this definition depends only on $\tX$, $\mathcal F$,  and the action of $H$.
\end{defn}

In the process of proving Proposition \ref{P:fgmodule} we will see that $\poly$ enjoys many of the properties that McMullen's Teichm\"uller polynomial does for $3$--manifolds.

\section{Cross sections, closed $1$--forms, and cohomology}
\label{S:x-sections_1-forms_H^1}

We will see that the McMullen polynomial is intimately related to the cross sections of the semiflow $\flow$. In fact, these cross sections will play a crucial role in our analysis and understanding of $\poly$. 
To this end, we discuss here the definitions and general properties of cross sections and the duality between cross sections and cohomology.   Along the way we recall the notion of a closed $1$--form on a topological space, and describe a class of closed $1$--forms which interact well with $\flow$.  We then introduce the cone of sections $\Csec$ and give two descriptions of it---one in terms of cross sections and the other in terms of closed $1$--forms.  In Appendix~\ref{S:characterizing_sections} we will describe a procedure for explicitly building cross sections dual to certain cohomology classes.  This provides a characterization of the classes in $\Csec$ in combinatorial terms and results in a concrete description of $\Csec$ allowing us to prove Theorem \ref{T:cone_of_sections}.

\subsection{Cross sections and flow-regular maps}
\label{S:cross_sections}

Given an open subset $W \subset X$ (or $\tX$), we say that a continuous map $\eta' \colon W \to Y$ with $Y = \sone$ or $Y=\R$ is {\em flow-regular} if for any $\xi \in X$ the map $\{ s \in \R_{\geq 0} \mid \psi_s(\xi) \in W \}\to Y$ defined by $s \mapsto \eta'(\psi_s(\xi))$ is an orientation preserving local diffeomorphism. With this terminology, we say that a finite embedded graph $\Theta \subset X$ is {\em transverse to $\psi$} if there is a neighborhood $W$ of $\Theta$ and a flow-regular map $\eta' \colon W \to \sone$ so that $\Theta = (\eta')^{-1}(x_0)$ for some $x_0 \in \sone$.

\begin{defn}[Cross section]
A finite embedded graph $\Theta \subset X$ which is transverse to $\psi$ is called  a \emph{cross section} (or simply \emph{section}) of $\flow$
if every flowline intersects $\Theta$ infinitely often, that is, if $\{s\in \R_{\ge 0} \mid \flow_s(\xi)\in \Theta\}$ is unbounded for every $\xi\in X$. The cross section then has an associated \emph{first return map} $f_\Theta\colon\Theta\to\Theta$, which is the continuous map defined by sending $\xi\in \Theta$ to $\flow_{T(\xi)}(\xi)$, where $T(\xi)$ is the minimum of the set $\{s > 0 \mid \flow_s(\xi)\in \Theta\}$.
\end{defn}

If  $\Theta$ is a cross section, there exists a reparameterization 
$\flow^{\Theta}$ of the semiflow $\flow$ such that this first return map is exactly the restriction of the time--$1$ map $\flow^{\Theta}_1$ to $\Theta$. Explicitly, if $K$ is larger than $T(\xi)$ for all $\xi\in \Theta$, then we may reparameterize $\flow$ inside the flow-regular neighborhood $W$ to obtain a semiflow $\flow'$ for which the first return map of $\Theta$ to itself is exactly $\flow'_K$. The desired reparameterization is then given by $\flow_t^{\Theta} = \flow'_{t/K}$.

\begin{prop-defn}[Flow-regular maps representing cross sections] \label{PD:flow regular to sone}
For any cross section  $\Theta\subset X$  there exists a flow-regular map $\fib_\Theta\colon X\to \sone$ such that  $\Theta = \fib_\Theta\inv(0)$. (Thus the flow-regular map $W\to \sone$ witnessing the fact that $\Theta$ is transverse to $\flow$ may in fact be taken to have domain all of $X$). We say that  any such $\fib_\Theta$  \emph{represents}  $\Theta$.
\end{prop-defn}
\begin{proof}
Note that the function 
\[\tau(\xi) = \min\{s > 0 \mid \flow^{\Theta}_s(\xi) \in \Theta\}\qquad\text{for $\xi\in X$}\]
is bounded above. Indeed, $\tau \equiv 1$ on $\Theta$ by construction, and the fact that $\Theta$ is transverse to $\flow$ implies that $\tau$ is also bounded on an open neighborhood $U$ of $\Theta$. 
As $\tau$ is continuous and therefore bounded on the compact set $X\setminus U$, the global boundedness of $\tau$ follows. This implies that every biinfinite flowline (that is, a map $\gamma\colon \R\to X$ satisfying $\flow^{\Theta}_s(\gamma(t)) = \gamma(s+t)$) must intersect $\Theta$ infinitely often in the backwards direction. Since every point of $X$ lies on a biinfinite flowline, it follows that
\begin{equation}
\label{eqn:flow_union_of_sections}
X = \bigcup_{s\geq 0}\flow^{\Theta}_s(\Theta) = \bigcup_{0 \leq s \leq 1} \flow_s^{\Theta}(\Theta).
\end{equation}
In particular, we see that $\tau(\xi) \leq 1$ for all $\xi\in X$ and that $\tau(\xi) = 1$ if and only if $\xi\in \Theta$. The assignment $\xi\mapsto 1-\tau(\xi)$, which is continuous on $X\setminus\Theta$, thus descends to a continuous map $\fib_\Theta\colon X\to \R/\Z = \sone$ that satisfies $\fib_\Theta(\flow^{\Theta}_s(\xi)) = s + \fib_\Theta(\xi)$ for all $\xi\in X$ and $s\geq 0$. This map $\fib_\Theta$ is a local isometry on each $\flow^{\Theta}$--flowline, and we also have $\Theta = \fib_\Theta\inv(0)$ by construction.  Since $\flow^{\Theta}$ is just a reparameterization of $\flow$, it follows that $\fib_\Theta$  is also flow-regular with respect to the original flow $\flow$.
\end{proof}

\subsection{Cross sections and cohomology}
\label{sec:sections_and_cohomology}

Every cross section $\Theta$ of $\flow$ determines a homomorphism $[\Theta]\colon G\to \Z$ as follows: Let $\eta'\colon W\to \sone$ be a flow-regular map for which $\Theta = (\eta')^{-1}(x_0)$ for some $x_0\in \sone$. Taking a sufficiently small neighborhood $W'\subset W$ of $\Theta$, any closed loop $\gamma\colon \sone\to X$ may then be homotoped so that $t\mapsto \eta'(\gamma(t))$ is a local homeomorphism on $\gamma\inv(W')$. (In fact, one may perform the homotopy by applying $\flow$ judiciously inside $W$ to arrange, for example, that each component of $\gamma\cap W'$ is an arc of a flowline). The value of $[\Theta]$ on $\gamma$ is then equal to the number of components of $\gamma\inv(W')$ on which $\eta'\circ\gamma$ is orientation preserving, minus the number on which it is orientation reversing.   Alternatively, if $\fib_\Theta \colon X \to \sone$ represents $\Theta$ as in Proposition-Definition~\ref{PD:flow regular to sone}, then $[\Theta] = (\eta_\Theta)_*$.

\begin{defn}[Duality]
\label{D:dual_class}
The cross section $\Theta$ and corresponding integral cohomology class $[\Theta]\in H^1(X;\Z)$ are said to be \emph{dual} to each other.
\end{defn}

For any cell structure $Y$ on $X$, there is a natural way to represent the class $[\Theta]$ by a cellular $1$--cocycle $z\in Z^1(Y;\Z)$: 
Adjusting $\Theta$ by a homotopy if necessary, first choose a small neighborhood $W'\subset W$ of $\Theta$, as above, such that $W'$ is disjoint from the $0$--skeleton $Y^{(0)}$. For each $1$--cell $\sigma$ of $Y$, one may then find an arc $\gamma\colon [0,1]\to Y$ that is homotopic to $\sigma$ rel $\partial \sigma$ and for which the assignment $t\mapsto \eta'(\gamma(t))$ is a local homeomorphism on each component of $\gamma\inv(W')$. The value of the cocycle $z$ on $\sigma$ is then defined, as above, to be the number of components of $\gamma\inv(W')$ on which $\eta'\circ \gamma$ is orientation preserving minus the number on which it is orientation reversing. By definition, $z$ then agrees with $[\Theta]$ on any $1$--cycle representing a closed loop in $X$. Therefore we see that $z$ is in fact a cocycle (since the boundary $\partial T$ of any $2$--cell $T$ is nullhomotopic and thus satisfies $z(\partial T) = [\Theta](\partial T) = 0$) and that the cohomology class of $z$ is equal to  $[\Theta]$.

\begin{proposition}
\label{P:connected_iff_primitive} 
A cross section $\Theta$ of $\flow$ is connected if and only if its dual cohomology class $[\Theta]$ is primitive.
\end{proposition}
\begin{proof}
Let $f_\Theta\colon\Theta\to\Theta$ denote the first return map. First suppose that $\Theta$ is a disjoint union of $k> 1$ connected components $\Theta_1,\dotsc,\Theta_k$. By continuity, $f_\Theta$ must map each $\Theta_i$ into some other connected component $\Theta_j$ and thus determines a self-map $\varsigma$ of the set $\{1,\dotsc,k\}$. 

We claim that $\varsigma$ must be surjective. To see this, suppose by contradiction that $1$ was not in the image of $\varsigma$. It then follows that every flowline $\{\flow_s(\xi) \mid s\geq 0, \xi\in X\}$ intersects the graph $\Theta_1$ at most once, for otherwise there would be some point of $\Theta_1$ that was mapped back into $\Theta_1$ by some iterate of $f_\Theta$. This fact shows that the subgraph $\Theta' = \Theta_2 \cup \dotsb \cup \Theta_k$, which is automatically transverse to $\flow$, is also a section of $\flow$. Equation \eqref{eqn:flow_union_of_sections} then shows that $X = \cup_{s\ge 0} \flow_s(\Theta')$;  in particular, there exist a point $\xi\in \Theta'$ and a time $s\ge 0$ for which $\flow_s(\xi)\in \Theta_1$. This contradicts the assumption that $1$ is not in the image of $\varsigma$.

Since $\varsigma$ is surjective it is also injective and therefore a permutation of the set $\{1,\dotsc,k\}$. In fact it must be a cyclic permutation: For each $i$, the set $X_i$ of points $\xi\in X$ that eventually flow into $\Theta_i$ is connected. If the action of $\varsigma$ partitioned $\{1,\dotsc,k\}$ into multiple orbits, then the corresponding sets $X_i$ would give a nontrivial decomposition of $X$ into disjoint closed subsets, contradicting the connectedness of $X$. 
It now follows that each component $\Theta_i$ of $\Theta$ is it itself a cross section and so determines a dual cohomology class $[\Theta_i]$. However, it is easy to see that these classes are all equal (for example, by using the fact that $\flow^{\Theta}_1(\Theta_i) = \Theta_{\varsigma(i)}$) and thus that their sum $[\Theta]_1+\dotsb + [\Theta_k]$, which necessarily equals $[\Theta]$, is not primitive.

Conversely, suppose that $\Theta$ is connected. Choose a point $\xi\in \Theta$ and consider the loop $\gamma\subset X$ obtained by concatenating the flowline from $\xi$ to $f_\Theta(\xi)$ with a path in $\Theta$ back to $\xi$. Taking a homotopic loop that is transverse to $\Theta$ (for example, $\flow_\epsilon(\gamma)$ for some small $\epsilon > 0$) we see that $\gamma$ intersects $\Theta$ once with positive orientation. Therefore $[\Theta](\gamma) = 1$ showing that $[\Theta]$ is primitive.
\end{proof}

\begin{example}[The cocycle $z_1$]
\label{Ex:(-1,2) introduced}
Let us construct a cross section and dual cohomology class for the running example. Using the unsubdivided cell structure introduced in Example~\ref{Ex:cell structure} and illustrated in Figure~\ref{F:01minimal}, consider the cocycle
\[ z_1 = [v_2^*,v_3^*] + [v_3^*,v_5^*] + 2[v_4^*,v_6^*] + [v_5^*,v_1^*] + [v_6^*,v_2^*].\]
Evaluating this expression on the boundary of each $2$--cell shows that it is indeed a $1$--cocycle,
and we may construct a cross section dual to the class $[z_1]$ as follows: Place $z_1(\sigma)$ vertices along each $1$--cell $\sigma$ of $X$, so in our case the cells $[v_2,v_3]$, $[v_3,v_5]$, $[v_5, v_1]$, and $[v_6,v_2]$ each get one vertex and $[v_4,v_6]$ gets two vertices. Since $z_1$ satisfies the cocycle condition, it is possible to ``connect the dots'' in each $2$--cell yielding a graph $\Theta_1$ intersecting $X^{(1)}$ at exactly these points. In this case it is moreover possible, as we have illustrated in Figure~\ref{F:(-1,2)introduced}, to construct $\Theta_1$ so that it is transverse to the flow. Every flowline of $\flow$ is seen to hit $\Theta_1$ infinitely often by inspection, so $\Theta_1$ is in fact a cross section. Applying the recipe following Definition~\ref{D:dual_class} to $\Theta_1$ yields exactly the cocycle $z_1$, so $\Theta_1$ is indeed dual to $[z_1]$ as desired.

\begin{figure}
\begin{center}
\includegraphics[width=13cm]{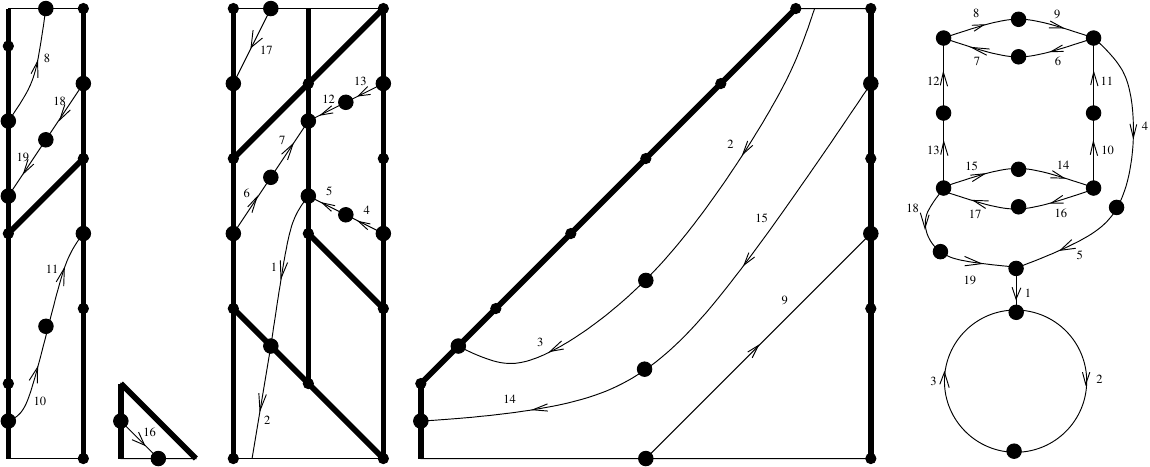}
\caption{A cross section $\Theta_1$ dual to the class $[z_1]$. The vertices of $X$ are small and the vertices of $\Theta_1$ are larger.}
\label{F:(-1,2)introduced}
\end{center}
\end{figure}

The abstract graph $\Theta_1$ is illustrated at the right of Figure~\ref{F:(-1,2)introduced}. We note that here we have used the the standard graph structure of Definition~\ref{D:standard_graph} to ensure that the first return map sends vertices to vertices. Following flowlines, we can then calculate the first return map $f_1 = f_{\Theta_1} \colon \Theta_1 \to \Theta_1$ as indicated in Figure \ref{F:(-1,2)continued1}. We thus find that the characteristic polynomial of the transition matrix $A(f_1)$ is
\[\zeta^{10}(\zeta^{9} - \zeta^{5} - \zeta^{4} - \zeta^{3} - \zeta^{2} - \zeta - 2).\]
One may also verify directly, or by using Peter Brinkmann's software package {\rm xtrain}\footnote{Available at http://gitorious.org/xtrain},  that $f_1$ is a train track map (c.f. Proposition~\ref{P:first_return_exp_irred_tt} below). 

By applying Stallings folds to the graph in the middle we produce the graph on the right. This folding operation $\kappa$ is a homotopy equivalence, and $f_1$ factors as the composition of $\kappa$ with the locally injective immersion from the graph on the right into $\Theta_1$. Since this immersion is locally injective but not injective, we see see that $(f_1)_*$ is an injective endomorphism that is {\em not} an automorphism.
\begin{figure} 
\labellist
\small\hair 2pt
\pinlabel $\kappa$ [b] at 310 185
\endlabellist
\begin{center}
\includegraphics[width=12cm]{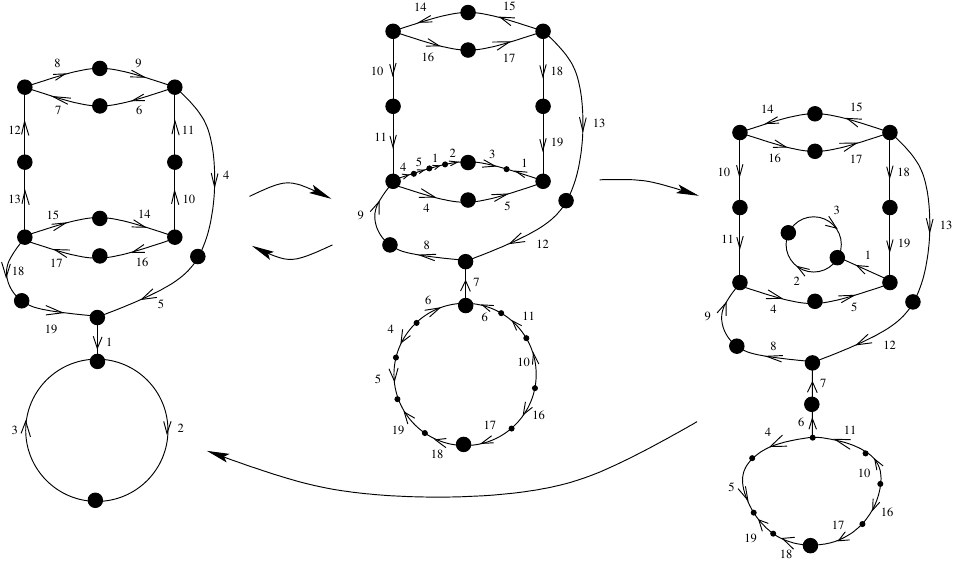}
\caption{The first return map map $f_1 \colon \Theta_1 \to \Theta_1$ for the section $\Theta_1$ dual to $[z_1]$ and the result of Stallings folds $\kappa$ on the right.}
\label{F:(-1,2)continued1}
\end{center}
\end{figure}
\end{example}

\begin{example}[The cocycle $z_2$]
\label{Ex:(-1,1) introduced}
Here we exhibit another cross section for the running example. Consider the graph $\Theta_2$ illustrated in Figure~\ref{F:(-1,1)introduced}. As in Example~\ref{Ex:(-1,2) introduced} above, one may verify by inspection that $\Theta_2$ is indeed a cross section and that its dual cohomology class $[\Theta_2]$ is represented by by the cocycle
\[z_2 = [v_1^*,v_2^*] + [v_2^*, v_3^*] - [v_4^*, v_5^*] - [v_5^*, v_6^*] + [v_1^*,v_3^*] + 2[v_6^*,v_2^*].\]
Notice that the cells $[v_4,v_5]$ and $[v_5,v_6]$ cross $\Theta_2$ ``against the flow'', that is, in the opposite direction as the flowlines do. The abstract graph $\Theta_2$ along with its first return map $f_2$ are illustrated in Figure~\ref{F:(-1,1)first_return}. Here we have again used the standard graph structure from Definition~\ref{D:standard_graph} to ensure that $f_2$ sends vertices to vertices. Applying Stallings folds to the graph on the right produces a graph with rank $4$, showing that $(f_2)_*$ is neither injective nor surjective. We have also verified, using  {\rm xtrain}, that $f_2$ is an expanding irreducible train track map and calculated that the characteristic polynomial of $A(f_2)$ is $\zeta^{13}(\zeta^{6} - 3\zeta^{3} - 3\zeta - 1)$.
\begin{figure}
\begin{center}
\includegraphics[width=11cm]{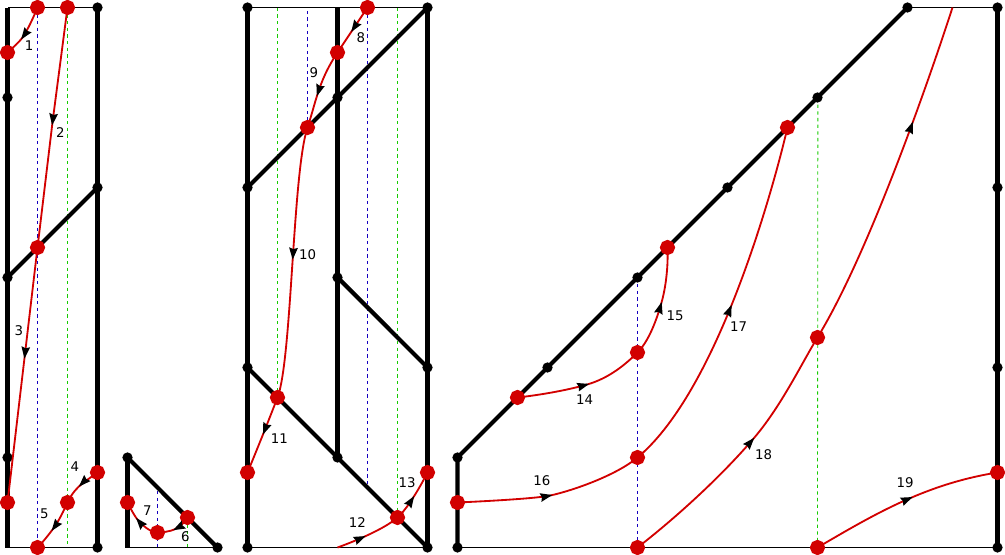}
\caption{A cross section $\Theta_2$ dual to the class $[z_2]$. Again the vertices of $X$ are small and the vertices of $\Theta_2$ are larger.}
\label{F:(-1,1)introduced}
\end{center}
\end{figure}
\begin{figure}
\begin{center}
\includegraphics[width=11cm]{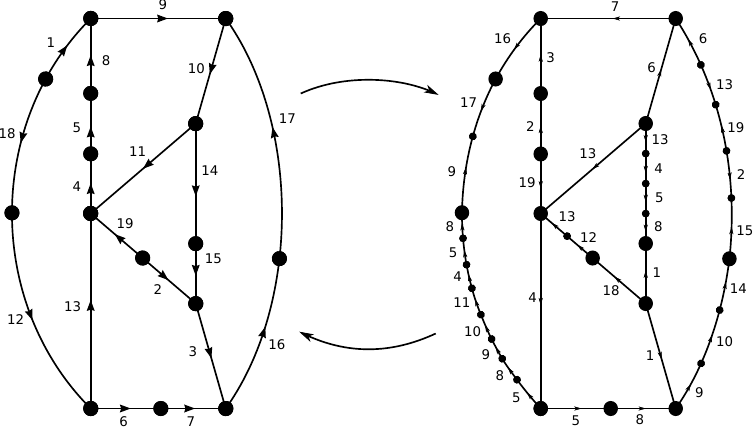}
\caption{The first return map $f_2\colon \Theta_2\to\Theta_2$.}
\label{F:(-1,1)first_return}
\end{center}
\end{figure}
\end{example}

\subsection{Closed $1$--forms and cohomology}
\label{S:closed 1-forms cohomology}

Cross sections provide a useful tool for analyzing certain cohomology classes, however working with cross sections confines us to studying only the integral classes. To consider more general cohomology classes we  must recall some additional machinery borrowed from differential geometry.

Following \cite{FGS10}, a {\em closed $1$--form} $\omega$ on a topological space $Y$ is a collection of functions
\[ \omega = \{\omega_U \colon U \to \R\}_{U \in \mathcal J}\]
defined on some open cover $\mathcal J$ of $Y$ with the compatibility property that if $U \cap V \neq \emptyset$ then there exist a locally constant function $\Delta_{U,V}$ so that on $U \cap V$, $\omega_V(\xi) = \omega_U(\xi) + \Delta_{U,V}(\xi)$.    We assume that the collection of functions and open sets are maximal in the sense that any other function on an open set $W \to \R$ that also satisfies this compatibility property is already in $\omega$.  Any collection of functions satisfying the compatibility condition is contained in a maximal collection which is unique up to addition of locally constant functions, so we typically ignore this technicality when no confusion arises.

Given a continuous path $\delta \colon [a,b] \to Y$, take a partition $a = a_0 < a_1 <\dotsb < a_k = b$ such that for all $j = 1,\ldots,k$, $\delta([a_{j-1},a_j]) \subset U_j$ for some $U_j \in \mathcal J$.  Then we can integrate $\omega$ over $\delta$:
\[ \int_\delta \omega = \sum_{j=1}^k \omega_{U_j}(a_j) - \omega_{U_j}(a_{j-1}).\]
This integral is independent of the covering used and does not change if we homotope $\delta$ relative to its endpoints.  In particular, the integral over a null-homologous loop is zero, and in this way a closed $1$--form defines a $1$--dimensional cohomology class.  If $\omega$ represents the cohomology class  $u$, then we will sometimes write $\omega = \omega^u$ or $u = [\omega]$.  Another consequence of these observations is that if $U$ is a simply connected open subset of $Y$ then by the maximality condition on the functions defining $\omega$, $U$ is a domain for a function $\omega_U \in \omega$. 

Closed $1$--forms can be added and multiplied by scalars (by applying these operations to the defining functions on simply connected domains, for example), and hence the closed $1$--forms on $Y$ form a real vector space.  If $Y$ is a CW-complex then the map from closed $1$--forms to $H^1(Y;\R)$ is an epimorphism.

Given a continuous function $\beta \colon Y' \to Y$ between topological spaces and a closed $1$--form $\omega$ on $Y$, we can pull $\omega$ back by $\beta$ to a closed $1$--form $\beta^*(\omega)$.  This is the collection of functions obtained from the compositions
\[ \omega_U \circ \beta|_{\beta^{-1}(U)} \colon \beta^{-1}(U) \to \R \]
and extended to a maximal collection.  

For the universal torsion-free abelian cover $p \colon \tX \to X$ and any closed $1$--form $\omega^u$ on $X$ representing $u \in H^1(X;\R)$ the pull-back $\widetilde \omega^u = p^*(\omega^u)$ contains a globally defined function $\widetilde \eta_u \colon \tX \to \R \in \widetilde \omega^u$.  Composing with any covering transformation $h \colon \tX \to \tX$, $h \in H$, gives rise to another such function $\widetilde \eta_u \circ h$, and these differ by translation by $u(h)$
\[ \widetilde \eta_u \circ h (\xi) = \widetilde \eta_u(\xi) + u(h).\]
That $\widetilde \eta_u$ and $\widetilde \eta_u \circ h$ differ by a translation follows from the definition of a closed $1$--form.  To see that the translation is by $u(h)$, take a loop $\gamma \colon [0,1] \to X$ in $G = \pi_1(X)$ representing the homology class $h$.  The covering transformation $h$ sends the initial point $\xi_0 = \tilde \gamma(0)$ of a lift of $\gamma$ to $\tX$ to the terminal point $\tilde \gamma(1) = h(\xi_0)$ and so
\[ \widetilde \eta_u \circ h(\xi_0) - \widetilde \eta_u (\xi_0) = \widetilde \eta_u (\tilde \gamma(1)) - \widetilde \eta_u(\tilde \gamma(0)) = \int_{\tilde \gamma} \tilde \omega^u = \int_\gamma \omega^u = u(\gamma) = u(h).\]
Conversely, given a function $\widetilde \eta_u \colon \tX \to \R$ equivariant with respect to the action of $H$ on $\R$ by translation determined by the homomorphism $u \colon H \to \R$, it is easy to see that this function belongs to the pull-back $p^*(\omega^u)$ for a unique closed $1$--form $\omega^u$ on $X$ representing $u$.

\subsection{Flow-regular closed $1$--forms and sections}
\label{S:flow-regular closed 1-forms and sections}

We will be interested in closed $1$--forms that interact with $\flow$ in some meaningful way.  To that end, we make the following

\begin{defn}[Flow-regular closed $1$--form]
Say that a closed $1$--form $\omega$ is {\em flow-regular} provided all the functions $\omega_U \in \omega$ are flow-regular.  If $\gamma_\xi(s) = \flow_s(\xi)$ is a flowline, then the {\em $\omega$--height} of the arc $\gamma_\xi([0,s_0])$ of the flowline is defined to be
\[ \int_{\gamma_\xi([0,s_0])} \omega. \]
We say that a flow-regular closed $1$--form $\omega$ is {\em tame}, if the restriction to every $2$--cell is smooth and if the restriction to each skew $1$--cell has finitely many critical points.
\end{defn}
Given a flow-regular closed $1$--form $\omega$, it is possible to construct another one which is tame and which represents the same cohomology class.  The point is to first adjust $\omega$ along the skew $1$--cells so that it is smooth with finitely many critical points, but without changing the integral of $\omega$ over the $1$--cell.  Then we further adjust the $1$--form over every $2$--cell.  When convenient we will pass from an arbitrary flow-regular closed $1$--form to a tame one.

For every section $\Theta$, the integral class $[\Theta]$ may be represented by a closed flow-regular $1$--form $\omega^{[\Theta]}$.  
Indeed,  consider a lift of the flow-regular map $\fib_{\Theta} \colon X \to \sone$  to a map
\[ \widetilde \fib_{\Theta} \colon \tX \to \R.\]
Then for every $h \in H$ we have $\widetilde \fib_{\Theta} \circ h = \widetilde \fib_{\Theta} + [\Theta](h)$, and hence $\widetilde \fib_{\Theta}$ determines a closed $1$--form $\omega^{[\Theta]}$ on $X$.  Since $\fib_{\Theta}$ is flow regular, so is $\widetilde \fib_{\Theta}$ and hence also $\omega^{[\Theta]}$.

Conversely, a flow-regular closed $1$--form $\omega^{[\Theta]}$ representing a integral class determines a flow-regular map $\fib_{\Theta} \colon X \to \sone$ by integrating along paths from a fixed basepoint, and any fiber of this map is a section.

Given this correspondence between sections and integral classes represented by closed flow-regular $1$--forms, the following definition is quite natural.
\begin{defn}[Cone of sections]
\label{D:cone_sections}
Let $\Csec\subset H^1(X;\R)$ denote the set of all cohomology classes represented by flow-regular closed $1$--forms.  We refer to $\Csec$ as the \emph{cone of sections.}
\end{defn}

The discussion above proves the following.
\begin{proposition}
\label{P:integral_classesin_Csec}
The integral classes in $\Csec$ are precisely the set of classes in $H^1(X;\R)$ that are dual to sections of $\flow$. \qed
\end{proposition}

\subsection{Convexity and the $\A$ cone.}

In \cite{DKL} we introduced the nonempty, open, convex cone $\A\subset H^1(X;\R)$ containing the epimorphism $u_0$ associated to the original splitting $G = F_N\rtimes_\varphi \Z$ and consisting of cohomology classes that may be represented by cellular $1$--cocycles that are positive on all $1$--cells of $X$. We then proved that every primitive integral class $u\in \A$ \emph{is} dual to a cross section of $\flow$ with some nice additional properties. More precisely, Theorem B of \cite{DKL} gives:

\begin{theorem} \label{T:DKL_BC}
Given a positive cellular $1$--cocycle $z \in Z^1(X;\R)$ representing a primitive integral element $u \in \A \subset H^1(X;\R) = H^1(G;\R)$, there exists a flow-regular map
$\fib_u \colon X \to \sone$ with $(\fib_u)_* = u \colon G \to \pi_1(\sone) = \Z$ and a fiber $\Theta_u = \fib_u^{-1}(y_0)$ for some $y_0 \in \sone$ so that
\begin{enumerate}
\item $\Theta_u$ is a finite, connected topological graph such that the
  inclusion $\Theta_u\subset X$ is $\pi_1$--injective and such that
  $\pi_1(\Theta_u)=\ker(u)\le \pi_1(X)=G$, and
\item $\Theta_u$ is a section of $\flow$ that is dual to $u$.
\end{enumerate}
\end{theorem}

From this theorem and basic properties of flow-regular closed $1$--forms we deduce the following.

\begin{proposition} \label{P:Csec convex}
The cone of sections $\Csec$ is an open convex cone in $H^1(X;\R)$ containing $\A$.  Moreover, $\Csec$ is the convex hull of the rays through integral classes dual to sections.
\end{proposition}
\begin{proof}
If $\omega_1$ and $\omega_2$ are closed flow-regular $1$--forms, then so is $t_1\omega_1 + t_2\omega_2$ for any $t_1,t_2 \in \R_+$, thus $\Csec$ is a convex cone.
According to Theorem \ref{T:DKL_BC}, every primitive integral element of $\A$ is dual to a section, and hence the rays through integral classes of $\A$ are contained in $\Csec$.  Since $\A$ is an open convex cone, these rays are dense and in fact $\A$ is the convex hull of this set of rays.  By convexity of $\Csec$, it follows that $\A \subset \Csec$.  

We now prove that $\Csec$ is open.  Given $u \in \Csec$, choose a tame, flow-regular closed $1$--form $\omega$ representing $u$.  The derivative of $\omega$ along flowlines is well-defined (independent of the choice of function in $\omega$ since any two differ by a constant) and we denote it $D_\flow\omega$.  Since $\omega$ is smooth on each $2$--cell, $D_\flow\omega$ is a continuous function on $X$.    By compactness, $D_\flow\omega$ is bounded above and below by positive constants, and we let $\epsilon > 0$ be a lower bound.

Next, since $\A$ is open, we can choose tame, flow-regular closed $1$--forms $\omega_1,\ldots,\omega_b$ representing elements of $\A$ which form a basis for $H^1(X;\R)$.  Let $K > 0$ be such that $D_\flow \omega_j \leq K$ for all $j = 1,\ldots,b$.  Now the set
\[ U = \left\{ \left. \left[\omega + \sum_{j=1}^b \epsilon_j \omega_j \right] \right| |\epsilon_j| < \epsilon/(bK) \mbox{ for all } j = 1,\ldots, b \right\} \]
is an open neighborhood of $[\omega]$ in $H^1(X;\R)$.  Furthermore, $D_\flow(\omega + \sum \epsilon_j \omega_j)$ with $|\epsilon_j| < \epsilon/(bK)$ is bounded below by $\epsilon - \sum |\epsilon_j K |  > \epsilon - bK\epsilon/(bK) = 0$, and hence this closed $1$--form is flow-regular.  It follows that $U$ is contained in $\Csec$ and hence $\Csec$ is open.  The last statement of the proposition is now immediate.
\end{proof}

As we will see in Example \ref{Ex:the_cones}, $\A$ can be a proper subcone of $\Csec$.

\section{A homological characterization of $\Csec$}
\label{S:cone_of_sections}

In this section we introduce a new cone $\D\subset H^1(X;\R)$ and investigate its properties. Studying $\D$ ultimately leads to a homological characterization of $\Csec$ and the proof of Theorem~\ref{T:cone_of_sections}.
Let $\frak O(\flow)$ denote the set of closed orbits $\mathcal O$ of $\flow$. Each $\mathcal O \in \frak O(\flow)$ may be thought of an integral homology class on $X$; for emphasis, we will sometimes denote this class by $[\mathcal O] \in H_1(X;\Z)$. Since this set of homology classes mimics Fried's cone of homology directions \cite{FriedS}, we call its dual the \emph{Fried cone}:
\begin{defn}
\label{D:Fried_cone}
The \emph{Fried cone} of $(X,\flow)$ is defined to be the set
\[\D = \{u\in H^1(X;\R) \mid u(\mathcal O) > 0 \mbox{ for all }\mathcal O\in\mathfrak{O}(\flow)\} \subset H^1(X;\R).\]
\end{defn}

To further analyze $\D$, we now describe a way to organize the set $\mathfrak O(\flow)$ that reveals some additional structure.
Recall that $f\colon \Gamma\to \Gamma$ is an expanding irreducible train track map and that the $(e,e')$--entry of its transition matrix $A = A(f)$ records the number of times that the edge path $f(e')$ crosses the edge $e$ in either direction. The matrix $A(f)$ has an associated transition graph $\mathcal{G}(f)$ whose vertex set is $E\Gamma$ and which has $A(f)_{(e,e')}$ directed edges from $e'$ to $e$ for each pair $(e,e')\in E\Gamma^2 = (V\mathcal{G}(f))^2$. 

Consider a combinatorial edge path $t_1\dotsb t_k$ in $\mathcal{G}(f)$ consisting of directed edges $t_i$ from $e_{i-1}$ to $e_i$. Such an edge path is said to be a \emph{circuit} in $\mathcal{G}(f)$ if $e_k = e_0$ and all of the vertices $e_0,\dotsc,e_{k-1}$ are distinct. Notice that in this case the edges $t_1,\dotsc,t_k$ are also distinct. Let $\mathcal{Y}$ denote the set of circuits in $\mathcal{G}(f)$. Circuits that differ by a cyclic permutation of its edges are not considered distinct, thus $t_1\dotsb t_k$ and $t_2\dotsb t_k t_1$ define the same element of $\mathcal{Y}$. Since $\mathcal{G}(f)$ is a finite graph, the set $\mathcal{Y}$ of circuits is finite.

Suppose that $y = t_1\dotsb t_k\in \mathcal{Y}$ is a circuit in $\mathcal{G}(f)$ with $t_i$ directed from $e_{i-1}$ to $e_i$. Then by definition of $\mathcal{G}(f)$, the directed edge $t_i$ corresponds to a particular occurrence of $e_{i}$ in the combinatorial edge path $f(e_{i-1})\subset \Gamma$. Similarly, the entire circuit determines a particular occurrence of $e_k = e_0$ in the edge path $f^k(e_0)\subset \Gamma$. By the linearity of $f$, there is then a subinterval $\alpha\subset e_0$ such that the restriction $(f^k)\vert_\alpha$ maps $\alpha$ homeomorphically and affinely onto this particular occurrence of $e_0$ in the edge path $f^k(e_0)$. The map $f^k\vert_\alpha \colon \alpha\to e_0$ then necessarily has a unique fixed point $p\in \alpha$, and we let
\[\mathcal{O}_y = \{\flow_s(p)\mid s\geq 0\} = \{\flow_s(p) \mid 0 \leq s \leq k\}\]
denote the closed orbit of $\flow$ through $p\in \Gamma\subset X$. Notice that this orbit is independent of the cyclic ordering on the directed edges $t_1,\dots,t_k$ of the circuit $y$. For each $y\in \mathcal{Y}$ we thus have a well-defined closed orbit $\mathcal{O}_y\subset X$.

More generally, each closed combinatorial edge path $w=t_1 \dotsb t_k$ in $\mathcal G(f)$ (not necessarily embedded) determines a closed orbit $\mathcal O_w$ in the same way. Conversely, any closed orbit of $\flow$ crosses a sequence $e_1,\dots,e_k$ of edges of $\Gamma\subset X$ and so determines a closed combinatorial edge path in $\mathcal G(f)$.  Thus we have a  surjective function from the set of closed edge paths in $\mathcal G(f)$ (up to cyclic permutation) to the set of closed orbits $\mathfrak O(\flow)$; this function is a bijection off the preimage of closed orbits through vertices of $\Gamma$.

\begin{lemma} \label{L:additive cone map}
Given any closed edge path $w$ in $\mathcal G(f)$, we may write $w$ as a union of circuits $w = y_1 \cup \cdots \cup y_n$ with $y_1,\ldots,y_n \in \mathcal Y$ and correspondingly write $[\mathcal O_w]$ as the sum
\[ [\mathcal O_w] = [\mathcal O_{y_1}] + \cdots + [\mathcal O_{y_n}] \in H_1(X;\Z).\]
\end{lemma}
Note that since $[w] = [y_1] + \cdots + [y_n] \in H_1(\mathcal G(f);\Z)$, it follows that the function from closed edge paths in $\mathcal G(f)$ to $\mathfrak O(\flow)$ descends to an additive map on the level of homology.
\begin{proof}
If $w = t_1 \dotsb t_k$ is already a circuit, then there is nothing to do.  Otherwise, we decompose $w$ as a union of circuits in the following recursive way.  Let $w_1 = w$. Since $w_1$ is closed, there must exist a subpath $t_i\dotsb t_j$, for some  $1 \leq i\leq j \leq k$, that is a circuit. In this case, we may `factor out' the subpath $y_1 = t_i\dotsb t_j\in \mathcal{Y}$ to obtain a shorter closed path $w_2 = t_1\dotsb t_{i-1} t_{j+1}\dotsb t_k$ in $\mathcal{G}(f)$. Since $w_2$ is again closed, we may factor out another circuit $y_2\in \mathcal{Y}$ giving yet a shorter closed path $w_3$. This process continues until we reach a closed path $w_n$ that is itself a circuit $y_n=w_n\in \mathcal{Y}$.

On the level of homology, we may perform an analogous factoring procedure to the orbit $\mathcal O_w \subset X$.  Specifically, we note that the closed orbits $\mathcal O_w$ and $\mathcal O_{y_1}$ start on some edge $e$ of $\Gamma$, cross the same set of edges (and at the same occurrences inside iterates of $f$ on $e$) before returning back to $e$ when $\mathcal O_{y_1}$ closes up.  At this time $\mathcal O_w$ switches and starts crossing the same edges as $\mathcal O_{w_2}$, finally returning to $e$ as both $\mathcal O_w$ and $\mathcal O_{w_1}$ close up.  Thus, the closed orbit $\mathcal O_w$ is homotopic to a closed curve constructed out of the closed orbits $\mathcal O_{y_1}$ and $\mathcal O_{w_2}$ together with a collection of subarcs of the edge $e$.  In particular, $[\mathcal O_w] - [\mathcal O_{y_1}] - [\mathcal O_{w_2}] \in H_1(X;\Z)$ is represented by a cocycle contained in $e$ which one may check  is null homologous, and hence $[\mathcal O_w] - [\mathcal O_{y_1}] - [\mathcal O_{w_2}] = 0 \in H_1(X;\Z)$.   Repeating this procedure for $[\mathcal O_{w_2}]$ and continuing recursively, we obtain the desired formula
\[ [\mathcal O_w] = [\mathcal O_{y_1}] + \cdots + [\mathcal O_{y_n}].\qedhere \]
\end{proof}

As a consequence, we see that $\D$ can be defined by requiring positivity on only finitely many orbits (compare \cite[Theorem I]{FriedS} and \cite[Theorem 1.6]{Wang}).

\begin{proposition} \label{P:Fried finiteness}
The Fried cone $\D$ is given by
\[ \D = \{ u \in H^1(X;\R) \mid u(\mathcal O_y) > 0 \mbox{ for all } y \in \mathcal Y \} .\]
In particular, $\D$ is an open, convex, polyhedral cone with finitely many rationally defined sides.
\end{proposition}
\begin{proof}
Since $\{\mathcal O_y\}_{y\in \mathcal Y} \subset \mathfrak O(\flow)$, the set in question clearly contains $\D$. For the other containment, we suppose $u(\mathcal O_u) > 0$ for all $y\in \mathcal Y$ and prove $u\in \D$. Every closed orbit $\mathcal O\in \mathfrak O(\flow)$ has the form $\mathcal O_w$ for some closed path $w$ in $\mathcal G(f)$. By Lemma~\ref{L:additive cone map} we may then write $[\mathcal O_w] = [\mathcal O_{y_1}] + \cdots + [\mathcal O_{y_n}]$, from which we obtain
\[u(\mathcal O_w) = u(\mathcal{O}_{y_1}) + \dotsb + u(\mathcal{O}_{y_n}) > 0.\]
This proves the first claim of the proposition.  The second claim is an immediate consequence of the first.
\end{proof}

Integrating a closed, flow-regular $1$--form over a closed orbit yields a positive number, and hence we have

\begin{proposition} \label{P:Fried contains Csec}
The Fried cone contains the cone of sections: $\Csec \subseteq \D$.
\end{proposition}
\begin{proof}
Given any class $u \in \Csec$, we can represent it by a closed flow-regular $1$--form $\omega = \omega^u$.  For any closed curve $\gamma$ we then have $u(\gamma) = \int_\gamma \omega^u$.  In particular, for each closed orbit $\mathcal O\in \mathfrak O(\flow)$, this gives
\[ u(\mathcal O) = \int_{\mathcal O} \omega^u > 0 \]
by flow-regularity.  Therefore $u \in \D$.
\end{proof}

We will see that the containment in this proposition is actually an equality. Proving this fact requires that we construct a section dual to any integral class $u \in \D$. 
This is a fairly technical construction which we carry out in Appendix~\ref{S:characterizing_sections}; see Proposition~\ref{P:dual_sections}. With this result in hand, we may now prove Theorem~\ref{T:cone_of_sections} from the introduction.

\begin{theorem:conesections}[Cone of sections]
There is an open convex cone $\Csec\subset H^1(X;\R)=H^1(G;\R)$ containing $\A$ (and thus containing $u_0$) such that a primitive integral class $u\in H^1(X_f,\R)$ is dual to a section of $\flow$ if and only if $u \in \Csec$. Moreover, $\Csec$ is equal to to the Fried cone $\D$, and there exist finitely many closed orbits $\mathcal{O}_1,\dotsc,\mathcal{O}_k$ of $\flow$ such that $u\in H^1(X;\R)$ lies in $\Csec$ if and only if $u(\mathcal{O}_i) > 0$ for each $1\leq i \leq k$.  In particular, $\Csec$ is an open, convex, polyhedral cone with finitely many rationally defined sides.
\end{theorem:conesections}
\begin{proof}
The first assertion follows from Propositions~\ref{P:integral_classesin_Csec} and \ref{P:Csec convex}. According to Proposition~\ref{P:Fried contains Csec}, we have $\Csec \subseteq \D$. On the other hand, Proposition~\ref{P:dual_sections} shows that all rational points of $\D$ are contained in $\Csec$. Since $\D$ is the convex hull of its rational points, we also have $\D \subseteq \Csec$, and thus $\Csec = \D$ as claimed. Taking the finite set of closed orbits $\{\mathcal O_1,\dotsc,\mathcal O_k\}$ to be $\{\mathcal O_y\}_{y \in \mathcal Y}$, the remaining claims then follow from Proposition~\ref{P:Fried finiteness}.
\end{proof}

\section{The algebra and dynamics of cross sections to $\flow$}
\label{S:geometry_of_sections} 

In this section we study the first return maps to cross sections of $\flow$. While technical considerations require that we restrict to certain nice cross sections, this is not a serious restriction as every integral class in $\Csec$ is seen to be dual to such a section. In \S\ref{S:endomorphisms} we go on to describe how each of these first return maps, which need not be homotopy equivalences, gives a realization of $G$ as an HNN-extension. This discussion culminates in the proof of Theorem~\ref{T:splittings}.

\subsection{First return maps}
\label{sec:first_return_maps}

Topologically, a cross section is merely a subset $\Theta = (\fib')\inv(x_0)$, for some flow-regular map $\fib'\colon X\to \sone$ and $x_0\in \sone$, which abstractly has the structure of an embedded topological graph. To aid in our analysis of the first return map $f_\Theta$, here we describe a topological graph structure on $\Theta$ that is tailored to the complex $X$ and flow $\flow$. We then show that $\Theta$ moreover supports a linear graph structure with respect to which $f_\Theta$ is an expanding irreducible train track map.

Suppose that $\eta'\colon X\to \sone$ is a flow-regular map and that $\Theta = (\eta')\inv(x_0)$ is a cross section of $\flow$ for some $x_0\in \sone$. It follows that for each $x\in \sone$ the subset $\Theta_x := (\eta')\inv(x)$ is again a cross section and that all of these cross sections $\{\Theta_x\}_{x\in \sone}$ are homotopic to each other. 

\begin{defn}[Compatible cross section]
\label{D:compatible_cross_section}
A cross section $\Theta\subset X$ is said to be \emph{compatible with the foliation $\mathcal F$} (or more succinctly \emph{$\mathcal F$--compatible}) if $\Theta\cap X^{(1)}$ consists of a finite set of points that all lie on vertex leaves of $\mathcal F$.
\end{defn}

\begin{remark}
Since the union of vertex leaves of $\mathcal F$ is dense in $X$, it is always possible to adjust any given cross section by a homotopy so that it is $\mathcal F$--compatible. Nevertheless, when we construct cross sections in \S\ref{S:constructing_sections} below, we will explicitly arrange for the constructed cross sections to be $\mathcal F$--compatible.
\end{remark}

Every $\mathcal F$--compatible cross section may be equipped with a convenient topological graph structure:

\begin{defn}[Standard graph structure]
\label{D:standard_graph}
Let $\Theta$ be an $\mathcal F$--compatible cross section. The \emph{standard (topological) graph structure} on $\Theta$ is then defined as follows. Firstly, $\Theta$ naturally inherits a topological graph structure in which the (finite) vertex set is $\Theta\cap X^{(1)}$ and every edge is an embedded arc in the interior of a $2$--cell of $X$. We subdivide this initial graph structure by declaring the vertex set of $\Theta$ to be
\[V\Theta = \left(\bigcup_{s\geq 0} \flow_s(\Theta\cap X^{(1)})\right) \cap \Theta.\]
Since $\mathcal F$--compatibility ensures that each point $\xi\in \Theta\cap X^{(1)}$ lies on a vertex leaf and thus that the $\flow$--orbit of $\xi$ eventually becomes periodic, we see that this is indeed a finite subdivision making $\Theta$ into a finite topological graph.
\end{defn}

\begin{remark}
\label{R:edges_are_transversals}
We emphasize that when an $\mathcal F$--compatible cross section $\Theta$ is equipped with its standard graph structure, then each edge of $\Theta$ is an embedded arc contained in a $2$--cell of $X$ with both endpoints lying in vertex leaves of $\mathcal F$. In particular each edge of $\Theta$ is a transversal to $\mathcal F$; this fact will play an important role in \S\ref{secc:module_isom} below.
\end{remark}

With these definitions, the arguments from \cite[\S7]{DKL} establishing \cite[Theorem C]{DKL} essentially go through verbatim to prove that the first return map to an $\mathcal F$--compatible cross section is an expanding irreducible train track map. Rather than repeat the entire proof, here we briefly go through the argument while recalling the relevant results from \cite{DKL}.

\begin{lemma}
\label{L:first_return_top_graph_map}
Let $\Theta$ be an $\mathcal F$--compatible cross section equipped with its standard graph structure and first return map $f_\Theta\colon\Theta\to\Theta$. Then for all $n\ge 1$ the map $f_\Theta^n$ is locally injective on the interior of each edge of $\Theta$. Moreover, $f_\Theta(V\Theta)\subset V\Theta$, and $f_\Theta$ is a regular topological graph map.
\end{lemma}
\begin{proof}
The fact that $f_\Theta(V\Theta)\subset V\Theta$ is ensured by the definition of the standard graph structure (Definition~\ref{D:standard_graph}). Lemma 7.2 of \cite{DKL} shows that for every arc $\alpha\colon(-\delta,\delta)\to \sigma$ in the interior of a $2$--cell $\sigma$ of $X$ and transverse to the flow $\flow$, the assignment $(t,s)\mapsto \flow_s(\alpha(t))$ defines a locally injective map $(-\delta,\delta)\times \R_{\ge 0} \to X$. Since each edge of $\Theta$ is by construction an embedded arc contained in the interior of $2$--cell and is transverse to $\flow$, the proof of \cite[Lemma 7.3]{DKL} goes through verbatim to establish the remaining claims of the lemma.
\end{proof}

\begin{lemma}
\label{L:first_return_linear}
Let $\Theta$ be an $\mathcal F$--compatible cross section equipped with its standard graph structure. Then there exists a linear structure on $\Theta$ with respect to which $f_\Theta\colon \Theta\to \Theta$ is a linear graph map with irreducible transition matrix $A(f_\Theta)$ and spectral radius $\lambda(f_\Theta) > 1$.
\end{lemma}
\begin{proof}
By assumption, the original map $f\colon \Gamma\to\Gamma$ defining the folded mapping torus $X$ is an expanding irreducible train track map. The proof of \cite[Lemma 7.4]{DKL} thus shows that the topological graph map $f_\Theta\colon \Theta\to\Theta$ is ``expanding on all scales'' (see Appendix A.2 of \cite{DKL}). It then follows from \cite[Theorem A.3]{DKL} that there exists a linear structure $\Lambda$ on $\Theta$ with respect to which $f_\Theta$ is a linear graph map, the transition matrix $A(f_\Theta)$ is irreducible, and the spectral radius $\lambda(f_\Theta)$ is larger than $1$. 
\end{proof}

\begin{proposition}
\label{P:first_return_exp_irred_tt}
Let $\Theta$ be an $\mathcal F$--compatible cross section equipped with the linear graph structure provided by Lemma~\ref{L:first_return_linear}. Then the first return map  $f_\Theta\colon \Theta\to \Theta$ is an expanding irreducible train track map.
\end{proposition}
\begin{proof}
By Lemmas~\ref{L:first_return_top_graph_map}--\ref{L:first_return_linear} we know that $f_\Theta$ is a regular linear graph map for which $A(f_\Theta)$ is irreducible, $\lambda(f_\Theta) >1$, and for which each iterate $f_\Theta^n$ is locally injective on each edge of $\Theta$. To prove that all iterates of $f_\Theta$ are locally injective near the degree-$2$ vertices of $\Theta$, we can follow the proof of \cite[Lemma 2.12]{DKL} verbatim with one difference: the assumption that the map was a homotopy equivalence was only used to guarantee that it was surjective.  Here surjectivity of $f_\Theta$ is immediate from the fact that $A(f_\Theta)$ is irreducible. Therefore the linear graph map $f_\Theta\colon\Theta\to\Theta$ satisfies all the conditions of being an expanding irreducible train track map, and so the claim holds.
\end{proof}

\begin{conv}
\label{conv:notation_convention}
Henceforth, for any primitive integral element $u\in \Csec$, we let $\fib_u\colon X\to \sone$ denote any flow-regular map whose induced map on fundamental group is $(\fib_u)_* = u\colon G\to \Z$ and whose fiber $\Theta_u\subset X$ is an $\mathcal F$--compatible cross section. Note that such a map $\fib_u$ always exists by Proposition~\ref{P:dual_sections}. We equip $\Theta_u$ with the standard graph structure; the first return map $f_u=f_{\Theta_u}\colon \Theta_u\to\Theta_u$ is then an expanding irreducible train track map by Proposition~\ref{P:first_return_exp_irred_tt}. Lastly, we use $\flow^u$ to denote the the reparameterized semiflow for which the restriction of the time--$1$ map $\flow^u_1$ to $\Theta_u$ is exactly $f_u$.
\end{conv}

\subsection{Sections and Endomorphisms}
\label{S:endomorphisms}

Here we explain how every section $\Theta$ of $\flow$ gives rise to
splitting of $G$ as a (possibly ascending) HNN-extension 
\begin{equation} 
\label{E:HNN from section} 
G = Q_{[\Theta]} \ast_{\phi_{[\Theta]}} 
\end{equation}
of a finitely generated free group $Q_{[\Theta]}$ along an injective endomorphism $\phi_{[\Theta]}$ with respect to which $[\Theta]\colon G\to \Z$ is exactly the natural projection $Q_{[\Theta]}\ast_{\phi_{[\Theta]}}\to \Z$ sending the stable letter to $1$ and $Q_{[\Theta]}$ to $0$. This general situation should be compared with the special case $[\Theta] \in \A$, where Theorem \ref{T:DKL_BC} shows that we can find $\Theta$ so that in the splitting (\ref{E:HNN from section}) we can take $Q_{[\Theta]} = \pi_1(\Theta)$ and let $\phi_{[\Theta]} = (f_{\Theta})_*$ be the {\em isomorphism} induced by the first return map. 

Let $u\in \Csec$ be any primitive integral class with dual cross section $\Theta_u$ as in Convention~\ref{conv:notation_convention}. Choosing a basepoint $v\in \Theta_u$ and a path $\beta\subset \Theta_u$ from $f_u(v)$ to $v$, we then let $(f_u)_*\colon \pi_1(\Theta_u,v)\to\pi_1(\Theta_u,v)$ be the endomorphism defined by $(f_u)_*(\gamma) = \beta f(\gamma)\beta\inv$. 
If we let $\tau_u\in G =\pi_1(X,v)$ be the loop obtained by concatenating the flowline from $v$ to $f_u(v)$ with $\beta$, it then follows that $u(\tau_u) = 1$ and therefore that the conjugation $g\mapsto \tau_u\inv g \tau_u$ defines an automorphism $\Phi_u\in \Aut(\ker(u))$ representing the monodromy $\fee_u$. Furthermore, the splitting (\ref{eqn:splitting}) gives a realization of $G$ as the HNN-extension $G = \ker(u)\ast_{\Phi_u}$ for which the natural projection is $u\colon G\to \Z$

The inclusion $\iota\colon \Theta_u\hookrightarrow X$ now induces a homomorphism
\[\iota_*\colon \pi_1(\Theta_u,v)\to\ker(u)\leq \pi_1(X,v) = G\]
that semiconjugates $(f_u)_*$ to $\Phi_u$.
Despite the fact that $\Theta_u$ need not $\pi_1$--inject into $X$ and that $(f_u)_*$ need not be an automorphism, van Kampen's Theorem still applies to yield the following:

\begin{lemma}\label{L:vanKampen}
Let $\Theta_u$ be a section of $\flow$. Then, with the above notation, the homomorphism $\iota_*\colon \pi_1(\Theta_u,v)\to \ker(u)$ induces an isomorphism
\begin{equation*} 
\langle \pi_1(\Theta_u,v), r \mid r^{-1} \gamma r = (f_u)_*(\gamma) \, \mbox{ for all } \gamma \in \pi_1(\Theta_u,v) \rangle \to \ker(u)\ast_{\Phi_u} =G
\end{equation*}
with respect to which $u\colon G\to \Z$ is given by sending $\pi_1(\Theta_u,v)$ to $0$ and $r$ to $1$.  
\end{lemma}

We can now prove Theorem~\ref{T:splittings}. For the statement, recall that the stable kernel $K_\phi$ of an arbitrary endomorphism $\phi\colon W\to W$ was defined in \S\ref{S:HNN-like} and that $\phi$ naturally descends to a homomorphism $\bar{\phi}$ of the corresponding stable quotient $W/K_\phi$.

\begin{theorem:splittings}[Splittings and ascending HNN-extensions]
Let $u\in \Csec$ be a primitive integral class with $\mathcal F$--compatible dual section $\Theta_u \subset X$ and first return map $f_u \colon \Theta_u \to\Theta_u$. Let $Q_u$ be the stable quotient of $(f_u)_*$ and let $\phi_u = \overline{(f_u)}_*$ be the induced endomorphism of $Q_u$. Then
\begin{enumerate}
\item $f_u$ is an expanding irreducible train track map.
\item $Q_u$ is a finitely generated free group and $\phi_u\colon Q_u\to Q_u$ is injective.
\item $G$ may be written as an HNN-extension
\[G \cong Q_u \ast_{\phi_u} = \langle Q_u, r \mid r\inv q r = \phi_u(q) \text{ for all }q\in Q_u\rangle\]
such that $u\colon G\to \Z$ is given by the assignment $r\mapsto 1$ and $Q_u\mapsto 0$.
\item If $J_u\leq \ker(u)$ denotes the image of $\pi_1(\Theta_u)$ induced by the inclusion $\Theta_u\hookrightarrow X$, then there is an isomorphism $Q_u\to J_u$ conjugating $\phi_u$ to $\Phi_u\vert_{J_u}$ for some automorphism $\Phi_u\in\Aut(\ker(u))$ representing the monodromy $\fee_u$.
\item The topological entropy of $f_u$ is equal to $\log(\lambda(\phi_u)) = \log(\lambda(\Phi_u\vert_{J_u}))$ and also to $\log(\Lambda(u))$.
\item $\ker(u)$ is finitely generated if and only if $\phi_u$ is an automorphism, in which case we have $\ker(u) \cong Q_u$ and that $\fee_u = [\phi_u] \in \Out(Q_u)$.
\end{enumerate}
\end{theorem:splittings}
\begin{proof}
We use the notation introduced before Lemma~\ref{L:vanKampen} and, in particular, implicitly use $v\in \Theta_u$ as a basepoint and take $\Phi_u\in\Aut(\ker(u))$ to be conjugation by $\tau_u\in \pi_1(X) =G$. Then $\Phi_u$ represents the monodromy $\fee_u\in \Out(\ker(u))$. Item (1) was established by Proposition~\ref{P:first_return_exp_irred_tt}, and by Proposition~\ref{P:stretch} this further implies that the topological entropy of $f_u$ is equal to $\log(\lambda((f_u)_*))$. Corollary~\ref{C:HNN_and_stretch_for_endos} then implies item (2) and additionally proves $\lambda(\phi_u) = \lambda((f_u)_*)$, thus establishing the first part of (5). The remaining assertions $\lambda(\phi_u) = \lambda(\Phi_u\vert_{J_u})$ and $\Lambda(u)=\lambda(\phi_u)$ of (5) will follow from items (4) and (3), respectively.

Since $\iota_*\colon\pi_1(\Theta_u)\to \ker(u)$ semiconjugates $(f_u)_*$ to $\Phi_u$, which is injective, it follows that the stable kernel $K_{(f_u)_*}$ is contained in $\ker(\iota_*)$. Therefore $\iota_*$ factors through the stable quotient $Q_u$ thus yielding a pair of homomorphisms
\[\pi_1(\Theta_u)\to Q_u\to \ker(u).\]
Since these homomorphisms respectively semiconjugate $(f_u)_*$ to $\phi_u = \overline{(f_u)}_*$ and $\phi_u$ to $\Phi_u$, they induce canonical homomorphisms
\[\pi_1(\Theta_u)\ast_{(f_u)_*} \to Q_u\ast_{\phi_u} \to \ker(u)\ast_{\Phi_u} = G\]
that respect the natural projections of these groups to $\Z$. The above composition is an isomorphism by Lemma~\ref{L:vanKampen}, and thus each map is itself an isomorphism. This proves (3). Since $Q_u$ embeds into $Q_u\ast_{\phi_u}$, this furthermore implies that $Q_u\to \ker(u)$ is injective. Therefore $Q_u\to \ker(u)$ is an isomorphism onto its image which, by definition, is equal to $J_u$. This isomorphism establishes (4).

It remains to prove (6). The fact that $u\colon G\to \Z$ is the natural projection $Q_u \ast_{\phi_u}\to \Z$ defined by $r\mapsto 1$ and $Q_u\mapsto 0$ implies that $\ker(u)$ is equal to the normal closure $ncl(Q_u)$ of $Q_u$ in $G$.  The HNN-extension presentation $G=Q_u \ast_{\phi_u}$ further implies that $ncl(Q_u)=\cup_{n=0}^\infty r^n Q_u r^{-n}$. Thus if $\phi_u$ is an automorphism, then $\ker(u)$ is equal to $Q_u$ and is thus finitely generated by (2). On the other hand, if $\phi_u$ is not an automorphism, then $\phi_u\colon Q_u\to Q_u$ is necessarily non-surjective (by (2)) and thus $ncl(Q_u)=\cup_{n=0}^\infty r_u^n Q_u r_u^{-n}$ is the union of an infinite strictly increasing chain of subgroups of $\ker(u)$. This shows that $\ker(u)$ is not finitely generated and completes the proof of (6).
\end{proof}

\begin{example}[The ascending HNN-extension for {$u_1 = [z_1]$}]
\label{Ex:(-1,2)_splitting}
Let us calculate the splitting described by Theorem~\ref{T:splittings}(2) for the homomorphism $u_1 = [z_1]$, where
\[ z_1 = [v_2^*,v_3^*] + [v_3^*,v_5^*] + 2[v_4^*,v_6^*] + [v_5^*,v_1^*] + [v_6^*,v_2^*]\]
is the cocycle introduced in Example~\ref{Ex:(-1,2) introduced}. Recall from equation (\ref{eqn:example_presentation}) that in our running example $G$ has the two-generator one-relator presentation
\begin{equation}
\label{eqn:re-present}
 G=\langle \gamma, r \mid  \gamma^{-1} r \gamma^{-1} r  \gamma^{-1} r^{-1} \gamma r \gamma r \gamma r^{-3} = 1\rangle.
\end{equation}
Thus we may apply Brown's method~\cite{Brown} to compute the kernel $\ker(u_1)$ and corresponding monodromy $\fee_1 = \fee_{u_1}$. Working in the unsubdivided cell structure on $X$, the loops $\gamma,r\subset X$ generating $G$ may be realized as $1$--chains
\[\gamma = \Big[ [v_1,v_2] + [v_2,v_4] + [v_4,v_5] - [v_3,v_5] - [v_1,v_3]\Big]\,\,\text{and}\,\, r = \Big[ [v_1,v_3] + [v_3,v_5] + [v_5,v_1]\Big].\]
Pairing these with $z_1$ we find that $u_1(\gamma) = -1$ and $u_1(r) = 2$. Setting $\beta=r\gamma^2$ and $\tau=\gamma^{-1}$, we therefore have $u_1(\tau)=1$ and $u_1(\beta)=0$. Conjugating $k \mapsto \tau\inv k\tau$ by $\tau$ thus yields an automorphism $\Phi_1\in \Aut(\ker(u_1))$ representing the monodromy $\fee_1$. Noting that $\gamma = \tau\inv$ and $r = \beta \tau^2$, we may rewrite the defining relation of (\ref{eqn:re-present}) in terms of the generators $\beta$ and $\tau$ to obtain the following presentation for $G$:
\begin{equation*}
\label{eqn:z_1-presentation1}
G=\langle \beta,\tau \mid 
\tau \beta \tau^{3} \beta \tau \beta^{-1}\tau^{-1}\beta \tau\beta \tau^{-1} \beta^{-1} \tau^{-2} \beta^{-1}\tau^{-2}\beta^{-1}
\rangle.
\end{equation*}
We now represent $G$ as an HNN-extension of another one-relator group, following the method introduced by McCool and Schupp~\cite{McSch}. For each $i\in \Z$ denote $\beta_i=\tau^{-i} \beta\tau^{i}$.  Then the defining relation above may be rewritten as
\[
\beta_{-1} \beta_{-4} \beta_{-5}^{-1} \beta_{-4}\beta_{-5} \beta_{-4}^{-1} \beta_{-2}^{-1}\beta_0^{-1} =1. 
\]
Hence we can rewrite the above presentation of $G$ as
\begin{equation}
\label{eqn:z_1-presentation2}
G = \left\langle \beta_{-5}, \dots, \beta_0,\tau \left| \begin{array}{l} \beta_{-1} \beta_{-4} \beta_{-5}^{-1} \beta_{-4}\beta_{-5} \beta_{-4}^{-1} \beta_{-2}^{-1}\beta_0^{-1}=1,\\
\quad  \tau^{-1}\beta_i\tau=\beta_{i+1}\text{ for } i=-5,\dots, -1\end{array}\right.\right\rangle.
\end{equation}

Consider now the one-relator group
\[
B=\langle  \beta_{-5},\dotsc, \beta_0 \mid  \beta_{-1} \beta_{-4} \beta_{-5}^{-1} \beta_{-4}\beta_{-5} \beta_{-4}^{-1} \beta_{-2}^{-1}\beta_0^{-1} \rangle.
\]
By Magnus' Freiheitssatz theorem for the one-relator group $B$, the
subgroups $L=\langle \beta_{-5},\dotsc,\beta_{-1}\rangle \le B$ and  $J=\langle\beta_{-4},\dotsc,\beta_0 \rangle \le B$ are free groups of rank $5$ and the indicated generating sets of $L$ and $J$ are their free bases. Hence the map $\phi_1\colon L\to J$, sending $\beta_i\mapsto \beta_{i+1}$ for $i=-5,\dots, -1$ is an isomorphism. Moreover, since the defining relation of $B$ implies that $\beta_0=\beta_{-1} \beta_{-4} \beta_{-5}^{-1} \beta_{-4}\beta_{-5} \beta_{-4}^{-1} \beta_{-2}^{-1}$, we see that the generator $\beta_0$ may be eliminated from the presentation of $B$ and thus that $L = \langle \beta_{-5},\dotsc,\beta_{-1}\rangle$ is in fact equal to $B$. Therefore $\phi_1\colon L\to J$ is actually an injective endomorphism $\phi_1\colon B\to B$ of the rank--$5$ free group $B$. This shows that the presentation~(\ref{eqn:z_1-presentation2}) gives a splitting of $G$ as an HNN-extension
\[G = B\ast_{\phi_1} = \langle B, \tau \mid \tau\inv b \tau = \phi_1(b)\text{ for }b\in B \rangle.\]

Note that with respect to the free basis $B = \langle \beta_{-5},\dotsc,\beta_{-1}\rangle$, the injective endomorphism $\phi_1$ is given by $\phi_1(\beta_{-1})=\beta_{-1} \beta_{-4} \beta_{-5}^{-1}\beta_{-4}\beta_{-5} \beta_{-4}^{-1} \beta_{-2}^{-1} $ and $\phi_1(\beta_i)=\beta_{i+1}$ for $i=-5,\dotsc, -2$. Using the Stallings folding method \cite{KM02,St83} we find that $\phi_1$ is not surjective: The fact that $\beta_{-5}$ occurs in $\phi_1(\beta_{-1})$ twice leads to the conclusion that the rose with petals labelled by the words $\phi_1(\beta_{-5}),\dotsc,\phi_1(\beta_{-1})$ \emph{does not} fold onto the rose with petals labelled $\beta_{-5},\dotsc,\beta_{-1}$. Therefore (\ref{eqn:z_1-presentation2}) is a strictly ascending HNN-extension presentation of $G$ with base $B=F(\beta_{-5},\dotsc,\beta_{-1})$ and stable letter $\tau$ along the injective but non-surjective endomorphism $\phi_1$ of $B$.

It follows that $\ker(u_1)$ is not finitely generated because it equal to the infinite union $ncl_G(B) = \cup_{i=0}^\infty \tau^i B \tau^{-i}$ of a strictly increasing chain of subgroups.
We also note that the free group endomorphism $\phi_1$ is just the restriction $\Phi_1\vert_{B}$ of $\Phi_1$ to $B\leq G$. Using the software package {\rm xtrain}, one may verify that $\phi_1$ does admit an expanding irreducible train track representative and that the spectral radius of the transition matrix for this train-track representative is $\approx 1.35827$. By Proposition~\ref{P:stretch} we may thus conclude that $\phi_1$ has a stretch factor of $\lambda(\phi_1) \approx 1.35827$.
\end{example}

\begin{remark} \label{Rm:infinitely generated kernels exist}
In light of the proof of Theorem~\ref{T:splittings}, one may also see that $\ker(u_1)$ is infinitely generated by noting, as we have done in Example~\ref{Ex:(-1,2) introduced}, that the first return map $f_1$ induces an injective but not surjective endomorphism $(f_1)_*$ of $\pi_1(\Theta_1)$: The fact that $(f_1)_*$ is injective implies that its stable kernel is trivial and that $\Theta_1$ $\pi_1$--injects into $G$. Thus the HNN-extension provided by Theorem~\ref{T:splittings}(2) is simply
\[G\cong \pi_1(\Theta_1)\ast_{(f_1)_*}.\]
Therefore, as $(f_1)_*$ is not surjective, Theorem~\ref{T:splittings}(6) implies that $\ker(u_1)$ is not finitely generated.
\end{remark}

\begin{example}[The splitting for {$u_2 = [z_2]$}]
\label{Ex:(-1,1)_splitting}
Let us also calculate the splitting of $G$ for the homomorphism $u_2 = [z_2]$ determined by the cocycle
\[z_2 = [v_1^*,v_2^*] + [v_2^*, v_3^*] - [v_4^*, v_5^*] - [v_5^*, v_6^*] + [v_1^*,v_3^*] + 2[v_6^*,v_2^*]\]
introduced in Example~\ref{Ex:(-1,1) introduced}. Again, by realizing the generators $\gamma,r\in G$ as $1$--cycles and pairing with $z_2$, we find that $u_2(\gamma) = -1$ and $u_2(r) = 1$. Setting $\beta = r\gamma$ and $\tau = \gamma\inv$, we may rewrite the presentation (\ref{eqn:re-present}) for $G$ in terms of the generators $\beta$ and $\tau$ as we did in (\ref{eqn:z_1-presentation2}) above to find that
\begin{equation}
\label{eqn:z_2-presentation2}
G = \left\langle \beta_{-4}, \dots, \beta_0,\tau \left| \begin{array}{l} \beta_{-1} \beta_{-3} \beta_{-4}^{-1} \beta_{-3}^2\beta_{-2}^{-1} \beta_{-1}^{-1} \beta_0^{-1}=1,\\
\quad  \tau^{-1}\beta_i\tau=\beta_{i+1}\text{ for } i=-4,\dots, -1\end{array}\right.\right\rangle,
\end{equation}
where we again denote $\beta_i = \tau^{-i}\beta \tau^i$ for $i\in \Z$. Setting
\[B' = \langle\beta_{-4},\dotsc,\beta_{0} \mid \beta_{-1} \beta_{-3} \beta_{-4}^{-1} \beta_{-3}^2\beta_{-2}^{-1} \beta_{-1}^{-1} \beta_0^{-1} \rangle,\]
Magnus' Freiheitssatz now tells us that the subgroups $L' = \langle\beta_{-4},\dotsc,\beta_{-1}\rangle\leq B'$ and $J' = \langle\beta_{-3}\dotsc,\beta_0\rangle \leq B'$ are freely generated by the indicated generating sets and that the assignment $\beta_i \mapsto \beta_{i+1}$ for $i=-4,\dotsc,-1$ gives an isomorphism $\phi_2\colon L'\to J'$. However, as the generator $\beta_0$ of $B'$ is evidently superfluous, we in fact have $L' = B'$. Thus $\phi_2$ is actually an endomorphism $\phi_2\colon B'\to B'$, and with respect to the free basis $B' =\langle \beta_{-4},\dotsc,\beta_{-1}\rangle$ it takes the form $\phi_2(\beta_{-1}) = \beta_{-1} \beta_{-3} \beta_{-4}^{-1} \beta_{-3}^2\beta_{-2}^{-1} \beta_{-1}^{-1}$ and $\phi_2(\beta_i) = \beta_{i+1}$ for $i=-4,\dotsc,-2$. Here the Stallings folding method shows that $\phi_2$ is both injective and surjective,  and therefore (\ref{eqn:z_1-presentation2}) presents $G$ as an HNN-extension $G\cong B'\ast_{\phi_2}$ over the free base group $B'$ along the automorphism $\phi_2$. It now follows that $\ker(u_2) = B'$ is finitely generated and in fact has rank $4$. Theorem~\ref{T:splittings} then implies that $\phi_2$ represents the monodromy $\varphi_2 = \varphi_{u_2}\in \Out(\ker(u_2))$, and the program xtrain calculates that the stretch factor here is $\lambda(\varphi_2) = \lambda(\phi_2)\approx 1.632992$.
\end{example}

\section{Cross sections and homology}
\label{S:sections_and_homology}

In this section we return our attention to the abelian group $H = H_1(G;\Z)/\text{torsion}$ and describe 
bases for $H$ and associated coordinates systems on $H^1(G;\R) = H^1(X;\R)$ that are
tailored to each integral class $u\in\Csec$.
Viewing a given primitive integral class $u\in \Csec$ as a surjective homomorphism $G \to \Z$, we let $\tX_u \to X$ denote the infinite cyclic cover corresponding to the kernel of $u$.  The inclusion $\Theta_u \to X$ lifts to an inclusion $\Theta_u \to \tX_u$ since $\pi_1(\Theta_u) < \ker(u)$.  Because $H$ is the maximal, abelian, torsion-free quotient of $G$, the homomorphism $u$ factors as a composition $G \to H \to \Z$, and we let $H_u < H$ denote the kernel of this homomorphism $H \to \Z$.

Observe that $H_u$ is the image in $H$ of $\pi_1(\tX_u) = \ker(u) < G$ under the homomorphism $G \to H$.  In fact, we claim that $H_u$ is the image of $\pi_1(\Theta_u)$ (if $u \in \A$, then $\pi_1(\Theta_u) = \pi_1(\tX_u)$ and this is obvious).  Before we prove the claim, we note that $\pi_1(\Theta_u)$ is only well-defined up to conjugacy in general, but any two conjugates have the same image in $H$.  Now, given an arbitrary element $[\alpha]$ of $H_u$, let $\alpha$ be a loop in $\tX_u$ sent to $[\alpha]$.  Lifting the semiflow $\psi$ to $\tX_u$ we flow $\alpha$ into some lift of $\Theta_u$ in $\tX_u$, producing a loop $\alpha'$ in some lift of $\Theta_u$, freely homotopic to $\alpha$.  Pushing down to $X$, we find $[\alpha] = [\alpha'] \in H_u$, and hence $[\alpha]$ is in the image of $\pi_1(\Theta_u)$, as required.

Write $p \colon \tX \to X$, let $\hTheta_u = p^{-1}(\Theta_u) \subset \tX$, and let $\tTheta_u$ denote some component.  Since $\tX \to \tX_u$ is the cover corresponding to the kernel of $\pi_1(\tX_u) \to H_u$, and $\pi_1(\Theta_u)$ maps onto $H_u$, $H_u$ is precisely the stabilizer in $H$ of $\tTheta_u$.  Furthermore, $H/H_u$ is the covering group of $\tX_u \to X$, and so $H/H_u \cong \Z$.  It follows that there is a (noncanonical) splitting $H = H_u \oplus \Z$.  Since $H$ has rank $b$, $H_u$ has rank $b-1$.

A useful alternative way to describe this splitting is the following.
Consider the subgroup of cohomology fixed by $f_u$, denoted
\[ H^1(\Theta_u;\Z)^{f_u} \subset H^1(\Theta_u;\Z).\]
Given a homotopy class $\alpha$ of loops in $\Theta_u$, evaluating cohomology classes on $\alpha$ defines an element of $\Hom(H^1(\Theta_u;\Z)^{f_u},\Z)$.  This is a quotient of $H_1(\Theta_u;\Z)$, and in fact it is the maximal quotient on which $f_u$ acts trivially.  On the other hand, the inclusion $\Theta_u \subset X$ induces a map $H_1(\Theta_u;\Z) \to H$, and it is straightforward to check that this map factors as a composition $H_1(\Theta_u;\Z) \to \Hom(H^1(\Theta_u;\Z)^{f_u},\Z) \to H$.  Moreover, this second map is actually an injection, so since $H_1(\Theta_u;\Z)$ is the abelianization of $\pi_1(\Theta_u)$, from the discussion above, we have a natural isomorphism
\[ H_u  = \Hom(H^1(\Theta_u;\Z)^{f_u},\Z).\]

We choose a basis $s_1,\ldots,s_{b-1}$ for $H_u$ and $w \in H$ a generator for the complement.  Here the group structure is additive, but it will also be convenient to have a multiplicative basis.   Specifically, we let $t_1,\ldots,t_{b-1},x \in H_1(X;\R_+)$, be given by $t_i =e^{s_i}$ and $x = e^w$.  Formally, this just means that $t_i$ and $x$ are the images of $s_i$ and $w$, respectively, under the isomorphism $H_1(X;\R) \to H_1(X;\R_+)$ determined by the isomorphism of coefficient groups $\exp \colon \R \to \R_+$.  When convenient we write ${\bf s} = (s_1,\ldots,s_{b-1})$ and ${\bf t} = e^{\bf s} = (t_1,\ldots,t_{b-1})$.
In this way, the integral group ring of $H$ is naturally identified with the ring of Laurent polynomials
\[ \Z[H] \cong \Z[t_1^{\pm 1},\ldots,t_{b-1}^{\pm 1},x^{\pm 1}] = \Z[{\bf t}^{\pm 1},x^{\pm}].\]
It is also convenient to think of $t_1,\ldots,t_{b-1},x$ as generators for the covering group $H$ of $\tX \to X$, and we will do so.  Finally, we assume that the generator $x$ is chosen so that it translates $\tX$ {\em positively} with respect to $\tpsi$.

It will be convenient to have a concise means for referring to all of these choices, so we make the following
\begin{defn} \label{D:adapted to z}
Given a primitive integral class $u \in \Csec$ with dual cross section $\Theta_u$, we say that the splitting $H =  H_u \oplus \Z$, and associated bases ${\bf s},w$ and ${\bf t} = e^{\bf s},x = e^w$ described above are {\em adapted to $u$}.
\end{defn}

\begin{conv} \label{Conv:homology coordinates}
The pairing between homology and cohomology allows us to view the homology classes $s_1,\ldots,s_{b-1},w,t_1,\ldots,t_{b-1},x$ as real-valued functions on $H^1(X;\R)$.  In fact, $({\bf s},w) \colon H^1(X;\R) \to \R^b$ are linear coordinate functions, and in these coordinates
\[ u = ({\bf s}(u),w(u)) = ({\bf 0},1).\]
Furthermore, the formal expressions ${\bf t} = e^{\bf s}$ and $x = e^w$ can be interpreted as a functional relation
\[ ({\bf t},x) = (e^{\bf s},e^w) \colon H^1(X;\R) \to \R_+^b.\]
\end{conv}

The bases ${\bf t}$ and ${\bf t},x$ for $H_u$ and $H$, respectively, determine isomorphisms of integral group rings with rings of integral Laurent polynomials:
\[ \Z[H_u] \cong \Z[{\bf t}^{\pm 1}] = \Z[t_1^{\pm 1},\ldots,t_{b-1}^{\pm 1}]\]
and
\[ \Z[H] \cong \Z[{\bf t}^{\pm 1},x^{\pm 1}] = \Z[t_1^{\pm 1},\ldots,t_{b-1}^{\pm 1},x^{\pm 1}].\]
Whenever we have chosen bases adapted to $u$, we will freely use these identifications.

\begin{example}[The cones $\A\subset\Csec$]
\label{Ex:the_cones}
Let us return to the running example introduced in Example~\ref{E:introduce_running}. If we take $u_0$ to be the homomorphism associated to the original splitting $G = F_N\rtimes_\fee \Z$, then $\Theta_{u_0}$ is the original graph $\Gamma$ shown in Figure~\ref{F:train_tracks}.  Let $H_0 = H_{u_0}$ as above giving rise to the splitting $H \cong H_0 \oplus \Z$.  Here $H_0 \le H$ is the cyclic group $\langle t\rangle$ generated by the image of either $\gamma_1$, $\gamma_2$, or $\gamma_3\in G$ under $G\to H = G^{ab}$, and its complement $\Z = \langle x \rangle$ is generated by the image of the stable letter $r\in G$.  

Let $s,w$ be the corresponding additive elements with $e^s = t$ and $e^w= x$, which we represent by $1$--cycles
\[ s = \Big[ [v_1,v_3] - [v_2,v_3] - [v_1,v_2] \Big]\quad\mbox{and}\quad w = \Big[ [v_1,v_3] + [v_3,v_5] + [v_5,v_1] \Big] .\]
Using $(s,w)$ as coordinates on $H^1(X;\R)$ we can express the classes $u_0, u_1, u_2$ as
\[ u_0 = (0,1), \quad u_1 = (-1,2), \quad u_2 = (-1,1).\]
All of these classes are dual to cross sections (see Examples~\ref{Ex:(-1,2) introduced}--\ref{Ex:(-1,1) introduced}) and are therefore contained in $\Csec$. By construction we also know that $\A$ contains $u_0$. On the other hand, since $\ker(u_1)$ is infinitely generated (see Example~\ref{Ex:(-1,2)_splitting}) we necessarily have $u_1\notin \A$. This gives us some information about $\A$ and $\Csec$, and we may in fact calculate these cones explicitly as follows.

First recall that every class in $\A$ may be represented by a positive cocycle on the unsubdivided cell structure. We note also that the coboundary operator $\delta \colon C^1(X;\R) \to C^2(X;\R)$ is linear, and hence that finding $1$--cocycles that represent points on the boundary of $\A$ amounts to finding $c \in C^1(X;\R)$ that maximizes/minimizes the linear function $c \mapsto c([v_1,v_3] - [v_2,v_3] - [v_1,v_2])$, subject to the system of linear equations and inequalities
\[  \left\{ \begin{array}{l}
\delta(c)  =  0,\\ 
c([v_1,v_3] + [v_3,v_5] + [v_5,v_1]) = 1,\\
c([v_i,v_j]) \geq 0, \mbox{ for all } i,j \mbox{ with } [v_i,v_j] \mbox{ a $1$--cell.}\end{array} \right.
\]
Here the first equation corresponds to the requirement that $c$ be a cocycle. The third inequality is the requirement that the cocycle is nonnegative. The second condition is a normalization that the cocycle evaluates to $1$ on $w$.  The function we are trying to maximize/minimize is the $s$--coordinate, and so if $\bar \A$ contained the ray through $(1,0)$ or $(-1,0)$, then there would be no maximum or minimum, respectively, subject to the constraints. This is a linear programming problem that can easily be solved in Mathematica or by hand. Doing so, we find that the the minimum and maximum values of the $s$--coordinate subject to these constraints are $-\nicefrac{1}{2}$ and $\nicefrac{1}{2}$, respectively. The $\A$ cone is therefore as illustrated in Figure~\ref{F:Acone-Scone}, and we note that $u_1$ in fact lies on its boundary $\partial \A$.

We now calculate $\Csec$ using the homological characterization given by Theorem~\ref{T:cone_of_sections}. Following the explicit construction of \S\ref{S:cone_of_sections}, we find that in the case of our running example the finitely many closed orbits $\mathcal{O}_i$ appearing in the statement of Theorem~\ref{T:cone_of_sections} consist of $7$ closed loops which are freely homotopic in $X$ to the following $7$ elements of $\pi_1(X) = \langle\gamma_1,\gamma_2,\gamma_3,r\rangle$:
\[\left\{\begin{array}{l}
r \gamma_3\inv\gamma_1\inv, 
\quad r^2\gamma_3\inv\gamma_1\inv\gamma_2\inv,
\quad r^2\gamma_3\inv,
\quad r^3,\vspace{4pt}\\
\quad r^2\gamma_3\inv r^2,
\quad r^2\gamma_3\inv\gamma_1\inv\gamma_2\inv r,
\quad r^2\gamma_3\inv\gamma_1\inv r\end{array}\right\}.\]
In terms of the additive basis $\{s,w\}$, the images of these elements in $H$ are
\[ (-2,1),\; (-3, 2),\; (-1,2),\; (0,3),\; (-1,4),\; (-3,3),\; (-2,3).\]
The cohomology classes $u$ which are positive on all of these vectors are exactly those whose $(s,w)$--coordinates satisfy both $w(u) > 0$ and $w(u) > 2s(u)$; such classes comprise the cone of sections $\Csec$ and are illustrated in Figure~\ref{F:Acone-Scone}.

\begin{figure}
\begin{center}
\begin{tikzpicture}[scale = 1.5]
\fill[black!30!white]
(0,0) -- (1.5,3) -- (-1.5,3) -- (0,0);
\fill[black!12!white]
(0,0) -- (-1.5,3) -- (-3.354,3) -- (-3.354,0) -- (0,0);

\draw[thin, ->] (-2,0) -- (2.2,0) node[above] {$s^*$};
\draw[thin, ->] (0,-.1) -- (0,3.2) node[right] {$w^*$} ;
\foreach \i in {-3,...,2}
{ \draw[thin] (\i,-.1) -- (\i,.1);}
\foreach \i in {0,...,3}
{ \draw[thin] (-.1,\i) -- (.1,\i);}
\draw[thick, dashed] (0,0) -- (1.5,3);
\draw[thick, dashed] (0,0) -- (-1.5,3);
\draw[very thick] (0,0) -- (1.5,3);
\draw[very thick] (0,0) -- (-3.354,0);

\fill (-1,2) circle (.06cm) node [left] {$u_1$};
\fill (-1,1) circle (.06cm) node [left] {$u_2$};
\fill (0,1) circle (.06cm) node [above right] {$[z_0]$};
\node (A)  at (.8,2.5) {$\A$};
\node (S) at (-2,2.5) {$\Csec$};

\end{tikzpicture}
\caption{The cones $\A\subset \Csec \subset H^1(X;\R)$ for our running example.}
\label{F:Acone-Scone}
\end{center}
\end{figure}
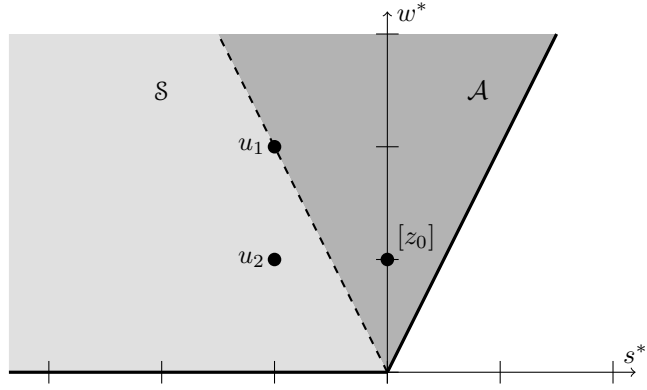
\end{example}

\begin{remark}
\label{R:S_has_both_kinds_of_kerels}
When $u\in \A$, Theorem~\ref{T:DKL_BC} ensures that $\ker(u)$ is finitely generated. However, Theorem~\ref{T:splittings} makes no claims about the finite or non-finite generation of $\ker(u)$ when $u\in \Csec\setminus\A$.
Indeed, since $\ker(u_1)$ is infinitely generated and $\ker(u_2)$ is finitely generated (see Examples~\ref{Ex:(-1,2)_splitting} and \ref{Ex:(-1,1)_splitting}), the above Example~\ref{Ex:the_cones} illustrates that both outcomes are possible for classes in $\Csec\setminus\A$.
\end{remark}

\section{Graph modules} \label{S:graph module}

Here we describe another, more combinatorial module that will connect the McMullen polynomial to the geometry of splittings of $G$.  Let $\tpsi^u$ be the lift of the reparameterized flow $\flow^u=\flow^{\Theta_u}$ so that $\tpsi_1^u$ maps $\tTheta_u$ to $x( \tTheta_u)$ (recall that $x \in H$ is a generator for the complement of $H_u$, which translates positively with respect to $\tpsi^u$). Thus $x^{-1} \circ \tpsi_1^u$ restricts to a graph map on $\tTheta_u$; in fact, $x^{-1} \circ \tpsi_1^u$ is just a lift of the first return map $f_u \colon \Theta_u \to \Theta_u$ of the flow $\psi^u$ to $\Theta_u\subset X$.

The covering group $H$ acts on (unoriented) edges of $\hTheta_u$, as with transversals, by taking preimages, so that $h \in H$ acts on an edge  $e$ by
\[ h \cdot e = e \cdot h = h^{-1}(e). \]
We can organize the edges of $\hTheta_u$ nicely as follows.  Let $\bar E = \{\bar \sigma_1,\ldots,\bar \sigma_m\}$ denote the set of (unoriented) edges of $\Theta_u$. For each $\bar \sigma_j$ we choose an edge $\sigma_j \in p^{-1}(\bar \sigma_j)$ contained in $\tTheta_u$, and we denote this set by $E = \{ \sigma_1,\ldots,\sigma_m\}$.
Because the action of $H$ is free, we see that the action on every orbit is also free. Thus the free $\Z$--module on the set of edges of $\hTheta_u$, which is naturally a $\Z[H]$--module, can be identified with the module
\[ \Z[H]^E = \left\{ \sum_{j=1}^m \sum_{h \in H} a_{j,h} h \cdot \sigma_j \left| \begin{array}{l} a_{j,h} \in \Z \mbox{ with } a_{j,h} = 0\\ \quad \mbox{ for all but finitely many } h \end{array} \right. \right\}. \]
Furthermore, since the elements of $E$ were chosen to lie in $\tTheta_u$, we see that the set of edges in $\tTheta_u$ is naturally identified with the $H_u$--orbit of $E$, and hence the free $\Z$--module on the edges of $\tTheta_u$ is precisely $\Z[H_u]^E$.

We now define the module of $\hTheta_u$, denoted $T(\hTheta_u)$, to be the largest quotient $\Z[H]^E \to T(\hTheta_u)$ in which the image $[e]$ of any edges $e$ satisfies the relations
\[ [e] = [e_1] + \ldots + [e_k], \]
where the edge path $\tpsi^u_1(e) = e_1\cdots e_k$ is the image of $e$ under $\tpsi^u_1$.  We refer to these above relations as the {\em basic relations} of $T(\hTheta_u)$. More precisely:

\begin{defn}[Module of $\hTheta_u$]
\label{D:graph module}
Let $M \le \Z[H]^E$ denote the submodule generated by all elements of the form $e - e_1 - \cdots - e_k$, with $e,e_1,\ldots,e_k$ as above. Note that the set of such elements is preserved by $H$. The \emph{module of $\hTheta_u$} is defined to be the quotient $\Z[H]$--module $T(\hTheta_u) = \Z[H]^E/M$.
\end{defn}

\subsection{Connection to transition matrices}
\label{S:transition_matrices}

Here we develop an alternate description of the module $T(\hTheta_u)$ closely related to the transition matrix of $\f_u$, ultimately culminating in a finite presentation.  Let ${\bf t} = t_1,\ldots,t_{b-1},x$ be a multiplicative basis for $H$ adapted to $u$ as in Definition \ref{D:adapted to z}.  For each edge $e \in \tTheta_u$, translating the edge path $\tpsi_1^u(e) = e_1\dotsb e_k$ by $x^{-1}$ gives a new edge path $x^{-1} \circ \tpsi_1^u(e)  = e_1' \cdots e_k' = x^{-1}(e_1 \cdots e_k)$, which now lies in $\tTheta_u$. Since $x^{-1}\circ\tpsi_1^u(e) = \tpsi_1^u(x^{-1}(e)) = \tpsi_1^u(x\cdot e)$, the corresponding basic relation becomes
\[ x \cdot [e] = [e_1'] + \ldots + [e_k'].\]
Now, for $e = \sigma_j \in E$, if we rewrite the sum on the right as a $\Z[H_u]$--linear combination of the edges in $E$ this relation becomes
\[ x \cdot [\sigma_j] = \sum_{i=1}^s A_{ij} \cdot [\sigma_i].\]
This gives a matrix $A = (A_{ij})$ with entries in $\Z[H_u]$; to clarify the dependence on $u$, we may sometimes write this matrix as $A_u = (A_{u,ij})$. Given the isomorphism with the ring of integral Laurent polynomials $\Z[H_u] = \Z[{\bf t}^{\pm 1}]$, $A_u$ becomes a matrix $A_u({\bf t})$ with coefficients in $\Z[{\bf t}^{\pm 1}]$.

Observe that $x^{-1} \circ \tpsi_1^u$ is a lift of the first return map $f_u \colon \Theta_u \to \Theta_u$, and so the following is not so surprising.

\begin{proposition} \label{P:evaluatetransition}
For every primitive integral $u \in \Csec$, the integral matrix $A_u({\bf 1})$, obtained by evaluating $A_u({\bf t})$ at ${\bf t} = {\bf 1} = (1,\dotsc,1)$, is exactly the transition matrix of the train track map $\f_u$.
\end{proposition}
\begin{proof}
The $(i,j)$--entry of $A_u$ is a $\Z$--linear combination of edges in the $H_u$--orbit of $\sigma_i$. Moreover, for each edge $e\in H_u\cdot \sigma_i$, the coefficient of $e$ in $A_{u,ij}$ is the number of times the path $x^{-1}\circ \tpsi^u_1(\sigma_j)$ crosses $e$. As these edges are exactly those which project to the edge $\bar \sigma_i$ in $\Theta_u$, the sum of these coefficients (namely $A_{z,ij}({\bf 1})$) is just number of times that the projected path $p(x^{-1}\circ\tpsi_1^u(\sigma_j))\subset\Theta_u$ crosses $\bar\sigma_i$. But this projected path is exactly $\f_u(\bar \sigma_j)$, and so the number of times it crosses $\bar \sigma_i$ is given by the $(i,j)$--entry of the transition matrix for $\f_u$.
\end{proof}

Furthermore, as the submodule $M\le \Z[H]^E$ of relations defining $T(\hTheta_u)$ is generated, as a $\Z[H]$--module, by the elements
\[ x \cdot \sigma_j - \sum_{i=1}^m A_{ij}({\bf t}) \cdot \sigma_i\]
for $j=1,\dotsc,m$, we have the following useful presentation for $T(\hTheta_u)$.
\begin{proposition} \label{P:presentationmatrix4graph}
The module $T(\hTheta_u)$ is finitely presented as a $\Z[H]$--module by
\[ \xymatrix{ \Z[H]^E \ar[r]^{D_u} & \Z[H]^E \ar[r] & T(\hTheta_u) \ar[r] & 0}, \]
where $D_u$ is the {\bf square} $m \times m$ matrix $xI - A_u({\bf t})$ with entries in $\Z[H]$.
\end{proposition}

We thus have the following immediate corollary:
\begin{corollary} \label{C:graphidealgenerator}
The g.c.d.~of the fitting ideal of $T(\hTheta_u)$ is given by
\[ \det(xI - A_u({\bf t}))\in \Z[H].\]
\end{corollary}

\begin{example} \label{Ex:calculating poly} 
Here we return again to our running example and calculate $A_u$ and $\det(xI-A_u)$ in this case. We take $u = u_0$ so that $\Theta_u$ is the original fiber $\Gamma$; the infinite cyclic cover $\tTheta_u\to \Gamma$ is then as shown in Figure~\ref{F:cyclic cover}. Here the generator $t$ of $H_u$ is pictured as translating to the right. The edges $\bar E$ of $\Gamma$ are labeled as in Figure~\ref{F:train_tracks}, and we have used the same labels for the representative edges $E$ of $\tTheta_u$; the rest of the edges of $\tTheta_u$ are obtained (and labeled) as translates of these representative edges under $H_u$. 

\begin{figure} [htb]
\begin{center}
\includegraphics[height=4.5cm]{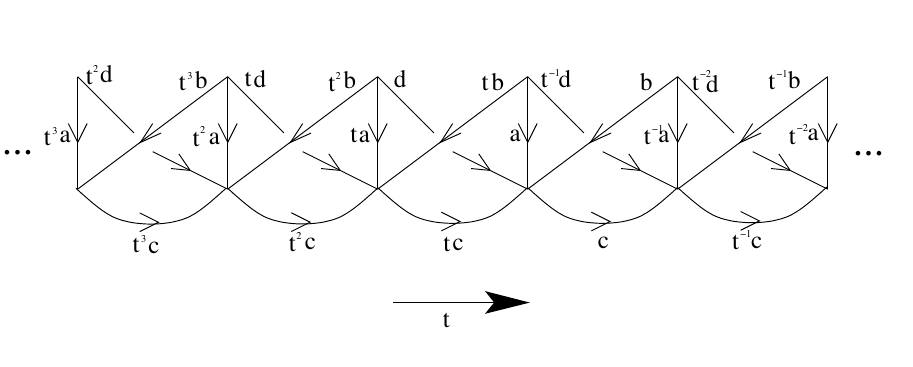} 
\caption{The infinite cyclic cover $\tTheta_u$ of $\Theta_u$ for our running example $f\colon\Gamma\to\Gamma$ when $u=u_0$ is the homomorphism associated to the original splitting $G = F_3\rtimes_\fee \Z$.}
\label{F:cyclic cover}
\end{center}
\end{figure}

The module $T(\hTheta_u)$ is then generated by the edges $\{\sigma_1,\dotsc,\sigma_4\} = \{a,b,c,d\}$. For each $\sigma_j\in \{a,b,c,d\}$, we calculate $x^{-1}\circ \tpsi_1(\sigma_j)$ by lifting the edge path $f(\bar \sigma_j)\subset \Gamma$ to $\tTheta_u$ and expressing it in the form $\sum_{i=1}^4 A_{ij}(t) \cdot [\sigma_j]$ with $A_{ij}(t) \in \Z[t^{\pm}]$. This gives the matrix:

\[A(t) = \begin{pmatrix}
0 & 1 & t^{-1} & t + t^3\\
0 & 0 & 1 & t^2 + t^3\\
0 & 0 & 0 & t\\
1 & 0 & 0 & t^2
\end{pmatrix}\]
Notice that $A(1)$ is the transpose of the transition matrix for $f$ found in \cite[Example 2.9]{DKL}. Calculating $\det(xI - A(t))$, we see that the fitting ideal of $T(\hTheta_u)$ is generated by the polynomial
\[ \poly(t,x) = x^4 -t^2x^3 - t^3 x^2 -t x^2 -t^3 x - t^2x - x -t\in \Z[H].\]
\end{example}

\section{The isomorphism of modules}
\label{secc:module_isom}

In this section we prove that there is a $\Z[H]$--module isomorphism between $T(\hTheta_u)$ and $T(\mathcal F)$ for every primitive integral class $u \in \Csec$.   We first note that by construction of $\Theta_u$, each edge of $\hTheta_u$ is in fact a transversal to $\mathcal F$; see Remark \ref{R:edges_are_transversals} and Convention~\ref{conv:notation_convention}. This gives a natural (injective) $\Z[H]$--module map $S \colon \Z[H]^E \to \mathbb F(\mathcal F)$. We also recall from Definitions~\ref{D:transversals module} and \ref{D:graph module} that the $\Z[H]$--modules $T(\hTheta_u)$ and $T(\mathcal F)$ are the quotients of $\Z[H]^E$ and $\mathbb F(\mathcal F)$ by submodules $M$ and $R$, respectively.  We can now state the main result of this section.

\begin{proposition} \label{P:moduleisomorphism}
The $\Z[H]$--module map $S \colon \Z[H]^E \to \mathbb F(\mathcal F)$ descends to a $\Z[H]$--module isomorphism $\bar{S}\colon T(\hTheta_u) \to T(\mathcal F)$. 
\end{proposition}

\begin{remark}
We emphasize that this isomorphism holds for \emph{every} primitive integral class $u\in \Csec$. In particular, the graph modules $T(\hTheta_u)$ are all isomorphic.
\end{remark}

\begin{proof}
First note that since $H$ preserves the set of basic relations generating $M$, we see that elements of $M$ are precisely $\Z$--linear combinations of elements of the form $(e - e_1 - \dotsb - e_k)$, where $e$ is an edge of $\hTheta_u$ and $\tpsi_1^u(e)$ is the edge path $e_1\dotsc e_k$ in $\hTheta_u$. Similarly, the elements of $R$ are exactly the $\Z$--linear combinations of elements of the form $(\tau - \tau_1 - \dotsb - \tau_k)$ or $(\tau - \tau')$, where $\tau,\tau',\tau_1,\dotsc,\tau_k$ are transversals to $\mathcal F$ as in the basic relations defining $R$ (more precisely, $\tau - \tau_1 - \cdots - \tau_k$ is obtained from the first relation by repeated subdivision).

To see that $S$ descends to the quotient, we need only check that $S(M)$ is contained in $R$. Consider a basic relation $m = e - e_1 - \dotsb - e_k$ in $M$ for edges $e,e_1,\ldots,e_k$ in $\hTheta_u$. By definition, the semiflow $\tpsi_1^u$ then maps $e$ to the edge path $e_1\dotsc e_k$. Let $\tau = S(e)$ and $\tau_i = S(e_i)$. Then $\tau$ may be subdivided as the concatenation $\tau'_1\dotsb\tau'_k$, where each $\tau'_i$ flows homeomorphically onto $\tau_i$ via $\tflow_1^u$. Thus we have
\[S(m) = \tau - \tau_1 - \dotsb - \tau_k = (\tau - \tau'_1 -\dotsb - \tau'_k) + \sum_{i=1}^k(\tau'_i - \tau_i)\in R.\]

To see that $\bar{S}$ is surjective, let $\tau$ be any transversal, and let $[\tau]$ denote its equivalence class in $T(\mathcal F)$. The transversal $\tau$ may be flowed forward locally homeomorphically onto an edge path $e_1\dotsb e_k$ in $\hTheta_u$ (necessarily this edge path lies in one component of $\hTheta_u$). This identification of $\tau$ with $e_1\dotsb e_k$ gives a subdivision of $\tau$ into transversals $\tau_1$, \dots, $\tau_k$, where each $\tau_i$ flows forward homeomorphically onto the edge $e_i$. Therefore
\[[\tau] = [\tau_1] + \dotsb + [\tau_k] = [e_1] + \dotsb + [e_k] = \bar{S}([e_1] + \dotsb [e_k]).\]

\begin{remark}
\label{R:vertex_leaves}
Note that this step uses the fact that transversals have endpoints on vertex leaves of $\mathcal F$.
\end{remark}

The only remaining difficulty, then, is to show that $\bar{S}$ is injective. Suppose that $w \in \Z[H]^E$ is such that $S(w) \in R$. It suffices to show that $w\in M$. By assumption, we may write $S(w)$ as a finite sum of basic relations in $R$:
\[S(w) = \sum_{\alpha\in A} r^\alpha\]
where $A$ is a finite index set and each $r^\alpha$ is either a basic subdivision relation $r^\alpha = \pm(\tau^\alpha - \tau^\alpha_1 - \dotsb - \tau^\alpha_{k_\alpha})$, or a basic flow relation $r^\alpha = \pm(\tau^\alpha -\tau'^\alpha)$ in which $\tau^\alpha$ flows forward homeomorphically onto $\tau'^\alpha$.

Since the sum $\sum_{\alpha} r^\alpha$ comprises only finitely many transversals, we may find some connected component $C$ of $\hTheta_u$ (the image of $\tTheta_u$ by a sufficiently high power of $x \in H$) so that each of these transversals flows forward into $C$. 
After translating by an even higher power of $x$, we may furthermore assume that $C$ is disjoint from each of these transversals (so that $C$ is \emph{strictly} higher than the transversals in question) and that the endpoints of these transversals all flow into vertices of $C$. 

Now, using the relations in $M$ we can push $w$ up into the the graph $C$. More precisely, the assumption on $S(w)$ implies that each edge $e$ occurring with nonzero coefficient in $w$ lives in a component of $\hTheta_u$ that is \emph{below} $C$. Adding a relation in $M$, we may effectively replace $e$ with the sum $e_1 + \dots + e_k$, where $\tpsi_1^u(e)$ is the edge path $e_1\dots e_k$. Doing this repeatedly for each edge in $w$ (and then again until all edges have been pushed up into $C$), we obtain an element $m\in M$ (a sum of basic relations) so that the element $w + m\in \Z[H]^E$ is a $\Z$--linear combination of only edges contained in $C$. More precisely, every edge $e\in \hTheta_u$ appearing with a nonzero coefficient in $w+m$ lives in the component $C$.

We claim that $w+m = 0 \in \Z[H]^E$ (so that $w \in M$ as desired). We have explicitly constructed $m$ as a sum of particular basic relations that only involve edges in the level $C$ or below.
Therefore, the above explicit proof that $S(M)\subset R$ shows us how to express $S(m)$ as a finite sum $S(m) = \sum_{\alpha\in B} r^\alpha$  of basic relations $r^\alpha\in R$ that only involve transversals that are either edges of $C$ or lie strictly below $C$.  We may now write
\[S(w+m) = \sum_{\alpha \in A \sqcup B} r^\alpha,\]
where this is an equality in $\mathbb F(\mathcal F)$. In particular, for every transversal $\tau$ to $\mathcal F$ that lies strictly below $C$, we know that the net coefficient of $\tau$ in the the right hand side must be zero.

Let $\tau$ be any transversal occurring in any of the basic relations $r^\alpha$, $\alpha \in A\sqcup B$. By construction, $\tau$ flows homeomorphically onto some edge path $e^\tau_1\dotsb e^\tau_{k_\tau}$ in $C$. (It may be, if $\tau$ is in one of the new relations coming from $m$, that $\tau$ is already a single edge in $C$, but every other transversal occurring in a $r^\alpha$ lies strictly below $C$.) Flowing the vertices of this edge path backwards onto $\tau$ we obtain a canonical subdivision $\rho^\tau_1\dotsb\rho^\tau_{k_\tau}$ of the transversal $\tau$. Now, for each occurrence $\pm\tau$ of $\tau$ in a basic relation $r^\alpha$, $\alpha \in 
A\sqcup B$, let us add the basic relation $\delta_{\pm\tau} = \pm(\rho^\tau_1 + \dotsb \rho^\tau_{k_\tau} - \tau)\in R$. Notice that the sum of all added relations is $0$ (in $\mathbb F(\mathcal F)$) since each $\tau$ either lies strictly below $C$ and thus occurs with net coefficient $0$ in $S(w+m)$, or is contained in $C$ in which case $\delta_\tau = \rho^\tau_1 - \tau= 0$ by construction (i.e., such a transversal does not get subdivided). Letting $\Omega_\alpha$ denote the set of (signed) transversals $\tau$ occurring in the relation $r^\alpha$ and grouping the new added relations $\delta_{\tau}$ along with the old, we may now write
\[S(w+m) = \sum_{\alpha\in A\sqcup B}\left(r^\alpha + \sum_{\tau \in \Omega_\alpha} \delta_\tau\right).\]

Let's regroup terms and see what we have. First consider a basic subdivision relation $r^\alpha = \tau^\alpha - \tau^\alpha_1- \dotsb - \tau^\alpha_{k_\alpha}$. Let us denote the canonical subdivision of $\tau^\alpha_i$ as $\beta_{i,1}\dotsb\beta_{i,j_i}$. Then since $\tau^\alpha_1\dotsb \tau^\alpha_{k_\alpha}$ was already a subdivision of $\tau^\alpha$, we see that the canonical subdivision of $\tau^\alpha$ must be
\[\tau^\alpha = (\beta_{1,1}\dotsb \beta_{1,j_1})(\beta_{2,1}\dotsb\beta_{2,j_2})\dotsb(\beta_{k_\alpha,1}\dotsb\beta_{k_\alpha,j_{k_\alpha}})\]
Therefore we have complete cancellation in the term $({r^\alpha + \sum_{\Omega_\alpha}\delta_\tau})$. That is, for each such subdivision relation $r^\alpha$ we have 
\begin{align*}
\left(r^\alpha + \sum_{\tau\in \Omega_\alpha}\delta_\tau\right)
&= \tau^\alpha - \tau^\alpha_1 - \dotsb \tau^\alpha_{k_\alpha} 
- \sum_{i = 1}^{k_\alpha}(\beta_{i,1} + \dotsb + \beta_{i,j_i} - \tau^\alpha_i)\\
&\phantom{=}\quad + (\beta_{1,1} + \dotsb + \beta_{1,j_1}+\dotsb + \beta_{\alpha_k,1} + \dotsb \beta_{\alpha_k,j_{\alpha_k}} - \tau^\alpha)\\
&= 0.
\end{align*}

Now consider a flow relation $r^\alpha = \pm(\tau - \tau')$ in which $\tau$ flows forward homeomorphically onto $\tau'$. In this case $\tau$ and $\tau'$ both flow homeomorphically onto the same edge path in $C$, and so in their canonical subdivisions $\beta_1\dotsb \beta_{k_\tau}$ and $\beta'_1\dotsb \beta'_{k_{\tau'}}$ we must have that $k_\tau = k_{\tau'}$ and that $\beta_i$ flows homeomorphically onto $\beta'_i$ for each $i$ (after possibly reindexing).
Therefore, the term $({r^\alpha + \sum_{\Omega_\alpha} \delta_\tau})$ may be reorganized
\begin{align*}
\left(r^\alpha + \sum_{\tau\in \Omega_\alpha} \delta_\tau\right)
=& \pm(\tau - \tau') \pm(\beta_1 + \dotsb + \beta_{k_\tau} - \tau) \mp (\beta'_1 + \dotsb \beta'_{k_\tau} - \tau')\\
=& \pm\left( (\beta_1 - \beta'_1) + \dotsb + (\beta_{k_\tau} - \beta'_{k_\tau})\right)
\end{align*}
into a sum of basic flow relations. Furthermore, each transversal appearing in this sum flows homeomorphically onto a single edge of $C$.

The above two paragraphs show that, in $\mathbb F(\mathcal F)$, we may write $S(w+m)$ as a finite sum $\sum_{i = 1}^N r_i$ of basic flow relations $r_i = \tau_i - \tau'_i$ between two transversals that both flow homeomorphically onto a single edge of $C$. Let $\Omega$ denote the set of all transversals appearing in the sum $\sum_i r_i$, and let us partition $\Omega$ into equivalence classes in which two transversals are equivalent if they both flow homeomorphically onto the same edge of $C$. Notice that, since the two transversals appearing in a relation $r_i = \tau_i - \tau'_i$ are in the same equivalence class, the sum of the coefficients of all transversals in a given equivalence class must be zero. On the the other hand, for each transversal $\tau\in \Omega$ that is not an edge of $C$, the the net coefficient of $\tau$ in the sum $\sum_i r_i$ must also be zero (since $\sum_i r_i = S(w+m)$ in $\mathbb F(\mathcal F)$, and the only edges of $\hTheta_u$ with nonzero coefficient in $w+m$ lie in $C$). As each equivalence class contains at most one transversal $\tau\in \Omega$ which is an edge of $C$, these two facts imply that the net coefficient of such a $\tau$ (i.e., of a transversal which is an edge of $C$) must also be zero. Therefore, for each $\tau\in \Omega$, the net coefficient of $\tau$ in the sum $\sum_i r_i$ is zero. This proves that $S(w+m) = \sum_i r_i = 0$ in $\mathbb F(\mathcal F)$.  On the other hand, $S$ is injective and hence $w + m = 0$ in $\Z[H]^E$ as required.
\end{proof}

\section{Specialization, characteristic polynomials, and stretch factors}
\label{S:specialization}

Propositions~\ref{P:moduleisomorphism} and \ref{P:presentationmatrix4graph} together prove  Proposition \ref{P:fgmodule} and thus verify that the McMullen polynomial $\poly$ is well-defined. Furthermore, Corollary \ref{C:graphidealgenerator} now provides an explicit way to calculate $\poly$:  Given a primitive integral class $u \in \Csec$, let ${\bf t},x$ be a basis for $H = H_u \oplus \Z$ adapted to $u$ as in Definition \ref{D:adapted to z}. Viewing $\poly$ as an integral Laurent polynomial $\poly({\bf t},x) \in \Z[{\bf t}^{\pm},x^{\pm 1}]$, we can now prove the ``determinant formula'' from the introduction.

\begin{theorem:determinant}[Determinant formula]
For any primitive integral $u \in \Csec$ we have
\[ \poly({\bf t},x) = \det(xI - A_u({\bf t}))\]
up to units in $\Z[{\bf t}^{\pm 1},x^{\pm 1}]$.
\end{theorem:determinant}
\begin{proof}  Since the g.c.d.~of the fitting ideal of a module depends only on the isomorphism type of the module up to units, the claim follows from Corollary~\ref{C:graphidealgenerator} and Proposition~\ref{P:moduleisomorphism}.
\end{proof}

The determinant formula also tells us about characteristic polynomials of transition matrices for the expanding irreducible train track maps $f_u \colon \Theta_u\to\Theta_u$. The relationship can be stated in terms of the specializations of $\poly$. To describe these, first note that any cohomology class $u\in H^1(G;\R)$ (not necessarily integral) determines a homomorphism $u\colon H\to \R$. Writing the McMullen polynomial as a finite sum
\[\poly = \sum_{h\in H} a_h h \in \Z[H],\]
the \emph{specialization of $\poly$ at $u$} is then the single variable power sum given by
\[ \poly_{u}(\zeta) = \sum_{h\in H} a_h \zeta^{u(h)}.\]
The exponents here are real numbers, but in the case that $u$ is integral, this is in fact an integral Laurent polynomial in $\zeta$.

\begin{corollary}[Specialization]
\label{C:specialization}
For any primitive integral $u \in \Csec$, the specialization $\poly_u(\zeta)$ is equal (up to a factor $\pm \zeta^j$ for some $j \in\Z$) to the characteristic polynomial of the transition matrix for the train track map $f_u \colon \Theta_u \to \Theta_u$.
\end{corollary}
\begin{proof}
We first recall  that the kernel of the homomorphism $u \colon H\to \Z$ is exactly $H_u< H$. By definition of the basis $t_1,\dots,t_{b-1},x$ of $H = H_u\oplus \Z$, we consequently see that $u(t_i) = 0$ for each $i = 1,\ldots,b-1$ and that $u(x) = 1$. Thus by Theorem~\ref{T:determinant formula} we have
\[ \poly_u(\zeta) = \poly(\zeta^0,\ldots,\zeta^0,\zeta^1) = \det(\zeta^1I - A_u(\zeta^0,\ldots,\zeta^0)) = \det(\zeta I - A_u(1,\ldots,1)).\]
But $A_u(1,\ldots,1)$ is the transition matrix for $f_u$ by Proposition \ref{P:evaluatetransition}.
\end{proof}

We may now prove Theorem~\ref{T:main_theorem} from the Introduction:
\begin{proof}[Proof of Theorem~\ref{T:main_theorem}]
The existence of $\poly$ follows from Proposition~\ref{P:fgmodule}, and Corollary~\ref{C:specialization} implies that the specialization $\poly_u(\zeta)$ at a primitive integral class $u\in \Csec$ is, up to a factor, equal to the characteristic polynomial of $f_u$. Thus the largest root of $\poly_u(\zeta)$ is the spectral radius $\lambda(f_u)$ of the transition matrix of $f_u$. Since $f_u$ is an irreducible train track map, Proposition~\ref{P:stretch} gives $h(f_u) = \log(\lambda(f_u))$ so that the equalities $\lambda(f_u) = \lambda((f_u)_\ast) = \lambda(\phi_u)= \Lambda(u)$ follow from Theorem~\ref{T:splittings}(5).
\end{proof}

\begin{example}[Specializing at $u_1$ and $u_2$]\label{Ex:specializing}
We now use Theorem~\ref{T:main_theorem} to calculate the stretch factors associated to the classes $u_1$ and $u_2$. Using the splitting $H = \Z[s]\oplus\Z[w]$ of $H$ adapted to the class $u = u_0$, the McMullen polynomial is then given, as calculated in Example~\ref{Ex:calculating poly}, by
\[ \poly(t,x) = x^4 -t^2x^3 - t^3 x^2 -t x^2 -t^3 x - t^2x - x -t\in \Z[H].\]
Since $(s(u_1),w(u_1)) = (-1,2)$, the specialization of $\poly$ at $u_1$ becomes
\begin{align*}
\poly_{u_1}(\zeta) 
&=\zeta^8 -\zeta^{-2}\zeta^6-\zeta^{-3}\zeta^4-\zeta^{-1}\zeta^4-\zeta^{-3}\zeta^2-\zeta^{-2}\zeta^2-\zeta^2-\zeta^{-1}\\
&=\zeta^{-1}(\zeta^9-\zeta^5-\zeta^4-\zeta^3-\zeta^2-\zeta-2).
\end{align*}
Up to a power of $\zeta$, this is exactly the characteristic polynomial of $A(f_1)$ for the first return map $f_1$ calculated in Example~\ref{Ex:(-1,2) introduced}. Furthermore,  the largest positive root of $\poly_{u_1}$ is $\lambda_{u_1}\approx 1.35827$, which we note agrees with the stretch factor of the injective endomorphism $\phi_1$ computed in Example~\ref{Ex:(-1,2)_splitting}.

On the other hand, as $(s(u_2),w(u_2)) = (-1,1)$, the specialization of $\poly$ at $u_1$ is
\begin{align*}
\poly_{u_2}(\zeta) 
&= \zeta^4 - \zeta^{-2}\zeta^3 - \zeta^{-3}\zeta^2 - \zeta^{-1}\zeta^2 - \zeta^{-3}\zeta - \zeta^{-2}\zeta - \zeta - \zeta^{-1}\\
&= \zeta^{-2}(\zeta^6 - 3\zeta^3 - 3\zeta - 1),
\end{align*} 
which, again up to a power of $\zeta$, is exactly the characteristic polynomial of $A(f_2)$ for the first return map $f_2$ found in Example~\ref{Ex:(-1,1) introduced}. We also find that the largest positive root of $\poly_{u_2}$ is $\lambda_{u_2} \approx 1.63299$, in agreement with the stretch factor of the monodromy $\fee_2 = \fee_{u_2}$ calculated in Example~\ref{Ex:(-1,1)_splitting}. Thus the data provided by $\poly$ for the classes $u_1$ and $u_2$ agrees with the conclusion of Corollary~\ref{C:specialization}.
\end{example}

\subsection{Subdivision}
\label{S:subdivision and poly calculation}

Recall from \S\ref{S:foliation} that the foliation $\mathcal F$ of $X$ has distinguished vertex leaves. These play an important role in the definition of the module $T(\mathcal F)$ (see Definitions~\ref{D:the_transversals}--\ref{D:transversals module}), and this is reflected in the McMullen polynomial $\poly$. 
 
For example, suppose that $\Theta\subset X$ is a cross section equipped with \emph{any} topological graph structure for which $f_\Theta$ sends vertices to vertices, and let $T(\hTheta)$ be the corresponding graph module as constructed in \S\ref{S:graph module}. If all vertices of $\Theta$ lie on vertex leaves of $\mathcal F$, then the proof of Proposition~\ref{P:moduleisomorphism} still goes through verbatim and gives an isomorphism $T(\hTheta)\to T(\mathcal F)$. Thus any finite presentation of $T(\hTheta)$ may be used to calculate the McMullen polynomial $\poly$.

However, if $V\Theta$ contains points that do not lie on vertex leaves, then Proposition~\ref{P:moduleisomorphism} fails (see Remarks~\ref{R:edges_are_transversals} and \ref{R:vertex_leaves}) and so the fitting ideal of $T(\hTheta)$ is conceivably different from that of $T(\mathcal F)$. Nevertheless, if we \emph{declare} the leaves of $\mathcal F$ through $V\Theta$ to be vertex leaves---effectively creating a new foliation $\mathcal F'$ with a new module $T(\mathcal F')$ and corresponding McMullen polynomial $\poly'$---then Proposition~\ref{P:moduleisomorphism} again gives an isomorphism $T(\hTheta')\to T(\mathcal F')$ that may be used to calculate $\poly'$. 

Theorem~\ref{T:determinant-subdivision} below explains exactly how the McMullen polynomial changes in this situation. To set up notation, let $\Theta_u$ be an $\mathcal F$--compatible cross section dual to $u\in \Csec$ and equipped with the standard graph structure, and let $\Theta'_u$ be the graph structure obtained by subdividing $\Theta_u$ along a finite set $\mathcal V\subset \Theta_u$ that is preserved by $f_u$. Notice that any such set $\mathcal V$ is necessarily disjoint from the vertex leaves of $\mathcal F$.  This is because any point contained in a vertex leaf eventually flows into a vertical $1$--cell of $X$ and thus onto a vertex of $\Theta_u$. Let $\mathcal F'$ be the foliation obtained by declaring the leaves through $\mathcal V$ to be vertex leaves, and let $\poly'$ be the corresponding McMullen polynomial (i.e., the g.c.d. of the fitting ideal of $T(\mathcal F')$).

Let $\{\mathbf{t},x\}$ be a basis of $H$ adapted to $u$. If $\tV \subset \tTheta_u$ denotes the preimage of $\mathcal V$, then the group $H_u\leq H$ acts freely on $\tV$. Thus if $\mathcal V_0 = \{v_1,\dotsc,v_k\}\subset \tV$ is any set of orbit representatives, then the free $\Z$--module on $\tV$ is isomorphic to the finitely generated free $\Z[H_u]$--module on $\mathcal V_0$, namely $\Z[H_u]^{\mathcal V_0}\cong \Z[H_u]^k$. Choosing a lift $\widetilde{f}_u \colon \tTheta'_u \to \tTheta'_u$ of $f_u$ then induces a $\Z[H_u]$--linear map $\Z[H_u]^{\mathcal V_0} \to \Z[H_0]^{\mathcal V_0}$ which we may represent by a matrix $B_u(\mathbf t)$ with entries in $\Z[H_u]$ as follows: Choose an orientation on each edge of $\Theta_u$ and lift to an orientation on the edges of $\tTheta_u$. For each vertex $v_j\in \mathcal V_0$, we have $\widetilde{f}_u(v_j) = h \cdot v_i$ for a unique $v_i\in \mathcal V_0$ and $h\in H_u$. Accordingly we set $B_{u,ij} = \pm h$, where the sign $\pm$ records whether or not $\widetilde{f}_u$ preserves the orientation at $v_j$. All other entries in the $j$--column of $B_u$ are set to $0$. We may now state the relationship between $\poly$, $\poly'$ and $B_u(\mathbf{t})$.

\begin{theorem}\label{T:determinant-subdivision}
Suppose $\Theta'_u$ is obtained by subdividing $\Theta_u$ along the finite $f_u$--invariant set $\mathcal V$ as described above. Let $\mathcal F'$ denote the foliation in which the leaves through $\mathcal V$ are declared to be vertex leaves, and let $\poly,\poly' \in \Z[H]$ be the McMullen polynomials corresponding to the modules $T(\mathcal F)$ and $T(\mathcal F')$. Then with $B_u({\bf t})$ as above, up to units in $\Z[H]$ we have
\[ \poly'(\mathbf{t},x) = \poly(\mathbf{t},x) \cdot \det(xI-B_u({\bf t})).\]
\end{theorem}
\begin{proof}
By Proposition~\ref{P:moduleisomorphism} we have $T(\mathcal F') \cong T(\hTheta'_u)$, thus the McMullen polynomial $\poly'$ may be calculated using any finite presentation of $T(\hTheta'_u)$.
By choosing this presentation carefully, the above formula will become obvious.

As in \S\ref{S:graph module}, let $\bar{E} = \{\bar\sigma_1,\dotsc,\bar\sigma_m\}$ be the edges of $\Theta_u$, and let $E = \{\sigma_1,\dotsc,\sigma_m\}$ denote chosen lifts in $\tTheta_u$. These choices determine the transition matrix $A_u(\mathbf t)$ as in \S\ref{S:transition_matrices}. Recall also that, as described above, $\mathcal V_0$ is any set of orbit representatives for the $H_u$ action on the preimage $\tV\subset\tTheta'_u$ of $\mathcal V$, and that $B_u(\mathbf{t})$ describes the $\Z[H_u]$--linear action of the lift $\widetilde{f}_u\colon \tTheta_u\to\tTheta_u$ on $\tV$.

We now choose convenient generators for $T(\hTheta'_u)$. As described above, fix an orientation on each edge of $\Theta_u$ and lift to an orientation on the edges of $\tTheta_u$. Each vertex $v_j\in \mathcal V_0$ is contained in a unique edge $\sigma$ of $\tTheta_u$, and we let $\alpha^-_j$ denote the initial subarc of $\sigma$ that terminates at $v_j$. We similarly let $\alpha_j^+$ denote the terminal subarc. Each of the arcs $\sigma_i,\alpha_j^\pm\subset\tTheta_u$ are comprised of edges of $\hTheta'_u$ and as such may be thought of as elements of the free $\Z$--module on the edges of $\hTheta'_u$. Moreover, since we can add and subtract $H$--translates of these to obtain representatives of every $H$--orbit, the $m+k$ elements $\{\sigma_1,\dotsc\sigma_m,\alpha_1^-,\dotsc,\alpha_k^-\}$ freely generate this module.  There is a bijection between this set and $E \cup \mathcal V_0$ which we use as a notational convenience.

By Definition~\ref{D:graph module}, we now see that $T(\hTheta'_u)$ is the quotient of $\Z[H]^{E\cup \mathcal V_0}$ by a certain submodule $M$ of relations. Namely each arc $\beta\in E\cup \mathcal V_0$ must satisfy the relation
\[[\beta] = \left[\tflow_1^u(\beta)\right],\]
where here $\beta$ and $\tflow_1^u(\beta)$ are shorthand for the sum of edges of $\hTheta'_u$ comprising those edge paths. Let us express these relations in terms of our generating set. For $\sigma_j\in E$ the corresponding relator may, as in \S\ref{S:transition_matrices}, be rewritten as $(x\cdot \sigma_j - \sum_i A_{ij}(\mathbf t)\cdot \sigma_i)$; notice that this does not involve any of the generators $\alpha_i^-$.

By definition of $B_u$, for each $v_j\in \mathcal V_0$ we have $\widetilde{f}_u(v_j) = \sum_{i} B_{u,ij}\cdot v_i = h\cdot v$ for some $h\in H_u$ and $v\in \mathcal V_0$. Notice that $\tflow_1^u(x\cdot \alpha_j^-)$ is an edge path in $\tTheta'_u$ that starts at a vertex of $\tTheta_u$, crosses several complete edges of $\tTheta_u$, and then terminates at $\widetilde{f}_u(v_j) = h \cdot v$. Therefore in $\Z[H]^{E\cup \mathcal V_0}$ this edge path may be expressed as
\begin{align*}
\tflow_1^u(x\cdot \alpha_j^-) &= h_1\cdot\sigma_1 + \dotsb + h_m\cdot\sigma_m + \sum_{i}B_{u,ij}(\mathbf t)\cdot\alpha_i^-,
\end{align*}
for some coefficients $h_i\in H$ (here we have used the fact that $\alpha_i^-+\alpha_i^+$ is a complete edge of $\hTheta_u$). Since the above $m+k$ relators (one for each element of $E\cup \mathcal V_0$)  generate the submodule $M$ of all relations, it follows that the $\Z[H]$--module $T(\hTheta'_u)$ is presented by a block matrix of the form
\[\begin{pmatrix} xI - A_u(\mathbf t) & \ast \\ 0 & xI - B_u(\mathbf t) \end{pmatrix}.\]
Since $\poly(\mathbf t,x) = \det(xI-A_u(\mathbf t))$ by Theorem~\ref{T:determinant formula}, the result follows.
\end{proof}

\begin{corollary}
If $\poly$ and $\poly'$ are McMullen polynomials as in Theorem~\ref{T:determinant-subdivision}, then for any integral class $u\in \Csec$ the specialization $\poly_u(\zeta)$ is a factor of $\poly'_u(\zeta)$.
\end{corollary}
\begin{proof}
For any  $u\in H^1(G,\Z)$, the assignment
\[\sum_{h\in H} c_h h  \mapsto \sum_{h\in H} c_h \zeta^{u(h)}\]
defines a ring homomorphism $\Z[H]\to \Z[\zeta^{\pm1}]$. By definition of specialization, this homomorphism sends $\poly$ to $\poly_u$ and $\poly'$ to $\poly'_u$.  Therefore, since Theorem~\ref{T:determinant-subdivision} implies that $\poly'$ is a product of $\poly$ with another element of $\Z[H]$, it follows that $\poly'_u$ is the product of $\poly_u$ and another element of $\Z[\zeta^{\pm1}]$.
\end{proof}

\begin{example} \label{Ex:subdivision matters}
There is a fixed point in the $d$--edge of our running example, and we can subdivide $\Gamma$ at this point to produce $\Gamma'$.  Looking at Figure \ref{F:cyclic cover} if we choose our representative point in the edge labeled $d$, then the image is contained in the edge labeled $t^2d$.    Thus, $B(t) = t^2$, and we see that
\[ \poly'(t,x) = \poly(t,x)(x-t^2).\]
\end{example}

\section{Real-analyticity, convexity, and divergence}
\label{S:convexity}

Theorem~D of \cite{DKL} showed that the logarithm of the stretch factor of $\fee_u$, for primitive integral $u\in \A$, extends to a continuous, convex, and homogeneous of degree $-1$ function $\mathfrak H \colon \A\to \R$. This theorem was an analogue of Fried's result in the setting of fibered hyperbolic $3$--manifolds \cite{FriedD}. Following McMullen's more recent approach to this $3$--manifold result \cite{Mc}, we now use the McMullen polynomial $\poly$ to give an alternate proof of \cite[Theorem D]{DKL} and to moreover extend this result to the entire cone of sections $\Csec$.

This discussion involves yet another cone $\C_X\subset H^1(X;\R)$, termed the \emph{McMullen cone}, that is naturally determined by the McMullen polynomial. We use the properties of $\poly$ to show that the above function $\mathfrak H\colon A\to \R$ naturally extends, with the same properties, to the entire cone $\C_X$ and that this extension diverges at $\partial \C_X$. On the other hand, the relationship between $\poly$ and the primitive classes $u\in \Csec$ (as illuminated by Theorem~\ref{T:main_theorem}) shows that this extension $\mathfrak H\colon \C_X\to \R$ is finite on $\Csec$ and therefore that $\Csec \subseteq \C_X$. Finally, we prove that $\mathfrak H$ in fact diverges as we approach the boundary of $\Csec$ by establishing the equality $\Csec = \C_X$, which is Theorem \ref{T:cones_are_equal}. This proof combines the characterization of $\Csec$ provided by Theorem~\ref{T:cone_of_sections} with another formula for $\poly$ that is due to \cite{AHR}.

\subsection{McMullen's Perron-Frobenius theory}

Given a primitive integral element $u \in \Csec$, let $H = H_u \oplus \Z$ and ${\bf s},{\bf t} = e^{\bf s},w$, and $x=e^w$ be adapted to $u$ as in Definition \ref{D:adapted to z}. We then express elements $u \in H^1(X;\R)$ in terms of their $({\bf s},w)$--coordinates following Convention \ref{Conv:homology coordinates}:
\[ u = ({\bf s}(u),w(u))= ({\bf s},w).\]

Given a nonzero polynomial $\mathfrak p({\bf t},x) \in \Z[{\bf t}^{\pm 1},x^{\pm 1}]$, we express this polynomial as
\[ \mathfrak p({\bf t},x) = \sum_{{\bf j} \in J} a_{\bf j} ({\bf t}x)^{\bf j}\]
where ${\bf j} = (j_1,\ldots,j_b)$ runs over some finite index set $J$, $({\bf t}x)^{\bf j} = t_1^{j_1}\cdots t_{b-1}^{j_{b-1}} x^{j_b}$, and $a_{\bf j} \neq 0$ for all ${\bf j} \in J$.  Then given ${\bf j} \in J$, the {\em dual cone of $a_{\bf j} ({\bf t}x)^{\bf j}$ for $\mathfrak p$} is defined by
\[ \C({\mathfrak p},a_{\bf j} ({\bf t}x)^{\bf j}) = \{ ({\bf s},w) \in \R^b \mid {\bf j} \cdot ({\bf s},w) > {\bf j}' \cdot ({\bf s},w) \mbox{ for all } {\bf j}' \in J \setminus \{ {\bf j} \} \} \]
Here, ``$\cdot$'' denotes the standard dot product on $\R^b$.  The cone $\C({\mathfrak p},a_{\bf j} ({\bf t}x)^{\bf j})$ is usually referred to as the dual cone of the Newton polytope of ${\mathfrak p}$ associated to $a_{\bf j} ({\bf t}x)^{\bf j}$.  The dual cones for terms $a_{\bf j} ({\bf t}x)^{\bf j}$ and $a_{\bf j'}({\bf t}x)^{\bf j'}$ of the same polynomial $\mathfrak p$ are either disjoint or equal.

As $({\bf s},w)$ are coordinates on $H^1(X;\R)$, we can view $\C({\mathfrak p},a_{\bf j} ({\bf t}x)^{\bf j}) \subset H^1(X;\R)$.  This cone depends only on $\mathfrak p \in \Z[H]$ and $({\bf t}x)^{\bf j} \in H$.  Furthermore, if $\mathfrak q = \pm ({\bf t}x)^{\bf i} \mathfrak p$, then
\[ \C({\mathfrak q},\pm a_{\bf j} ({\bf t}x)^{\bf j + i}) = \C({\mathfrak p},a_{\bf j} ({\bf t}x)^{\bf j}) \subset H^1(X;\R).\]

A polynomial in $\Z[{\bf t}^{\pm 1}]$ is {\em nonnegative} if all coefficients are nonnegative, and {\em positive} if it is nonnegative and nonzero.
Following \cite{Mc}, a matrix $P({\bf t})$ with entries in $\Z[{\bf t}^{\pm 1}]$ is said to be {\em Perron-Frobenius} if each entry is nonnegative, for every $i,j$ there exists $k > 0$ so that $(P({\bf t})^k)_{ij}$ is positive, and $P({\bf 1})$, the matrix obtained by evaluating ${\bf t}$ at $(1,\ldots,1)$, is not a permutation matrix. In \cite[Appendix A.1]{Mc}, McMullen proves the following theorem. 

\begin{remark}
This is a slightly weaker definition of Perron-Frobenius matrix than that given in \cite{Mc}, but the proof in that paper goes through in this setting as well with only minor modifications.
\end{remark}

\begin{theorem}[McMullen \cite{Mc}]
\label{T:PF}
Let $E({\bf t})$ be the leading eigenvalue of a $k \times k$ Perron-Frobenius matrix $P({\bf t})$ with entries in $\Z[{\bf t}^{\pm 1}]$.  Then:
\begin{enumerate}
\item The function $\xi({\bf s}) = \log(E(e^{\bf s}))$ is a  convex function of ${\bf s} \in \R^{b-1}$
\item The graph $\mathfrak G = \{ ({\bf s},w) \mid w = \xi({\bf s}) \}$ meets each ray from the origin in $\R^{b-1} \times \R$ at most once.
\item The rays passing through $\mathfrak G$ coincide with the dual cone $\C({\mathfrak p},x^k)  \subset \R^{b-1} \times \R$ for $\mathfrak p = \det(xI-P({\bf t})) \in \Z[{\bf t}^{\pm 1},x^{\pm 1}]$.
\end{enumerate}
\end{theorem} 

Note that the terms of $\det(xI-P({\bf t}))$ other than $x^k$ have a smaller exponent on $x$, and hence $({\bf 0},1) \in \R^{b-1} \times \R$ is contained in $\C({\mathfrak p},x^k)$.

\begin{remark}
In fact the statement in \cite{Mc} is sharper, but the above theorem will suffice for our purposes.
\end{remark}

\begin{lemma}
\label{L:perron-frobenius}
For any primitive integral class  $u\in \Csec$, the transition matrix
\[ A_u({\bf t}) \in M_{m \times m}\left( \Z[{\bf t^{\pm 1}}] \right)\]
is a Perron-Frobenius matrix with entries in $ \Z[{\bf t^{\pm 1}}] = \Z[H_u]$.
\end{lemma}
\begin{proof}
The entries of $A_u({\bf t})$ are nonnegative Laurent polynomials in $\Z[{\bf t}^{\pm 1}]$.  By Proposition~\ref{P:evaluatetransition}, $A_u({\bf 1})$ is the transition matrix for $f_u$ which is an expanding irreducible train track map by Theorem \ref{T:DKL_BC}.  Therefore, $A_u({\bf 1})$ is not a permutation matrix.  Moreover, for every $i,j$ there exists $k$ so that $(A_u({\bf 1})^k)_{ij} > 0$, which implies $(A_u({\bf t})^k)_{ij}$ is a positive Laurent polynomial, as required.
\end{proof}

\subsection{The McMullen cone} \label{S:mcmullen cone}

From Theorem~\ref{T:determinant formula}, we have $\poly = \det(xI - A_u({\bf t}))$, and we let
\[ \C_X(u) = \C(\det(xI - A_u(\mathbf{t})),x^m) = \C(\poly,x^m) \subset H^1(X;\R) \]
be the cone from Theorem \ref{T:PF}, which contains $u$.  Since any two primitive integral elements $u,u' \in \Csec$ define the same polynomial $\poly \in \Z[H]$ up to units in $\Z[H]$, it follows that we either have $\C_X(u)= \C_X(u')$ or $\C_X(u) \cap \C_X(u') = \emptyset$.

\begin{lemma} \label{L:contained}
For any two primitive integral elements $u,u' \in \Csec$, $\C_X(u) = \C_X(u')$.  Furthermore, writing $\C_X = \C_X(u)$ for any (hence every) primitive integral $u \in \Csec$, we have $\Csec \subseteq \C_X$.
\end{lemma}
\begin{proof}
Suppose $\C_X(u) \neq \C_X(u')$ for two primitive integral classes $u,u' \in \Csec$, so that as mentioned above $\C_X(u) \cap \C_X(u') = \emptyset$. Because the cones are defined by linear inequalities with integer coefficients, it follows that the line segment in $H^1(X;\R)$ between $u$ and $u'$ meets the boundary of $\overline{\C_X(u)}$ at some rational point.  This rational point is a rational multiple of some primitive integral point $u'' \in \Csec$, and hence $u''$ lies in the boundary of $\overline{\C_X(u)}$.  Since the cones are open, it follows that $\C_X(u'') \cap \C_X(u) \neq \emptyset$, and so $\C_x(u) = \C_x(u'')\ni u''$. This contradicts the fact that $u''$ is in the boundary of $\overline{\C_X(u)}$.   Therefore, $\C_X(u) = \C_X(u')$ for all primitive integral $u,u' \in \Csec$ and we let this common cone be denoted $\C_X$.

It follows that every primitive integral point and hence every rational line of $\Csec$ is contained in $\C_X$. By convexity, $\C_X$ then contains the convex hull of such rays which, by Proposition~\ref{P:Csec convex}, is equal to $\Csec$.
\end{proof}

\begin{defn}\label{D:McMcone}
We call the cone $\C_X\subseteq H^1(G,\R)$ constructed in Lemma~\ref{L:contained} the \emph{McMullen cone} associated to the folded mapping torus $X$.
\end{defn}

In fact the containment $\Csec \subseteq \C_X$ can be promoted to an equality:

\begin{theorem:cones-are-equal}[McMullen polynomial detects $\Csec$]
The McMullen cone $\C_X$ is equal to the cone of sections $\Csec$.
\end{theorem:cones-are-equal}

The proof of Theorem~\ref{T:cones_are_equal} is somewhat involved and so is postponed to the end of this section.  For now we use it to prove that $\mathfrak H$ extends to all of $\Csec$ with many nice properties.

\begin{example}[Computing the McMullen cone $\C_X$]
\label{E:A}
The McMullen polynomial for our running example is, as calculated in Example~\ref{Ex:calculating poly}, given by
\[ \poly(t,x) = x^4 -t^2x^3 - t^3 x^2 -t x^2 -t^3 x - t^2x - x -t\in \Z[H].\]
Accordingly, we calculate that the dual cone $\C(\poly(t,x),x^4)$ is the interior of the convex hull of the rays through $(1,2)$ and $(-1,0)$; this is exactly equal to the cone of sections $\Csec$ as calculated in Example~\ref{Ex:the_cones}, and so $\C_X$ has already been illustrated in Figure~\ref{F:Acone-Scone}. We note that the equality $\C_X=\Csec$ here is in agreement with Theorem~\ref{T:cones_are_equal}.
\end{example}

\subsection{Stretch factors}

In the notation we have established the specialization of $\poly(\mathbf t,x)$ at any point $({\bf s},w) \in H^1(X;\R)$ is given by
\[  \poly_{({\bf s},w)}(\zeta) = \poly(\zeta^{\bf s},\zeta^w).\]
We now define a function $k\colon \Csec\to \R_+\cup \{\infty\}$ by setting
\[ k({\bf s},w) = \sup \{ \zeta \in \R_+ \mid \poly_{({\bf s},w)}(\zeta) = 0 \} \]
for $({\bf s},w)\in \Csec$.
If there does not exist $\zeta \in \R_+$ such that $\poly_{({\bf s},w)}(\zeta) = 0$, we interpret the above definition as $k({\bf s},w) =\infty$.
We note that from the definition we have $k(qu) = k(u)^{1/q}$ for any $u \in \Csec$ and $q \in \R_+$ (with the convention that $\infty^{1/q} = \infty$).

Note that for any primitive integral element $({\bf s},w)=u\in \Csec$,  Theorem~\ref{T:main_theorem} shows that $k({\bf s},w) = \lambda(\f_u)$, the Perron-Frobenius eigenvalue of the transition matrix $A(f_u)$ of the expanding irreducible train track map $f_u\colon \Theta_u\to\Theta_u$. Thus for such $({\bf s},w)\in \Csec$ we have $k({\bf s},w) > 1$, and consequently $k>1$ on rays through integral points in $\Csec$.
A priori, $k$ could be equal to infinity on rays through non-integral points.  The next proposition shows that this is in fact not the case, and $1 < k({\bf s},w) < \infty$ for all $({\bf s},w)\in \Csec$.

\begin{proposition} \label{P:PFconsequence}
Let $A_u(\mathbf t)$ be the Perron-Frobenius matrix from Lemma~\ref{L:perron-frobenius}, and let $\mathfrak G$ be the corresponding graph as in Theorem~\ref{T:PF}. The following hold:
\begin{enumerate}
\item For any $({\bf s},w)\in \mathfrak G$ we have $\poly_{({\bf s},w)}(e) = 0$.
\item  $\mathfrak G$ is the level set $k({\bf s},w) = e$ in $\Csec$.  
\item For all $({\bf s},w) \in \Csec$, $1 < k({\bf s},w) < \infty$ and the function $k\colon\Csec \to (1,\infty)$ is real-analytic.
\end{enumerate}
\end{proposition}
\begin{proof}
Fix $({\bf s},w) \in \mathfrak G$, so that $w = \log(E(e^{\bf s}))$.  According to Theorem \ref{T:PF} the ray though $({\bf s},w)$ is contained in $\C_X$, and by Theorem \ref{T:cones_are_equal}, it follows that $({\bf s},w) \in \Csec$.  We claim that $k({\bf s},w) = e$.  To prove this claim, first observe that
\[ \poly_{({\bf s},w)} (e) =  \poly(e^{\bf s},e^w) = \poly(e^{\bf s},e^{\log(E(e^{\bf s}))}) = \poly(e^{\bf s},E(e^{\bf s})) = 0.\]
This last equality follows from the fact that $E(e^{\bf s})$ is an eigenvalue of $A_u(e^{\bf s})$, and $\poly(e^{\bf s},x)$ is its characteristic polynomial.
Thus (1) is verified.

By definition, we now have $e \leq k({\bf s},w)$.  Suppose that $k({\bf s},w) = k_0 > e$ for some $({\bf s},w)$.  Set $\lambda_0 = \log(k_0) > 1$.  We then have
\[ 0 = \poly_{({\bf s},w)}(k_0) = \poly(k_0^{\bf s},k_0^w) = \poly(e^{\log( k_0^{\bf s})},e^{\log (k_0^w)}) = \poly(e^{\lambda_0 {\bf s}},e^{\lambda_0 w}) = \poly_{(\lambda_0 {\bf s},\lambda_0 w)}(e).\]
But this says that $x = e^{\lambda_0 w}$ is a zero of $\poly(e^{\lambda_0 {\bf s}},x)$ and hence is an eigenvalue of $A_u(e^{\lambda_0 {\bf s}})$.  By definition $E(e^{\lambda_0 {\bf s}})$ is the largest such eigenvalue and thus it follows that
\[ E(e^{\lambda_0 {\bf s}}) \geq e^{\lambda_0 w}. \]
Therefore
\[ \log(E(e^{\lambda_0 {\bf s}})) \geq \lambda_0 w.\]
We claim that this is a contradiction. To see this, first observe that by Theorem \ref{T:PF} the ray through $({\bf s},w)$ intersects $\mathfrak G$ in exactly one point, namely the point $({\bf s},w)$.  For all $\lambda > 1$ we must have $\lambda ({\bf s},w)$ above the graph $\mathfrak G$---that is, $\lambda w > \log(E(e^{\lambda {\bf s}}))$.    Since $\lambda_0 > 1$, we arrive at the desired contradiction. Thus (2) is verified.

To see that (3) holds, let $({\bf s},w)\in \Csec$ be arbitrary. Then there exists a unique $q>0$ such that $({\bf s}_\ast,w_\ast)=({\bf s}/q,w/q)\in \mathfrak G$. Then  $k({\bf s}_\ast,w_\ast)=e$ and
\[
k({\bf s},w)=k(q{\bf s}_\ast,qw_\ast)=e^{1/q}, \text{ so that } 1 < k({\bf s},w) < \infty.
\]
Since $\mathfrak G$ is the graph of a real-analytic function (see the proof of Theorem \ref{T:PF} in \cite[Appendix A]{Mc}), the above formula also shows that the function $k\colon \Csec \to (1,\infty)$ is real-analytic.  Compare with the proof of \cite[Corollary 5.4]{Mc}
\end{proof}

We now prove Theorem~\ref{T:continuity/convexity again} from the introduction:

\begin{theorem:continuity-convexity}[Convexity of stretch factors]
There exists a real-analytic, homogeneous of degree $-1$ function $\mathfrak H\colon \Csec \to \R$ such that:
\begin{enumerate}
\item $1/\mathfrak H$ is positive and concave, hence $\mathfrak H$ is convex.
\item For every primitive integral $u\in \Csec\subset H^1(X;\R)$ with dual compatible section $\Theta_u$, first return map $f_u$, and injective endomorphism $\phi_u = \overline{(f_u)}_*$ as in Theorem~\ref{T:splittings} we have 
\[\mathfrak H(u)=\log(\Lambda(u)) = \log(\lambda(\phi_u)) =\log(\lambda(f_u)) = h(f_u).\]
\item $\mathfrak H(u)$ tends to infinity as $u\to\partial \Csec$.
\end{enumerate}
\end{theorem:continuity-convexity}
\begin{proof}
In view of Proposition~\ref{P:PFconsequence} and Lemma~\ref{L:contained}, there is a real-analytic function $k\colon\Csec\to (1,\infty)$ such that for 
every $u=({\bf s},w)\in \Csec$ we have
\[ k({\bf s},w) = \sup \{ \zeta \in \R_+ \mid \poly_{({\bf s},w)}(\zeta) = 0 \}.\]

Define $\mathfrak H\colon \Csec\to R$ as $\mathfrak H(u):=\log k(u)$ for $u\in \Csec$. Then for any primitive integral $u=({\bf s},w)\in \Csec$, Theorem~\ref{T:main_theorem} implies that $\mathfrak H(u)>0$ is equal to the topological entropy $h(f_u)$ of $f_u$, and that $e^{\mathfrak H(u)}$ is equal to the spectral radius of the transition matrix of $f_u$ and also to the stretch factor $\Lambda(u)=\lambda(\phi_{u})$, proving (2). Since the function $k(u)$ satisfies $k(qu)=k(u)^{1/q}$ for any $u\in \Csec$ and $q>0$, it follows that $\mathfrak H(u)=\log k(u)$ is homogeneous of degree $-1$, that is, $\mathfrak H(qu)=\mathfrak H(u)/q$ for any $u\in \Csec$ and $q>0$. By Proposition~\ref{P:PFconsequence}(2), we know that $\mathfrak G$ is the the level set of $\mathfrak H(u)=1$. Therefore the concavity of $1/\mathfrak H$ follows from convexity of the hypersurface $\mathfrak G$, as proved by McMullen in his proof of \cite[Corollary 5.4]{Mc}, thus (1) follows.

All that remains is to prove (3).  For this, we suppose $\{ u_n \} \subset \Csec$ is a sequence in $\Csec$ converging to $u \in \partial \Csec$ for which $\mathfrak H(u_n)$ is bounded.  By homogeneity of $\mathfrak H$ there exists such a sequence with $u_n \in \mathfrak G$ for all $n$.  But since $\mathfrak G$ is a closed subset of $H^1(X;\R)$, it follows that $u \in \mathfrak G \cap \partial \Csec$. This is impossible since $\mathfrak G \subset \C_X$ by Theorem \ref{T:PF}, $\C_X = \Csec$ by Theorem \ref{T:cones_are_equal}, and $\Csec \cap \partial \Csec = \emptyset$ since $\Csec$ is an open cone.
\end{proof}

\subsection{The McMullen cone is the cone of sections}
\label{S:cones are cones}

We now return to the proof of Theorem \ref{T:cones_are_equal} and show that $\C_X = \Csec$. 

We fix attention on the original monodromy $f \colon \Gamma \to \Gamma$ defining $X$.  Let $u_0 \in H^1(X;\R)$ be the associated primitive integral class, $H = H_0 \oplus \Z$ with $H_0 = H_{u_0}$ the corresponding splitting, and ${\bf s},w$ and ${\bf t} = e^{\bf s},x = e^w$ the bases adapted to $u_0$ as in Definition \ref{D:adapted to z}.  Let $\bar E = \{\bar \sigma_1,\ldots,\bar \sigma_m\}$ be representative edges in $\tGamma \subset \hGamma$ of the $H$--orbit of edges.  Recall from \S \ref{S:graph module} the construction of the matrix $A({\bf t})= A_{u_0}({\bf t})$ with entries in $\Z[H_0] = \Z[{\bf t}^{\pm 1}]$ in terms of $E$ so that $A({\bf 1}) = A(1,\ldots,1)$ is the transition matrix of $f \colon \Gamma \to \Gamma$ by Proposition \ref{P:evaluatetransition}.

Next we let $\mathcal G$ denote the directed graph associated to $A({\bf 1})$ and $\mathcal Y$ the set of all closed oriented circuits in $\mathcal G$, as in \S\ref{S:cone_of_sections}.   Following the construction in \cite{AHR} (and also used by Hadari \cite{HadariComm}), we now associate to the matrix $A({\bf t})$ a labeling on the edges of $\mathcal G$ by elements of $H_0$ (this labeling differs slightly from \cite{AHR} due to our convention on the action of $H$ on $T(\mathcal F) \cong T(\hGamma)$). First note that $A_{ij}({\bf t}) \in \Z[{\bf t}^{\pm 1}]$ is the coefficient of $[\sigma_i]$ in the expression of $x \cdot [\sigma_j]$ as a $\Z[{\bf t}^{\pm 1}]$--linear combination of the elements of $E$.  
\begin{conv}
For the remainder of this subsection we assume that $A_{ij}({\bf t})$ is written as a positive sum of elements of $H$, that is, it is a sum of monomials in $t_1^{\pm 1},\ldots,t_{b-1}^{\pm 1}$ with coefficient $1$.  We refer to these simply as the {\em terms} or the {\em monomials} of $A_{ij}({\bf t})$.  We also consider the monomial $1$, so $A_{ij}({\bf t})$ may also have several terms which are just $1$.
\end{conv}
The terms in the polynomial $A_{ij}({\bf t})$ are in a 1-1 correspondence with the edges of $\mathcal G$ from the $j^{th}$ vertex to the $i^{th}$ vertex (again appealing to Proposition \ref{P:evaluatetransition}), and we use any such correspondence to label those edges by the terms of $A_{ij}({\bf t})$.  Denote the labeled graph as $\mathcal G({\bf t})$.

Given a circuit $y \in \mathcal Y$, let $|y|$ denote the length of the circuit;  that is, $|y|$ is the number of vertices in $y$.  We also observe that $y$ gives rise to a monomial $p_y({\bf t})$ which is a product of the monomials in the circuit $y \subset \mathcal G({\bf t})$.

Next let $\mathcal Y'$ denote the set of subsets of pairwise disjoint circuits in $\mathcal G$.  That is,
\[ \mathcal Y' = \{ {\bf y} = \{y_1,\ldots,y_k \} \subset \mathcal Y \mid y_i \cap y_j = \emptyset \mbox{ for all } i \neq j  \} \]
Every ${\bf y} = \{y_1,\ldots,y_k \} \in \mathcal Y'$ also has a length, $|{\bf y}| = |y_1| + \cdots + |y_k|$, and an associated monomial $p_{\bf y}({\bf t}) = p_{y_1}({\bf t})\cdots p_{y_k}({\bf t})$.   We also record the number of circuits in ${\bf y}$ as $\#({\bf y}) = k$.  For a set with one element $\{ y \} \in \mathcal Y'$, we denote this $y \in \mathcal Y'$.

Finally, for every $n \geq 1$ we let $\mathcal Y^{(n)}$ denote the subset of $\mathcal Y'$ consisting of ${\bf y}$ with $|{\bf y}| = n$.  Thus $\mathcal Y'$ is the disjoint union of $\mathcal Y^{(n)}$ over all $n$ from $1$ up to $m$, the number of vertices of $\mathcal G$.  

We now have the following very useful formula which is essentially the ``cycle polynomial'' of \cite{AHR}, but differs slightly, again due to our conventions on the actions of $H$ on $T(\hGamma)$. This formula  is also closely related to the ``clique polynomial'' recently studied by McMullen \cite{McClique} in the general setting of directed graphs.

\begin{theorem} \label{T:AHR determinant} \cite{AHR}
With the notation above we have
\begin{equation} \label{Eq:AHR formula}
\poly({\bf t},x) = \det(xI - A({\bf t})) = x^m + \left( \sum_{n = 1}^m \left(\sum_{{\bf y} \in \mathcal Y^{(n)}} (-1)^{\#({\bf y})} p_{\bf y}({\bf t}) \right) x^{m-n} \right)
\end{equation}
\end{theorem}
For completeness, and since our definition differs slightly from \cite{AHR}, we include the proof of this theorem. 
\begin{proof}
We use the definition of the determinant as a product
\begin{equation} \label{Eq:defn of det}
\prod_{\rho \in S_m} \left( {\rm sgn}(\rho) \prod_{i=1}^m (xI - A({\bf t}))_{\rho(i) i} \right)
\end{equation}
over all permutations $\rho$ of the set $\{1,\ldots,m\}$, where ${\rm sgn}(\rho) \in \{ \pm 1 \}$ is the sign of the permutation.  Expanding this expression out (without cancellation), we analyze each nonzero monomial $q_0({\bf t},x)$ of this expression and produce exactly one monomial in the expansion of \eqref{Eq:AHR formula} (also without cancellation).

The monomial $q_0({\bf t},x)$ is obtained from some permutation $\rho$ by taking the product for $i = 1,\ldots,m$ of some choice of a single term of the polynomial $(xI-A({\bf t}))_{\rho(i) i}$, for each $i$.  To better understand $q_0({\bf t},x)$, first write $\rho$ in its disjoint cycle notation, and to simplify the notation, we assume that this representation of $\rho$ takes the following form:
\[ (m \, m-1 \, \ldots \, k_k+1 ) (k_k \, \ldots \, k_{k-1}+1) \cdots (k_2 \, \ldots \, k_1+1)(k_1)(k_1-1) \cdots (2)(1).\]
Here we are composing {\em right to left instead of left to right} as is classical because this better matches our other conventions.  In particular $(3 \, 2 \, 1)$ sends $1$ to $2$, $2$ to $3$ and $3$ to $1$.

Next note that the only cycles in this decomposition that can contribute a power of $x$ to $q_0({\bf t},x)$ are the $1$--cycles since these correspond precisely to entries on the diagonal.  We further assume that by reindexing if necessary (and without loss of generality) these $1$--cycles are the initial $k_0$ appearing, for some $0 \leq k_0 \leq k_1$.  Thus $q_0({\bf t},x) = x^{k_0}q_1({\bf t})$.  To see what $q_1({\bf t})$ is, we write the product of monomials defining it as a product, over each of the remaining cycles in the disjoint cycle representation of $\rho$, of some monomials of the associated entries of the matrix.  So, for the remaining $1$--cycles $(k_1),(k_1-1),\ldots,(k_0+1)$ we obtain monomials
\[ m_{k_1}({\bf t}), \ldots , m_{k_0+1}({\bf t}) \]
where for each $i = k_0 + 1,\ldots,k_1$, the monomial $m_i({\bf t})$ is a term of $-A_{ii}({\bf t})$.

Every other cycle $(k_{i+1} \, \ldots \, k_i + 1)$ gives us a monomial which is a product of monomials
\[ {\bf m}_{k_i+1}({\bf t}) = m_{k_i+1 , k_{i+1}}({\bf t}) m_{k_{i+1},k_{i+1}-1}({\bf t}) \cdots m_{k_i+3 , k_i+2}({\bf t}) m_{k_i+2 , k_i+1}({\bf t}),\]
and each $m_{i,j}({\bf t})$ is a term of $-A_{ij}({\bf t})$.
Then
\[ q_1({\bf t}) = {\rm sgn}(\rho){\bf m}_{k_k+1}({\bf t}) \cdots {\bf m}_{k_1+1}({\bf t}) m_{k_1}({\bf t}) \cdots m_{k_0+1}({\bf t}).\]

Now we observe that, because $q_1({\bf t}) \neq 0$, each of our cycles (except the first $k_0$) determines a circuit in $\mathcal G({\bf t})$: specifically, the $1$--cycle $(j)$ determines a circuit $y$ of length $1$ around the loop edge of $\mathcal G({\bf t})$ labeled by $-m_j({\bf t})$.  That is to say, the monomial associated to this $y$ is precisely $p_y({\bf t}) = -m_j({\bf t})$.

On the other hand, the cycle $(k_{i+1} \, \ldots \, k_i + 1)$ determines a circuit $y$ of length $|y| = k_{i+1} - k_i$ (equal to the length of the cycle) labeled by the monomials (which as we traverse the circuit are read right-to-left in the following list)
\[ \left( -m_{k_i+1 , k_{i+1}}({\bf t}) \right) , \left( -m_{k_{i+1} , k_{i+1}-1}({\bf t}) \right) , \ldots , \left( -m_{k_i+3 , k_i+2}({\bf t}) \right) , \left( -m_{k_i+2 , k_i+1}({\bf t}) \right). \]
The monomial associated to this circuit is precisely
\[ p_y({\bf t}) = (-1)^{|y|} {\bf m}_{k_i+1}({\bf t}). \]
Finally note that because these are disjoint cycles, the vertices of the circuits are all distinct, hence the circuits are pairwise disjoint, and therefore their union determines an element ${\bf y} \in \mathcal Y'$.  Writing ${\bf y} = \{ y_1 , \ldots, y_r \}$ where $r = \#({\bf y})$ is the number of circuits in ${\bf y}$, the previous calculations prove that
\[ q_1({\bf t}) = {\rm sgn}(\rho) (-1)^{|{\bf y}|}p_{\bf y}({\bf t}) .\]
Since the sign of a permutation is the product of the signs of its disjoint cycles, and the sign of a $j$--cycle is $(-1)^{j+1}$, we see that
\[ {\rm sgn}(\rho) = (-1)^{|y_1| + |y_2| + \cdots + |y_r| + r} = (-1)^{|{\bf y}| + \#({\bf y})}.\]
Combining this with the previous equation, and the fact that $q_0({\bf t},x) = q_1({\bf t})x^{k_0}$ and $k_0 = m - |{\bf y}|$ we find that our original monomial has the form
\[ q_0({\bf t},x) = (-1)^{\#({\bf y})} p_{\bf y}({\bf t}) x^{m-|{\bf y}|}.\]
This is a term in the sum \eqref{Eq:AHR formula}.  Reversing this discussion, we can find a corresponding term in the expansion of the determinant, proving \eqref{Eq:AHR formula}.
\end{proof}

We need one more ingredient before proving Theorem~\ref{T:cones_are_equal}. Recall from \S\ref{S:cone_of_sections} that  every circuit $y \in \mathcal Y$ corresponds to a closed orbit $\mathcal O_y \subset X$ of $\flow$. This closed orbit defines an element of $H$ which we denote $[\mathcal O_y]$.
\begin{lemma} \label{L:orbit form}
For every $y \in \mathcal Y$ we have $[\mathcal O_y] = (p_y({\bf t}))^{-1}x^{|y|}$.
\end{lemma}
Here $p_y({\bf t})$ is a monomial with coefficient $1$ so that it and its inverse, can be viewed of as elements of $H$.
\begin{proof}
To simplify the proof of the lemma, we suppose that the closed circuit $y$ runs around the vertices $1,2,\ldots,k-1,k,1$ in order, where $k = |y|$.   Recall that $\bar E = \{ \bar \sigma_1,\ldots,\bar \sigma_k \}$ are the edges of $\Gamma$ with $\sigma_i$ an edge in the preimage of $\bar \sigma_i$.  The closed orbit $\mathcal O_y$ starts at a point $ \xi \in \bar \sigma_1$, then runs through $\bar \sigma_2,\bar \sigma_3,\ldots,\bar \sigma_k$ before returning to $\bar \sigma_1$ at the point $\flow_k(\xi) = \xi$.  To see what $[\mathcal O_y]$ is as an element of $H$, we lift the path $s \mapsto \flow_s(\xi)$ for $s \in [0,k]$ to a path in $\tX$, and find the covering transformation taking the initial point to the terminal point.  The lift is obtained by simply picking a point $\tilde \xi \in \tX$ in the preimage of $\xi$ and considering the lifted flow: $s \mapsto \widetilde \flow_s(\tilde \xi)$, for $s \in [0,k]$.   Since the covering group is abelian, it does not matter which point $\tilde \xi$ we pick, so we assume $\tilde \xi \in \sigma_1$.

Now consider the monomials on the edges of the circuit $y$, and denote these $p_1({\bf t}),\ldots,p_k({\bf t})$. By definition, $p_i({\bf t}) \in H_0$ has the property that the edge path $x^{-1}(\widetilde \flow_1(\sigma_i))$ contains the edge $p_i({\bf t}) \cdot \sigma_{i+1} = (p_i({\bf t}))^{-1}(\sigma_{i+1})$ for all $i$, with $i+1$ taken modulo $k$ (recall the convention for the action of $H$ on the edges of $\hGamma$ in \S\ref{S:graph module}).  Therefore, we see that $\widetilde \flow_1(\sigma_i)$ contains the edge $(x^{-1} p_i({\bf t})) \cdot \sigma_{i+1}$ for all $i$.  
From this, we have
\[  \begin{array}{lllllll} \widetilde \flow_k(\sigma_1) & \supset  \quad x^{-1}p_1({\bf t}) \cdot \widetilde \flow_{k-1}(\sigma_2) \quad \supset \quad  x^{-2}p_1({\bf t})p_2({\bf t}) \cdot \widetilde \flow_{k-2}(\sigma_3) \quad \supset \quad \cdots \\\\
 & \cdots \quad  \supset  \quad x^{-k+1}p_1({\bf t})p_2({\bf t}) \cdots p_{k-1}({\bf t}) \cdot \widetilde \flow_1(\sigma_k) \quad  \supset  \quad x^{-k}p_1({\bf t})p_2({\bf t}) \cdots p_k({\bf t}) \cdot \sigma_1.
\end{array} \]
It follows that $\widetilde \flow_k(\tilde \xi)$ lies in this edge, and hence the covering transformation $h$ taking $\tilde \xi$ to $\widetilde \flow_k(\tilde \xi)$ is precisely
\[ h = x^k (p_y({\bf t}))^{-1} = x^{|y|}(p_y({\bf t}))^{-1} \]
as required.
\end{proof}

\begin{proof}[Proof of Theorem \ref{T:cones_are_equal}]
We have already seen in Lemma \ref{L:contained} that $\Csec \subseteq \C_X$.  We let $u \in \partial \Csec$ and prove that $u \in \partial \C_X$.  By convexity of $\C_X$ this will prove the theorem.  We note that the containment of cones implies $u \in \overline{ \C_X}$, so it suffices to prove that $u \not \in \C_X$.  Let $y \in \mathcal Y$ be the shortest length circuit (that is, with minimal $|y|$) for which $u([\mathcal O_y])= 0$.   Since $\Csec = \D$ by Theorem \ref{T:cone_of_sections}, we are guaranteed that some such $\mathcal O_y$ exists.

\begin{claim} \label{Cl:orbit term}
In $\poly(x,{\bf t}) = \det(xI - A({\bf t}))$, the term $p_y({\bf t})x^{m-|y|}$ has a nonzero coefficient $c_0$ (after expanding and cancellation).
\end{claim}
\begin{proof}
We note that there is a term of this form after expanding \eqref{Eq:AHR formula} and {\em before canceling}, so we just need to make sure it survives the cancellation.  
Note that after expanding, the coefficient of $p_y({\bf t})x^{m-|y|}$ is just $(-1)^{\#(y)} = -1$.  Thus, for there to be cancellation, there must be some ${\bf y} = \{y_1,\ldots,y_k\} \in \mathcal Y'$ so that $p_{\bf y}({\bf t})x^{m-|{\bf y}|} = p_y({\bf t})x^{m-|y|}$, but for which the coefficient after expanding is $+1$ (all coefficients are equal to $\pm 1$, and if they were all $-1$, there could be no cancellation).  But since the coefficient is $(-1)^{\#({\bf y})} = 1$, we must have $\#({\bf y}) = k > 1$.  

Since $p_{\bf y}({\bf t}) =p_{y_1}({\bf t}) \cdots p_{y_k}({\bf t})$, applying Lemma \ref{L:orbit form}, we have
\[ [\mathcal O_{y_1}]\cdots [\mathcal O_{y_k}] = (p_{y_1}({\bf t}))^{-1}x^{|y_1|} \cdots (p_{y_k}({\bf t}))^{-1}x^{|y_k|} = (p_{\bf y}({\bf t}))^{-1}x^{|{\bf y}|} = (p_y({\bf t}))^{-1}x^{|y|}.\]
On the other hand, $u \in \partial \Csec = \partial \D$ and so $u([\mathcal O_{y_i}]) \geq 0$ for all $i = 1,\ldots,k$.  The previous equation, together with the fact that $u$ is a homomorphism implies
\[ 0 = u((p_y({\bf t}))^{-1}x^{|y|}) = u((p_{\bf y}({\bf t}))^{-1}x^{|{\bf y}|}) = u([\mathcal O_{y_1}]) + \cdots + u([\mathcal O_{y_k}]) \]
Since the terms on the right are all greater than or equal to $0$, they must all be equal to zero.  This contradicts the minimality assumption on the length of $y$ (since each $y_i$ has length less than that of $y$), and this completes the proof of the claim.
\end{proof}

Writing $p_y({\bf t}) = t_1^{j_1} \cdots t_{b-1}^{j_{b-1}}x^{|y|}$, Lemma \ref{L:orbit form} implies that in the additive basis for $H$ we have
\[ [\mathcal O_y] = -j_1 s_1 - j_2 s_2 - \cdots  - j_{b-1} s_{b-1} + |y| w \]
Thus the condition $u([\mathcal O_y]) = 0$ implies
\begin{equation}
\label{Eq:key cone relation}
|y|w(u) = j_1 s_1(u) + \cdots + j_{b-1} s_{b-1}(u).
\end{equation}

We set ${\bf j} = (j_1,\ldots,j_{b-1},m-|y|)$ so that $c_0 ({\bf t}x)^{\bf j}$ is a nonzero term of $\poly({\bf t},x)$ according to Claim \ref{Cl:orbit term}.
The ``leading term'' $x^m$ is the one that defines the McMullen cone
\[ \C_X = \C(\poly,x^m)\]
as described in \S\ref{S:mcmullen cone}.
Writing ${\bf j}' = (0,0,\ldots,0,m)$, we have $x^m = ({\bf t}x)^{{\bf j}'}$.  However, for the class $u$, Equation \eqref{Eq:key cone relation} implies that if we write $u$ in the coordinates, $({\bf s}(u),w(u))$, we have
\begin{eqnarray*}
{\bf j} \cdot ({\bf s}(u),w(u)) & = & (j_1,\ldots,j_{b-1},m-|y|)\cdot ({\bf s}(u),w(u))\\
& = & j_1s_1(u) + \cdots + j_{b-1}s_{b-1}(u) + (m-|y|)w(u)\\
& = & |y|w(u) + mw(u) -|y|w(u) = m w(u) \\
& = &{\bf j}' \cdot ({\bf s}(u),w(u)).
\end{eqnarray*}
But then $u$ cannot be in $\C(\poly,x^m)$ since this requires ${\bf j}' \cdot ({\bf s}(u),w(u)) > {\bf j}({\bf s}(u),w(u))$.  This completes the proof of  Theorem~\ref{T:cones_are_equal}.
\end{proof}

\section{Closed $1$--forms and foliations}
\label{S:closed 1-forms and foliations}

Given $u \in \Csec$, let $\omega = \omega^u$ be a tame, flow-regular closed $1$--form representing $u$.  In this section, we describe the foliation $\Omega_u$ associated to $\omega^u$.  This is analogous to the foliation tangent to the kernel of a closed $1$--form in the classical smooth manifold setting.  This discussion quickly leads to a proof of Theorem~\ref{T:BNS}.  After this, we specialize to the case $u \in \A$, where $\omega^u$ can be chosen so that it and $\Omega_u$ satisfy stronger geometric and topological conditions analogous to those satisfied in the $3$--manifold setting.

For each $\xi \in X$, setting $\gamma_\xi(s) = \flow_s(\xi)$ we let
\[ r_\xi(s) = \int_{\gamma_\xi([0,s])} \omega^u.\]
This can be used to define a reparameterization of $\flow$ by $\flow^u_s(\xi) = \flow_{r_\xi^{-1}(s)}(\xi)$ so that arcs of flowlines map isometrically to $\R$ by functions in $\omega^u$.  We lift this reparameterization $\widetilde \flow^u_s$ to $\tX$.

Let $(\widetilde \fib_u \colon \tX \to \R) \in p^*(\omega^u)$ be a function defined on all of $\tX$.   Then as in \S\ref{S:closed 1-forms cohomology} we have
\[ \widetilde \fib_u \circ h (\xi) = \widetilde \fib_u(\xi) + u(h)\]
for all $\xi \in \tX$ and $h \in H$.
Let $\tOmega_{y,u}$ denote the fiber $\tOmega_{y,u} = \widetilde \fib_u^{-1}(y)$, for $y \in \R$.  This $\tOmega_{y,u}$ is naturally a graph with vertex set $\tOmega_{y,u}^{(0)}$ equal to the intersection with the $1$--skeleton $X^{(1)}$.  As such, the valence of $\tOmega_{y,u}$ is uniformly bounded.  Each of these fibers $\tOmega_{y,u}$ covers a graph $\Omega_{y,u}$ that includes into $X$.

We define the {\em foliation $\Omega_u$ defined by $\omega^u$} to be the decomposition of $X$ into the sets $\Omega_{y,u}$, which we call the  {\em leaves}.  We note that two leaves $\Omega_{y,u} = \Omega_{y',u}$ if and only if $y$ and $y'$ differ by an element of $u(H)$.  As such, the leaves of $\Omega_u$ are in a one-to-one correspondence with elements of $\R / u(H)$.  Thus, when we write $\Omega_{y,u}$, it should be understood that $y$ is representing the coset $y + u(H)$ in $\R/u(H)$.  

\begin{lemma}
The reparameterized semiflow $\flow^u$ maps leaves to leaves, $\flow^u_s(\Omega_{y,u}) = \Omega_{y+s,u}$.
\end{lemma}
\begin{proof}
We verify that the corresponding statement holds in $\tX$, and then the lemma follows.  By construction of the reparameterization, we know that $s \mapsto \widetilde \fib_u(\widetilde \flow^u_s(\xi))$ is an isometry.  From this fact we deduce that if $\xi \in \tOmega_{y,u}$ then
\[ \widetilde \fib_u (\widetilde \flow^u_s(\xi)) = \widetilde \fib_u(\xi) + s = y + s.\]
So, $\widetilde \flow^u_s(\xi) \in \tOmega_{y+s,u}$, as required.
\end{proof}

We note that in the special case that $u$ is primitive integral, then the leaves are compact (in fact they are sections).  Indeed, in this case $\omega^u$ defines a map $\fib_u \colon X \to \R/u(H) = \R/\Z = \sone$, the descent of $\widetilde \fib_u \colon \tX \to \R$, and $\Omega_{y,u} = \Theta_{y,u}$.

Recall that $\BNS(G) \subset (H^1(G;\R) \setminus \{ 0 \})/\R_+$ denotes the BNS-invariant of $G$, which consists precisely of those $u \in H^1(G;\R)$ for which $H$ is finitely generated over a finitely generated submonoid of $u^{-1}([0,\infty)) < G$ (see the remarks following the Meta-Theorem).  There is an alternate, more geometric description of $\BNS(G)$ due to Bieri-Neumann-Strebel \cite[Theorem G]{BNS} which allows us to deduce the following.

\begin{proposition} 
\label{P:S is in BNS}
The cone $\Csec \subset H^1(X;\R) = H^1(G;\R)$ projects into $\BNS(G)$.
\end{proposition}
\begin{proof}
The construction of \cite{BNS} associates to any $u \in H^1(G;\R) \setminus \{ 0 \}$ a $u$--equivariant continuous map $\chi' \colon \tX \to \R$.  The authors prove that $u \in \BNS(G)$ if and only if $\pi_1(\chi'^{-1}([0,\infty)))$ surjects onto $\pi_1(\tX)$ (or if $\chi'^{-1}([0,\infty))$ is disconnected, one considers the fundamental group of the unbounded component).  They further note that modifying $\chi'$ to $\chi''$ by an equivariant, bounded homotopy, does not change the property of $\pi_1$--surjectivity.  Therefore, any continuous $u$--equivariant map $\tX \to \R$ will suffice in applying this characterization. 

Now suppose $u \in \Csec$ and let $(\widetilde \fib_u \colon \tX \to \R) \in \widetilde \omega^u$, where $\omega^u$ is a closed flow-regular $1$--form representing $u$.   According to \cite[Theorem G]{BNS}, we need only prove that $\pi_1(\widetilde \fib_u^{-1}([0,\infty))$ surjects onto $\pi_1(\tX)$.  For this we take any loop $\gamma \in \tX$, and let $-s \in \R$ be the minimum value of $\widetilde \fib_u$ on $\gamma$.  Then $\flow^u_s(\gamma) \subset \widetilde \fib_u^{-1}([0,\infty))$, and $\flow^u_s(\gamma)$ is clearly homotopic to $\gamma$.
\end{proof}

In light of Proposition~\ref{prop:stretch_is_locally_bounded} and Theorem~\ref{T:continuity/convexity again}, it follows that $\C_X$ in fact detects a full component of $\BNS(G)$:

\begin{theorem:cone-is-BNScpt}[McMullen polynomial detects component of $\BNS(G)$]
The McMullen cone $\C_X$ projects onto a full component of the BNS-invariant $\BNS(G)$. That is, $\{[u] \mid u\in \C_X\}$ is a connected component of $\BNS(G)$.
\end{theorem:cone-is-BNScpt}
\begin{proof}
Let $\widehat{\BNS}(G) = \{u\in H^1(G;\R) \mid [u]\in \BNS(G)\}$, and let $\widehat{\BNS}_0(G)$ denote the component containing $u_0$. Proposition~\ref{P:S is in BNS} shows that $\C_X\subset \widehat{\BNS}_0(G)$. Supposing the reverse inclusion fails, then $\widehat{\BNS}_0(G)$ necessarily contains a point of $\partial\C_X$. Since $\widehat{\BNS}_0(G)$ is open and the faces of $\C_X$ are rationally defined, this implies that $\widehat{\BNS}_0(G)$ contains a rational point $u\in \partial \C_X\cap \QBNS(G)$. By Theorem~\ref{T:continuity/convexity again} we know that $\log(\Lambda)$ agrees with $\mathfrak{H}$ and is consequently unbounded on $\C_X\cap \QBNS(G)\cap U$ for any open neighborhood $U$ of $u$. As this contradicts Proposition~\ref{prop:stretch_is_locally_bounded}, we conclude $\widehat{\BNS}_0(G)\subset \C_X$ as claimed.
\end{proof}

\subsection{In the cone $\A$}
\label{S:foliations in A}

For $u \in \A$, the closed $1$--form $\omega^u$ can be chosen to enjoy better properties than for an arbitrary $u \in \Csec$ (as is true for the flow-regular maps to $\sone$ in the integral case).
Before we describe the properties we will want for these closed $1$--forms, we recall the following from \cite[Definition 5.2]{DKL}.
\begin{defn}[Local model]
\label{D:local model}
A {\em local model} for $\flow$ is a subset of $X$ of the form
\[ \M(K,s_0) = \bigcup_{0 \leq s \leq s_0} \flow_s(K)\]
where $s_0 > 0$ and $K$ is a closed contractible neighborhood of a point of $\Gamma_t$, for some $t$.
We call $K$ the {\em bottom} of $\M$ and $\flow_{s_0}(K)$ the {\em top} and we call $\M$ minus the top and bottom the {\em flow-interior} of $\M$.  For any $0 \leq s_1 \leq s_2 \leq s_0$,
\[ \bigcup_{s_1 \leq s \leq s_2} \flow_s(K) \]
is also a local model, and we call it a {\em local submodel} of $\M(K,s_0)$.
\end{defn}
The local models were used in \cite{DKL} to ensure that for integral classes $u \in \A$, the reparameterized flows $\flow^u_s$ mapped each fiber $\fib_u^{-1}(y) = \Theta_{y,u}$ to the fiber $\Theta_{y+s,u}$, for $y,y+s \in \sone$, by a homotopy equivalence;  see \cite[Corollary 6.20]{DKL}.  The key point was that we could cover $X^{(1)}$ by finitely many flow-interiors of local models so that each component of the fiber $\Theta_{y,u}$ intersected with any of the local models $\M$ in the finite set is contained in a fiber of the original fibration $\fib \colon X \to \sone$.

\begin{defn}[non-singular closed $1$--form]
We will say that a closed $1$--form $\omega$ on $X$ is {\em nonsingular} if it is flow-regular and the following holds
\begin{enumerate}
\item For any point $\xi$ in the $1$--skeleton, there is a simply-connected neighborhood $U$ of $\xi$ containing a local model $\M(K,s_0)$ such that the fibers of $\omega_U$ intersected with $\M(K,s_0)$ have the form $\psi_s(K)$ for some $0 \leq s \leq s_0$ (and are hence contained in fibers of the original fibration $\fib \colon X \to \sone$).
\item The integral over any oriented $1$--cell of $X$ is positive.
\end{enumerate}
\end{defn}
In particular, note that if $\omega$ is a nonsingular closed $1$--form, then by the second condition the $1$--cocycle $z$ determined by integrating $\omega$ over $1$--cells is a positive $1$--cocycle, and hence $u = [z] \in \A$.  We write $\omega = \omega^z = \omega^u$ in this case.  

If $u  \in \A$ is primitive integral and $\omega^u$ is constructed as in \S\ref{S:flow-regular closed 1-forms and sections} from the flow-regular map $\fib_u$ from Theorem \ref{T:DKL_BC}, then \cite[Lemma 6.16]{DKL} guarantees that $\omega^u$ is non-singular.    This fact can be used to show that every element of $\A$ is represented by a nonsingular closed $1$--form.

\begin{proposition} \label{P:non-singular 1-forms exist}
For every class $u \in \A$, there is a nonsingular closed $1$--form $\omega^u$ representing $u$.
\end{proposition}
\begin{proof}
Suppose $u_1,\ldots,u_n \in \A$ are integral classes such that $u$ can be written as a positive linear combination of these elements
\[ u = \sum_{i=1}^n \zeta_i u_i.\]
Let $\omega^{u_i}$ be a nonsingular closed $1$--form as above.   Then $u$ is represented by the closed $1$--form
\[ \omega^u = \sum_{i=1}^n \zeta_i \omega^{u_i} .\]
Since positive combinations of flow-regular maps are flow-regular, it follows that $\omega^u$ is flow-regular.

If $z_i$ is the positive cocycle representing $u_i$ obtained from $\omega_i$, then the cocycle $z$ representing $u$ obtained from $\omega^u$ is precisely $\sum \zeta_i z_i$.  Therefore, the the second conditions of non-singularity is satisfied.

By choosing the covering of $X^{(1)}$ in \cite[Lemma 6.16]{DKL} by sufficiently small local models (see also \cite[Proposition 5.6]{DKL}), we can assume that the same covering works for all of $u_1,\ldots,u_n$, and that each are contained in simply connected open sets in $X$.  The first property of non-singularity then also follows.
\end{proof}

Now for any $u \in \A$, we assume that $u$ is represented by a nonsingular closed $1$--form and $\Omega_u$ is the foliation defined by $\omega^u$.  We also write $\widetilde \fib_u \colon \tX \to \R$ with fibers $\tOmega_{y,u} \subset \tX$ covering the leaves $\Omega_{y,u}$.

\begin{theorem} \label{T:like kernels of 1-forms}
Suppose $\omega^u$ is a nonsingular closed $1$--form on $X$ defining $u \in \A$.  Then for every $y \in \R$, the inclusion of the fibers $\tOmega_{y,u} = \widetilde \fib_u^{-1}(y) \to \tX$ is a homotopy equivalence (in particular, the fibers are connected).  Furthermore, for every $s \geq 0$, the restriction of the reparameterized flow
\[ \tflow^u_s \colon \tOmega_{y,u} \to \tOmega_{y+s,u} \]
is a homotopy equivalence.
\end{theorem}

Before we give the proof, we observe the following corollary.  Here $\BNS_s(G) = \BNS(G) \cap -\BNS(G)$ is the {\em symmetrized} BNS-invariant.

\begin{corollary} \label{C:A is in BNS_0}
The cone $\A$ projects into $\BNS_s(G)$.
\end{corollary}
\begin{proof}
For $u \in \A$, we have the inclusion of $\tOmega_{0,u} = \widetilde \fib_u^{-1}(\{0 \}) \to \tX$ is a homotopy equivalence by Theorem \ref{T:like kernels of 1-forms}.  Therefore, $\widetilde \fib_u^{-1}([0,\infty))$ and $\widetilde \fib_u^{-1}((-\infty,0])$ are $\pi_1$--surjective.  As in the proof of Proposition \ref{P:S is in BNS}, we can appeal to Theorem G of \cite{BNS} to conclude that both $u$ and $-u$ are in $\BNS(G)$ and hence $u \in \BNS_s(G)$.
\end{proof}

\begin{remark}
For a fibered hyperbolic $3$--manifold, every cohomology class in the cone on the fibered face of a hyperbolic $3$--manifold is represented by a nowhere vanishing closed $1$--form that evaluates positively on the vector field defining the suspension flow (of any fiber in the cone on the face).  In particular, the kernel of the $1$--form is tangent to a taut foliation, and so by the Novikov-Rosenberg Theorem the leaves $\pi_1$--inject into the $3$--manifold; see e.g.~\cite{Calegari07}.  Furthermore, reparameterizing the flow (so that the vector field evaluates to $1$ on the $1$--form) the leaves map homeomorphically from one onto another.
\end{remark}

\begin{proof}
We first note that when $u$ is a scalar multiple of a primitive integral $u' \in \A$, then $\tOmega_{y,u}$ is a component of the preimage in $\tX$ of the graph $\Theta_{y',u'}$ (though the vertex set for $\tOmega_{y,u}$ might be a proper subset of those coming from the standard graph structure of $\Theta_{y',u'}$ if $\Theta_{y',u'}$ is $\mathcal F$--compatible).  In this case, the conclusion follows from \cite[Corollaries 6.20 and 6.21]{DKL}.
We therefore assume that $u$ is not a scalar multiple of an integral element.  We will prove the last statement, and then deduce the rest from this.  

We assume that all local models considered in what follows satisfy the condition (1) of nonsingularity for $\omega^u$.  We choose a finite cover $\MM$ of $X^{(1)}$ by flow-interiors of such local models satisfying the following conditions for some $\ell > 0$.
\begin{enumerate}
\item[(i)] Any $\M \in \MM$ intersects $X^{(1)}$ in a contractible subset containing at most $1$ vertex;
\item[(ii)] Given $\M(K,s_0),\M(K',s_0') \in \MM$, if $\flow_s(K) \cap \flow_{s'}(K') \neq \emptyset$ for some $s \in [0,s_0],s' \in [0,s_0']$, then $\flow_s(K) \cap \flow_{s'}(K) \cap X^{(1)} \neq \emptyset$ is a single point;
\item[(iii)] Any arc of a flowline $\gamma_\xi(s) = \flow_s(\xi)$ that nontrivially intersects $X^{(1)}$, and for which the integral of $\omega^u$ is at most $\ell$, is contained in a local model.
\end{enumerate}
Appealing to \cite[Proposition 5.6]{DKL} we can find a finite cover of $X^{(1)}$ by flow-interiors of local models so that the first and second conditions are satisfied.  This is done by first choosing sufficiently small, pairwise disjoint flow-interiors of local models, $\M_1,\ldots,\M_k$, one containing each $0$--cell of $X$, and intersecting $X^{(1)}$ in a contractible set.  Then we appropriately choose very tiny flow-interiors of local models about each point $\xi \in X^{(1)} \setminus (\M_1 \cup \cdots \cup \M_k)$ which are also all disjoint from the $0$--cells.  These are chosen so that they intersect $X^{(1)}$ in an arc (necessarily contained in a single $1$--cell), and are so small that any two which intersect different $1$--cells of $X^{(1)}$ are disjoint.  By compactness, we can take a finite subcover $\MM$, which we do, noting that the flow-interiors covering the $0$--cells chosen first are necessarily in this set.
The existence of $\ell > 0$ so that the third condition is satisfied is then also a consequence of compactness of $X^{(1)}$.

Now we lift each local model in $\MM$ to a local model in $\tX$ via the covering map $p \colon \tX \to X$ (which is possible since each $\M \in \MM$ is contained in a simply-connected open subset of $X$), and denote the collection as $\widetilde \MM$.  Observe that the conditions (i)--(iii) are also satisfied by $\widetilde \MM$ with the same $\ell > 0$ (though this is no longer a finite cover).

Fix any $y \in \R$, look at the fiber $\tOmega_{y,u} = \widetilde \fib_u^{-1}(y)$, and consider the set
\[ \tOmega_{[y,y+\ell],u} = \bigcup_{0 \leq s \leq \ell} \psi_s^u(\tOmega_{y,u}) = \widetilde \fib_u^{-1}([y,y+\ell]).\]
Suppose that the arc $\gamma_\xi([0,\ell])$ of the the flowline through $\xi \in \tOmega_{y,u}$ nontrivially intersects $\tX^{(1)}$.  Then this arc is contained in a local model $\widetilde \M \in \widetilde \MM$ by condition (iii).   The intersection $\tOmega_{[y,y+\ell],u} \cap \widetilde \M$ is a local submodel $\widetilde \M' \subset \widetilde \M$ (most likely not a local model in $\widetilde \MM$).

Note that $\widetilde \M' \cap \tX^{(1)}$ is a component of $\tOmega_{[y,y+\ell],u} \cap \tX^{(1)}$ by (i).    We enumerate the components of $\tOmega_{[y,y+\ell],u} \cap \tX^{(1)}$ and choose $\widetilde \M_i \in \widetilde \MM$ containing the $i^{th}$ component of $\tOmega_{[y,y+\ell],u} \cap \tX^{(1)}$ and let
\[ \widetilde \M'_i = \widetilde \M_i \cap \tOmega_{[y,y+\ell],u}.\]
By condition (ii), $\widetilde \M'_i \cap \widetilde \M'_j \neq \emptyset$ if and only if $\widetilde \M'_i \cap \widetilde \M'_j \cap \tX^{(1)} \neq \emptyset$.  Since $\widetilde \M'_i \cap \tX^{(1)}$ is the $i^{th}$ component, it follows that $\widetilde \M'_i \cap \tX^{(1)} = \widetilde \M'_j \cap \tX^{(1)}$ and hence $i = j$.

For any $0 \leq s \leq \ell$, we can subdivide $\tOmega_{y,u}$ and $\tOmega_{y+s,u}$ so that the intersections of these with any $\widetilde \M'_i$ are subgraphs $\tOmega_{y,u,i} \subset \tOmega_{y,u}$ and $\tOmega_{y+s,u,i} \subset \tOmega_{y+s,u}$, respectively.  The restriction of the reparameterized flow $\widetilde \flow_s^u \colon \tOmega_{y,u,i} \to \tOmega_{y+s,u,i}$ is controlled rel boundary in the sense of \cite[Definition 5.3]{DKL}.  In any complementary component, $\widetilde \flow_s^u$ restricts to a homeomorphism (since this is contained in the interior of a $2$--cell where the semiflow restricts to an honest local flow).  We can now apply \cite[Proposition 5.5]{DKL} to guarantee that $\widetilde \flow_s^u$ restricts to a homotopy equivalence from $\tOmega_{y,u}$ to $\tOmega_{y+s,u}$.
\begin{remark}
Proposition 5.5 of \cite{DKL} was only stated for finite graphs because that was the only case of interest.  The proof is valid in the current setting as well.
\end{remark}

Thus we have proved the last part of the theorem for any $y \in \R$ and any $0 \leq s \leq \ell$.  Since $\widetilde \flow_s^u$ is a semiflow, it also holds for any $s \geq 0$.

Next we claim that every fiber is connected. To see this, observe that since the semiflow restricts to a homotopy equivalence of fibers, it maps distinct components to distinct components.  Thus flowing forward and backward (i.e. taking the preimage under the semiflow) for any two components of a fiber defines disjoint open sets.  Thus we have a bijection between the components of a fiber and the components of $\tX$, but $\tX$ is connected, and hence so is every fiber.

Finally, to see that the inclusion $\tOmega_{y,u} \to \tX$ is a homotopy equivalence, we note that both spaces are Eilenberg-Maclane spaces for (possibly infinitely generated) free groups, so we need only prove that the map is an isomorphism on fundamental groups.  Any loop in $\tX$ based at a point $\xi_0$ of $\tOmega_{0,u}$ is contained in the region between two fibers $\tOmega_{y,u}$ and $\tOmega_{y',u}$ for some $y < 0 < y'$.  This loop flows forward onto a loop in $\tOmega_{y',u}$, but since the flow restricts to homotopy equivalences between fibers, it follows that the loop is in the image of $\pi_1(\Omega_{0,u}, \xi_0)$.  Similarly, for any loop in ${0,u}$ based at $\xi_0$, suppose there is a null-homotopy of this loop to a point.  This null-homotopy can be flowed forward to some fiber $\tOmega_{y,u}$ and since the flow restricts to a homotopy equivalence $\tOmega_{0,u} \to \tOmega_{y,u}$, the loop must have already been null-homotopic in $\tOmega_{0,u}$.
This completes the proof.
\end{proof}

Theorem \ref{T:like kernels of 1-forms} quickly implies our theorem from the introduction.

\begin{theorem:transversefoliations} [$\pi_1$--injective foliations]
Given $u \in \A$, there exists a flow-regular closed $1$--form
$\omega^u$ representing $u$ with associated foliation $\Omega_u$ of
$X$ having the following property.  There is a reparameterization of $\flow$, denoted $\flow^u$, so that for each $y \in \R$ the inclusion of the fiber $\Omega_{y,u} \to X$ is $\pi_1$--injective and induces an isomorphism $\pi_1(\Omega_{y,u}) \cong \ker(u)$.
Furthermore, for every $s \geq 0$ the restriction
\[ \flow^u_s \colon \Omega_{y,u} \to \Omega_{y+s,u} \]
of $\flow_s^u$ to any leaf $\Omega_{y,u}$ is a homotopy equivalence.
\end{theorem:transversefoliations}
\begin{proof} 
Let $\omega^u$ be the nonsingular closed $1$--form from Theorem \ref{T:like kernels of 1-forms}.

We observe that the image of $\ker(u)$ in $H$, which we denote $\ker_H(u)$, preserves the fiber $\tOmega_{y,u} = \widetilde \fib_u^{-1}(y)$ in $\tX$ via the action of $H$ on $\tX$.  Indeed, $\ker_H(u)$ is precisely the stabilizer in $H$ of $\tOmega_{y,u}$.  Since the inclusion of $\tOmega_{y,u} \to \tX$ is a homotopy equivalence by Theorem \ref{T:like kernels of 1-forms}, so is $\tOmega_{y,u}/\ker_H(u) \to \tX/\ker_H(u)$.  The further quotient $\tX/\ker_H(u) \to X$ maps $\tOmega_{y,u}/\ker_H(u)$ homeomorphically onto $\Omega_{y,u}$ since it is the quotient by the action of $H/\ker_H(u) = H/stab_H(\tOmega_{y,u})$.  Since $\tX/\ker_H(u) \to X$ is the cover corresponding to $\ker(u)$, we see that the inclusion $\Omega_{y,u}$ induces an isomorphism $\pi_1(\Omega_{y,u}) \cong \pi_1(\tX/\ker_H(u)) = \ker(u) < G$, as required.

By Theorem \ref{T:like kernels of 1-forms} $\widetilde \flow^u_s \colon \tOmega_{y,u} \to \tOmega_{y+s,u}$ is a homotopy equivalence for all $y,s \in R$, $s \geq 0$.   Since $\tX \to X$ is a covering map (hence restricting to a covering map on all fibers of $\widetilde \eta$) and since this covering map is semiflow equivariant, it follows that $\flow^u_s \colon \Omega_{y,u} \to \Omega_{y+s,u}$ is also a homotopy equivalence.
\end{proof}

\section{Lipschitz flows}
\label{S:lipschitz flows}

Associated to every primitive integral class $u \in \Csec$, there is a rich geometric structure on $X$ described by the following.

\begin{proposition} \label{P:motivating metric}
For every primitive integral $u \in \Csec$, there exists a geodesic metric $d_u$ on $X$ such that for the reparameterized semiflow $\psi^u$ we have:
\begin{enumerate}
\item The semiflow-line $s \mapsto \psi_s^u(\xi)$ is a local geodesic for all $\xi \in X$;
\item The metric $d_u$ restricts to a path metric on each fiber $\Theta_{y,u}$ and the maps
\[ \{\psi_s^u \colon \Theta_{y,u} \to \Theta_{y+s,u}\}_{s \geq 0}\]
are $\lambda^s$--homotheties on the interior of every edge, where $\lambda = \lambda(f_u)$, is the stretch factor of the first return map $f_u$.
\end{enumerate}
\end{proposition}

The metric can be built explicitly by taking the canonical eigenmetric on $\Theta_u$, defining a metric on the fibers so that (2) holds, and then extending to a metric on $X$ so that the semiflow-lines from (1) have unit speed and are ``perpendicular'' to the fibers (in the interior of each $2$--cell, this makes sense, and then the $2$--cells are glued together by isometries).

The main result of this section is that one can carry out an analogous construction for {\em any} cohomology class $u \in \Csec$.  This mirrors the {\em Teichm\"uller flow} from Theorem 1.1 of \cite{Mc} for fibered hyperbolic $3$--manifolds.  There the flow mapped leaves of the foliation tangent to the kernel of a $1$--form to leaves and was a Teichm\"uller mapping.  In our construction, all maps are Lipschitz with constant stretch factor, and so we call these {\em Lipschitz flows}.

\begin{theorem:Lipschitzflows} [Lipschitz flows]
For every $u \in \Csec$, let $\mathfrak H(u)$ be as in Theorem \ref{T:continuity/convexity again}, $\omega^u$ a tame flow-regular closed $1$--form representing $u$, $\flow^u$ the associated reparameterization of $\flow$ and $\Omega_u$ the foliation defined by $\omega^u$.  Then there is a geodesic metric $d_u$ on $X$ such that:
\begin{enumerate}
\item The semiflow-lines $s \mapsto \psi_s^u(\xi)$ are local geodesics for all $\xi \in X$,
\item The metric $d_u$ determines a path metric on each (component of a) leaf of $\Omega_{y,u}$ of the foliation $\Omega_u$ defined by $\omega^u$ making it into a (not necessarily finite) simplicial metric graph,
\item The restriction of the reparameterized semiflow
\[ \{\psi_s^u \colon \Omega_{y,u} \to \Omega_{y+s,u}\}_{s \geq 0}\]
are $\lambda^s$--homotheties on the interior of every edge, where $\lambda = e^{\mathfrak H(u)}$.
\item The induced path metric from $d_u$ on the interior of any $2$--cell $U$ is locally isometric to a constant negative curvature Riemannian metric.
\end{enumerate}
\end{theorem:Lipschitzflows}

In the special case that $u \in \Csec$ is primitive integral, the $1$--form $\omega^u$ is given by restricting the fibration $\eta_u \colon X \to \sone$ to simply connected open sets and lifting the resulting maps to the universal covering $\R \to \sone$.  In particular, the leaves $\Omega_{y,u}$ are precisely the fibers $\Theta_{y,u}$ and the metric $d_u$ satisfies the conclusion of Proposition \ref{P:motivating metric}.  Thus, Proposition \ref{P:motivating metric} will follow from Theorem \ref{T:lipschtiz flows}.  We recall that the label $y$ for the leaves $\Omega_{y,u}$ makes sense as an element of $\R/u(H)$.

\subsection{Twisted transverse measures on $\mathcal F$.}
\label{S:twisted transverse measures}

For the remainder of this section we fix a primitive integral class $u_0 \in \A$, which we may as well assume is the class associated to our original fibration $\fib \colon X \to \sone$ with $\Theta_{u_0} = \Gamma$ and first return map $f_{u_0} = f \colon \Gamma \to \Gamma$.  We write $H_0 = H_{u_0} < H$, and let ${\bf s} = (s_1,\ldots,s_{b-1})$, $(t_1,\ldots,t_{b-1}) = {\bf t} = e^{\bf s} =  (e^{s_1},\ldots,e^{s_{b-1}})$, $w$, $x = e^w$ be adapted to $u_0$ as in Definition \ref{D:adapted to z}. 

Any class $u \in H^1(X;\R) = \Hom(G,\R)$ can be used to turn the additive group $\R$ into a module over $\Z[H]$ as follows.  Since $u \colon G \to \R$ factors through $H$, we can view the exponential of $u$ as a homomorphism $\rho =  \rho_u = e^u \colon H \to \R_+$ to the multiplicative group of positive real numbers.  This defines an action of $H$ on $\R$ by $h \cdot y = \rho(h) y$ for $h \in H$ and $y \in \R$.  Let $\R_u$ denote $\R$ with this module structure.  Given ${\bf j} = (j_1,\ldots,j_{b-1})$, if we write ${\bf t}^{\bf j} = t_1^{j_1} \cdots t_{b-1}^{j_{b-1}} \in H$, then we can think of this as a function $H^1(X;\R) \to \R_+$, and we have
\[ \rho({\bf t}^{\bf j}) = {\bf t}(u)^{\bf j} = t_1(u)^{j_1} \cdots t_{b-1}(u)^{j_{b-1}}.\]
Similarly, $\rho(x^j) = x(u)^j$.

\begin{remark}
To clarify the notation we note that ${\bf t}(u) \in \R_+^n$, while for ${\bf j} = (j_1,\ldots,j_{b-1})$ we have ${\bf t}(u)^{\bf j} \in \R_+$.
\end{remark}

The construction of the metric $d_u$ uses the following proposition.   Recall that $\mathfrak G \subset \C_X \subset H^1(X;\R) \cong \R^b$ is the graph of ${\bf s} \mapsto \log(E(e^{\bf s}))$ as in Theorem \ref{T:PF}.

\begin{proposition} \label{P:twistedMF1}
For any $u \in \mathfrak G$, there exists a homomorphism of $\Z[H]$--modules
\[ \mu = \mu_u \colon T(\mathcal F) \to \R_u \]
such that for every transversal $\tau$ we have $\mu([\tau]) > 0$.  Moreover, $\mu$ is unique up to scaling.
\end{proposition}
The condition that $\mu$ is a $\Z[H]$--module homomorphism means that it is a group homomorphism, and for every $h \in H$ and transversal $\tau$ we have
\[ \mu(h \cdot [\tau]) =\rho(h) \mu([\tau]).\]

\begin{proof}
Recall that $u_0 \in \A$ is our initial primitive integral class coming from the graph $\Gamma = \Theta_{u_0}$.  As in Section \ref{S:graph module}, we let $E = \{ \sigma_1,\ldots,\sigma_m \}$ be the representative edges in $\tGamma \subset \hGamma$ and $A({\bf t}) = A_{u_0}({\bf t})$ be the associated matrix for the action $\Z[H_0]^E \to \Z[H_0]^E$ for the lift $x^{-1} \circ \tpsi_1$ of $f \colon \Gamma \to \Gamma$.  

We will construct a homomorphism of $\Z[H]$--modules
\[ \bar \mu \colon \Z[H]^E \to \R_u\]
so that for every edge $e$ of $\hGamma$ we have $\bar \mu (e) > 0$ and if $\tpsi_1(e) = e_1 \cdots e_j$, then
\begin{equation} \label{Eq:bar mu homomorphism}
\bar \mu(e) = \bar \mu (e_1) + \ldots + \bar \mu(e_j).
\end{equation}
By definition of $T(\hGamma)$, such a homomorphism will descend to a homomorphism of $\Z[H]$--modules
\[ \mu \colon T(\hGamma) \cong T(\mathcal F) \to \R_u.\]
Since every transversal $\tau$ flows forward into an edge path of $\hGamma$, hence is equivalent in $T(\mathcal F)$ to a positive linear combination of edges of $\hGamma$, we will have $\mu([\tau]) >0$ as required.

Now, since $u \in \mathfrak G$ we have $w(u) = \log(E(e^{{\bf s}(u)}))$ or equivalently
\[ x(u) = E({\bf t}(u)).\]
Let ${\bf U}({\bf t})$ be the (left) Perron-Frobenius eigenvector associated to $E({\bf t})$.  Thus, we have
\[ {\bf U}({\bf t}) A({\bf t}) = E({\bf t})  {\bf U}({\bf t}).\]
Given $\ell \in \{ 1, \ldots,m \}$, define
\[ \bar \mu (\sigma_\ell) = {\bf U}_\ell({\bf t}(u)). \]
For an arbitrary element ${\bf t}^{\bf j} x^j$ with ${\bf t}^{\bf j} = t_1^{j_1} \cdots t_{b-1}^{j_{b-1}} x^j$ and $\ell \in \{1,\ldots,m\}$ we define
\[ \bar \mu({\bf t}^{\bf j} x^j \cdot \sigma_\ell) = \rho({\bf t}^{\bf j} x^j) \bar \mu(\sigma_\ell) = {\bf t}(u)^{\bf j} x(u)^j {\bf U}_\ell({\bf t}(u)).\]

Since $\Z[H]^E$ is a free abelian group, there is a unique extension to a group homomorphism.  By construction, this is a $\Z[H]$--module homomorphism:
\begin{eqnarray*}
\bar \mu({\bf t}^{\bf j'} x^{j'} \cdot ({\bf t}^{\bf j} x^j \cdot \sigma_\ell)) & = & \bar \mu({\bf t}^{\bf j' + j} x^{j' + j} \cdot \sigma_\ell)\\
& = & {\bf t}(u)^{\bf j' + j} x(u)^{j' + j} \bar \mu (\sigma_\ell) \\
&  = & {\bf t}(u)^{\bf j'}x(u)^{j'}({\bf t}(u)^{\bf j}x(u)^j \bar \mu (\sigma_\ell))\\
& = & \rho({\bf t}^{\bf j'}x^{j'}) \bar \mu({\bf t}^{\bf j} x^j \cdot \sigma_\ell)
\end{eqnarray*}

Finally, to prove (\ref{Eq:bar mu homomorphism}), we note that an arbitrary edge $e$ has the form ${\bf t}^{\bf j} x^j \cdot \sigma_\ell$ and so by definition of $A({\bf t})$ and the fact that $\bar \mu$ is a homomorphism, we are left to prove that for every $\ell = 1,\ldots,m$ we have
\[ \bar \mu(x \cdot \sigma_\ell) = \bar \mu \left( \sum_{i=1}^m A_{i,\ell} ({\bf t}) \cdot \sigma_i \right) .\]
For this, we note that since $x(u) = E({\bf t}(u))$ we have
\begin{eqnarray*}
\bar \mu(x \cdot \sigma_\ell) & = & x(u) {\bf U}_\ell({\bf t}(u))\\
& = &  E({\bf t}(u)) {\bf U}_\ell({\bf t}(u))\\
& = & \sum_{i=1}^m A_{i,\ell}({\bf t}(u)) {\bf U}_i({\bf t}(u))\\
& = & \sum_{i=1}^m A_{i,\ell}(\rho(t_1),\ldots,\rho(t_{b-1})) \bar \mu (\sigma_i)\\
& = & \bar \mu \left( \sum_{i=1}^m A_{i,\ell}({\bf t}) \cdot \sigma_i \right)
\end{eqnarray*}
as required.  Thus, $\bar \mu$ descends to $\mu \colon T(\hGamma) = T(\mathcal F) \to \R_u$.

To prove the uniqueness, suppose $\mu'$ is some other homomorphism with $\mu'([\tau]) > 0$ for all transversals $\tau$.  We note that any homomorphism is determined by its values on $[\sigma_1],\ldots,[\sigma_m]$, and so we need only verify that there is some $\lambda > 0$ so that
\[ \mu'([\sigma_\ell]) = \lambda \mu([\sigma_\ell]) \]
for all $\ell  = 1,\ldots,m$.  For this, we write $\sigma = ([\sigma_1],\ldots,[\sigma_m])$ as a row vector, and write $\mu'({\bf \sigma}) = (\mu'([\sigma_1]),\ldots,\mu'([\sigma_m]))$ and $x \cdot {\bf \sigma} = (x \cdot [\sigma_1],\ldots,x \cdot [\sigma_m])$.  With this convention, the relations can be expressed in matrix form as
\[ x \cdot {\bf \sigma} = {\bf \sigma} \cdot A({\bf t}).\]
Then we have
\begin{eqnarray*} \mu'({\bf \sigma}) & = & \mu'(x^{-1} x \cdot {\bf \sigma})\\
& = & x(u)^{-1} \mu'(x \cdot {\bf \sigma})\\
& = &E({\bf t}(u))^{-1} \mu'( {\bf \sigma} \cdot A({\bf t})) \\
& = & E({\bf t}(u))^{-1} \mu'( {\bf \sigma}) A({\bf t}(u))
\end{eqnarray*}
This shows that $\mu'({\bf \sigma})$ is an eigenvector for the Perron-Frobenius eigenvalue $E({\bf t}(u))$.  Since $\mu(\sigma) = {\bf U}({\bf t}(u))$ is the unique eigenvector for $E({\bf t}(u))$ up to scalar multiples, it follows that $\mu'(\sigma)$ is a scalar multiple of $\mu(\sigma)$, and we are done.
\end{proof}

We think of $\mu$ as defining a ``transverse measure'' on $\mathcal F$ which is twisted by the action of $H$ and the homomorphism $\rho_u$.   It is only defined on transversals, and not arbitrary arcs in $2$--cells transverse to $\mathcal F$.  This transverse measure can be extended to any arc, as we explain in Remark \ref{Rm:extending to transverse measure}.

Before we proceed, we note that $\mu$ satisfies a monotonicity condition.

\begin{lemma} \label{L:monotonicity}
For any pair of transversals $\tau,\tau'$, if $\tau \subseteq \tau'$ then
\[ \mu([\tau]) \leq \mu([\tau'])\]
with equality if and only if $\tau = \tau'$.

If $\tau_1,\tau_2,\ldots,\tau_n \subset \tau$ are subtransversals of the transversal $\tau$, pairwise intersecting at most in their endpoints, then
\[ \mu([\tau_1]) + \cdots + \mu([\tau_n]) \leq \mu([\tau]).\]
Equality holds if and only if $\tau = \tau_1 \cup \cdots \cup \tau_n$.
\end{lemma}
\begin{proof}
We can subdivide $\tau'$ as $\tau' = \tau_1 \tau_2 \tau_3$, where $\tau_2 = \tau$, and where either or both of $\tau_1,\tau_3$ may be empty.  Then
\[ \mu(\tau') = \mu(\tau_1) +\mu(\tau_2) + \mu(\tau_3) \geq \mu(\tau_2) =\mu(\tau).\]
Equality holds if and only if $\mu(\tau_1) = \mu(\tau_3) = 0$ and hence if and only if $\tau_1$ and $\tau_3$ are empty, which is to say $\tau = \tau_1 = \tau'$.

The last statement holds for similar reasoning by subdividing $\tau$ into a concatenation of subtransversals consisting of $\tau_1,\ldots,\tau_n$, together with a (possibly empty) collection of other subtransversals.
\end{proof}

\subsection{Constructing the metric}

Recall that the foliation $\Omega_u$ is the descent to $X$ of the foliation by fibers $\tOmega_u$ of $\tX$.
We want to define a metric on fibers $\tOmega_{y,u}$ using $\mu$ and then push it down to a metric on the leaves of $\Omega_u$, but there is a technical difficulty we must address first.  To begin, we just define a notion of length for certain paths.

Given a transversal $\tau$ contained in some $\tOmega_{y,u}$, we define the $\mu$--length of $\tau$ to be
\[ \ell_\mu(\tau) = e^y \mu([\tau]).\]
We can similarly define the $\mu$--length of any path $\tau' = \tilde \psi_s^u(\tau) \subset \tOmega_{s+y,u}$, where $\tau$ is a transversal in $\tOmega_{y,u}$, by
\[ \ell_\mu(\tau') = e^{y+s} \mu([\tau]).\]
We note that while $\tau$ is really a subset of $\tX$, we need to think of $\tau'$ as a parametrized path (though the parametrization is unimportant) since the restriction of $\tpsi_s^u$ to $\tau$ need not be injective.
It is straightforward to check that this length is independent of the transversal $\tau$ that flows into $\tau'$.  We will call such a path $\tau'$ a {\em flowed transversal}.

\begin{lemma} \label{L:mu length under flow}
For any $y,s \in \R$ and $s \geq 0$ and any flowed transversal $\tau'$ we have
\[ \ell_\mu(\tpsi_s^u(\tau')) = e^s \ell_\mu(\tau').\]
\end{lemma}
\begin{proof}
Take a transversal $\tau \in \tOmega_{y,u}$ and $s_0$ so that $\tpsi_{s_0}^u(\tau) = \tau'$.  Then
\[ \ell_\mu(\tpsi_s^u(\tau')) = e^{y + s_0 + s} \mu([\tau]) = e^s (e^{y+s_0} \mu([\tau])) = e^s \ell_\mu(\tau'). \qedhere \]
\end{proof}

\begin{lemma} \label{L:mu length descends}
The $H$--action preserve $\mu$--lengths of flowed transversals:
\[ \ell_\mu(\tau) = \ell_\mu(h(\tau)).\]
\end{lemma}
\begin{proof}
Since $\tpsi^u$ commutes with the action $H$, appealing to Lemma \ref{L:mu length under flow} we may assume that $\tau$ is a transversal contained in $\tOmega_{y,u}$.  Suppose $h \in H$, and then $h(\tau) = h^{-1} \cdot \tau$ is contained in the fiber $\tOmega_{y+u(h),u}$.  Because $\mu$ is a homomorphism of $\Z[H]$--modules we have
\[ \ell_\mu(h(\tau)) = e^{y+u(h)} \mu(h^{-1} \cdot [\tau]) = e^y e^{u(h)} e^{-u(h)} \mu([\tau]) = e^y \mu([\tau]) = \ell_\mu(\tau). \qedhere \]
\end{proof}

By a transversal $\tau$ in a leaf $\Omega_{y,u} \subset X$ we mean the image of a transversal in $\tOmega_{y,u}$ under the covering projection $p \colon \tX \to X$.   We similarly define flowed transversal $\tau$ in $\Omega_{y,u} \subset X$ and note that by Lemma \ref{L:mu length descends}, the $\mu$--length $\ell_\mu(\tau)$ of any such $\tau$ is well-defined as the $\mu$--length back in $\tX$.  Moreover, from Lemma \ref{L:mu length under flow} we have
\[ \ell_\mu(\psi^u_s(\tau)) = e^s \ell_\mu(\tau)\]
for all $s \in \R$ and any transversal $\tau$ in any leaf $\Omega_{y,u}$.

\begin{lemma} \label{L:no atoms}
Let $y,\epsilon \in \R$ with $\epsilon > 0$, and $\xi \in \Omega_{y,u}$ be given.  Then there is a finite set of paths $\tau_1,\ldots,\tau_j$ in $\Omega_{y,u}$ whose union forms a neighborhood of $\xi$, which are each flowed transversals, and for which the sum of the $\mu$--lengths is less than $\epsilon$.  If $\xi$ is not a vertex of $\Omega_{y,u}$ then we can take $j = 1$ and $\tau = \tau_1$ to be a transversal.
\end{lemma}
\begin{proof}
We give the proof for $\xi$ not a vertex.

We suppose that the conclusion is false.  Thus there is a sequence of flowed transversals $\{ \tau_n\}$ in $\Omega_{y,u}$ containing $\xi$ in their interiors for which
\[ \ell_\mu(\tau_n) \geq \epsilon \]
and for which the intersection is $\{ \xi\}$.  Since $\xi$ is not a vertex of $\Omega_{y,u}$, it lies in the interior of a $2$--cell, and hence these flowed transversals can be taken to be transversals.  Such a sequence exists since vertex leaves of $\mathcal F$ are dense.

Suppose first that $\xi$ lies on a periodic orbit of $\psi^u$.  That is, for some $r >0$ we have $\psi_{rj}^u(\xi) = \xi$ for all integers $j \geq 0$.  Then note that for any $j$, there is some $n$ so that $\tau_n$ is so small that
\[ \psi_{rj}^u(\tau_n) \subset \tau_1.\]
On the other hand, together with Lemma \ref{L:monotonicity} this implies
\[ \ell_\mu(\tau_1) \geq \ell_\mu(\psi_{rj}^u(\tau_n)) = e^{rj} \ell_\mu(\tau_n) > e^{rj} \epsilon\]
Since the left- and right-hand side does not depend on $n$, and the inequality between these holds for all $j$, this contradicts the fact that $\ell_\mu(\tau_1)$ is some finite number.

A similar argument works if $\xi$ flows into a point which is periodic, so we assume that this is not the case.  Therefore, we can assume that $\{\flow^u_s(\xi)\}_{s \geq 0}$ accumulates somewhere.  It follows that there is a sequence of times $s_j \to \infty$ so that $\flow^u_{s_j}(\xi) \in \Omega_{y',u}$ for all $j$ and some $y'$, and that these points are contained in a single transversal $\tau$ in $\Omega_{y',u}$.   Then, for any $j \geq 0$, we can find $n$ so that $\tau_n$ is so small that
\[ \flow^u_{s_1}(\tau_n), \flow^u_{s_2}(\tau_n), \ldots, \flow^u_{s_j}(\tau_n) \]
are pairwise disjoint transversals inside $\tau$.   Now we note that $\ell_\mu(\tau)$ is some finite number, whereas by Lemma \ref{L:monotonicity} we have
\[ \ell_\mu(\tau) \geq \ell_\mu(\flow^u_{s_1}(\tau_n)) + \cdots + \ell_\mu(\flow^u_{s_j}(\tau_n)) \geq j \epsilon.\]
Since the right-hand side tends to infinity, this is a contradiction.

This completes the proof for the case that $\xi$ is not a vertex.   The case that $\xi$ is a vertex is similar and we leave that to the reader.
\end{proof}

\begin{remark} \label{Rm:extending to transverse measure}
Lemma~\ref{L:no atoms} can also be used to extend $\mu$ to any arc transverse to $\mathcal F$ (say, contained in a $2$--cell, so that transversality makes sense).  The reason is that we can approximate any such arc by transversals, and the lemma tells us that the limit of the values of $\mu$ will be independent of the sequence of transversals used.
\end{remark}

For every $y$ and every edge $e$ of $\Omega_{y,u}$, we next define a path metric called the {\em $\mu$--path-metric} on $e$ inducing the given topology, so that the length of any transversal is the $\mu$--length.  For any two points $\xi,\xi'$ contained in $e$, we consider a sequence of transversals whose endpoints limit to $\xi$ and $\xi'$.  By Lemma~\ref{L:monotonicity}, the limit of the $\mu$--lengths of these edges is independent of the sequence, and we call this the {\em $\mu$--path distance between $\xi$ and $\xi'$}.
The transversals and flowed transversals define a basis for the topology on $e$, and so Lemmas~\ref{L:monotonicity} and \ref{L:no atoms}, together with the fact that every transversal has positive $\mu$--length (Proposition \ref{P:twistedMF1}), further implies that the $\mu$--path distance is a path metric inducing the given topology.

We now define a metric on $X$ for which the induced path metric on every edge $e$ of $\Omega_{y,u}$ is precisely the $\mu$--path-metric.
This metric on $X$ is obtained as a quotient path metric of a collection of ``nice'' subsets of the hyperbolic plane, and will be constructed in a couple steps.

First, for every $y \in \R$ and every transversal $\tau \subset \Omega_{y,u} \subset X$ let $\beta_\tau \colon [0,r_0] \to \tau$ denote a unit speed parameterization with respect to the $\mu$--path-metric (note that $\tau$ is contained in some edge).
Next, let $s_0 > 0$ be any positive real number and define a map
\[ B_{\tau,s_0} \colon [0,r_0] \times [0,s_0] \to X \]
by
\[ B_{\tau,s_0}(r,s) = \psi^u_s(\beta_\tau(r)).\]
We give $[0,r_0] \times [0,s_0]$ the Riemannian metric $e^{2s} dr^2 + ds^2$, and note that this is isometric to the region in the hyperbolic plane bounded by two asymptotic geodesics and two tangent horocycles with its induced path metric.   We call this the {\em hyperbolic metric on $[0,r_0] \times [0,s_0]$}.   Also observe that $B_{\tau,s_0}$ maps $[0,r_0] \times \{s\}$ into $\Omega_{y+s,u}$ and for any arc of $[0,r_0] \times \{s\}$ that maps into an edge $e$ of $\Omega_{y+s,u}$, the map on this arc is a path isometry to the $\mu$--path-metric $e$.

We will also want to consider this same construction in a slightly more general setting, namely when $\tau$ is an injective flowed transversal.  In this case, $\tau$ will be a concatenation of finitely many transversals and the $\mu$--path metric makes sense on each of these and extends to a metric on all of $\tau$.  We then proceed choosing a unit speed parameterization $B_\tau \colon [0,r_0] \to \tau$ and constructing the map $B_{\tau,s_0} \colon [0,r_0] \times [0,s_0] \to X$ as above.

We call any $B_{\tau,s_0}$ which is a homeomorphism onto its image a {\em hyperbolic model map}.   Suppose $\delta \colon [0,1] \to X$ is a path contained in the image of some hyperbolic model map $B_{\tau,s_0}$.  Define the $d_u$--length of $\delta$ to be
\[ \ell_u(\delta) = \ell(B_{\tau,s_0}^{-1}(\delta)) \]
where on the right, $\ell$ denotes length in the hyperbolic metric on $[0,r_0] \times [0,s_0]$.  Given a path $\delta \colon [0,1] \to X$ which can be subdivided into subpaths, each of which is contained in the image of some hyperbolic model map, we define the $d_u$--length to be the sum of the $d_u$--lengths of the subpaths.  We call such a path a {\em tame path}.  The $d_u$--length of any tame path is independent of the set of hyperbolic model maps used to calculate its length---this follows from the fact that on the overlap of any two, the composition $B_{\tau,s_0} \circ B_{\tau',s_0'}^{-1}$ (where defined) is a path isometry.

Now we define a pseudo-metric $d_u$ on $X$ by
\[ d_u(\xi,\xi') = \inf \{ \ell_u(\delta) \mid \delta \mbox{ is a tame path with } \delta(0) = \xi, \delta(1) = \xi' \}.\]
\begin{proposition} \label{P:pseudo-metric is metric}
The pseudo-metric $d_u$ is a geodesic metric on $X$.  Furthermore, for every point $\xi$ in the interior of a $2$--cell, there is a hyperbolic model map $B_{\tau,s_0}$ and a disk $\Delta \in [0,r_0] \times [0,s_0]$ isometric to a disk in the hyperbolic plane so that $B_{\tau,s_0}\big|_{\Delta}$ maps isometrically onto a neighborhood of $\xi$.

For any other point $\xi$, there exists a finite set of hyperbolic model maps $\{ B_{\tau_i,s_{0,i}} \}_{i=1}^n$ and disks $\Delta_i \subset [0,r_{0,i}] \times [0,s_{0,i}]$ isometric to disks in the hyperbolic plane such that $B_{\tau_i,s_{0,i}}$ restricts to an isometric embedding and the union of the images defines a neighborhood of $\xi$.
\end{proposition}
\begin{proof}
For a point $\xi$ in the interior of a $2$--cell, we can find a hyperbolic model map that is a homeomorphism onto a neighborhood of $\xi$.  Restricting to a sufficiently small disk provides the required neighborhood.

\begin{figure}[htb]
\begin{center}
\includegraphics[width=13cm]{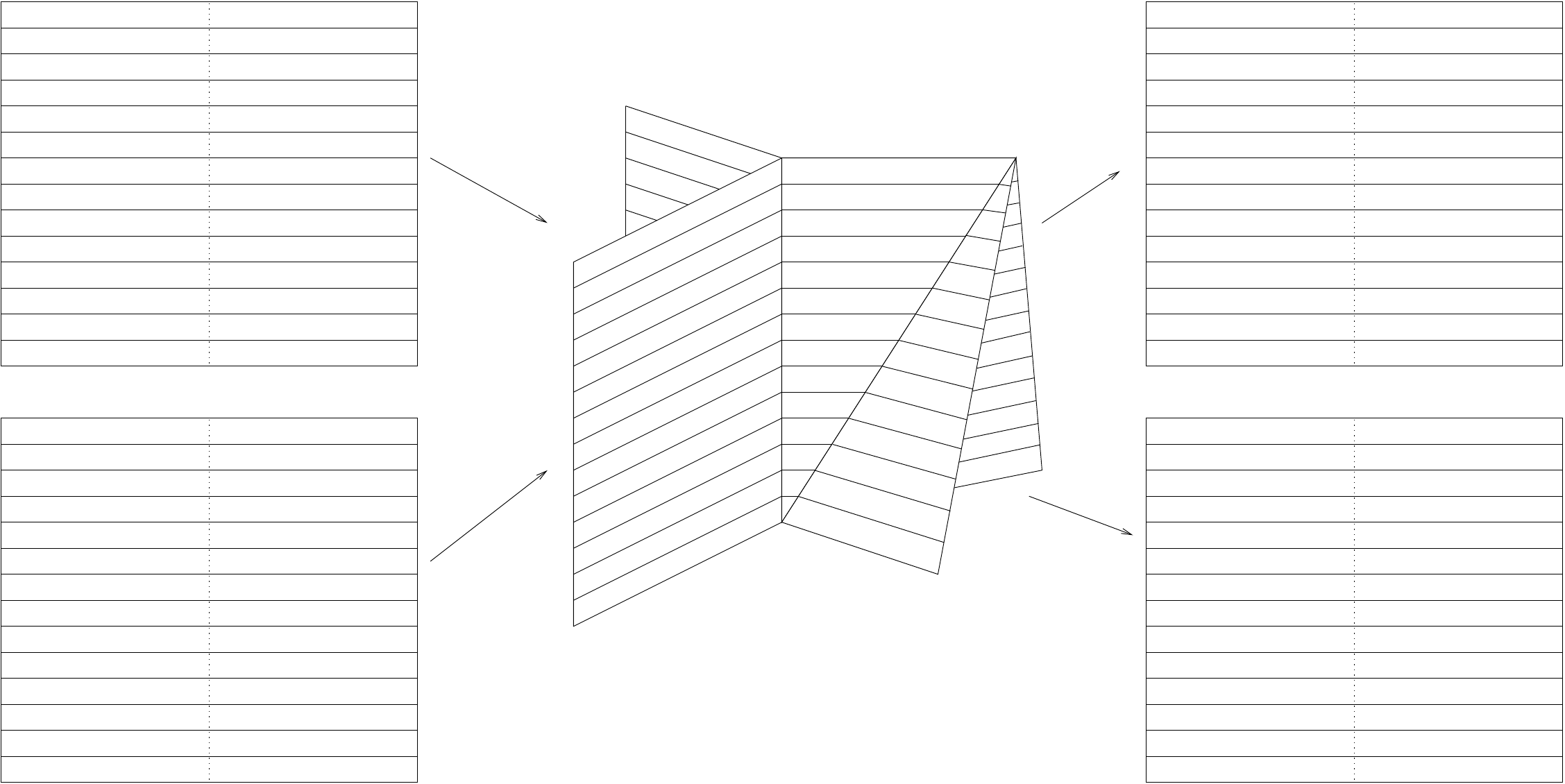}
\caption{The maps $B_{\tau_1,s_0},B_{\tau_2,s_0}$ from $[0,r_0] \times [0,s_0]$ (left) to the neighborhood $V$ (middle), and the projections $\Pi_1$ and $\Pi_2$ back onto $[0,r_0] \times [0,s_0]$.  In this example, the map $\Pi_2 \circ B_{\tau_1,s_0}$ maps the upper left rectangle to the lower right rectangle, and is given by $(r,s) \mapsto (r',s)$ where $r' = r$ if $r \geq r_0/2$ and $r' = r_0/2$ if $r \leq r_0/2$.} \label{F:gluing models}
\end{center}
\end{figure}

For a point $\xi$ which is not in the interior of a $2$--cell, we choose a finite number of hyperbolic model maps $\{B_{\tau_i,s_{0,i}}\}_{i=1}^n$ for which the images, denoted $\{ V_i \}_{i=1}^n$, contain $\xi$ and so that the union $V = V_1 \cup \ldots \cup V_n$ is a (closed) neighborhood of $\xi$.  By shrinking the model maps if necessary, we can assume that (1) the domains are all the same, so $s_{0,i} = s_0$ and $r_{0,i} = r_0$ for some $r_0,s_0 > 0$ and all $i$, (2) the defining transversals (or flowed transversals) $\tau_1,\ldots,\tau_n$ are contained in a single leaf $\Omega_{y,u}$, and (3) $\xi = B_{\tau_i,s_0}(r_0/2,s_0/2)$ for all $i$.  These conditions are easy to arrange and primarily serve to simplify the notation.

We note that the ``top'' $V_i^+ = B_{\tau_i,s_0}([0,r_0], \times \{s_0\}) \subset V_i$ is contained in the leaf $\Omega_{y+s_0,u}$ for all $i$, and the union $V^+ = V_1^+ \cup \ldots \cup V_n^+$ forms a neighborhood of the point $\flow_{s_0/2}^u(\xi)$ (which is also the point $B_{\tau_i,s_0}(r_0/2,s_0)$, for all $i$) inside $\Omega_{y+s_0,u}$.   By shrinking the models further if necessary, we can assume that $V^+$ is a contractible subset of the graph $\Omega_{y+s_0,u}$, hence $V^+$ is a tree.

\begin{claim}
For each $i = 1,\ldots,n$, there is a map $\Pi_i \colon V \to V_i $ such that
\[ \Pi_i \circ B_{\tau_j,s_0} \colon [0,r_0] \times [0,s_0] \to [0,r_0] \times [0,s_0] \]
satisfies the following conditions with respect to the hyperbolic metrics on domain and range:
\begin{enumerate}
\item[(i)] for every $i,j$, $\Pi_i \circ B_{\tau_j,s_0}$ is $1$--Lipschitz, and
\item[(ii)] for every $i$, $\Pi_i \circ B_{\tau_i,s_0}$ is the identity.
\end{enumerate}
\end{claim}
See Figure \ref{F:gluing models}.
\begin{proof}[Proof of Claim.]
For each $i$, the top $V_i^+$ is an arc, and we first define the restriction of $\Pi_i$ to $V^+$ to be the ``closest point projection'' $V^+ \to V_i^+$; since $V^+$ is a tree, this closest point projection is independent of the choice of a path metric used to define it.    For each $i,j = 1,\ldots,n$, we compose $\Pi_i$ with the model maps and their inverses to obtain a map $\pi_{i,j} \colon [0,r_0] \to [0,r_0]$ given by
\[ \pi_{i,j}(r) = B_{\tau_i,s_0}^{-1}(\Pi_i(B_{\tau_j,s_0}(r,s_0))).\]
Observe that when $i=j$, we have $\pi_{i,i}$ is the identity.   Furthermore, when $i \neq j$, $\pi_{i,j}$ is $1$--Lipschitz: in fact, because of the way closest point projections work in a tree, we can subdivide $[0,r_0]$ into at most three segments, so that on each segment $\pi_{i,j}$ is either an isometry, or constant.

Now, given any point $\xi_0 \in V$ we define $\Pi_i(\xi_0)$ to be the point obtained by first flowing up for some time $s$ to $V^+$, projecting by closest point projection to $V_i^+$, then {\em within $V_i$} flowing backward a time $s$.  This last map in the composition is possible because inside $V_i$, the semiflow restricts to a local flow:  since $B_{\tau_i,s_0}$ is a homeomorphism onto its image, this local flow is just the conjugate of the local flow $\nu$ on $[0,r_0] \times [0,s_0]$ defined by $\nu_t(r,s) = (r,s+t)$.  With this notation we can write down the map $\Pi_i$
\[ \Pi_i(\xi_0) = \nu_{-s}(B_{\tau_i,s_0}^{-1}(\Pi_i(\flow_s^u(\xi_0)))) \]
where $s \geq 0$ is the smallest number so that $\flow_s(\xi_0) \in V^+$.   

For any $i,j = 1,\ldots,n$, the composition $\Pi_i \circ B_{\tau_j,s_0}$ takes the simple form
\[ \Pi_i \circ B_{\tau_j,s_0}(r,s) = (\pi_{i,j}(r),s).\]
It now follows from the properties of $\pi_{i,j}$ mentioned above that $\Pi_i$ satisfies (i) and (ii) from the claim, for all $i$ and $j$.
\end{proof}

Now suppose that $(r,s),(r',s') \in [0,r_0] \times [0,s_0]$ are any two points.  If $\delta$ is tame path between $B_{\tau_i,s_0}(r,s)$ and $B_{\tau_i,s_0}(r',s')$ contained in $V$, then $\Pi_i(\delta)$ is a path between $(r,s)$ and $(r',s')$ with hyperbolic length no more than the $d_u$--length of $\delta$ (since $\Pi$ is $1$--Lipschitz on each hyperbolic model where we compute lengths).  Therefore, with respect to the path pseudo-metric on $V$ induced by $d_u$--lengths, the map $B_{\tau_i,s_0}$ preserves distances.  In particular, this pseudo-metric on $V$ is a metric.

Finally, let $\epsilon > 0$ be half the hyperbolic distance from $(r_0/2,s_0/2)$ to the boundary $[0,r_0] \times \{0,s_0\} \cup \{0,r_0\} \times [0,s_0]$, and let $\Delta$ be a hyperbolic ball of radius $\epsilon$ centered at $(r_0/2,s_0/2)$.
\begin{claim}
The restriction of $B_{\tau_i,s_0}$ to $\Delta$ is distance preserving with respect to the $d_u$ pseudo-metric. 
\end{claim}
\begin{proof}[Proof of Claim.]
Observe that from the discussion above, in the path metric on $V$ induced by $d_u$--lengths, the union of the images of $\Delta$
\[ B_{\tau_1,s_0}(\Delta) \cup \ldots \cup B_{\tau_n,s_0}(\Delta)\]
is precisely the $\epsilon$--ball about $\xi$.  Furthermore, for any two points $(r,s),(r',s') \in \Delta$, a path $\delta$ between $B_{\tau_i,s_0}(r,s)$ and $B_{\tau_i,s_0}(r',s')$ which leaves $V$ must contain two subpaths, each connecting the frontier of $B_{\tau_1,s_0}(\Delta) \cup \ldots \cup B_{\tau_n,s_0}(\Delta)$ to the frontier of $V$, completely contained in $V$.  Considering appropriate projections $\Pi_j$ and $\Pi_k$, we see that these subpaths must each have length at least $\epsilon$, and so the path $\delta$ has length at least $2 \epsilon$.  Since the hyperbolic distance between $(r,s)$ and $(r',s')$ is less than $2 \epsilon$, it follows that the hyperbolic geodesic in $\Delta$ between these points projects to a path whose $d_u$--length is strictly shorter than any path not entirely contained in $V$.  But then distances between points in $B_{\tau_i,s_0}(\Delta)$ are computed in terms of $d_u$--lengths of paths in $V$.  We have already seen that $B_{\tau_i,s_0}$ is distance preserving in the $d_u$--path metric on $V$, and so $B_{\tau_i,s_0}$ is distance preserving as a map from $\Delta$ into $X$.
\end{proof}
All that remains is to see that $d_u$ is actually a metric, and not a pseudo-metric---as a path metric on a compact space, it will necessarily be geodesic.  For that, we just need to know that no two points have distance $0$.  This follows from the arguments just given: a path from some point $\xi$ to any other point $\xi'$ is either entirely contained in the neighborhood of $\xi$ built from hyperbolic models, and hence is bounded away from zero, or else must exit the neighborhood and hence has length bounded away from $0$ (by the radius of the hyperbolic disk $\Delta$).  This completes the proof.
\end{proof}

\begin{corollary} \label{C:induced path metric on leaves}
For any $y \in \R$ the path metric on $\Omega_{y,u}$ induced by $d_u$ makes $\Omega_{y,u}$ into a simplicial metric graph.  Moreover, the path metric on any edge $e$ induced by $d_u$ is precisely the $\mu$--path-metric on $e$.
\end{corollary}
\begin{proof}
Because $\omega^u$ is tame, the restriction to any skew $1$--cell has only finitely many critical points.  Therefore the components of intersection of the neighborhoods from Proposition \ref{P:pseudo-metric is metric} with a leaf $\Omega_{y,u}$ define neighborhoods of points in $\Omega_{y,u}$ which are obtained by gluing finitely many flowed transversals of the form $B_{\tau_i,s_0}([0,r_0] \times \{s \})$ along finitely many arcs.  Furthermore, the metric $d_u$ induces the $\mu$--path-metric on these flowed transversals.  Since every point in any $\Omega_{y,u}$ has such a neighborhood, the corollary follows.
\end{proof}

We are now ready to give the
\begin{proof}[Proof of Theorem~\ref{T:lipschtiz flows}]
Suppose first that $u \in \Csec$ lies in $\mathfrak G$ so that $\lambda = e^{\mathfrak H(u)} = e$.   Let $d_u$ be the metric constructed above.  Proposition \ref{P:pseudo-metric is metric} guarantees that $d_u$ is a metric, and provides enough local information to guarantee that $d_u$ satisfies all the properties of the theorem as we now explain.

First, the flowlines of the local flow $\nu_t(r,s) = (r,t+s)$ on the $[0,r_0] \times [0,s_0]$ are unit speed geodesics in the hyperbolic metric on $[0,r_0] \times [0,s_0]$ and the hyperbolic model maps $B_{\tau_i,s_0}$ conjugate the local flow to the restriction of the semiflow on the image.   Since $B_{\tau_i,s_0}$ is an isometric embedding, (1) follows.  Part (2) follows from Corollary \ref{C:induced path metric on leaves}.  In the hyperbolic metric on $[0,r_0] \times [0,s_0]$, $\nu_s$ maps $[0,r_0] \times \{ s' \}$ to $[0,r_0] \times \{ s' + s \}$ and is an $e^s$--homothety.  Since $\lambda = e$ and $B_{\tau_i,s_0}$ conjugates $\nu$ to the restriction of $\flow^u$, (3) is also true.  Finally, part (4) is an immediate consequence of the first part of Proposition \ref{P:pseudo-metric is metric} since it says that the metric on the interior of every $2$--cell is locally isometric to the hyperbolic plane, hence has constant curvature $-1$.

For an arbitrary class $u \in \Csec$, we use the fact that $\Csec$ is contained in the McMullen cone $\C_X$ hence the ray through $u$ intersects $\mathfrak G$ in a unique point $u'$.   Let $t > 0$ be such that $u' = t u$ then $\lambda = e^{\mathfrak H(u)} = e^{\mathfrak H(\frac{1}{t}u')} = e^{t \mathfrak H(u')} = e^t$.   Let $\omega^{u'}$ be the closed $1$--form used above since $u' \in \mathfrak G$.  Then $\omega^u = t \omega^{u'}$ represents $u$, and the foliation $\Omega_u$ is exactly the same as the foliation $\Omega_{u'}$ except that the leaves are given by $\Omega_{y,u} = \Omega_{ty,u'}$ (since $\widetilde \fib_{u'} = t \widetilde \fib_u$).
Similarly, the reparameterization $\flow^u$ is related to $\flow^{u'}$ by $\flow^u_s = \flow^{u'}_{ts}$. 

Now let $d_u = \frac{1}{t} d_{u'}$.  Since $\flow^u_s = \flow^{u'}_{ts}$, the $\flow^u_s$--flow lines are local geodesics, and on any edge of any leaf of $\Omega_u$ ($=\Omega_{u'}$), $\flow^u_s = \flow^{u'}_{ts}$ is still an $e^{ts}$--homothety (this is unaffected by scaling the metric).  Since $\lambda = e^t = e^{\mathfrak H(u)}$, this completes the proof.
\end{proof}


\appendix

\section{Characterizing sections}
\label{S:characterizing_sections}

The goal of this appendix is to prove Proposition~\ref{P:dual_sections}, which explicitly constructs dual cross sections for all integral classes in the Fried cone $\D$.  
To facilitate the construction, we give a combinatorial characterization of $\D$ in terms of a positivity condition for cellular cocycle representatives with respect to a certain trapezoidal subdivision of $X$.
This characterization is similar in spirit to the definition of the cone $\A$ from \cite{DKL}.  This allows us to borrow many of the techniques from \cite{DKL} for constructing sections and apply them to classes in the Fried cone $\D$.  The constructions require a number of modifications and new ideas, and while the resulting sections still provide a wealth of information about the associated cohomology class, they do  not (and can not) behave as nicely as those from Theorem~\ref{T:DKL_BC} in general.  

We begin by recalling some of the relevant notions about Trapezoidal subdivisions from \cite{DKL}.  Then we introduce the combinatorial characterization of $\D$ and analyze some of its consequences. Next we describe a procedure for subdividing the cell structure and refining the cocycle so that it more closely resembles a cocycle representing an element of $\A$.  From this we finally construct a section dual to any primitive integral class $u \in \D$ following the construction in \cite{DKL} with appropriate modifications.

\subsection{Trapezoidal subdivisions}
\label{S:trapezoidal_subdivisions}

Our construction of cross sections will involve working with various trapezoidal subdivisions of the cell structure on $X$. To this end we recall from \cite[\S6.3]{DKL} the relevant structure of these subdivisions as well as some terminology and notation for working with them. A \emph{trapezoidal cell structure} on $X$ is one in which every $1$--cell is either vertical or skew, and  every $2$--cell is a trapezoid. Each $1$--cell $e$ is equipped with a globally defined positive orientation so that the restriction $\eta\vert_{e}$ is orientation preserving. The boundary of a trapezoidal $2$--cell $T$ consists of four arcs $\ell_-(T)$, $\ell_+(T)$, $e_-(T)$, and $e_+(T)$ which we refer to as the left, right, bottom and top arcs of $T$, respectively. Each of these arcs is a union of $1$--cells and may therefore be regarded as a cellular $1$--chain in $X$. The trapezoid induces a \emph{$T$--orientation} on each of these arcs in which $\ell_\pm(T)$ and $e_-(T)$ are given the positive orientation; the sides $\ell_\pm(T)$ are then distinguished by the convention that $e_-(T)$ is oriented from $\ell_-(T)$ to $\ell_+(T)$. The $T$--orientation on the top arc $e_+(T)$ is defined so that it is also oriented from $\ell_-(T)$ to $\ell_+(T)$. This $T$--orientation on $e_+(T)$ may not agree with the globally defined positive orientation, and we define the \emph{sign}  $\zeta(T)\in \{\pm1\}$ of $T$ so that the $1$--cells comprising the $1$--chain $\zeta(T)e_+(T)$ appear with positive sign. 

In any trapezoidal cell structure we require that the sides $\ell_\pm(T)$ of each trapezoid consist of vertical $1$--cells (however $\ell_+(T)$ may degenerate to a vertex), and that the bottom arc $e_-(T)$ consists of a \emph{single} skew $1$--cell.  Conversely each skew $1$--cell is equal to $e_-(T)$ for a unique trapezoid $T$. On the other hand, the top arc $e_+(T)$ of a trapezoid may consist of several skew $1$--cells. Finally, we note that each $0$--cell is the initial endpoint of a unique (positively oriented) vertical $1$--cell.

\subsection{A combinatorial characterization of the Fried cone $\D$}
\label{S:the_cone}

Recall that $f\colon \Gamma\to \Gamma$ is an expanding irreducible train track map with transition matrix $A(f)$ and associated transition graph $\mathcal G(f)$.  The circuits $\mathcal Y$ of $\mathcal G(f)$ determine a finite set of closed orbits $\{ \mathcal O_y\}_{y \in \mathcal Y}$ as explained in \S\ref{S:cone_of_sections}. 

We subdivide $X$ along the closed orbits $\mathcal O_y$ for all circuits $y\in \mathcal{Y}$. More precisely, for a given circuit $y\in \mathcal Y$, each skew $1$--cell $\sigma$ of $X$ that intersects $\mathcal O_y$ is subdivided by adding a $0$--cell at $\sigma\cap \mathcal O_y$, and each trapezoidal $2$--cell $T$ of $X$ that intersects $\mathcal O_y$ is subdivided into two trapezoids by adding the vertical $1$--cell $T\cap \mathcal O_y$. The cell structure obtained by performing these subdivisions for each $y\in \mathcal{Y}$ is called the \emph{circuitry cell structure} on $X$ and will be denoted $\XC$. Since $\mathcal{Y}$ is finite, we note that $\XC$ is indeed a finite trapezoidal subdivision of $X$.

\begin{defn}[Vertical positivity]
\label{D:B-cone}
A cellular $1$--cocycle $z\in Z^1(Y;\R)$ of a trapezoidal subdivision $Y$ of $X$ is said to be \emph{vertically positive} if $z(\sigma) > 0$ for every vertical $1$--cell $\sigma$ of $Y$. 
\end{defn} 

\begin{proposition}
\label{P:homological_characterization}
The Fried cone $\D$ is equal to the set of cohomology classes that can be represented by vertically positive $1$--cocycles of the circuitry cell structure $\XC$.
\end{proposition}
\begin{proof}
Recall from Proposition \ref{P:Fried finiteness} that $\D\subset H^1(X;\R)$ consists exactly of those classes which are positive on all orbits $\mathcal O_y$ for $y\in \mathcal Y$. Since each $\mathcal O_y$ is comprised of vertical $1$--cells of $\XC$, it follows immediately that every vertically positive cocycle of the circuitry cell structure $\XC$ represents a class in $\D$. 

It remains to prove that each class in $\D$ can be represented by a vertically positive cocycle of $\XC$. Let $u \in \D$ so that $u(\mathcal{O})>0$ for every closed orbit $\mathcal O\in \mathfrak{O}(\flow)$ of $\flow$. Let us say that a vertical $1$--cell of $\XC$ is periodic if it lies on a closed orbit of $\flow$ and that it is nonperiodic otherwise. As closed orbits of $\flow$ are necessarily disjoint, we see that the periodic $1$--cells form a disjoint union of finitely many $\gamma_1,\dotsc,\gamma_k$ closed orbits of $\flow$ (some of which are of the form $\mathcal{O}_y$ for $y\in \mathcal Y$). The (oriented) nonperiodic $1$--cells, on the other hand, form a disjoint union of oriented rooted trees $T_1,\dotsc,T_n$ (oriented towards the root) with the root of each tree lying on one of the closed orbits $\gamma_i$.  

Choose any cocycle $z_0\in Z^1(\XC;\R)$ representing $u$. Since $u\in \mathcal{D}$, we then have $z_0(\gamma_i)> 0$ for each of the closed orbits $\gamma_1,\dotsc,\gamma_k$. 
Therefore, by adding coboundaries associated to the vertices of $\gamma_i$ to distribute the value $z_0(\gamma_i)$ equally among all periodic $1$--cells comprising $\gamma_i$, we obtain a cohomologous cocycle $z'_0$ that is positive on all periodic $1$--cells.

Having accomplished this, we may now freely add coboundaries associated to any vertex in any tree $T_j$, except for the root, without affecting the value of $z'_0$ on the periodic $1$--cells. Since each $T_j$ is a tree, it straightforward to add such coboundaries -- working from the root down -- so that the resulting cocycle $z$ is positive on all nonperiodic $1$--cells as well.  Thus $z$ is a vertically positive representative of $u$, as desired.
\end{proof}

Recall that our objective is to prove $\D\subseteq \Csec$ by constructing cross sections dual to all integral classes $u\in \D$. The first step towards this goal is the following lemma, which is a quantified strengthening of Proposition~\ref{P:Fried finiteness}.

\begin{lemma}[Positive on flowlines]
\label{L:positive_on_flowlines}
Given $u \in \D$, there exists $\kappa > 0$ with the following property. Let $X'$ be any trapezoidal subdivision of $X$ and suppose that $\xi\in X$ and $t> 0$ are such that $\xi$ and $\flow_t(\xi)$ both lie on the same skew $1$--cell $\sigma$ of $X'$. If $\gamma\subset X$ denotes the closed loop obtained by concatenating the flowline from $\xi$ to $\flow_t(\xi)$ with the arc of $\sigma$ connecting $\flow_t(\xi)$ to $\xi$, then we have 
\[u(\gamma) \geq \kappa(t - 1).\]
\end{lemma}
\begin{proof}
Let $\delta>0$ be the minimum of $u(\mathcal O_y)$ for $y\in \mathcal{Y}$ and let $L\geq 1$ denote the longest (combinatorial) length of any circuit in $\mathcal{G}(f)$.   Set $\kappa = \delta/L$.

Choose the smallest $s \geq 0$ so that $q = \flow_s(\xi)\in \Gamma\subset X$, and consider the the homotopic loop $\flow_s(\gamma)$. Let $e\in E\Gamma$ be the edge of $\Gamma$ containing $q$. By pushing the `slanted' portion of the loop $\flow_s(\gamma)$ (running along $\flow_s(\sigma)$) onto $e$, we see that there is some integer $k\geq 1$ (near to $t$) so that $\flow_s(\gamma)$ is homotopic to the closed loop $\gamma'$ obtained by concatenating the flowline from $q$ to $\flow_{k}(q)$ with the arc along $e$ from $\flow_k(q)$ to $q$. In fact, $\abs{t-k}$ here is equal to the length of the image of the arc along $\sigma$ from $\flow_t(\xi)$ to $\xi$ under the map $\fib\colon X\to \sone$. Since this is a subarc of some skew $1$--cell of $X$ and, by construction (see \cite[\S4.4]{DKL}), every skew $1$--cell of $X$ projects under $\eta\colon X\to \sone$ to an arc of length less than $1$, we have $k\in(t-1,t+1)$.

There is a subinterval $\alpha \subset e$ containing $q$ on which $f^k$ is an affine homeomorphism onto $e$, and we let $q_0 \in \alpha$ denote the fixed point of $f^k$.  There is thus a closed orbit $\mathcal O_w$ through $q_0$, which is homotopic to $\gamma'$ and which also crosses $\Gamma$ $k$ times.  It follows that the associated closed path $w$ in $\mathcal G(f)$ has combinatorial length $k$.  By Lemma \ref{L:additive cone map} we can now write
\[ [\gamma] = [\gamma'] = [\mathcal O_w] = [\mathcal O_{y_1}] + \cdots + [\mathcal O_{y_n}]\]
for some circuits $y_i\in \mathcal{Y}$ whose combinatorial lengths sum to $k$. Since each $y_i$ has length at most $L$, we have
\[u(\gamma) = u(\mathcal{O}_{y_1}) + \dotsb + u(\mathcal{O}_{y_n}) \geq n\delta \geq k\delta/L \geq (t-1)\delta/L = \kappa(t-1).\qedhere\]
\end{proof}

\subsection{Height}
\label{S:height}

Recall that $p\colon \tX\to X$ is the universal torsion-free abelian cover of $X$; the induced homomorphism $p_*\colon H_1(\tX;\R)\to H_1(X;\R)$ is then trivial by definition of $\tX$.  Given a trapezoidal subdivision $Y$ of $X$, let $\widetilde{Y}$ denote the lifted trapezoidal cell structure on the universal torsion-free abelian cover and choose a vertex $v_0\in\widetilde{Y}^{(0)}$ to serve as a basepoint. Every cocycle $z\in Z^1(Y;\R)$ then determines a \emph{height function} $h_z\colon \widetilde{Y}^{(0)}\to \R$ whose value at a vertex $v\in \widetilde{Y}^{(0)}$ is defined to be
\[h_z(v) = \tilde{z}(\Sigma),\]
where $\tilde{z} = p^*(z)$ is the pull-back of $z$ and $\Sigma$ is any cellular $1$--chain on $\widetilde{Y}$ with boundary $\partial \Sigma = v - v_0$. If $\Sigma'$ is any other $1$--chain with $\partial \Sigma' = v-v_0$, then $\Sigma'-\Sigma$ is a cycle and so $p_*(\Sigma'-\Sigma)$ is nullhomologous in $H_1(X;\R)$. Therefore $\tilde{z}(\Sigma' - \Sigma) = z(p_*(\Sigma'-\Sigma)) = 0$ showing that $h_z(v)$ is well-defined. We note that for any two vertices $v,v'\in \widetilde{Y}^{(0)}$, the difference $h_z(v') - h_z(v)$ is equal to $\tilde{z}(\Sigma)$ for any $1$--chain in $\widetilde{Y}$ with $\partial \Sigma = v' - v$. For notational convenience, we similarly define the height of a $1$--cell $\sigma$ of $\widetilde{Y}$ as $h_z(\sigma) = \min\{h_z(t(\sigma)),h_z(o(\sigma))\}$.

\begin{lemma}
\label{L:short_flowlines_have_bounded_height}
Let $Y$ be a trapezoidal subdivision of $X$ with a cellular cocycle $z\in Z^1(Y;\R)$. Then there exists $L\ge 0$ such that if $\widetilde{T}$ is a trapezoid of $\widetilde{Y}$ and $\sigma_\pm$ are skew $1$--cells along the arcs $e_\pm(\widetilde T)$, then $\abs{h_z(\sigma_+) - h_z(\sigma_-)}\leq L$.
\end{lemma}
\begin{proof}
For any trapezoid $T$ of $Y$, let $M(T)$ denote the quantity
\[M(T) = \sum_{\sigma\in\partial T} \abs{z(\sigma)},\]
where the sum is over all $1$--cells that occur in the $1$--chain $\partial T$. Take $L$ to be the maximum value of $M(T)$ among all (finitely many) trapezoids $T$ of $Y$. For skew $1$--cells $\sigma_\pm$ as in the statement of the lemma, let $v_\pm$ denote the vertex of $\sigma_\pm$ achieving $h_z(\sigma_\pm)$. Then there exists a $1$--chain $\Sigma$ consisting of cells in $\partial\widetilde T$ for which $\partial \Sigma = v_+ - v_-$. By the above definition of $L$, it then follows that
\[\abs{h_z(\sigma_+) - h_z(\sigma_-)} = \abs{h_z(v_+) - h_z(v_-)} = \abs{\tilde{z}(\Sigma)} \leq L.\qedhere\]
\end{proof}

Combining this with Lemma~\ref{L:positive_on_flowlines}, we can now prove the following.

\begin{lemma}[Linear ascent]
\label{L:linear_ascent}
Let $Y$ be a trapezoidal subdivision of $X$ and let $z\in Z^1(Y;\R)$ be a cocycle representing a class $[z]\in \D$. Then there exists $\kappa > 0$ and $N\geq 0$ satisfying the following property. Suppose that $\xi\in \widetilde Y$ and $t> 0$ are such that $\xi$ and $\tflow_t(\xi)$ lie on skew $1$--cells $\sigma$ and $\sigma'$ of $\widetilde{Y}$, respectively. Then
\[h_z(\sigma') - h_z(\sigma) \geq \kappa t - N.\]
\end{lemma}
\begin{proof}
Let us declare two skew $1$--cells of $\widetilde{Y}$ to be equivalent if they cover the same $1$--cell of $Y$, and let us denote the corresponding set of equivalence classes by $\mathcal{W}$. Note that $\mathcal{W}$ is finite. Let $\kappa>0$ be the constant provided by Lemma~\ref{L:positive_on_flowlines}, and let $L\geq 0$ be as in Lemma~\ref{L:short_flowlines_have_bounded_height}.

Write $\alpha_1 = \sigma$ and let $W_1\in \mathcal{W}$ be its equivalence class. Set $s_1 = 0$ and define
\[t_1 = \max\{s\in [0,t] \mid \tflow_s(\xi)\text{ lies in a skew $1$--cell in }W_1\}\]
to be the last time that the flowline from $\xi$ to $\tflow_t(\xi)$ hits a skew $1$--cell in $W_1$. Denote the skew $1$--cell containing $\tflow_{t_1}(\xi)$ by $\beta_1\in W_1$. If $t_1=t$ then we stop, otherwise let
\[s_2 = \min\{s \in (t_1,t] \mid \tflow_s(\xi)\text{ lies in a skew $1$--cell}\}\]
be the next time that the flowline hits a skew $1$--cell. We again let $\alpha_2$ denote the skew $1$--cell containing $\tflow_{s_2}(\xi)$ and let $W_2\in \mathcal{W}$ be its equivalence class. Notice that $W_2\neq W_1$ by definition of $t_1$ and $s_2$. We now similarly define
\[t_2 = \max\{s\in [0,t] \mid \tflow_s(\xi)\text{ lies in a skew $1$--cell in }W_2\}\]
to be the last time the flowline hits a skew $1$--cell in $W_2$ and let $\beta_2\in W_2$ denote that particular $1$--cell. Note that $t_2\ge s_2$ by definition.

Continuing in this manner, we find a sequence of times
\[0 = s_1 \leq t_1 < s_2 \leq t_2 < \dotsb < s_n \leq t_n = t\]
such that for each $1\leq i \leq n$ the points $\tflow_{s_i}(\xi)$ and $\tflow_{t_i}(\xi)$ lie on skew $1$--cells $\alpha_i$ and $\beta_i$, respectively, that are both members of the same equivalence class $W_i$, and that moreover these equivalence classes $W_1,\dotsc,W_n$ are all distinct. In particular we note that $n\leq \abs{\mathcal{W}}$.

Now, the trapezoidal cell structure on $X$ was explicitly constructed so that every flowline must hit a skew $1$--cell of $X$ (and thus also of $Y$) within time $2$; see \cite[\S4.4]{DKL}. Therefore for each $1 \leq i < n$ we have $s_{i+1} - t_{i} \leq 2$ by definition of $s_{i+1}$. Furthermore, since the flowline from $\tflow_{t_i}(\xi)$ to $\tflow_{s_{i+1}}(\xi)$ is disjoint from all other skew $1$--cells, it must be that $\beta_i$ and $\alpha_{i+1}$ lie along the bottom and top arcs $e_-(\widetilde T)$ and $e_+(\widetilde T)$, respectively, for some trapezoidal $2$--cell $\widetilde T$ of $\widetilde Y$. Lemma~\ref{L:short_flowlines_have_bounded_height} therefore implies $\abs{h_z(\alpha_{i+1}) - h_z(\beta_i)}\leq L$ for all $1\leq i < n$.

On the other hand, since $\alpha_i$ and $\beta_i$ are in the same equivalence class, Lemma~\ref{L:positive_on_flowlines} implies that $h_z(\beta_i) - h_z(\alpha_i) \geq \kappa(t_i - s_i - 1)$. Therefore we may conclude
\begin{align*}
h_z(\sigma') - h_z(\sigma)
&= \left(\sum_{i=1}^nh_z(\beta_i) - h_z(\alpha_i)\right) + \left(\sum_{i=1}^{n-1} h_z(\alpha_{i+1})- h_z(\beta_i)\right)\\
&\geq \left(\sum_{i=1}^n \kappa(t_i - s_i - 1)\right)  - L(n-1)\\
&= \kappa\left(t - \sum_{i=1}^{n-1} s_{i+1} - t_i\right) - \kappa n - L(n-1)\\
&\geq \kappa t - (3\kappa + L)\abs{\mathcal{W}}\qedhere
\end{align*}
\end{proof}

This ``linear ascent'' is a crucial feature of classes in the Fried cone.  For an integral class $u\in \D$ with cover $X_u\to X$ corresponding to $\ker(u)$, this shows that lifting any biinfinite flowline to $X_u$ yields a path that exits both ends of $X_u$ (compare \cite[Theorem A]{FriedS} and \cite[Theorem 1.1]{Wang}).

\subsection{Depth}
\label{S:depth}

In order to construct cross sections dual to certain cohomology classes, we will need to represent those classes by cocycles that enjoy certain properties. Here we introduce the first of these properties and show that classes in the cone $\D$ have such representatives.

A biinfinite flowline in $\tX$ is a path $\tilde{\gamma}\colon \R\to \tX$ such that for all $t\in \R$ and $s\geq 0$ we have $\tflow_s(\tilde\gamma(t)) = \tilde\gamma(s+t)$. In this case we say that the set $\tilde\gamma((-\infty,0])$ is a \emph{backwards} flowline of $\tilde\gamma(0)$. Notice that every point of $\tX$ has infinitely many backwards flowlines.

For a trapezoidal subdivision $Y$ of $X$ and a vertex $v\in \widetilde{Y}^{(0)}$,  a skew $1$--cell $\sigma$ of $\widetilde{Y}$ is said to be \emph{behind} $v$ if there is a backwards flowline of $v$ that intersects $\sigma$. 

\begin{defn}[Depth]
\label{D:depth}
Let $Y$ be a trapezoidal subdivision of $X$ and let $z\in Z^1(Y;\R)$ be a cellular cocycle. The \emph{$z$--depth} of a vertex $v\in \widetilde{Y}^{(0)}$ is defined to be the number of $1$--cells $\sigma$ of $\widetilde{Y}$ for which there exists some biinfinite flowline $\tilde{\gamma}$ through $v = \tilde{\gamma}(0)$ and a time $t < 0$ such that
\begin{itemize}
\item  $\tilde{\gamma}(t)\in \sigma$, and
\item $h_z(v)\leq h_z(\sigma')$ for every skew $1$--cell $\sigma'$ of $\widetilde{Y}$ that intersects $\tilde{\gamma}([t,0])$.
\end{itemize}
The equivariance of $\tflow$ and $\tilde{z}$ under deck transformations of $\widetilde{Y}$ implies that vertices of $\widetilde{Y}$ that project to the same point of $Y$ necessarily have the same $z$--depth. Thus we also define the $z$--depth of a vertex of $Y$ to the $z$--depth of any of its lifts. 
\end{defn}

While the $z$--depth of a vertex could in general be infinite, this is never the case when the cocycle $z$ represents a class in $\D$:

\begin{lemma}
\label{L:finite_depth}
Let $Y$ be a trapezoidal subdivision of $X$ and let $z\in Z^1(Y;\R)$ be a cocycle for which $[z]\in \D$. Then for any $M\geq 0$ and any vertex $v$ of $\widetilde{Y}$, there are only finitely many skew $1$--cells $\sigma$ of $\widetilde{Y}$ behind $v$ satisfying
\[h_z(\sigma) \geq h_z(v) - M.\]
In particular, every vertex of $\widetilde{Y}$ has finite $z$--depth.
\end{lemma}
\begin{proof}
Let $\kappa > 0$ and $N\geq 0$ be the constants provided by Lemma~\ref{L:linear_ascent}. Since every flowline of $X$ eventually hits a skew $1$--cell, we can find $s\geq 0$ so that $\tflow_s(v)$ lies in a skew $1$--cell $\sigma'$ of $\widetilde{Y}$. Set $L = \kappa\inv(h_z(\sigma') - h_z(v) + M + N)$.

Suppose that $\tilde{\gamma}$ is any biinfinite flowline through $v =\tilde{\gamma}(0)$ and that $t < 0$ is such that $\tilde{\gamma}(t)$ lies on a skew $1$--cell $\sigma$ of $\widetilde{Y}$ satisfying $h_z(\sigma)\geq h_z(v) - M$. Then by Lemma~\ref{L:linear_ascent} we have
\[h_z(\sigma') - h_z(\sigma) \geq \kappa(s-t) - N,\]
which together with the assumption  on $\sigma$ implies
\[\kappa(s-t) - N \leq h_z(\sigma') - h_z(v) + M.\]
In particular we see that
\[-t \leq  \kappa\inv\big(h_z(\sigma')- h_z(v) +M + N\big)  -s \leq  L.\]

The above paragraph shows that any skew $1$--cell $\sigma$ of $\widetilde{Y}$ behind $v$ that satisfies $h_z(\sigma)\geq h_z(v) - M$ must intersect $\tilde{\gamma}([-L,0])$ for some backwards flowline of $v$. But if we declare two biinfinite flowlines to be equivalent if they agree on $[-L,0]$, then there are are only finitely many equivalence classes of biinfinite flowlines $\tilde{\gamma}$ satisfying $\tilde{\gamma}(0) = v$. Thus, taking the union of $\tilde{\gamma}([-L,0])$ over all flowlines $\tilde{\gamma}$ satisfying $\tilde{\gamma}(0) = v$, we see that there are only finitely many skew $1$--cells intersecting this set.
\end{proof}

Once one knows that every vertex has finite depth, it is straightforward to build a refined cocycle for which every vertex in fact has zero depth. 
We recall here that a \emph{refinement} of a cellular cocycle $z\in Z^1(Y;\R)$ to a subdivision $Y'$ of $Y$ is a cocycle $z'\in Z^1(Y';\R)$ such that $z'(\sigma) = z(\sigma)$ for every $1$--cell $\sigma$ of $Y$, where in the expression $z'(\sigma)$ we regard $\sigma$ as a $1$--chain in $Y'$.

\begin{lemma}[A shallow refinement]
\label{L:shallow_refinement}
Let $Y$ be a trapezoidal subdivision of $X$ and let $z\in Z^1(Y;\R)$ be a vertically positive cocycle for which $[z]\in \D$. Then there exists a trapezoidal subdivision $Y'$ of $Y$ and a vertically positive refinement $z'\in Z^1(Y';\R)$ of $z$ such that every vertex of $Y'$ has $z'$--depth equal to zero.
\end{lemma}
\begin{proof}
Let $K$ denote the largest depth of any vertex in $Y$, and choose a vertex $\tilde v$ of $\widetilde{Y}$ with $z$--depth equal to $K$. We may assume $K \geq 1$. Lemma~\ref{L:finite_depth} implies that the set
\[\{h_z(\tilde v) - h_z(\tilde \sigma) \mid \text{$\tilde \sigma$ is a skew $1$--cell behind $\tilde v$}\}\]
is discrete. Therefore we may choose $\delta > 0$ such that every skew $1$--cell $\tilde \sigma$ behind $\tilde v$ satisfying $h_z(\tilde\sigma) < h_z(\tilde v)$ also satisfies $h_z(\tilde \sigma) < h_z(\tilde v) - \delta$.

Since $\tilde v$ has $z$--depth equal to $K\geq 1$, Definition~\ref{D:depth} implies that there exists biinfinite flowline $\tilde{\gamma}$ with $\tilde{\gamma}(0) = \tilde v$, a time $t< 0$, and a skew $1$--cell $\tilde\sigma$ of $\widetilde{Y}$ satisfying $h_z(\tilde v)\leq h_z(\tilde\sigma)$ such that $\tilde{\gamma}(t)\in \tilde\sigma$ and such that $\tilde{\gamma}((t,0))$ is disjoint from all skew $1$--cells of $Y$. Letting $\widetilde T$ denote the trapezoid of $\widetilde{Y}$ for which $e_-(\widetilde{T}) = \tilde\sigma$, it follows that $\tilde v\in e_+(\widetilde{T})$. Now, if the arc $\tilde{\gamma}([t,0])$ were equal to $\ell_\pm(\widetilde{T})$ and thus comprised of vertical $1$--cells, then the vertical positivity of $z$ would imply $h_z(\tilde\sigma) < h_z(\tilde v)$. As this is not the case, the arc $\tilde{\gamma}([t,0])$ must therefore traverse the interior of $\widetilde{T}$.

Recalling that $p\colon \widetilde{Y}\to Y$ is the covering map, let $T = p(\widetilde T)$ and $\sigma = p(\tilde \sigma)$. It follows that $p(\tilde v)$ lies in the interior of $e_+(T)$ and that $w = p(\tilde{\gamma}(t))$ is an interior point of $e_-(T) = \sigma$. We now subdivide $\sigma$ into two skew $1$-cells $\alpha_1,\alpha_2$ by adding a vertex at $w$ and subdivide $T$ into two trapezoids $T_1$ and $T_2$ by adding the vertical $1$--cell $\beta = p(\tilde{\gamma}([t,0]))$. Call the resulting subdivision $Y'$ and let $\tilde w$ and $\tilde \alpha_i$ respectively denote the unique lifts of $w$ and $\alpha_i$ that are contained in $\widetilde{T}$. We refine $z$ to a cocycle $z'\in Z^1(Y';\R)$ by setting $z'(\beta) = \delta$ and defining $z'(\alpha_i)$ subject to the cocycle condition for $T_i$. Notice that $z'$ is again vertically positive.

Let us now consider the $z'$--depth of vertices in $Y'$. First recall that $h_z(\tilde v)\leq h_z(\tilde\sigma) = \min\{h_z(t(\tilde\sigma)),h_z(o(\tilde\sigma))\}$. Since $h_{z'}(\tilde w) =  h_{z'}(\tilde v) - \delta$ by construction, it follows that $h_{z'}(\tilde w) < h_z(t(\tilde\sigma)),h_z(o(\tilde\sigma))$. Therefore $h_{z'}(\tilde\alpha_i) < h_z(\tilde\sigma)$ for $i=1,2$, and the same holds for all equivariant translates of $\tilde\alpha_1,\tilde\alpha_2,\tilde\sigma$. In particular, we see that whenever $\tilde\alpha_i$ is behind some vertex $\tilde{v}_0$ of $\widetilde{Y}$ satisfying $h_z(\tilde\sigma) < h_z(\tilde{v}_0)$, then we also have $h_{z'}(\tilde\alpha_i) < h_{z'}(\tilde{v}_0)$ for $i=1,2$. This shows that the $z'$--depth of every vertex of $\widetilde{Y}$ is bounded above by its $z$--depth. Moreover, since $h_{z'}(\tilde \alpha_i) = h_{z'}(\tilde v) - \delta$ we see that neither of the new skew $1$--cells $\tilde\alpha_i$ contribute to the $z'$--depth of $\tilde v$, whereas $\tilde\sigma$ did contribute to the $z$--depth of $\tilde v$. Thus the $z'$--depth of $\tilde v$ is strictly less than its $z$--depth.

It remains to analyze the $z'$--depth of the new vertex $w$. Notice that the set of skew $1$--cells of $\widetilde{Y}'$ behind $\tilde w$ is a subset of the ones behind $\tilde v$. Furthermore any skew $1$--cell $\tilde\sigma_0$ of $\widetilde{Y}'$ behind $\tilde v$ satisfying $h_{z'}(\tilde\sigma_0) < h_{z'}(\tilde v)$ also satisfies $h_{z'}(\tilde\sigma_0) < h_{z'}(\tilde w) = h_{z'}(\tilde v) - \delta$ by the choice of $\delta$. Thus the set of skew $1$--cells of $\widetilde{Y}'$ counting towards the $z'$--depth of $\tilde w$ is a subset of those counting towards the $z'$--depth of $\tilde v$. This shows that the $z'$--depth of $\tilde w$ bounded by the $z'$--depth of $\tilde v$, which we have seen is at most $K-1$.

Recall that $K$ was chosen to be the maximal $z$--depth of any vertex in $Y$. We have constructed a subdivision $Y'$ and cocycle refinement $z'\in Z^1(Y';\R)$ so that no vertex has $z'$--depth more than $K$ and so that the number of vertices in $Y'$ with $z'$--depth equal to $K$ is strictly less than the corresponding number of vertices in $Y$. By repeating this recursively, we eventually arrive at the desired subdivision in which every vertex has depth zero.
\end{proof}

\subsection{Consistent signage and the refinement procedure}
\label{S:signage}

We now introduce the second property we require of our cocycles in order to construct dual cross sections. Recall that if $T$ is a trapezoidal $2$--cell in a trapezoidal subdivision $Y$ of $X$, then we may write its top arc as a cellular $1$--chain $e_+(T) = \zeta(T)(\sigma_1 + \dotsb +\sigma_k)$ where each $\sigma_i$ is a (positively oriented) skew $1$--cell in $Y$.

\begin{defn}[Consistently signed]
\label{D:consistently_signed}
Let $Y$ be a trapezoidal subdivision of $X$ and let $z\in Z^1(Y;\R)$ be a cocycle. We say that $z$ is \emph{consistently signed} if it is vertically positive and for every trapezoid $T$ of $Y$ with top arc $e_+(T) = \zeta(T)(\sigma_1 + \dotsb + \sigma_k)$, all of the numbers $z(\sigma_i)$ are nonzero and have the same sign.
\end{defn}

Before showing how to arrange for a cocycle to be consistently signed, we first recall the subdivision procedure introduced in \S6.4 of \cite{DKL}. Briefly, if $Y$ is a trapezoidal subdivision of $X$, then following \cite[Definition 6.8]{DKL} the \emph{standard subdivision} $\widehat{Y}$ of $Y$ is obtained as follows: For each trapezoid $T$ of $Y$ whose top arc $e_+(T) = \zeta(T)(\sigma_1 + \dotsb + \sigma_k)$ contains at least $k\geq 2$ skew $1$--cells, we subdivide the bottom arc $e_-(T)$ (which is a single skew $1$--cell) into $k$ skew $1$--cells that flow homeomorphically onto the skew $1$--cells $\sigma_1,\dotsc,\sigma_k$ comprising $e_+(T)$, and we add $k-1$ vertical $1$--cells through $T$ connecting the newly added vertices on $e_-(T)$ to the already existing vertices along $e_+(T)$. (Note that the first step of the standard subdivision procedure, in which all invariant bands are subdivided, has no effect here since $f$ is an expanding irreducible train track map; see Remark 6.7 of \cite{DKL}.)

\begin{lemma}
\label{L:can_sign_consistently}
Let $Y$ be a trapezoidal subdivision of $X$ and let $z\in Z^1(Y;\R)$ be a vertically positive cocycle such that every vertex of $v$ has $z$--depth equal to zero. Then there exists a trapezoidal subdivision $Y'$ of $Y$ and a consistently signed cocycle $z'\in Z^1(Y';\R)$ that refines $z$.
\end{lemma}
\begin{proof}
In fact we will take $Y'$ to be the standard subdivision $\widehat{Y}$. Let $T'$ be any trapezoidal $2$--cell of $Y'$. Then, by nature of the standard subdivision procedure, the top arc $\zeta(T')e_+(T')$ of $T'$ is equal, when considered as a $1$--chain in $Y$, to a single skew $1$--cell $\sigma$ of $Y$. This skew $1$--cell $\sigma$ is in turn equal to $e_-(T)$ for a unique trapezoid $T$ of $Y$. 
Writing the top arc of $T$ as a concatenation of skew $1$--cells of $Y$, we let $v_0,v_1,\ldots,v_k$ be the vertices of $Y$ along $e_+(T)$ (traversed in that order with the $T$--orientation), so that $v_0$ is the top endpoint of $\ell_-(T)$ and $v_k$ is the top endpoint of $\ell_+(T)$. Then in the subdivision procedure, new vertices $u_1,\dotsc,u_{k-1}$ are placed along $\sigma$ (indexed so that $u_i$ flows onto $v_i$), and for each $1\le i < k$ a vertical $1$--cell $\beta_i$ is added connecting $u_i$ to $v_i$. We also let $u_0=o(\sigma)$ and $u_k = t(\sigma)$ denote the bottom endpoints of $\ell_-(T)$ and $\ell_+(T)$, respectively. Thus in $Y'$, the skew $1$--cell $\sigma$ is subdivided into the $k$ skew $1$--cells $\alpha_i = (u_{i-1},u_i)$ for $1 \leq i \leq k$.

With this notation, writing the top arc of $T'$ as a $1$--chain in $Y'$ we now have $e_+(T') = \zeta(T')(\alpha_1 + \dotsb + \alpha_k)$. Thus we must define $z'$ on the new $1$--cells $\alpha_i$ and $\beta_i$ so that $z'$ refines $z$ and all of the numbers $z(\alpha_1),\dotsc,z(\alpha_k)$ have the same sign. Notice that $\sigma$ is behind each of the vertices $v_0,\dotsc,v_k$. Since every vertex of $Y$ has zero $z$--depth by assumption, we therefore have $h_z(\sigma) < h_z(v_i)$ for each $0\leq i \leq k$. Given any partition of unity $\rho_1 + \dotsb +\rho_k = 1$ with each $\rho_i > 0 $, if we define $z'(\alpha_i) = \rho_iz(\sigma)$ for $1\leq i \leq k$, then the numbers $h_{z'}(u_1),\dotsc, h_{z'}(u_{k-1})$ will interpolate between $h_{z'}(u_0) = h_z(o(\sigma))$ and $h_{z'}(u_k) = h_z(t(\sigma))$. Therefore, since each $v_i$ satisfies $h_z(\sigma) < h_z(v_i)$, we may choose the partition $\rho_1+\dotsb+\rho_k=1$ so $h_{z'}(u_i) < h_{z}(v_i)$ for all $1\leq i < k$ (in fact, we could even arrange to have $h_{z'}(u_i) < \min\{h_z(v_0),\dotsc,h_z(v_{k})\}$ for each $1\leq i < k$). Defining each $z'(\alpha_i)$ in this way, we then set $z'(\beta) = h_{z}(v_i) - h_{z'}(u_i) > 0$ for $1\leq i < k$ and leave the value of $z$ unchanged on the vertical $1$--cells of $\ell_\pm(T)$. This ensures that $z'$ satisfies the cocycle condition and that it is moreover a vertically positive refinement of $z$. 
\end{proof}

In \S6.5 of \cite{DKL} we also introduced an accompanying \emph{standard refinement} procedure for positive cocycles.
However, the reader may check that Definitions 6.10 (skew ratios) and 6.11 (standard refinement) of \cite{DKL} make sense verbatim if each occurrence of the word `positive' is replaced with `consistently signed'. The proofs of Lemmas 6.12--6.13 of \cite{DKL} also go through verbatim in this more general context to give the following.

\begin{lemma}
\label{L:skews_tend_to_zero}
Let $Y_0$ be a trapezoidal subdivision of $X$ and $z_0\in Z^1(Y_0;\R)$ a consistently signed cocycle. Consider the sequence of standard subdivisions $Y_0,Y_1,\dotsc$ with corresponding consistently signed cocycle refinements defined recursively by $Y_{n+1} = \widehat{Y_n}$ and $z_{n+1} = \widehat{z_n}$. Then the numbers
\[S_n =\max\big\{ \abs{z_n(\sigma)} : \text{ $\sigma$ is a skew $1$--cell of $Y_n$}\big\}\]
tend to zero as $n\to \infty$.
\end{lemma}

\subsection{Constructing cross sections}
\label{S:constructing_sections}

Finally, we construct cross sections dual to all integral classes $u\in \D$. For the proof, we first recall from \cite[Definition 6.14]{DKL} that a trapezoidal $2$--cell $T$ is said to be \emph{unconstrained}
by a cocycle $z$ if
\[\max\big\{0,\,\,z(e_-(T))\big\} \leq \min\big\{z(\ell_-(T)),\,\, z(\ell_-(T)) + z(e_+(T))\big\}.\]
By the cocycle condition, this is equivalent to $\abs{z(e_-(T))} \leq z(\ell_\pm(T))$ when $z$ is vertically positive. We also recall that a trapezoidal $2$--cell is said to be \emph{degenerate} if its right
side $\ell_+(T)$ consists of a single vertex (as in Figure~8 of \cite{DKL}).

\begin{proposition}
\label{P:dual_sections}
Given a primitive integral class $u\in \D$, there exists a flow-regular map $\fib_u \colon X \to \sone$ with $(\fib_u)_* = u \colon G \to \pi_1(\sone) = \Z$ and a fiber $\Theta_u = \fib_u^{-1}(y_0)$ for some $y_0 \in S^1$ such that
\begin{enumerate}
\item $\Theta_u$ is a finite, connected topological graph such that $\iota_*(\pi_1(\Theta_u))\le \ker(u)$, where $\iota_*$ is the homomorphism induced by the inclusion $\iota\colon\Theta_u\hookrightarrow X$,
\item $\Theta_u$ is a section of $\flow$ dual to $u$, and
\item $\Theta_u$ is $\mathcal F$--compatible (and so the first return map $f_u\colon \Theta_u\to\Theta_u$ is an expanding irreducible train track map by Proposition~\ref{P:first_return_exp_irred_tt}).
\end{enumerate}
\end{proposition}

\begin{remark}
\label{R:non_injectivity}
We note that the sections produced by Proposition~\ref{P:dual_sections} will generally not have all the properties ensured by Theorem~\ref{T:DKL_BC}. Namely, for $u\in\A$ the fiber $\Theta_u$, as provided by Theorem~\ref{T:DKL_BC}, will $\pi_1$--inject into $X$ and the first return map will be a homotopy equivalence. However, both of these properties will fail whenever $u\in \D$ lies outside of $\bar{\A}$.
\end{remark}

\begin{proof}
By Proposition~\ref{P:homological_characterization}, we can represent $u\in \D$ by a vertically positive cocycle $z\in Z^1(\XC,\R)$.
Using Lemma~\ref{L:shallow_refinement}, we may then pass to a trapezoidal subdivision of $\XC$ and a corresponding vertically positive refinement of $z$ with respect to which every vertex has zero depth. By Lemma~\ref{L:can_sign_consistently}, we may pass to a further subdivision $Y_0$ on which there is a refinement $z_0\in Z^1(Y_0;\R)$ of $z$ that is consistently signed. We then consider the sequence $Y_n$ of standard subdivisions with corresponding refinements $z_n$ defined recursively by $Y_{n+1} = \widehat{Y_n}$ and $z_{n+1} = \widehat{z_n}$. 

Let $D_n$ denote the set of degenerate trapezoids in $Y_n$. Notice that the cardinality of $D_n$ is constant independent of $n$. Indeed, when a degenerate trapezoid $T$ of $Y_n$ is subdivided, exactly one of the resulting trapezoids $T'\subset T$ is degenerate, and its degenerate side $\ell_+(T')$ is equal to $\ell_+(T)$. Thus if we define
\[U_N = \bigcup_{T\in D_n} \overline{T}\qquad\text{and}\qquad V_n = \{\ell_+(T) \mid T\in D_n\} \]
to be the union of degenerate trapezoids and the set of ``degenerate sides'' of trapezoids in $Y_n$, respectively,  then all of the the sets $V_0,V_1,\dotsc$ are equal. Furthermore, the infinite intersection $\cap_n U_n$ is exactly $V_0$. This follows from the fact that the union $\cup_n Y_n^{(0)}$ is dense in every skew $1$--cell of $Y_0$, which is a consequence $f$ being an expanding irreducible train track map. Therefore, we may choose an index $N$ such that all trapezoids in $D_N$ have disjoint closure. Set $U = U_N$

Let $M>0$ denote the minimum of $z_N(\sigma)$ over all vertical $1$--cells of $Y_N$. By Lemma~\ref{L:skews_tend_to_zero}, we may choose $n \gg N$ so that every skew $1$--cell $\sigma$ of $Y_n$ has $\abs{z_n(\sigma)} < M$. Furthermore, as in the proof of Proposition 6.15 of \cite{DKL}, we additionally have $z_n(\beta) \geq M$ for every vertical $1$--cell $\beta$ of $Y_n$ that is not contained in $U$. Therefore, every trapezoidal $2$--cell $T$ of $Y_n$ that is not contained in $U$ satisfies $\abs{z_n(e_-(T))}\leq z_n(\ell_\pm(T))$ and is consequently unconstrained by $z_n$ 

Set $Y = Y_n$ and by an abuse of notation let $z = z_n$. We now use this trapezoidal subdivision to construct the desired map $\fib_u\colon Y\to \sone$. The construction proceeds largely as in the proof of \cite[Lemma 6.16]{DKL}: 
Choose a maximal tree $Q \subset Y^{(1)}$ and fix a vertex $v_0 \in Q$ to serve as a basepoint. Define a map $\hat \fib_u\colon Q \to \R$ so that $\hat \fib_u(v_0) = 0$ and so that for every $1$--cell $e$ in $Q$ the restriction to $e$ is a diffeomorphism onto its image and so that 
\[ \hat \fib_u(t(e)) - \hat \fib_u(o(e)) = z(e). \]
While this equation may fail for $1$--cells not contained in $Q$, since $z$ represents an integral cohomology class we will still have 
\[ \hat \fib_u(t(e)) - \hat \fib_u(o(e)) - z(e) \in \Z \]
for such $1$--cells. Therefore, the composition of $\hat \fib_u$ with $\pi\colon\R\to\sone$ may be extended to a map $\fib_u'\colon Y^{(1)} \to \sone$  with the property that for any $1$--cell $e$, any lift $\widetilde \fib_u'|_e\colon  e \to \mathbb R$ is an injective local diffeomorphism for which
\[ \widetilde \fib_u'|_e(t(e)) - \widetilde \fib_u'|_e(o(e)) = z(e). \]

Since the trapezoids contained in $U$ may not be unconstrained by $z$, we must take care to ensure that $\fib'_u$ can be extended to a flow-regular map. Let $T$ be a trapezoidal $2$--cell contained in $U$. Flowing the bottom arc of $T$ onto the top using $\flow$  gives a homeomorphism $h_T\colon e_-(T)\to e_+(T)$ which is orientation preserving with respect to the $T$--orientation on $e_\pm(T)$. Consider any lift $\widetilde \fib_u'\vert_{\partial T}$ of the restriction of $\fib_u'$ to $\partial T$. Since $z$ is consistently signed, this maps the arcs $e_\pm(T)$ homeomorphically onto intervals $I_\pm = \widetilde \fib_u'\vert_{\partial T}(e_\pm(T))\subset\R$ which we equip with orientations so that these identifications $e_\pm(T)\to I_\pm$ are orientation preserving. Let $A\colon I_+\to I_-$ be the unique orientation preserving (with respect to the above orientations) affine homeomorphism from $I_+$ to $I_-$. We now redefine $\fib_u'$ on the $1$--cell $e_-(T)$ by declaring
\[\fib_u'\vert_{e_-(T)} = \pi \circ A\circ \widetilde \fib_u'\vert_{\partial T} \circ h_T.\]
Notice that this agrees with the original value of $\fib_u'$ on the endpoints of $e_-(T)$. Since the interior of $e_-(T)$ is disjoint from the closure of every other trapezoid contained in $U$ (this follows from the choice of $N)$ we may freely make these adjustments on the bottom arcs of all trapezoids in $U$ without conflict. 
Let us denote the resulting (adjusted) map by $\fib_u\colon Y^{(1)}\to \sone$. Then for any trapezoid $T$ contained in $U$ and any lift $\widetilde \fib_u\vert_{\partial T}\colon \partial T \to \R$, we have
\begin{equation}
\label{eqn:transverse_condition}
\widetilde \fib_u\vert_{\partial T} (\xi) < \widetilde \fib_u\vert_{\partial T}(h_T(\xi))
\end{equation}
for every interior point $\xi\in e_-(T)$. In fact, since all other trapezoids are unconstrained by $z$, this property holds for every trapezoid of $Y$.

Choose a point $y_0\in \sone \setminus \fib_u(Y^{(0)})$. Since the vertex leaves of $\mathcal{F}$ are dense in $X$, after adjusting $\fib_u$ by a (small) homotopy rel $Y^{(0)}$, maintaining all of the properties above, we may assume that $(\fib_u)\inv(y_0)\subset Y^{(1)}$ consists of a finite set of points lying on vertex leaves of $\mathcal{F}$.

We now extend $\fib_u$ over the $2$--cells of $Y$. Consider any trapezoid $T$ and a lift $\widetilde \fib_u\vert_{\partial T} \colon \partial T \to \R$ of $\fib_u\vert_{\partial T}$ to the universal cover. As in the proof of \cite[Lemma 6.16]{DKL}, we may identify $\overline T$ with the Euclidean trapezoid provided by \cite[Proposition 4.20]{DKL} and then extend $\widetilde \fib_u\vert_{\partial T}$ over $\overline{T}$ in such a way that the fibers are straight lines (Euclidean) lines. Equation \eqref{eqn:transverse_condition} then ensures that each fiber has nonvertical slope and is thus transverse to the foliation of $\overline T$ by vertical flowlines. We then extend $\fib_u$ to the interior of $T$ by declaring it to equal $\pi\circ \widetilde \fib_u\vert_{\overline T}$ on $\overline T$. Extending in this way over every $2$--cell of $Y$, we obtain a map $\fib_u\colon Y\to \sone$. 

Since the fibers of $\fib_u$ are transverse to the foliation of $X$ by flowlines, we see that $\fib_u$ is flow-regular on the interior of each $2$--cell. However, since $\fib_u$ was extended to the various $2$--cells independently, it need not be flow-regular on a neighborhood of $Y^{(1)}$. Nevertheless, adjusting by a homotopy rel $Y^{(1)}$, we may smooth out and ``straighten'' the fibers of $\fib_u$ in a neighborhood of $Y^{(1)}$ so that $\fib_u\colon Y\to \sone$ is flow-regular on all of $Y$.  

The map $\fib_u\colon Y\to\sone$ now satisfies all of the conclusions of the proposition. By construction the fiber $\Theta_u = \fib_u\inv(y_0)\subset X$ is a connected topological graph which, since $\fib_u$ is flow-regular, is necessarily a cross section dual to $(\fib_u)_*$. Indeed, Lemma~\ref{L:linear_ascent} explicitly shows that every flowline hits $\Theta_u$ infinitely often, and Proposition~\ref{P:connected_iff_primitive} ensures that $\Theta_u$ is connected. Moreover, our choice of $y_0\in \sone$ guarantees that $\Theta_u\cap Y^{(1)}$ is contained in vertex leaves of $\mathcal F$ and thus that $\Theta_u$ is $\mathcal{F}$--compatible. Finally, the desired relationship $(\fib_u)_*=u$ follows from the fact that for each $1$--cell $e$ of $Y$, any lift of the restriction $\fib_u\vert_e$ maps $e$ homeomorphically onto an interval whose endpoints differ by $z(e)$. This in turn implies that $\iota_*(\pi_1(\Theta_u))\leq \ker(u)$.
\end{proof}

\section{Local boundedness of the stretch function}
\label{S:bounding_stretch}
Our goal in this appendix to prove Proposition~\ref{prop:stretch_is_locally_bounded} and thereby show that the stretch function $\Lambda\colon \QBNS(G)\to\R_+$ defined in \S\ref{sec:BNS_stretch} is locally bounded. To this end, let $u_1 \in H^1(G;\R)$ be a primitive integral point that lies in $\BNS(G)$. Letting $G = B \ast_{\phi_1}$ be an HNN-extension compatible with $u_1$ (that is, with $B\le \ker(u_1)$ a finite rank free group and $\phi_1 = \fee_{u_1}\vert_{B} \colon B \to B$ an injective endomorphism), we then have $\Lambda(u_1) = \lambda(\phi_1)$ by definition.

\subsection{An expanding representative}

Let $f_1 \colon \Rose \to \Rose$ be a regular, irreducible expanding graph map on a rose $\Rose$ with an identification $B\cong \pi_1(\Rose)$ so that $(f_1)_* = \phi_1$.   To simplify the discussion, we assume (as we may) that the transition matrix is not only irreducible, but is in fact positive.  Let $v_0$ be the wedge point of $\Rose$ and for $K \geq 0$ set
\[ \mathcal V_k = \bigcup_{k=1}^K f_1^{-k}(v_0) \quad \mbox{ and } \quad \mathcal V_\infty = \bigcup_{k \geq 0} \mathcal V_k.\]
Since $f_1$ has a positive transition matrix, $\mathcal V_K$ nontrivially intersects in the interior of every $1$--cell of $\Rose$ for every $K \geq 1$.

\begin{lemma} \label{L:subdivided injectivity}
For any $K \geq 0$, $0 \leq k \leq K+1$, and any arc $e$ in the complement of $\mathcal V_K$,  the restriction of $f_1^k$ to $e$ is locally injective.
\end{lemma}
\begin{proof}
The map $f_1$ is locally injective on the interior of every edge of $\Rose$ since $f_1$ is regular.  It follows that local injectivity of $f_1^2$ can fail only at points $f_1^{-1}(v_0) = \mathcal V_1$;  hence $f_1^2$ is locally injective on any arc in the complement of $\mathcal V_1$.  As compositions of locally injective maps are locally injective, it follows by induction that failure of locally injectivity of $f_1^k$ can only occur at points of $\mathcal V_{k-1}$, and there are no such points in $e$.
\end{proof}

For technical reasons which will be clear later, we give each $1$--cell $e$ of $\Rose$ an auxiliary linear structure determined by a characteristic map $[0,1] \to e$ which maps the rational points bijectively onto $\mathcal V_\infty \cap e$.  This is possible because both are countable dense sets (for $\mathcal V_\infty$ this is because $f_1$ is expanding and irreducible).  Furthermore, we assume, as we may, that the new linear structure and the original linear structure define the same smooth structure on the edges.  It follows that for any arc $\alpha$ of an edge $e$ with endpoints in $\mathcal V_\infty$, a linear map $[0,1] \to \alpha$ also sends $\Q \cap [0,1]$ onto $\mathcal V_\infty \cap \alpha$ (since such a linear map differs from the restriction of $[0,1] \to e$ by a rationally defined affine map of $\R \to \R$).

\subsection{The mapping torus and its cell structure.}

Let $Y$ be the mapping torus of $f_1$, and let $\psi$ denote the suspension semi-flow.  By van Kampen's Theorem, $\pi_1(Y) =  G$.   For a map $\fib \colon Y \to S^1 = \R/\Z$ or to $\R$, if the composition of any flowline $x \mapsto \psi_s(x)$ with $\fib$ is smooth, we will say that $\fib$ is {\em $\psi$--smooth}, and the map $\psi(\fib) \colon Y \to \R$, defined as the derivative of $\eta$ with respect to the flow-parameter will be called the {\em $\psi$--derivative}.

For any $K \geq 1$, we define a cell structure on $Y$, denoted $Y_K$, as follows.   We embed $\Rose$ into $Y$ as $\Rose \times \{0\}$ in $\Rose \times [0,1]$.  The $0$--cells are the image of $\mathcal V_K$ by this embedding $\Rose \to Y$.  The $1$--cells are the edges of the resulting subdivision of $\Rose$ (viewed in $Y$) together with the following arcs of flowlines
\[ \left\{ \bigcup_{0 \leq s \leq 1} \psi_s(v) \right\}_{v \in \mathcal V_K}.\]
We refer to these two types of $1$--cells as {\em skew} and {\em vertical}, respectively.  We give the skew $1$--cells the linear structure induced from the auxiliary linear structure on $\Rose$, and the vertical $1$--cells the linear structure for which the flow parameter is linear.  The complement of the $1$--skeleton $Y_K^{(1)}$ is a union of topological disks which we define to be the $2$--cells of $Y_K$.  Whenever we write $Y$ we will assume it is given the cell structure $Y_1$, unless otherwise stated.

\begin{lemma} \label{L:flow local injectivity}
For any $K \geq 1$, let $\alpha \colon (0,1) \to Y_K$ be an arc contained in (the interior of) a $2$--cell, transverse to $\psi$.  Let $H_\alpha \colon (0,1) \times [0,K) \to Y$ be given by $H_\alpha(t,s) = \psi_s(\alpha(t))$.  Then $H_\alpha$ is locally injective.
\end{lemma}
\begin{proof}
The bottom of the $2$--cell is a skew $1$--cell of $Y_K$, and we let $\sigma \colon (0,1) \to \Rose \subset Y_K$ be a characteristic map for this cell. We can similarly define $H_\sigma \colon (0,1) \times [0,K+1) \to Y$, which is locally injective by Lemma \ref{L:subdivided injectivity}. Note that $H_\alpha$ ``factors through'' $H_\sigma$ in the sense that there is an embedding $h \colon (0,1) \times [0,K) \to (0,1) \to [0,K+1)$ so that $H_\sigma \circ h = H_\alpha$.  Since $H_\sigma$ is locally injective, so is $H_\alpha$.
\end{proof}

For any $K \geq 1$, the boundary of each $2$--cell of $Y_K$ is a union of two vertical $1$--cells, a single skew $1$--cell on the ``bottom'', and a finite union of skew $1$--cells on the ``top''.  Note that the interior of the bottom and top are each contained in a single one of the original edges of $\Rose$, and with respect to the auxiliary linear structure, the flow identifies the bottom and top by a diffeomorphism.  We use this, together with the flow parameter, to define a smooth structure on each $2$--cell of $Y_K$, viewed as a product.

\subsection{Maps to $S^1$.}

Let $\fib_{1,0} \colon Y \to S^1 = \R/\Z$ be the projection onto the second coordinate, plus $1/2$.  Then the preimage of $0$ is identified with $\Rose \times \{1/2 \} \cong \Rose$.  We perturb $\fib_{1,0}$ outside $\Rose \times \{1/2\}$ to a map $\fib_1$ so that
\begin{itemize}
\item the restriction to the $1$--cells of $\Rose$ is affine and nonconstant and sends vertices to rational points, i.e.~points of  $\Q/\Z$
\item $\fib_1$ is a local diffeomorphism on each flow-line with $\psi$--derivative in the interval $(.9,1.1)$, and
\item $\fib_1$ is smooth on each $2$--cell of $Y_2$.
\end{itemize}
We use $\fib_1$ to prescribe a preferred orientation on every edge by requiring $\fib_1$ to be orientation preserving.  The map $\fib_1$ determines a rational cellular $1$--cocycle $z_1\in Z^1(Y;\Q)$ that is positive on each $1$--cell of $Y$.

Next, let $z_2,\ldots, z_b$ be rational cellular cocycles, so that $[z_1],\ldots,[z_b]$ is a basis for the integral lattice in $H^1(G;\R) = H^1(Y;\R)$.  For each $z_i$, $i=2,\ldots,b$, define a map $\fib_i \colon Y \to S^1$ representing $z_i$ so that 
\begin{itemize}
\item the restriction to each $1$--cell is affine, and sends vertices to rational points,
\item $\fib_i$ is $\psi$--smooth with $\psi$--derivative bounded above and below, and
\item $\fib_i$ is smooth on each $2$--cell of $Y_2$.
\end{itemize}

Now for any integer $b$--tuple ${\bf a} = (a_1,\ldots,a_b)$ we have a map
\begin{equation} \label{E:linear combination map} \fib_{\bf a} = \sum_{i=1}^b a_i \fib_i \colon Y \to S^1,
\end{equation}
and this determines a $1$--cocycle $z_{\bf a}$.  Let $r > 0$ be such that if $a_1 > r \sum_{i=2}^b |a_i|$, then $z_{\bf a}$ is positive and the $\psi$--derivative of $\psi(\fib_{\bf a}) = \sum_{i} a_i \psi(\fib_i)$ is bounded below by $\frac{a_1}{2}$.  Consequently, $\fib_{\bf a}$ is flow regular and so 
\[ \Upsilon_r = \left\{ \left[ \sum_{i=1}^b a_i z_i \right] \in H^1(Y;\R) \mid a_i \in \R, a_1 > r \sum_{i=2}^b |a_i| \right\} \]
is a cone in $H^1(Y;\R)$ for which every integral element is represented by the flow regular map $\fib_{\bf a}$ as in \eqref{E:linear combination map} that is linear on the $1$--cells  of $Y_2$, smooth on the $2$--cells of $Y_2$, and defines a positive cocycle $z_{\bf a}$.  Furthermore, we take $r$ even larger if necessary so that any integral point in $[z_{\bf a}] \in \Upsilon_r$ with ${\bf a} = (a_1,\ldots,a_b) \neq (1,0,0,\ldots,0)$ has $a_1 \geq 2$.

\begin{lemma} \label{L:Good intersections}
For any integral $b$--tuple ${\bf a}$ with $[z_{\bf a}] \in \Upsilon_r$ and any skew $1$--cell $e$, $\fib_{\bf a}^{-1}(0) \cap e \subset \mathcal V_\infty$. (Where here we use the identification of $\Rose$ with $\Rose\times \{0\}$ to view $\mathcal V_\infty$ inside $Y$.)
\end{lemma}
\begin{proof}
The $1$--cell $e$ is the image of a $1$--cell in the subdivision of $\Rose$ along $\mathcal V_1$.  Because of our choice of linear structure, the characteristic map $\sigma \colon [0,1] \to e$ has $\sigma(\Q \cap [0,1]) \subset \mathcal V_\infty$.  Compose this with $\fib_{\bf a}$ to produce a map $\sigma_{\bf a} = \fib _{\bf a} \circ \sigma$, and choose a lift to $\mathbb R$
\[ \tilde \sigma_{\bf a} \colon [0,1] \to \mathbb R.\]
Because each $\fib_i$ was assumed affine with vertices sent to rational points, it follows that
\[ \tilde \sigma_{\bf a}(x) = cx + d,\]
where $c,d \in \Q$.  Suppose $\fib_{\bf a}(\sigma(x)) = 0$.  Then $cx + d  \in \Z$, and hence $x \in \Q$ so that $\sigma(x) \in \mathcal V_\infty$ as required.
\end{proof}

\subsection{Reparameterized semi-flow and first return maps}

Fix an integral $b$--tuple ${\bf a}$ with $[z_{\bf a}] \in \Upsilon_r$. We reparameterize the semiflow $\psi$ as $\psi^{\bf a}$ so that the composition of a flowline with $\fib_{\bf a}$ is a local isometry to $S^1$. The first return map of $\psi$ (or equivalently of $\psi^{\bf a}$) to $\Theta_{\bf a}:= \fib_{\bf a}^{-1}(0)$ is then the restriction of the time $1$ map, $\psi^{\bf a}_1$, and we denote this by $f_{\bf a} \colon \Theta_{\bf a} \to \Theta_{\bf a}$.

Next, let $Z = \Theta_{\bf a} \cap Y_2^{(1)}$ be the intersection of $\Theta_{\bf a}$ with the $1$--skeleton of $Y_2$, which is a finite set of points.  From Lemma \ref{L:Good intersections}, every point of $Z$ flows forward into $v_0$, and we let $Z_{\bf a}$ be the intersection of $\Theta_{\bf a}$ with the flow lines starting at $Z$.  This is also a finite subset of $\Theta_{\bf a}$, and $f_{\bf a}(Z_{\bf a}) \subset Z_{\bf a}$ by construction.

\begin{proposition}
Given an integral $b$--tuple ${\bf a}$ with $\Theta_{\bf a}$, $Z_{\bf a}$, $\psi^{\bf a}$ and $f_{\bf a}$ be as above.  Then $\Theta_{\bf a}$ is a graph with vertex set $Z_{\bf a}$ and $f_{\bf a} \colon \Theta_{\bf a} \to \Theta_{\bf a}$ is a regular graph map.  Furthermore, $f_{\bf a}$ is expanding and irreducible.
\end{proposition}
\begin{proof}
The complement of $Z_{\bf a}$ in $\Theta_{\bf a}$ is a collection of arcs, each contained in a $2$--cell, transverse to $\psi$ since $f_{\bf a}$ is flow-regular, thus $\Theta_{\bf a}$ is a graph.  We have already observed that $f_{\bf a}(Z_{\bf a}) \subset Z_{\bf a}$.  Next we prove that the restriction of $f_{\bf a}$ to any edge $e$ of $\Theta_{\bf a}$ is locally injective.  If ${\bf a} = (1,0,0,\ldots,0)$, then this $f_{\bf a} = f \colon  \Rose \to \Rose$, with $\Rose$ subdivided at $\mathcal V_2$, and we're done.  So, we assume that ${\bf a} \neq (1,0,0,\ldots,0)$, and hence $a_2 \geq 2$.

Choose a parameterization $\alpha \colon (0,1) \to e$, and consider the two maps
\[ H_\alpha^{\bf a} \colon (0,1) \times [0,1] \to Y \mbox{ and } H_\alpha \colon (0,1) \times [0,2) \to Y \]
given by $H_\alpha^{\bf a}(t,s) = \psi_s^{\bf a}(\alpha(t))$ and $H_\alpha(t,s) = \psi_s(\alpha(t))$.  Because the $\psi$--derivative of $\fib_{\bf a}$ is at least $\frac{a_2}{2} \geq \frac{2}{2} = 1$, it follows that the reparameterization at any point $x$ is defined by a monotone function $h_x \colon [0,\infty) \to [0,\infty)$, so that $\psi_s^{\bf a}(x) = \psi_{h_x(s)}(x)$ with $h_x' \leq 1$.  In particular, there is an embedding $h_\alpha \colon (0,1) \times [0,1] \to (0,1) \times [0,2)$ so that $H_\alpha^{\bf a} = H_\alpha \circ h_\alpha$.   According to Lemma \ref{L:flow local injectivity}, $H_\alpha$ is locally injective, and therefore, so is $H_\alpha^{\bf a}$.  Consequently, $f_{\bf a} = \psi_1^{\bf a}|_{\Theta_{\bf a}}$ is locally injective on $e$, and since $e$ was arbitrary, $f_{\bf a}$ is regular.

The fact that $f_{\bf a}$ is expanding and irreducible follows, just as in \S\ref{sec:first_return_maps}, from the assumption that $f_1 \colon \Rose \to \Rose$ is expanding and irreducible: any arc $\alpha_0$ contained in an edge $e_0$ of $\Theta_{\bf a}$ flows forward by $\psi$ onto a non-degenerate arc $\alpha_1$ contained in an edge $e_1$ of $\Rose$.  This in turn eventually flows forward onto all of $\Rose$ (since $f_1$ is expanding and irreducible), and consequently can be flowed further forward to hit any point of $\Theta_{\bf a}$.
\end{proof}

\subsection{Bounding the stretch factor}
Note that $\Upsilon_r\subset H^1(Y;\R) = H^1(G;\R)$ is an open cone containing $u_1$. For any primitive integral $b$--tuple ${\bf a}$ with $[z_{\bf a}]\in \Upsilon_r$, consider the graph $\Theta_{\bf a}$ and first return map $f_{\bf a}$. By construction $\Theta_{\bf a}$ is a cross section of $\psi$ and is dual to the class $[z_{\bf a}]$. Since $[z_{\bf a}]$ is primitive integral, Proposition~\ref{P:connected_iff_primitive} ensures that $\Theta_{\bf a}$ is connected. Van Kampen's theorem therefore gives $G = \pi_1(\Theta_{\bf a})\ast_{(f_{\bf a})_*}$ and so just as in the train track case (\S\ref{S:endomorphisms}) we may write $G = Q_{\bf a}\ast_{\phi_{\bf a}}$, where $\phi_{\bf a}$ is the descent to the stable quotient $Q_{\bf a}$ of $(f_{\bf a})_* \colon \pi_1(\Theta_{\bf a}) \to \pi_1(\Theta_{\bf a})$. In particular, this proves that $u_{\bf a}=[z_{\bf a}]\in \QBNS(G)$ and that $\Lambda(u_{\bf a}) = \lambda(\phi_{\bf a})$. Furthermore $\log(\Lambda(u_{\bf a}))$ is bounded from above by the log of the Perron-Frobenius eigenvalue of the transition matrix of $f_{\bf a}$, which is the entropy $h(f_{\bf a})$:
\begin{equation} \label{E:bounding stretch} 
\log(\Lambda(u_{\bf a})) = \log(\lambda((f_{\bf a})_*)) \leq h(f_{\bf a}).
\end{equation}
For any integer $k > 0$, $h(f_{k{\bf a}}) = \frac{1}{k} h(f_{\bf a})$, so that ${\bf a} \mapsto h(f_{\bf a})$ naturally extends by degree $-1$ homogeneity to the rational rays as well, and hence \eqref{E:bounding stretch} holds for all rational points in $\Upsilon_r$.

\begin{proposition} For any rational $u = [z_{\bf a}] \in \Upsilon_r \subset H^1(G;\R)$, there is a neighborhood of $U\subset \Upsilon_r$ of $u$ and a constant $R >0$ so that $\Lambda \leq R$ on all rational points of $U$.
\end{proposition}
Since every point of $\QBNS(G)$ is contained in such an open cone neighborhood $\Upsilon_r$, this proves Proposition~\ref{prop:stretch_is_locally_bounded}.
\begin{proof}
From \eqref{E:bounding stretch}, it suffices to prove that the assignment $[z_{\bf a}] \mapsto h(f_{\bf a})$ extends continuously (and hence homogeneously with degree $-1$) to all of $\Upsilon_r$.  For this we can repeat the arguments given in Section 8 of \cite{DKL}, essentially verbatim.  We sketch this here, but refer the reader to \cite{DKL} for the details.

Continuity follows by first passing to the natural extension of $(Y,\psi)$, which we denote $(Y',\Psi)$.   Here $\Psi$ is an honest flow, and there is a projection $\pi \colon Y' \to Y$ which is equivariant with respect to $\R_+$, that is $\pi(\Psi_s(x')) = \psi_s(\pi(x'))$ 
for all $x' \in Y'$  and $s \geq 0$.
The preimage of a section $\Theta_{\bf a}$ is a section $\Theta_{\bf a}' = \pi^{-1}(\Theta_{\bf a})$ of $\Psi$, and $\pi$ conjugates the associated first return map $f_{\bf a}'\colon \Theta_{\bf a}' \to \Theta_{\bf a}'$ to $f_{\bf a}$. Moreover, the entropies are equal $h(f_{\bf a}) = h(f_{\bf a}')$, and so it suffices to prove $[z_{\bf a}]\to h(f_{\bf a}')$ extends continuously.

For this, we proceed as in \cite{DKL}:  Applying the Variational Principal and Abramov's Theorem, we conclude that for any integral class $[z_{\bf a}] \in \Upsilon_r$, 
\[ \frac{1}{h(f_{\bf a}')} = \inf_\mu \frac{\mu_{\bf a}(Y')}{h_\mu(\Psi)},\]
where $\mu$ ranges over all $\Psi$--invariant probability measures on $Y'$, $h_\mu(\Psi)$ is the measure theoretic entropy of $\Psi$ with respect to $\mu$, and $\mu_{\bf a}$ is an auxiliary measure associated to $\mu$ and ${\bf a}$ via a measure $\bar \mu_0$ on the section $\Theta_{\bf a}'$; see \cite{DKL} for details.
The key property here is that given integral classes $[z_{\bf a}],[z_{\bf b}] \in \Upsilon_r$,
\[ \mu_{\bf a} + \mu_{\bf b} = \mu_{\bf a + b}.\]
This follows from the fact that $\fib_{\bf a} + \fib_{\bf b} = \fib_{\bf a + b}$, as in \cite{DKL}. 

For each $\mu$, this shows that $[z_{\bf a}] \mapsto \mu_{\bf a}(Y')$ extends from the integral points of $\Upsilon_r$, to a linear function $\mu_* \colon H^1(G;\R) \to \R$. Consequently,  we obtain a continuous, concave, positive function on all of $\Upsilon_r$ given by
\[ W([z_{\bf a}]) = \inf_\mu \frac{\mu_*([z_{\bf a}])}{h_\mu(\Psi)},\]
where the infimum is over all $\Psi$--invariant probability measures on $Y'$. Thus $[z_{\bf a}]\mapsto 1/W([z_{\bf a}])$ is the desired continuous extension of $[z_{\bf a}]\mapsto h(f_{\bf a})$ to all of $\Upsilon_r$.
\end{proof}

\bibliography{mcpolynomial}{}
\bibliographystyle{alphanum}

\end{document}